\documentclass[12pt]{amsart}
\usepackage{float} 
\usepackage{amsmath,amsthm,amssymb,enumitem,verbatim,stmaryrd,xcolor,microtype,graphicx,aliascnt,needspace}
\usepackage[T1]{fontenc}
\usepackage[utf8]{inputenc}
\usepackage[english]{babel} 
\usepackage[top=3.5cm,bottom=3.55cm,left=3.5cm,right=3.5cm]{geometry}
\usepackage[bookmarksdepth=2,linktoc=page,colorlinks,linkcolor={red!70!black},citecolor={red!80!black},urlcolor={blue!80!black},pdfpagemode=UseOutlines,pdfstartview={fitH}]{hyperref}
\usepackage{tikz}\usetikzlibrary{matrix,arrows,decorations.markings,shapes,calc,intersections,through}\def\centerarc[#1](#2)(#3:#4:#5){\draw[#1] ($(#2)+({#5*cos(#3)},{#5*sin(#3)})$) arc (#3:#4:#5);}

\usepackage{tikz-cd}
\pgfrealjobname{foundations2}
\usepackage[mathcal]{euscript}                               

\usepackage{mathptmx}                                        
\usepackage[colorinlistoftodos,bordercolor=orange,backgroundcolor=orange!20,linecolor=orange,textsize=footnotesize]{todonotes} \makeatletter \providecommand \@dotsep{5} \def\listtodoname{List of Todos} \def\listoftodos{\@starttoc{tdo}\listtodoname} \makeatother 

\allowdisplaybreaks 

\usepackage{etoolbox}
\makeatletter
\patchcmd{\@startsection}{\@afterindenttrue}{\@afterindentfalse}{}{}             
\patchcmd{\part}{\bfseries}{\bfseries\LARGE}{}{}
\patchcmd{\section}{\scshape}{\bfseries}{}{}\renewcommand{\@secnumfont}{\bfseries} 
\patchcmd{\@settitle}{\uppercasenonmath\@title}{\large}{}{}
\patchcmd{\@setauthors}{\MakeUppercase}{}{}{}
\addto{\captionsenglish}{} 
\addto{\captionsenglish}{} 
\addto{\captionsenglish}{} 
\makeatother

\usepackage{fancyhdr}

\addto\extrasenglish{}  
\theoremstyle{plain}
\newtheorem{thm}{Theorem}[section] 
\newaliascnt{lemma}{thm}\newtheorem{lemma}[lemma]{Lemma}\aliascntresetthe{lemma}
\newaliascnt{cor}{thm}\newtheorem{cor}[cor]{Corollary}\aliascntresetthe{cor}
\newaliascnt{prop}{thm}\newtheorem{prop}[prop]{Proposition}\aliascntresetthe{prop}
\newtheorem{thmA}{Theorem} 

\newtheorem*{thm*}{Theorem}
\newtheorem*{lem*}{Lemma}
\newtheorem*{cor*}{Corollary}

\theoremstyle{definition}
\newaliascnt{df}{thm}\newtheorem{df}[df]{Definition}\aliascntresetthe{df}
\newaliascnt{rem}{thm}\newtheorem{rem}[rem]{Remark}\aliascntresetthe{rem}
\newaliascnt{ex}{thm}\newtheorem{ex}[ex]{Example}\aliascntresetthe{ex}
\newaliascnt{conj}{thm}\aliascntresetthe{conj}
\newaliascnt{problem}{thm}\newtheorem{problem}[problem]{Problem}\aliascntresetthe{problem}

\newtheorem*{df*}{Definition}
\newtheorem*{ex*}{Example}
\newtheorem*{rem*}{Remark}

\theoremstyle{remark}

\DeclareRobustCommand{\gobblefour}[5]{}    

\DeclareFontFamily{OT1}{pzc}{}                                
\DeclareFontShape{OT1}{pzc}{m}{it}{<-> s * [1.10] pzcmi7t}{}
\DeclareMathAlphabet{\mathpzc}{OT1}{pzc}{m}{it}
\DeclareSymbolFont{sfoperators}{OT1}{bch}{m}{n} \DeclareSymbolFontAlphabet{\mathsf}{sfoperators} \makeatletter\def\operator@font{\mathgroup\symsfoperators}\makeatother 
\DeclareSymbolFont{cmletters}{OML}{cmm}{m}{it}              
\DeclareSymbolFont{cmsymbols}{OMS}{cmsy}{m}{n}
\DeclareSymbolFont{cmlargesymbols}{OMX}{cmex}{m}{n}
\DeclareMathSymbol{\myjmath}{\mathord}{cmletters}{"7C}     \let\jmath\myjmath 
\DeclareMathSymbol{\myamalg}{\mathbin}{cmsymbols}{"71}     
\DeclareMathSymbol{\mycoprod}{\mathop}{cmlargesymbols}{"60}
\DeclareMathSymbol{\myalpha}{\mathord}{cmletters}{"0B}     \let\alpha\myalpha 
\DeclareMathSymbol{\mybeta}{\mathord}{cmletters}{"0C}      \let\beta\mybeta
\DeclareMathSymbol{\mygamma}{\mathord}{cmletters}{"0D}     \let\gamma\mygamma
\DeclareMathSymbol{\mydelta}{\mathord}{cmletters}{"0E}     \let\delta\mydelta
\DeclareMathSymbol{\myepsilon}{\mathord}{cmletters}{"0F}   \let\epsilon\myepsilon
\DeclareMathSymbol{\myzeta}{\mathord}{cmletters}{"10}      \let\zeta\myzeta
\DeclareMathSymbol{\myeta}{\mathord}{cmletters}{"11}       \let\eta\myeta
\DeclareMathSymbol{\mytheta}{\mathord}{cmletters}{"12}     \let\theta\mytheta
\DeclareMathSymbol{\myiota}{\mathord}{cmletters}{"13}      \let\iota\myiota
\DeclareMathSymbol{\mykappa}{\mathord}{cmletters}{"14}     \let\kappa\mykappa
\DeclareMathSymbol{\mylambda}{\mathord}{cmletters}{"15}    \let\lambda\mylambda
\DeclareMathSymbol{\mymu}{\mathord}{cmletters}{"16}        \let\mu\mymu
\DeclareMathSymbol{\mynu}{\mathord}{cmletters}{"17}        \let\nu\mynu
\DeclareMathSymbol{\myxi}{\mathord}{cmletters}{"18}        \let\xi\myxi
\DeclareMathSymbol{\mypi}{\mathord}{cmletters}{"19}        \let\pi\mypi
\DeclareMathSymbol{\myrho}{\mathord}{cmletters}{"1A}       \let\rho\myrho
\DeclareMathSymbol{\mysigma}{\mathord}{cmletters}{"1B}     \let\sigma\mysigma
\DeclareMathSymbol{\mytau}{\mathord}{cmletters}{"1C}       \let\tau\mytau
\DeclareMathSymbol{\myupsilon}{\mathord}{cmletters}{"1D}   \let\upsilon\myupsilon
\DeclareMathSymbol{\myphi}{\mathord}{cmletters}{"1E}       \let\phi\myphi
\DeclareMathSymbol{\mychi}{\mathord}{cmletters}{"1F}       \let\chi\mychi
\DeclareMathSymbol{\mypsi}{\mathord}{cmletters}{"20}       \let\psi\mypsi
\DeclareMathSymbol{\myomega}{\mathord}{cmletters}{"21}     \let\omega\myomega
\DeclareMathSymbol{\myvarepsilon}{\mathord}{cmletters}{"22}\let\varepsilon\myvarepsilon
\DeclareMathSymbol{\myvartheta}{\mathord}{cmletters}{"23}  \let\vartheta\myvartheta
\DeclareMathSymbol{\myvarpi}{\mathord}{cmletters}{"24}     \let\varpi\myvarpi
\DeclareMathSymbol{\myvarrho}{\mathord}{cmletters}{"25}    \let\varrho\myvarrho
\DeclareMathSymbol{\myvarsigma}{\mathord}{cmletters}{"26}  \let\varsigma\myvarsigma
\DeclareMathSymbol{\myvarphi}{\mathord}{cmletters}{"27}    \let\varphi\myvarphi

\DeclareMathOperator{\Hom}{Hom}
\DeclareMathOperator{\Aut}{Aut}

\DeclareMathOperator{\Sym}{Sym}

\DeclareMathOperator{\sign}{sign}

\DeclareMathOperator{\rk}{rk}
\DeclareMathOperator{\pr}{pr}
\DeclareMathOperator{\eq}{eq}

\DeclareMathOperator{\colim}{colim\,}

\DeclareMathOperator{\Pastures}{{Pastures}}

\DeclareMathOperator{\Emb}{{Emb}}
\DeclareMathOperator{\USL}{{USL}}

\newcommand\B{{\mathbb B}}
\newcommand\C{{\mathbb C}}
\newcommand\D{{\mathbb D}}
\newcommand\E{{\mathbb E}}
\newcommand\F{{\mathbb F}}
\newcommand\G{{\mathbb G}}
\renewcommand\H{{\mathbb H}}

\newcommand\K{{\mathbb K}}

\renewcommand\P{{\mathbb P}}
\newcommand\Q{{\mathbb Q}}
\newcommand\R{{\mathbb R}}
\renewcommand\S{{\mathbb S}}
\newcommand\T{{\mathbb T}}
\newcommand\U{{\mathbb U}}
\newcommand\V{{\mathbb V}}
\newcommand\W{{\mathbb W}}

\newcommand\Z{{\mathbb Z}}
\newcommand\bI{{\mathbf I}}
\newcommand\bJ{{\mathbf J}}
\newcommand\cA{{\mathcal A}}

\newcommand\cC{{\mathcal C}}
\newcommand\cD{{\mathcal D}}
\newcommand\cE{{\mathcal E}}
\newcommand\cF{{\mathcal F}}

\newcommand\cH{{\mathcal H}}

\newcommand\cL{{\mathcal L}}

\newcommand\cN{{\mathcal N}}

\newcommand\cR{{\mathcal R}}
\newcommand\cS{{\mathcal S}}

\newcommand\cX{{\mathcal X}}

\newcommand\Funpm{{\F_1^\pm}}
\newcommand\id{\textup{id}}

\newcommand\octa{{\scalebox{0.7}{$\diamondsuit$}}}

\renewcommand{\min}{\textup{min}}
\newcommand{\fundtype}{\textup{fund}}
\newcommand{\rank}{\textup{rk}}
\renewcommand\geq{\geqslant}
\renewcommand\leq{\leqslant}
\renewcommand\check{\checkmark}
\newcommand\norep{\ensuremath{\tikz[x=6pt,y=6pt,thick]{\draw(0,0)--(1,1);\draw(1,0)--(0,1);}}}
\newcommand{\FF}[1]{\ensuremath{F_{\ref{#1}}}}
\newcommand{\gen}[1]{\langle #1 \rangle}
\newcommand{\genn}[1]{\langle\!\!\langle #1 \rangle\!\!\rangle}

\newcommand{\past}[2]{#1\!\sslash\!#2}

\newcommand{\pastgenn}[3]{#1(#2)\!\sslash\!\langle\!\!\langle #3 \rangle\!\!\rangle}

\newcommand{\cross}[5]{\mathchoice{\scalebox{1.3}{$\big[$}\,\raisebox{1pt}{$\begin{matrix}{\scalebox{0.9}{$#1$}}\hspace{-5pt}&{\scalebox{0.9}{$#2$}}\\[-2pt]{\scalebox{0.9}{$#3$}}\hspace{-5pt}&{{\scalebox{0.9}{$#4$}}}\end{matrix}$}\,\scalebox{1.3}{$\big]$}_{#5}}{\big[\begin{smallmatrix}{#1}&{#2}\\{#3}&{#4}\end{smallmatrix}\big]_{#5}}{}{}}   
\newcommand{\crossinv}[5]{\mathchoice{\scalebox{1.3}{$\big[$}\,\raisebox{1pt}{$\begin{matrix}{\scalebox{0.9}{$#1$}}\hspace{-5pt}&{\scalebox{0.9}{$#2$}}\\[-2pt]{\scalebox{0.9}{$#3$}}\hspace{-5pt}&{{\scalebox{0.9}{$#4$}}}\end{matrix}$}\,\scalebox{1.3}{$\big]$}_{#5}^{-1}}{\big[\begin{smallmatrix}{#1}&{#2}\\{#3}&{#4}\end{smallmatrix}\big]_{#5}^{-1}}{}{}}   
\renewcommand{\setminus}{\backslash}
\newcommand{\minor}[2]{\setminus #1 / #2}
\newcommand{\hyperplus}{\mathrel{\,\raisebox{-1.1pt}{\larger[-0]{$\boxplus$}}\,}}

\renewcommand\emptyset\varnothing

\newcommand\ol{\texorpdfstring{\textbf{Oliver:} }{Oliver: }}
\newcommand\mb{\texorpdfstring{\textbf{Matt:} }{Matt: }}

\title{Foundations of matroids\\[10pt] \normalsize Part 2: Further theory, examples, and computational methods}

\author{Matthew Baker}
\address{\rm Matthew Baker, School of Mathematics, Georgia Institute of Technology, Atlanta, USA}
\email{mbaker@math.gatech.edu}

\author{Oliver Lorscheid}
\address{\rm Oliver Lorscheid, University of Groningen, the Netherlands, and IMPA, Rio de Janeiro, Brazil}
\email{oliver@impa.br}

\author{Tianyi Zhang}
\address{\rm Tianyi Zhang, School of Mathematics, Georgia Institute of Technology, Atlanta, USA}
\email{kafuka@gatech.edu}

\hypersetup{pdftitle={Foundations of matroids - Part 2: Further theory, examples, and computational methods},pdfauthor={Matthew Baker, Oliver Lorscheid and Tianyi Zhang}}

\begin{document}

\begin{abstract}
 In this sequel to ``Foundations of matroids - Part 1,'' we establish several presentations of the foundation of a matroid in terms of small building blocks. For example, we show that the foundation of a matroid $M$ is the colimit of the foundations of all embedded minors of $M$ isomorphic to one of the matroids \ $U^2_4$, \ $U^2_5$, \ $U^3_5$, \ $C_5$, \ $C_5^\ast$, \ $U^2_4\oplus U^1_2$, \ $F_7$, \ $F_7^\ast$, \ and we show that this list is minimal. We establish similar minimal lists of building blocks for the classes of 2-connected and 3-connected matroids. We also establish a presentation for the foundation of a matroid in terms of its lattice of flats. Each of these presentations provides a useful method to compute the foundation of certain matroids, as we illustrate with a number of concrete examples. Combining these techniques with other results in the literature, we are able to compute the foundations of several interesting classes of matroids, including whirls, rank-2 uniform matroids, and projective geometries. In an appendix, we catalogue various `small' pastures which occur as foundations of matroids, most of which were found with the assistance of a computer, and we discuss some of their interesting properties.
\end{abstract}

\maketitle

\begin{small} \tableofcontents \end{small}


\section*{Introduction}
\label{introduction}

We assume throughout this Introduction that the reader is familiar with the basic theory of matroid representations over pastures, as described for example in \cite{Baker-Lorscheid20}. We refer the reader to the detailed introduction to \cite{Baker-Lorscheid20} for the definitions of, and motivation for, some of the concepts mentioned below (including pastures, the foundation of a matroid, and universal cross-ratios); see also \autoref{section: background} and \autoref{section: the foundation of a matroid} below for a brief summary.

\medskip

In \cite{Baker-Lorscheid20}, the authors initiated a systematic study of the \emph{foundation} of a matroid $M$, an algebraic object canonically attached to $M$ which governs the representability of $M$ over arbitrary pastures.
In particular, the foundation $F_M$ determines the set of projective equivalence classes of representations of $M$ over partial fields, as well as the set of reorientation classes of orientations of $M$.
More precisely, for any pasture $P$, the set of (weak) $P$-representations of $M$, modulo rescaling equivalence, is canonically identified with the set of pasture homomorphisms from $F_M$ to $P$.
Some advantages of the foundation over the earlier concepts of ``inner Tutte group'', due to Dress and Wenzel, or ``universal partial field'', due to Pendavingh and van Zwam, include:
\begin{itemize}
\item Unlike the inner Tutte group, the foundation also has an additive structure rather than just a multiplicative one. This additive structure is crucial for determining the representations of $M$.
\item If $M$ is not representable over any field, the universal partial field does not exist, but the foundation of $M$ is defined for every matroid $M$. And it carries information about representations of $M$ over interesting pastures such as the signed or tropical hyperfields, in addition to partial fields.
\item Unlike both the inner Tutte group and universal partial field, the foundation of $M$ can be characterized intrinsically through a universal property (as the unique object representing a certain functor). Both the inner Tutte group and universal partial field, by contrast, are defined in terms of generators and relations; for the foundation, such characterizations are merely descriptions.
\item The category of pastures is much more flexible and robust than the categories of fuzzy rings or partial fields; for example, it admits arbitrary limits and colimits and has both an initial and final object. In particular, one has a tensor product operation which is lacking in the earlier theories, but plays an important role in the representation theory of matroids.
\end{itemize}

\medskip

In \cite{Baker-Lorscheid20}, we made use of Tutte's homotopy theory and the companion results of Gelfand--Rybnikov--Stone to give an explicit presentation for $F_M$ by generators (``universal cross ratios'') and relations (the ``GRS relations''), as well as a formula expressing $F_M$ as the colimit of $F_N$ over an explicit and universal set of ``small'' embedded minors $N$ of $M$.
In the case of ternary matroids, or more generally matroids without $U^2_5$ or $U^3_5$ minors (which we refer to as ``large uniform minors''), we proved that $F_M$ decomposes as a tensor product of an explicit finite set of basic ``building blocks'' $\{ \U,\ \D,\ \H,\ \F_3,\ \F_2 \}$.
Using this ``structure theorem for matroids without large uniform minors'', we were able to give new proofs and generalizations of a number of results in the matroid theory literature, for example the Lee--Scobee theorem that a matroid is both ternary and orientable if and only if it is dyadic.

\medskip

In the present paper, we continue this study, complementing the results of \cite{Baker-Lorscheid20} with new theoretical insights which also serve as computational tools that we apply to numerous concrete examples of interest. 
We now summarize several of these specific enhancements to the material in \cite{Baker-Lorscheid20}.

\subsection*{Foundations of direct sums}
The following result (\autoref{thm: foundations of direct sums}) on the foundation of the direct sum $M\oplus N$ of two matroids was stated without proof in \cite{Baker-Lorscheid20}:

\begin{thmA}\label{thmA}
 $F_{M\oplus N}=F_M\otimes F_N$.
\end{thmA}

\subsection*{Fundamental presentations}

A minor embedding $N\hookrightarrow M$ induces a morphism $F_N\to F_M$ between the respective foundations. It is shown in \cite{Baker-Lorscheid20} that the foundation of $M$ is the colimit of the foundations of all embedded minors $N\hookrightarrow M$ on at most $7$ elements. We explore this result in more depth in the present paper, with the goal of making it both more precise and more generalizable.

Let $\cC$ be a class of isomorphism types of matroids and $\cE_{M,\cC}$ the diagram of all embedded minors $N\hookrightarrow M$ with isomorphism type in $\cC$ together with all minor embeddings $N\hookrightarrow N'$ between such minors. Taking foundations and the induced morphisms yields a diagram $F(\cE_{M,\cC})$ of pastures.

It is shown in \cite{Baker-Lorscheid20} that $F_M$ is the colimit of $F(\cE_{M,\cC_0'})$ for a certain list $\cC_0'$ of matroids on up to $7$ elements, but this list fails to be minimal. We reduce this list to $\cC_0 = \{ U^2_4, \ U^2_5, \ U^3_5, \ C_5, \ C_5^\ast, \ U^2_4\oplus U^1_2, \ F_7, \ F_7^\ast \}$, call $\cE_M=\cE_{M,\cC_0}$ the \emph{fundamental diagram of $M$}, and establish the \emph{fundamental presentation of $M$} (\autoref{thm: fundamental presentation}):

\begin{thmA}\label{thmB}
 $F_M=\colim F(\cE_{M})$.
\end{thmA}

In fact, $\cC_0$ is the unique minimal set such that \autoref{thmB} holds for all matroids $M$. 

\medskip

We also derive from \autoref{thmB} a description of $F_M$ in terms of the lattice of flats of $M$ (\autoref{thm: fundamental presentation}). For this, we establish some technical refinements of the famous ``Scum Theorem'' which may be of independent interest (\autoref{prop: embedded minor surject onto upper sublattices} and \autoref{lemma: going up}). In particular, a minor embedding $N\hookrightarrow M$ defines a sublattice $\Lambda_N$ of the lattice of flats of $M$, and the induced morphism $F_N\to F_M$ depends only on the corresponding lattice embedding.

The \emph{fundamental lattice diagram of $M$} is the diagram $\cL_M$ of all sublattices that stem from embedded minors of types $U^2_4$, $U^2_5$, $U^3_5$, $C_5$, $F_7$ and $F_7^\ast$, together with all inclusions of sublattices. Taking foundations yields the \emph{fundamental lattice presentation $F(\cL_M)$ of $M$} (\autoref{thm: fundamental lattice presentation}): 

\begin{thmA}\label{thmC}
 $F_M=\colim F(\cL_M)$.
\end{thmA}

\autoref{thmC} has some useful variants; cf.\ \autoref{thm: fundamental lattice presentation by upper sublattices of small rank} and \autoref{thm: variant of foundation as colimit of all upper sublattices of small rank}.

\subsection*{Fundamental types}
More generally, we show (\autoref{thm: fundamental type}) that for each class $\cC$ of matroids, there is a unique minimal set $\cC_0$ of isomorphism classes of matroids in $\cC$ such that the foundation of every matroid $M$ in $\cC$ is the colimit of the foundations of all embedded minors of $M$ whose isomorphism type is in $\cC_0$. We call $\cC_0$ the \emph{fundamental type} of $\cC$. 

We determine the fundamental type for several classes of matroids, with the most interesting examples being the classes of $2$-connected matroids (\autoref{thm: fundamental presentation for 2-connected matroids}) and $3$-connected matroids (\autoref{thm: fundamental presentation for 3-connected matroids}). Using the standard nomenclature from Oxley's book \cite{Oxley92}, we have:

\begin{thmA}\label{thmD}
 The fundamental type of the class of all $2$-connected matroids is 
 \[
  \cC_0 \ = \ \{ U^2_4, \ U^2_5, \ U^3_5, \ C_5, \ C_5^\ast, \ F_7, \ F_7^\ast \},
 \] 
 i.e.\ $F_M=\colim F(\cE_{M,\cC_0})$ for every $2$-connected matroid $M$.
\end{thmA}

\begin{thmA}\label{thmE}
 The fundamental type of the class of all $3$-connected matroids is 
 \[
  \cC_0 \ = \ \{ U^2_4, \ U^2_5, \ U^3_5, \ W^3, \ Q_6, \ P_6, \ F_7, \ F_7^\ast \}, 
 \]
 i.e.\ $F_M=\colim F(\cE_{M,\cC_0})$ for every $3$-connected matroid $M$.
\end{thmA}

The proofs of these results are quite non-trivial: \autoref{thmD} requires a detailed analysis of the minimal 2-connected extensions of $U^2_4\oplus U^1_2$, which is achieved through the Cunningham-Edmonds tree decomposition for 2-connected matroids and a rather elaborate induction; \autoref{thmE} uses a strengthening of Seymour's splitter theorem due to Bixby--Coullard and an exhaustive computer search on matroids up to 8 elements.\footnote{While a number of results in this paper were verified using the Macaulay2 package \textsc{Pastures}, the explicit determination of the fundamental type for 3-connected matroids is the only place in the paper (outside the Appendix) which actually requires a computer-assisted proof.}

As with \autoref{thmC}, every fundamental type has a corresponding fundamental type of sublattices; cf., for example, \autoref{thm: fundamental lattice presentation for 3-connected matroids}.

\subsection*{Map of examples} 
We illustrate how to use the various fundamental presentations of the foundation in concrete examples. The list in \autoref{table: list of examples} shows which fundamental presentation of $F_M$ enters which example.

\begin{table}[htb]
 \caption{List of which fundamental presentation appears in which example}
 \label{table: list of examples}
 \begin{tabular}{|c|c|c|c|}
  \hline 
  fundamental presentation & ref.\ to result & matroid & ref.\ to example \\
  \hline \hline 
  by cross ratios & \autoref{thm: fundamental presentation of foundations in terms of bases} & $U^2_5$ & \autoref{subsection: foundations of the uniform matroid U25} \\
  \hline 
  $F_M=\colim F(\cE_M)$ & \autoref{thmB} / \autoref{thm: fundamental presentation} & $Q_6$ & \autoref{subsection: foundation of Q6} \\
  \hline 
  $F_M=\colim F(\cL_M)$ & \autoref{thmC} / \autoref{thm: fundamental lattice presentation} & $AG(2,3)\setminus e$ & \autoref{subsection: foundation of AG23-e} \\
  \hline 
  $F_M=\colim F(\cE^{(2)}_M)$ & \autoref{thmD} / \autoref{thm: fundamental presentation for 2-connected matroids} & whirls & \autoref{subsection: foundation of whirls} \\
  \hline 
  $F_M=\colim F(\cE^{(3)}_M)$ & \autoref{thmE} / \autoref{thm: fundamental presentation for 3-connected matroids} & $F_7^-$ & \autoref{subsection: foundation of F7-} \\
  \hline 
  $F_M=\colim F(\cL^{(3)}_M)$ & \autoref{thm: fundamental lattice presentation for 3-connected matroids} & $P_7$ & \autoref{subsection: foundation of P7} \\
  \hline 
  $F_M=\colim F(\cL^{(\leq3)}_M)$ & \autoref{thm: fundamental lattice presentation by upper sublattices of small rank} & $T_8$ & \autoref{subsection: foundation of T8} \\
  \hline 
 \end{tabular}
\end{table}

Moreover, we draw on results from the literature to determine that the foundation of the uniform matroid $U^2_{k+3}$ is Semple's $k$-regular partial field (\autoref{prop: U_k is the foundation of U2k+3}), the foundation of the $d$-dimensional projective space $PG(d,q)$ is the finite field $\F_q$ (\autoref{prop: foundation of projective spaces}), and the foundation of any non-Desarguesian projective plane is $\K$ (\autoref{prop: foundation of non-Desarguesian projective plane}).

\subsection*{A note on the examples appearing in this paper}

While our techniques can be used, in principle, to algorithmically compute the foundation of any matroid, the complexity of these computations increases exponentially with the size of the matroid. 
So for ``large'' matroids, such direct calculations are best done with the aid of a computer. This has been systematically implemented by Justin Chen and the third author, who have written the Macaulay2 package \textsc{Pastures} and an accompanying paper \cite{Chen-Zhang}.

\medskip

Rather than relying exclusively on computer calculations in order to present interesting examples, we have decided to work out some computations of foundations ``by hand'' in this paper, partly in order to illustrate some of the different theoretical techniques presented here. For this, we systematically make use of tactical shortcuts, leaning on known results about the representability of the above-mentioned matroids in order to make the calculations human-readable. A typical computation of the foundation $F_M$ of a matroid $M$ in our list of examples might proceed as follows: 
\begin{enumerate}
 \item We leverage known results about the representability of $M$ to gain information about which pastures can possibly appear as $F_M$. In many cases, $M$ is without large uniform minors,
  which allows us to apply the structure theorem for this class of matroids.
 \item We write down the ``fundamental diagram'' of $M$ and observe, in the cases of interest, that it is connected.
 This implies that $F_M$ does not decompose into the tensor product of several non-trivial factors. 
 \item\label{proof3}
 We extract sufficient relations between the cross ratios of $M$ from the fundamental presentation of $F_M$ to narrow down the possibilities of $F_M$ to a unique pasture, which concludes the proof.
\end{enumerate}

\subsection*{Inductive approach.} We also sometimes employ an inductive approach to computing the foundation of a matroid in terms of the foundations of its minors. This approach is illustrated, for example, in the computations of the foundations of $F_7^-$ (\autoref{prop: foundation of the non-Fano matroid}) and $T_8$ (\autoref{prop: foundation of T8}).
The inductive approach also underlies the theoretical fact that we can compute the fundamental type of a class of matroids in terms of minimal extensions of the matroids appearing in certain other fundamental types (\autoref{prop: fundamental type of subclasses with U24}). 

\medskip

At the end of \autoref{subsection: fundamental presentation for subclasses}, we formulate some open problems concerning fundamental types.

\subsection*{List of foundations}

In \autoref{appendix: some interesting foundations}, we describe some interesting foundations, which we found partly based on theory and partly by using the Macaulay2 package \textsc{Pastures} developed by Chen and the third author in \cite{Chen-Zhang}. 

By \autoref{thm: foundations of direct sums} and a result whose proof appears in \cite{Baker-Lorscheid-Walsh-Zhang}, the foundation of a direct sum or $2$-sum decomposes into a tensor product of the foundations of the $3$-connected components of the matroid. We therefore concentrate on foundations of $3$-connected matroids.

There are two infinite families of foundations discussed in \autoref{appendix: some interesting foundations}: the foundations of uniform matroids, which by \autoref{prop: U_k is the foundation of U2k+3} include Semple's $k$-regular partial fields,
and finite fields, which by \autoref{prop: foundation of projective spaces} are foundations of projective geometries.

We continue by describing various other pastures which occur as foundations, including the foundations of all 3-connected matroids on up to $8$ elements. We see that many of the partial fields described in \cite{Pendavingh-vanZwam10a} occur as the foundation of a matroid, while others occur as the universal partial field but not necessarily the foundation.

We conclude the appendix with some remarks on 
foundations of non-representable matroids.

\subsection*{Acknowledgements}

We thank Nathan Bowler for sharing many helpful insights with us, Zach Walsh for his help with \autoref{FKrasner}, and Peteris Silins for his feedback. We thank Steffen M\"uller and Willard Verschoore de la Houssaije for their help with the determination of the $S$-units of various universal rings in the appendix, partially using the software developed in \cite{Verschoore24}. The first author was supported by NSF grant DMS-2154224 and a Simons Fellowship in Mathematics. The second author was supported by Marie Sk{\l}odowska Curie Fellowship MSCA-IF-101022339.


\section{Background}
\label{section: background}

In this section we give a quick reminder of some basic notions from \cite{Baker-Lorscheid20}, such as pastures, matroid representations, universal cross ratios, and the foundation of a matroid. We refer the reader to \cite{Baker-Lorscheid20} and \cite{Baker-Lorscheid21} for a more extensive discussion of these notions.

\subsection{Pastures}
\label{subsection: pastures}

A \emph{pointed monoid} is a (multiplicatively written and commutative) monoid $A$ with neutral element $1$ and absorbing element $0$, i.e.\ $1\cdot a=a$ and $0\cdot a=0$ for all $a\in A$. We denote by $P^\times$ the subgroup of invertible elements. A \emph{morphism of pointed monoids} is a multiplicative map $f:A_1\to A_2$ that preserves $1$ and $0$. We denote by $\Sym_3(A)$ the quotient of $A^3$ by the permutation action of $S_3$ on the coefficients of $(a,b,c)\in A^3$. We denote the equivalence classes in $\Sym_3(A)$ by $a+b+c=[(a,b,c)]$.

A \emph{pasture} is a pointed monoid $P$ such that $P^\times=P-\{0\}$ together with an involution $a\mapsto-a$ and a nonempty subset $N_P$ of $\Sym_3(P)$, called the \emph{null set of $P$}, such that: 
\begin{enumerate}
\item For all $a+b+c\in N_P$ and $d\in P$ we have $da+db+dc\in N_P$. 
\item $a+b+0\in N_P$ if and only if $b=-a$. 
\end{enumerate}
It follows that $-0=0$, $(-1)^2=1$, and $-(-a)=a$. 

A \emph{pasture morphism} is a morphism $f:P_1\to P_2$ of pointed monoids such that $f(a)+f(b)+f(c)\in N_{P_2}$ if $a+b+c\in N_{P_1}$. It follows that $f(-a)=-f(a)$. We denote the category of pastures by $\Pastures$.

We write $0$ for $0+0+0$ and $a+b$ for $a+b+0$. Note that if $a+0\in N_P$, then $a=0$. We also write $a+b-c$ for $a+b+(-c)$ and $a+b=c$ if $a+b-c\in N_P$. Note also that the inversion $a\mapsto-a$ is determined by the null set $N_P$.

\subsubsection{Examples}
Some first examples of importance are:
\begin{itemize}
 \item the \emph{regular partial field} $\Funpm=\{0,1,-1\}$ with null set $N_{\Funpm}=\{0,\ 1+(-1)\}$, which is an initial object in $\Pastures$;
 \item the \emph{Krasner hyperfield} $\K=\{0,1\}$ with null set $N_\K=\{0,\ 1+1,\ 1+1+1\}$, which is a terminal object in $\Pastures$;
 \item the field with two elements $\F_2=\{0,1\}$ with null set $N_{\F_2}=\{0,\ 1+1\}$;
 \item the field with three elements $\F_3=\{0,1,-1\}$ with null set 
 \[
  N_{\F_3} \ = \ \{0,\ 1+(-1),\ 1+1+1,\ (-1)+(-1)+(-1)\};
 \]
 \item the \emph{sign hyperfield} $\S=\{0,1,-1\}$ with null set 
 \[
  N_\S \ = \ \{0, 1+(-1),\ 1+1+(-1),\ 1+(-1)+(-1)\}.
 \]
\end{itemize}

In fact, every partial field and every hyperfield defines a pasture. Given a partial field $(G,R)$, where $R$ is a ring and $G$ is a subgroup of $R^\times$ that contains $-1$, we define the associated pasture as $P=G\cup\{0\}$ with null set $N_P=\{a+b-c\mid a+b=c\text{ in }R\}$. Given a hyperfield $F$, we define the associated pasture as $P=F$ (as a multiplicative monoid) with null set $N_P=\{a+b-c\mid c\in a\hyperplus b\text{ in }F\}$. In this text, we consider partial fields and hyperfields as pastures. Note that the respective notions of morphisms for partial fields and for hyperfields coincide with the corresponding notion of pasture morphisms.

\subsubsection{Free algebras and quotients}
Given a pasture $P$ and indeterminants $x_i$, indexed by $i\in I$, the \emph{free algebra $P\gen{x_i\mid i\in I}$} is the pasture whose unit group is the product of $P^\times$ with the free abelian group $\gen{x_i\mid i\in I}$ generated by the $x_i$, and whose null set consists of all elements of the form $da+db+dc$ where $a+b+c\in N_P$ and $d\in P\gen{x_i\mid i\in I}$.

Given a pasture $P$ and a subset $\{a_j+b_j+c_j\mid j\in J\}$ of $\Sym^3(P)$, where we assume that $a_i,b_i\in P^\times$, we define $\past{P}{\genn{a_j+b_j+c_j\mid j\in J}}$ as the following pasture: its monoid is the quotient monoid $\overline{P}=P/\sim$ of $P$ by the equivalence relation generated by the relations $da_j\sim -db_j$ for all $d\in P$ and all $j\in J$ for which $c_j=0$. We write $[a]$ for the class of $a\in P$ in $\overline P$. The null set of $\past{P}{\genn{a_j+b_j+c_j\mid j\in J}}$ consists of all expressions of the form $[a]+[b]+[c]$ with $a+b+c\in N_P$, together with all expressions $[da_j]+[db_j]+[dc_j]$ for $d\in P$ and $j\in J$. 

This allows us to write every pasture in the form 
\[
 P \ = \ \pastgenn\Funpm{x_i\mid i\in I}{a_j+b_j+c_j\mid j\in J}
\]
by choosing generators and relations. 
Indeed, to see that this is always possible, note that for any pasture $P$ we have 
\[
 P \ \cong \ \pastgenn\Funpm{x_a\mid a\in P}{S},
\]
where $S$ consists of all binary relations $x_{a_1}\cdots x_{a_k} - 1 = 0$ corresponding to multiplicative relations of the form $a_1\cdots a_k = 1$ in $P$, together with all ternary relations $x_a + x_b + x_c = 0$ corresponding to $a+b+c \in N_P$. It is easy to see that the map sending $a$ to $x_a$ is an isomorphism of pastures.


The following pastures, which we present via generators and relations, all play an important role in the theory of partial-field representations of matroids:

\begin{itemize}
 \item the \emph{near-regular partial field} $\U=\pastgenn{\Funpm}{x,y}{x+y-1}$;
 \item the \emph{dyadic partial field} $\D=\pastgenn{\Funpm}{x}{x-1-1}$;
 \item the \emph{hexagonal partial field} $\H=\pastgenn{\Funpm}{\zeta_6}{\zeta_6^3+1,\ \zeta_6+\zeta_6^{-1}-1}$;
 \item the \emph{golden ratio partial field} $\G=\pastgenn{\Funpm}{x}{x^2+x-1}$.
\end{itemize}

\subsubsection{Categorical properties}
The category of pastures is complete and cocomplete, as proven in \cite{Creech21}. For the present purposes, it suffices to understand finite colimits. To start with, we describe the \emph{coproduct}, or \emph{tensor product},  of two pastures $P_1$ and $P_2$. As a pointed monoid, the tensor product is 
\[
 P_1\otimes P_2 \ = \ \{0\}\cup (P_1^\times \oplus P_2^\times) \, / \, \{\pm(1,1)\};
\]
we write $a_1\otimes a_2=\{(a_1,a_2),(-a_1,-a_2)\}$ for its cosets (if $a_1\neq0$ and $a_2\neq0$), and define $a_1\otimes 0=0=0\otimes a_2$. Its null set is
\begin{align*}
 N_{P_1\otimes P_2} \ &= \ \Big\{ a\otimes d+b\otimes d+c\otimes d \in \Sym_3(P_1\otimes P_2) \, \Big| \, a+b+c\in N_{P_1},\ d\in P_2 \Big\} \\
                      & \, \cup \; \ \Big\{ d\otimes a+d\otimes b+d\otimes c\in \Sym_3(P_1\otimes P_2) \, \Big| \, d\in P_1,\ a+b+c\in N_{P_2} \Big\}.
\end{align*}
The tensor product comes with the canonical inclusions $\iota_1:P_1\to P_1\otimes P_2$ and $\iota_2:P_2\to P_1\otimes P_2$ that are defined by $\iota_1(a)=a\otimes1$ and $\iota_2(b)=1\otimes b$, respectively. It satisfies the universal property of the coproduct: for every pair of morphisms $f_1:P\to Q$ and $f_2:P_2\to Q$ into a pasture $Q$, there is a unique morphism $f:P_1\otimes P_2\to Q$ such that $f_1=f\circ\iota_1$ and $f_2=f\circ\iota_2$.

The tensor product $P_1\otimes\dotsb\otimes P_n$ of pastures $P_1,\dotsc,P_n$ is their categorical coproduct. It can be derived from the pairwise tensor product as
\[
 P_1\otimes\dotsb\otimes P_n \ = \ (((P_1\otimes P_2)\otimes P_3)\dotsb \otimes P_{n-1})\otimes P_n.
\]
When $n=0$, we define the empty tensor product to be the initial pasture $\Funpm$.


The \emph{colimit} of a finite diagram of $\cD$ of pastures $\{P_i\}_{i\in I}$ and pasture morphisms $\{f_j:P_{s_j}\to P_{t_j}\}_{j\in J}$ is the quotient
\[
 \colim\cD \ = \ \past{\bigotimes_{i\in I} \ P_i \ }{ \ \genn{1\otimes\dotsb a_{s_j} \dotsb \otimes 1 \ - \ 1\otimes\dotsb \underbrace{f_j(a_{s_j})}_{\in \ P_{t_j}} \dotsb \otimes 1\mid j\in J,\ a_{s_j}\in P_{s_j}}}
\]
of the coproduct $\bigotimes P_i$. 

\begin{ex}
 Consider the diagram
\[
 \begin{tikzcd}[column sep=30pt]
  \cD \ = \ \Big( \ \H \ar[r,shift left=3pt,"\id"] \ar[r,shift right=3pt,"f"'] & \H \ \Big)
 \end{tikzcd}
\]
where $\H=\pastgenn\Funpm{\zeta_6}{\zeta_6^3+1,\ \zeta_6+\zeta_6^{-1}-1}$ is the hexagonal partial field and $f:\H\to\H$ is defined by $f(\zeta_6)=\zeta_6^{-1}$. Then
\begin{multline*}
 \colim\cD \ = \ \past{\H\otimes\H \ }{ \, \genn{\zeta_6\otimes1-1\otimes \id(\zeta_6),\ \zeta_6\otimes1-1\otimes f(\zeta_6)}} \\ 
             \underset{\tiny (\zeta_6\otimes1\sim1\otimes\zeta_6)}= \ \past{\H \, }{\genn{\zeta_6-\zeta_6^{-1}}} \ \underset{\tiny (\zeta_6\sim\zeta_6^3\sim-1)}= \ \F_3.
\end{multline*}
\end{ex}

\subsection{Matroid representations}
\label{subsection: matroid representations}

Throughout the text, $M$ denotes a matroid of rank $r$ with ground set $E=\{1\dotsc,n\}$. We write $\bI=(e_1,\dotsc,e_s)$ for elements of $E^s$ and $I=|\bI|$ for the subset $\{e_1,\dotsc,e_s\}$ of $E$. We write $\bI f=(e_1,\dotsc,e_s,f)$ and $I f=\{e_1,\dotsc,e_s,f\}$. 

Let $P$ be a pasture. A \emph{weak Grassmann-Pl\"ucker function of rank $r$ with values in $P$} is a function $\Delta:E^r\to P$ such that
\begin{enumerate}[label = (GP\arabic*)]\setcounter{enumi}{-1}
 \item The support $\{ |\bI| \; : \; \Delta(\bI)\neq0 \}$ of $\Delta$ is the set of bases of a rank $r$ matroid on $E$.
  \item $\Delta$ is \emph{alternating}, i.e.\ $\Delta(e_{\sigma(1)},\dotsc,e_{\sigma(r)})=\sign(\sigma)\cdot\Delta(e_1,\dotsc,e_r)$ for all $e_1,\dotsc,e_r\in E$ and $\sigma\in S_r$; and
 \item $\Delta$ satisfies the \emph{$3$-term Pl\"ucker relations}
       \[
        \Delta(\bJ e_1e_2) \cdot \Delta(\bJ e_3e_4) \ - \ \Delta(\bJ e_1 e_3) \cdot \Delta(\bJ e_2 e_4) \ + \ \Delta(\bJ e_1 e_4) \cdot \Delta(\bJ e_2 e_3) \ \in \ N_P
       \]
       for all $\bJ\in E^{r-2}$ and $e_1,\dotsc,e_4\in E$.
\end{enumerate}

A \emph{$P$-representation of $M$} is a weak Grassmann-Pl\"ucker function $\Delta:E^r\to P$ such that $\Delta(\bI)\neq0$ if and only if $|\bI|$ is a basis of $M$.
 
A trivial but useful observation is that every matroid $M$ has a unique $\K$-representation $\Delta:E^r\to\K$, given by setting $\Delta(\bI)=1$ if $|\bI|$ is a basis of $M$ and $\Delta(\bI)=0$ if not. This yields an identification of matroids with weak Grassmann-Pl\"ucker functions $\Delta:E^r\to\K$.

We say that $M$ is \emph{$P$-representable} if it has a $P$-representation $\Delta:E^r\to P$. This streamlines and extends other notions of representability: $M$ is representable over a partial field $(G,R)$ (in the usual sense) if and only if $M$ is $P$-representable, where $P$ is the pasture associated with $(G,R)$. In particular, $M$ is regular if and only if $M$ is $\Funpm$-representable, $M$ is binary if and only if $M$ is $\F_2$-representable, and $M$ is ternary if and only if $M$ is $\F_3$-representable. Moreover, $M$ is orientable if and only if $M$ is $\S$-representable.

\subsection{Universal cross ratios and the inner Tutte group}
\label{subsection: universal cross ratios}

Let $M$ be a matroid. We denote by $\Omega_M$ the collection of all tuples $(I,a,b,c,d)$, where $I \subset E$ is an $(r-2)$-subset and $a,b,c,d\in E$ such that $I ac$, $I ad$, $I bc$, $I bd$ are bases of $M$. We denote by $\Omega_M^\octa$ the subset of \emph{non-degenerate} tuples $(I,a,b,c,d)$, for which also $I ab$ and $I cd$ are bases of $M$. 

Let $\T_M$ be the Tutte group $M$, which is generated by symbols $-1$ and $T_\bI/T_\bJ$ for $\bI,\bJ\in E^r$ such that $|\bI|$ and $|\bJ|$ are bases of $M$; see \cite{Wenzel91} for details. Let $\T_M^{(0)}$ be the inner Tutte group, which consists of the multi-degree zero elements of $\T_M$.

For $(I,a,b,c,d)\in\Omega_M$, we define the \emph{universal cross ratio} as the element
\[
 \cross abcdI \ = \ \frac{T_{\bI ac}T_{\bI bd}}{T_{\bI ad}T_{\bI bc}}
\]
of the inner Tutte group $\T^{(0)}_M$ of $M$, where $\bI$ is any ordering of $I$. Note that $\cross abcdI$ does not depend on the ordering of $\bI$. A fundamental fact is the following:

\begin{prop}[{\cite[Prop.\ 6.4]{Wenzel91}}]\label{prop: cross ratios generated the inner Tutte group}
 The inner Tutte group $T_M^{(0)}$ is generated by $-1$ and the universal cross ratios $\cross abcdI$ for $(I,a,b,c,d)\in \Omega_M^\octa$.
\end{prop}

Let $\Theta_M$ be the collection of all \emph{modular quadruples of hyperplanes for $M$}, which are tuples $(H_1,H_2,H_3,H_4)$ of hyperplanes $H_1,\dotsc,H_4$ of $M$ such that $F=H_1\cap\dotsb\cap H_4$ is a flat of corank $2$ and such that $F=H_i\cap H_j$ for $i\in\{1,2\}$ and $j\in\{3,4\}$. Let $\Theta_M^\octa$ be the subset of \emph{non-degenerate} quadruples $(H_1,\dots,H_4)$, for which also $F=H_1\cap H_2=H_3\cap H_4$. 

Let $\gen{-}$ be the closure operator of $M$. Then the association 
\[
 \begin{array}{cccc}
  \Psi: & \Omega_M    & \longrightarrow & \Theta_M \\
        & (I,a,b,c,d) & \longmapsto     & \big(\gen{Ia},\gen{Ib},\gen{Ic},\gen{Id}\big)
 \end{array}
\]
is a surjection, which restricts to a surjection $\Omega_M^\octa\to\Theta_M^\octa$. It follows from \cite[Lemma 1.4]{Dress-Wenzel89} (see also \cite[Prop.\ 3.6]{Baker-Lorscheid20}) that $\cross abcdI=\cross{a'}{b'}{c'}{d'}{I'}$ as elements of $\T_M^{(0)}$ if $\Psi(I,a,b,c,d)=\Psi(I',a',b',c',d')$. This allows us to define $\cross{H_1}{H_2}{H_3}{H_4}{}=\cross abcdI$ whenever $(H_1,H_2,H_3,H_4)\in\Theta_M$ with $(H_1,H_2,H_3,H_4)=\Psi(I,a,b,c,d)$.

\subsection{Fundamental elements and hexagons}
\label{subsection: fundamental elements and hexagons}

Let $P$ be a pasture. A \emph{fundamental pair of $P$} is an ordered pair $(a,b)$ of elements $a,b\in P^\times$ such that $a+b-1\in N_P$. A \emph{fundamental element of $P$} is an element $a\in P^\times$ that appears in a fundamental pair $(a,b)$. 
Every fundamental pair $(a,b)$ defines the set 
\[\textstyle
 \Xi(a,b) \ = \ \Big\{ \ (a,b), \quad (b,a), \quad (\frac 1a, -\frac ba), \quad (-\frac ba,\frac 1a), \quad (\frac 1b,-\frac ab), \quad (-\frac ab,\frac 1b) \ \Big\}
\]
of fundamental pairs, which are not necessarily pairwise distinct. We call such a set $\Xi(a,b)$ a \emph{hexagon in $P$}, which refers to the way of illustrating the involved fundamental elements as 
 \[
 \beginpgfgraphicnamed{tikz/fig44}
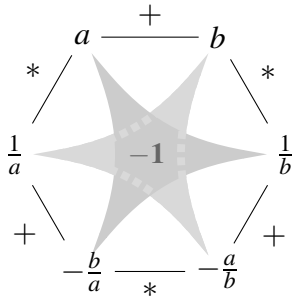
  \begin{tikzpicture}[baseline={([yshift=-.5ex]current bounding box.center)},x=0.9cm,y=0.9cm]
   \draw[line width=3pt,color=gray!30,fill=gray!30,bend angle=20] (60:1.4) to[bend left] (180:1.4) to[bend left] (300:1.4) to[bend left] cycle;
   \draw[line width=3pt,color=gray!40,fill=gray!40,bend angle=20] (0:1.4) to[bend left] (120:1.4) to[bend left] (240:1.4) to[bend left] cycle;
   \draw[line width=3pt,color=gray!30,dotted,bend angle=20] (60:1.4) to[bend left] (180:1.4) to[bend left] (300:1.4) to[bend left] cycle;
   \node (-1) at (0,0) {\small \textcolor{black!60}{{$\mathbf{-1\ }$}}};
   \node (a) at (120:2) {$a$};
   \node (b) at ( 60:2) {$b$};
   \node (c) at (180:2) {$\frac 1a$};
   \node (d) at (  0:2) {$\frac 1b$};
   \node (e) at (240:2) {$-\frac ba$};
   \node (f) at (300:2) {$-\frac ab$};
   \path (a) edge node[auto] {$+$} (b);
   \path (b) edge node[auto] {$*$} (d);
   \path (d) edge node[auto] {$+$} (f);
   \path (f) edge node[auto] {$*$} (e);
   \path (e) edge node[auto] {$+$} (c);
   \path (c) edge node[auto] {$*$} (a);
  \end{tikzpicture}
 \endpgfgraphicnamed
 \]
where an edge with label $*$ indicates that its vertices multiply to $1$ and an edge with label $+$ indicates that its vertices $x$ and $y$ add up to $1$, i.e.\ $x+y-1\in N_P$. The vertices of the triangles multiply to $-1$.

A basic observation (\cite[Prop.~3.6]{Baker-Lorscheid21}) is that the hexagons of $P$ are in bijective correspondence with $N_P^\times/P^\times$, where $N_P^\times$ is the subset of all terms $a+b+c\in N_P$ with $abc\in P^\times$. Since every element in $N_P-N_P^\times$ is of the form $a-a+0$ for some $a\in P$, the null set $N_P$ is determined by the hexagons in $P$.

\section{The foundation of a matroid}
\label{section: the foundation of a matroid}

The foundation of a matroid is defined in \cite{Baker-Lorscheid21b}, and further applications to the representation theory of matroids are developed in \cite{Baker-Lorscheid20} and \cite{Baker-Lorscheid21}. In this section, we recall the definition of the foundation and its fundamental presentation in terms of generators and relations, which forms the basis for nearly all results in this paper.

\subsection{The foundation}
\label{subsection: the foundation}

Using \cite[Cor.\ 7.13]{Baker-Lorscheid21b}, we can phrase the definition of the foundation of a matroid $M$ as follows.

\begin{df}
 The \emph{foundation of $M$} is the pasture $F_M$ with unit group $F_M^\times=\T_M^{(0)}$ and whose null set is generated by the elements $\cross abcdI+\cross acbdI-1$ for all $(I,a,b,c,d)\in\Omega_M^\octa$.
\end{df}

We can rescale a representation $\Delta:E^r\to P$ of $M$ by an element $t=(d,t_1,\dotsc,t_n)\in (P^\times)^{n+1}$ via 
\[
 t.\Delta(\bI) \ = \ d\cdot\Big(\prod_{e\in|\bI|} t_e\Big)\cdot\Delta(\bI),
\]
which defines a group action of $(P^\times)^{n+1}$ on the set of $P$-representations of $M$. We denote by $\cX_M(P)$ the set of equivalence classes under this action, which we call the \emph{$P$-rescaling classes of $M$}. The characterizing property of $F_M$ (which can also be taken as the \emph{definition} of $F_M$) is the following:

\begin{thm}[{\cite[Cor.\ 7.28]{Baker-Lorscheid21b}}]\label{thm: the foundation represents the space of rescaling classes}
 Let $M$ be a matroid with foundation $F_M$, and let $P$ be a pasture. Then there is a canonical bijection $\Hom(F_M,P)\to\cX_M(P)$ which is functorial in $P$.
\end{thm}

This has the following immediate consequence, which makes the foundation a useful tool for studying matroid representations:

\begin{cor}\label{cor: M is representable over P iff there is a morphism from F_M to P}
 The matroid $M$ is $P$-representable if and only if there is a morphism $F_M\to P$. \qed
\end{cor}

The foundation behaves well with respect to several natural matroid constructions. For instance, we recall from \cite[Prop.\ 4.8 and 4.9]{Baker-Lorscheid20}:

\begin{prop}\label{prop: foundation of dual matroids and simplifications}
 Let $M$ be a matroid with foundation $F_M$. Then the foundations of the dual matroid $M^\ast$ and of the simplification of $M$ are canonically isomorphic to $F_M$.
\end{prop}


\subsection{A presentation of the foundation by cross ratios}
\label{subsection: presentation of the foundation by cross ratios}

We recall the description of the foundation of a matroid in terms of generators and relations from \cite[Thm.\ 4.21]{Baker-Lorscheid20}. This result is derived from Gelfand, Rybnikov, and Stone's description of a complete set of relations between the universal cross ratios in \cite{Gelfand-Rybnikov-Stone95}, which is itself a consequence of Tutte's homotopy theorem; cf.\ \cite{Tutte58a}.

\begin{thm}\label{thm: fundamental presentation of foundations in terms of bases}
 Let $M$ be a matroid with foundation $F_M$. Then
 \[\textstyle
  F_M \ = \ \Funpm \, \big\langle \, \cross {e_1}{e_2}{e_3}{e_4}J \, \big| \, (J;e_1,\dotsc,e_4)\in\Omega_M \, \big\rangle \, \sslash \, S,
 \]
 where $S$ is defined by the multiplicative relations
 \[\tag{R--}\label{R-}
  -1=1 
 \]
 if the Fano matroid $F_7$ or its dual $F_7^\ast$ is a minor of $M$;
 \[\tag{R$\sigma$}\label{Rs}
  \cross {e_1}{e_2}{e_3}{e_4}J \ = \ \cross {e_2}{e_1}{e_4}{e_3}J \ = \ \cross {e_3}{e_4}{e_1}{e_2}J \ = \ \cross {e_4}{e_3}{e_2}{e_1}J
 \]
 for all $(J;e_1,\dotsc,e_4)\in\Omega_M^\octa$;
 \[\tag{R0}\label{R0}
  \cross {e_1}{e_2}{e_3}{e_4}J \ = \ 1
 \]
 for all degenerate $(J;e_1,\dotsc,e_4)\in\Omega_M$;
 \[\tag{R1}\label{R1}
  \cross {e_1}{e_2}{e_4}{e_3}J \ = \ \crossinv {e_1}{e_2}{e_3}{e_4}J
 \]
 for all $(J;e_1,\dotsc,e_4)\in\Omega_M^\octa$;
 \[\tag{R2}\label{R2}
  \cross {e_1}{e_2}{e_3}{e_4}J \cdot \cross {e_1}{e_3}{e_4}{e_2}J  \cdot \cross {e_1}{e_4}{e_2}{e_3}J \ = \ -1
 \]
 for all $(J;e_1,\dotsc,e_4)\in\Omega_M^\octa$;
 \[\tag{R3}\label{R3}
  \cross {e_1}{e_2}{e_3}{e_4}{J} \cdot \cross {e_1}{e_2}{e_4}{e_5}J \cdot \cross {e_1}{e_2}{e_5}{e_3}J \ = \ 1
 \]
 for all $e_1,\dotsc,e_5\in E$ and $J\subset E$ such that each of $(J;e_1,e_2,e_3,e_4)$, $(J;e_1,e_2,e_4,e_5)$ and $(J;e_1,e_2,e_5,e_3)$ is in $\Omega_M$;
 \[\tag{R4}\label{R4}
  \cross {e_1}{e_2}{e_3}{e_4}{Je_5} \cdot \cross {e_1}{e_2}{e_4}{e_5}{Je_3} \cdot \cross {e_1}{e_2}{e_5}{e_3}{Je_4} \ = \ 1
 \]
 for all $e_1,\dotsc,e_5\in E$ and $J\subset E$ such that $(Je_5;e_1,e_2,e_3,e_4)$, $(Je_3;e_1,e_2,e_4,e_5)$ and $(Je_4;e_1,e_2,e_5,e_3)$ are in $\Omega_M$; 
 \[\tag{R5}\label{R5}
  \cross {e_1}{e_2}{e_3}{e_4}{Je_5} \ = \ \cross {e_1}{e_2}{e_3}{e_4}{Je_6} 
 \]
 for all $e_1,\dotsc,e_6\in E$ and $J\subset E$ such that $\gen{Je_5}=\gen{Je_6}$ and such that $(Je_5;e_1,e_2,e_3,e_4)$ and $(Je_6;e_1,e_2,e_3,e_4)$ are in $\Omega_M^\octa$; 
 \\[7pt]
 as well as the additive Pl\"ucker relations
 \[\tag{R+}\label{R+}
  \cross {e_1}{e_2}{e_3}{e_4}J + \cross {e_1}{e_3}{e_2}{e_4}J  \ = \ 1
 \]
 for all $(J;e_1,\dotsc,e_4)\in\Omega_M^\octa$.
\end{thm}

\begin{rem}
It is worth noting that it is possible to have $-1=1$ in $F_M$ even if $M$ has no $F_7$ or $F_7^*$-minor; cf.\ \autoref{F22a} and \autoref{F53c} for examples. 
\end{rem}

\subsection{The universal partial field}
\label{subsection: the universal partial field}

The universal partial field of a matroid is introduced in \cite{Pendavingh-vanZwam10a} and \cite{Pendavingh-vanZwam10b} as a tool for studying representations of matroids over partial fields, in a similar way as the foundation can be used to study representations over arbitrary pastures. In fact, the first two authors show in \cite[Lemma 7.48]{Baker-Lorscheid21b} that the universal partial field can be derived in a functorial way from the foundation of the matroid. 

Since this result is formulated in the language of ordered blueprints, and a direct interpretation in terms of pastures would require additional results from \cite[sections 7.2 and 7.3]{Baker-Lorscheid21b}, we include in this section an exposition that is independent of \cite{Baker-Lorscheid21b}. We cite certain results about pastures from \cite{Baker-Lorscheid21}, but the proofs of those results do not require any knowledge of the theory of ordered blueprints.


\begin{df}
 Let $P$ be a pasture. The \emph{universal ring of $P$} is the ring
 \[
  R_P \ = \ \Z[P^\times]\,/\,\gen{N_P},
 \]
 where we consider the elements $a+b+c$ of the null set $N_P$ as elements of the group ring $\Z[P^\times]$. The association $a\mapsto [a]$ defines a multiplicative map $\rho_P:P\to R_P$ that satisfies $\rho_P(a)+\rho_P(b)+\rho_P(c)=0$ for all $a+b+c\in N_P$.
 
 A \emph{mock partial field} is a pasture $P$ with $R_P\neq0$. For a mock partial field $P$, we define its \emph{associated partial field $\Pi(P)$} as the pasture
 \[
  \Pi(P) \ = \ \past{\overline{P}}{\genn{a+b+c\mid a+b+c=0\text{ in }R_P}},
 \]
 where $\overline{P}=\rho_P(P)$ is the image of $\rho_P:P\to R_P$. The map $\rho_P$ restricts to a surjective pasture morphism $\pi_P:P\to\Pi(P)$.
\end{df}

The following summarizes Lemmas 2.12 and 2.14 of \cite{Baker-Lorscheid21}.

\begin{lemma}\label{lemma: associated partial field of a mock partial field}
 Let $P$ be a pasture and $f:P\to P'$ a pasture morphism into a partial field $P'$. Then $P$ is a mock partial field, and there is a unique morphism $\Pi(f):\Pi(P)\to P'$ such that $f=\pi_P\circ\Pi(f)$. In particular, $\Pi$ extends to a functor from mock partial fields to the subcategory of partial fields in such a way that
 \[
  \Hom(\Pi(P), \ P') \ = \ \Hom(P, \ P')
 \]
 are naturally identified. Moreover, $\pi_P:P\to\Pi(P)$ is an isomorphism if and only if $P$ is a partial field.
\end{lemma}

The universal partial field $\P_M$ of a representable matroid is defined in \cite[section 4]{Pendavingh-vanZwam10a} as the partial field generated by the cross ratios inside the bracket ring $\B_M$, which
itself is the ring generated by symbols $T_B$, together with their multiplicative inverses, for all bases $B$ of $M$, modulo the ideal generated by the $3$-term Pl\"ucker relations. 

The following result gives an independent interpretation of the universal partial field $\P_M$ as $\Pi(F_M)$. 

\begin{prop}\label{prop: the universal partial field as quotient of the foundation}
 Let $M$ be a representable matroid with foundation $F_M$ and universal partial field $\P_M$. Then $F_M$ is a mock partial field, and $\P_M$ is canonically isomorphic to $\Pi(F_M)$. In particular, composing a morphism $\P_M\to P$ with $\rho_{F_M}:F_M\to\Pi(F_M)=\P_M$ defines a bijection
 \[
  \Hom(\P_M, \ P) \ = \ \Hom(F_M,\ P)
 \]
 for every partial field $P$.
\end{prop}

\begin{proof}
 A representation of $M$ over a field $k$ induces a morphism $F_M\to k$, so \autoref{lemma: associated partial field of a mock partial field} implies that $F_M$ is a mock partial field.
 
 Comparing the definitions of the universal ring $R_M=R_{F_M}$ and the bracket ring $\B_M$ from \cite[section 4.1]{Pendavingh-vanZwam10a} shows that $R_P$ is the subring of $\B_M$ generated by the image of $\pi_{F_M}:F_M\to R_M$, considered as a submonoid of $R_M\subset \B_M$. 
 
 The foundation $F_M$ is generated over $\Funpm$ by the universal cross ratios of $M$. The map $\pi_{F_M}:F_M\to\B_M$ sends universal cross ratios of $F_M$ to cross ratios of $\B_M$ in the sense of \cite[section 4.2]{Pendavingh-vanZwam10a}. Since the universal partial field $\P_M$ is generated by such cross ratios together with $-1$ in $\B_M$, the partial field $\Pi(F_M)$ agrees with $\P_M$ as multiplicative submonoids of $\B_M$. 
 
 A comparison of $\gen{N_P}$ with the defining ideal of the bracket ring $\B_M$ shows that all defining relations stem from (possibly degenerate) $3$-term Pl\"ucker relations for $M$. This means that the null sets of $\Pi(F_M)$ and $\P_M$ are equal and the identity map $\Pi(F_M)\to\P_M$ is an isomorphism of pastures.
 
 The equality $\Hom(\P_M, P) = \Hom(F_M,P)$ follows at once from \autoref{lemma: associated partial field of a mock partial field}.
\end{proof}


\section{First examples}
\label{section: first examples}

We determine the foundation for some first classes of matroids: we recall the result for binary and regular matroids from \cite{Baker-Lorscheid21b} and determine the foundation of uniform matroids of rank $2$.

\subsection{The foundation of binary and regular matroids}
\label{subsection: foundations of binary and regular matroids}

By relation \eqref{R0}, the foundation $F_M$ of a matroid $M$ without any $U^2_4$-minors is a quotient of $\Funpm$. In particular, the foundation of a binary matroid is either $\Funpm$ or $\F_2$; cf.\ \cite[Thm.\ 7.32]{Baker-Lorscheid21b}. The foundations of both the Fano matroid $F_7$ and its dual $F_7^\ast$ are isomorphic to $\F_2$. A matroid is regular if and only it has foundation $\Funpm$; cf.\ \cite[Thm.\ 7.35]{Baker-Lorscheid21b}. 

\subsection{The foundation of \texorpdfstring{$U^2_4$}{U(2,4)}}
\label{subsection: foundations of the uniform matroid U24}

We review the account from \cite[section 4.5]{Baker-Lorscheid20}. The uniform matroid $U^2_4$, with ground set $E=\{1,2,3,4\}$, has universal cross ratios
 \[
  x \ = \ \cross1234{}, \qquad y \ = \ \cross1324{}, \qquad \cross1243{}, \qquad \cross1342{}, \qquad \cross1423{}, \qquad \cross1432{}
 \]
 (up to the relations \eqref{Rs}), which satisfy the relations
 \[
  \cross1243{} \ = \ x^{-1}, \qquad \cross1342{} \ = \ y^{-1}, \qquad \cross1423{} \ = \ -x^{-1}y, \qquad \cross1432{} \ = \ -xy^{-1} \qquad 
 \]
 (using the relations \eqref{R1} and \eqref{R2}) and $x+y=1$ (relation \eqref{R+}). Thus the foundation of $U^2_4$ is 
 \[
  \U \ = \ \pastgenn\Funpm{x,y}{x+y-1}.
 \]

\subsection{The foundation of \texorpdfstring{$U^2_5$}{U(2,5)}}
\label{subsection: foundations of the uniform matroid U25}

In this section, we establish a fact which was stated without proof in \cite[Prop. 5.4]{Baker-Lorscheid20}. The assertion is that the foundations of $U^2_5$ and $U^3_5$ are isomorphic to the pasture
\[
 \V \ = \ \pastgenn{\Funpm}{x_1,\dotsc,x_5}{x_i+x_{i-1}x_{i+1}-1|i=1,\dotsc,5},
\]
where the subscript $i$ has to be read `modulo $5$,' i.e.,\ $x_0=x_5$ and $x_6=x_1$.

\begin{prop}\label{prop: foundation of U25}
 The foundation of $U^2_5$ is isomorphic to $\V$.
\end{prop}

\begin{proof}
 For the proof, we choose $E= \Z / 5\Z = \{1,\dotsc,5\}$ as the ground set for $U^2_5$. 
 Since the set of bases consist of precisely all $2$-subsets of $E$, we obtain for each tuple $(\emptyset,i,j,k,l)$ with pairwise distinct $i,j,k,l\in E$ a cross ratio
 \[
  \cross ijkl \ = \ \cross ijkl\emptyset.
 \]
 
 Using the relations of type \eqref{Rs}, \eqref{R1}, and \eqref{R2}, we conclude that the foundation $F$ of $U^2_5$ is generated by the cross ratios of the form
 \[
  x_i = \cross{i+1}{i+2}{i+4}{i+3}{} \qquad \text{and} \qquad y_i \ = \ \cross{i+1}{i+4}{i+2}{i+3}{} 
 \]
 for $i=1,\dotsc,5$, which satisfy the relations $x_i+y_i-1\in N_F$ by \eqref{R+}. Since $U^2_5$ has no (dual) Fano minor, the relation \eqref{R-} does not occur. Since every $2$-subset of $E$ is a basis, there are no relations of type \eqref{R0}. Since the rank of $U^2_5$ is smaller than $3$, there are no relations of types \eqref{R4} and \eqref{R5}.
 
 We show in the following that the relations of type \eqref{R3} come down to $y_i=x_{i-1}x_{i+1}$, which concludes the proof of our claim. These relations are of the form
 \[
  \cross ijkl{} \cdot \cross ijlm{} \cdot \cross ijmk{} \ = \ 1
 \]
 for $\{i,j,k,l,m\}=E$. Note that permuting $\{i,j\}$ and $\{k,l,m\}$ yields the same relation, up to permuting the factors and possibly taking inverses of all factors. This leaves us with 10 relations of type \eqref{R3}, one for each basis $\{i,j\}\subset E$. 
 
 More precisely, we choose the following representatives for each basis. For a basis of the form $\{i,i+1\}$, we consider
 \begin{align*}
  1 \ &= \ \cross{i}{i+1}{i+2}{i+3}{} \cdot \cross{i}{i+1}{i+3}{i+4}{} \cdot \cross{i}{i+1}{i+4}{i+2}{} \\
      &= \ x_{i+4}^{-1} \cdot x_{i+2}^{-1} \cdot y_{i+3}
 \end{align*}
 where we use the relations of types \eqref{Rs} and \eqref{R1} to express the cross ratios in terms of the $x_i$ and $y_i$. This yields at once the desired relation $y_i=x_{i-1}x_{i+1}$ and thus $x_i+x_{i-1}x_{i+1}-1\in N_F$. 
 
 We are left with showing that the additional $5$ relations for bases of the forms $\{i,i+2\}$ do not endow additional relations between the $x_i$. We consider the following representatives for these relations:
 \begin{align*}
  1 \ &= \ \cross{i}{i+2}{i+3}{i+4}{} \cdot \cross{i}{i+2}{i+4}{i+1}{} \cdot \cross{i}{i+2}{i+1}{i+3}{} \\
      &= \ y_{i+1}^{-1} \cdot (-x_{i+3}^{-1} \ y_{i+3}) \cdot (-x_{i+4}^{-1} \ y_{i+4}),
 \end{align*}
 where we use the relations of types \eqref{Rs}, \eqref{R1} and \eqref{R2} to express the cross ratios in terms of the $x_i$ and $y_i$. When we substitute $y_i$ by $x_{i-1}x_{i+1}$ in this relation, we find that 
 \[
  y_{i+1}^{-1} \cdot (-x_{i+3}^{-1}\ y_{i+3}) \cdot (-x_{i+4}^{-1}\ y_{i+4}) \ = \ x_i^{-1} \ x_{i+2}^{-1} \ x_{i+3}^{-1} \ x_{i+2} \ x_{i+4} \ x_{i+4}^{-1} \ x_{i+3} \ x_i \ = \ 1
 \]
 holds already in $F$, which shows that the latter family of relations do not imply additional relations between the cross ratios $x_1,\dotsc,x_5$. This concludes the proof.
\end{proof}

\begin{cor}\label{cor: foundation of U35}
 The foundation of $U^3_5$ is isomorphic to $\V$.
\end{cor}

\begin{proof}
 This follows at once from \autoref{prop: foundation of U25} by taking duals; cf.\ \autoref{prop: foundation of dual matroids and simplifications} and \cite[Prop.\ 4.8]{Baker-Lorscheid20}.
\end{proof}


\subsection{The foundation of \texorpdfstring{$U^2_n$}{U(2,n)}}
\label{subsection: foundation of U2n}

In this section, we identify the foundation of the uniform matroid $U^2_{k+3}$ with Semple's $k$-regular partial field, cf.\ \cite{Semple97}, \cite{Semple98}. Interestingly, this gives a different presentation of $\V$ from \autoref{prop: foundation of U25}, a fact that we comment on further at the end of this section.

Semple's $k$-regular partial field is defined as $\cR_k=\big(\cA_k,\Q(\alpha_1,\dotsc,\alpha_k)\big)$ where $\alpha_1,\dotsc,\alpha_k$ are algebraically independent elements over $\Q$ and 
\[
 \cA_k \ = \ \bigg\{ \, \prod_{i=1}^s (\beta_i-\gamma_i)^{n_i} \, \bigg| \, s\geq0,\, \beta_i,\gamma_i\in\{0,1,\alpha_1,\dotsc,\alpha_k\},\, n_i\in\Z \, \bigg\}.
\]
Note that $\cA_k$ contains $0=0-0$ and $-a=(0-1)a$ for every $a\in\cA_k$. Thus $\cR_k$ is indeed a partial field. We define $\U_k$ as the pasture associated with $\cR_k$, and (by abuse of terminology) refer to it also as the $k$-regular partial field.

Note that the $0$-regular partial field $\U_0$ coincides with the regular partial field $\Funpm$, and the $1$-regular partial field $\U_1$ is isomorphic to the near-regular partial field $\U$. 

An explicit description of the pasture $\U_k$ in terms of generators and relations is as follows.

\begin{lemma}\label{lemma: presentation of the k-regular partial field}
 Let $k\geq0$ and define $\alpha_{-1}=0$ and $\alpha_0=1$. Then
 \[
  \U_k \ = \ \pastgenn\Funpm{\alpha_j-\alpha_i\mid -1\leq i<j\leq k,\ j\neq 0}{S_k},
 \]
 where $S_k$ consists of the terms
 \[
  \frac{\alpha_i-\alpha_l}{\alpha_j-\alpha_l} \ + \ \frac{\alpha_i-\alpha_j}{\alpha_l-\alpha_j} \ - \ 1 
 \]
 for all $-1\leq i<j<l\leq k$ and 
 \[
  \frac{(\alpha_i-\alpha_l)(\alpha_j-\alpha_m)}{(\alpha_i-\alpha_m)(\alpha_j-\alpha_l)} \ + \ \frac{(\alpha_i-\alpha_j)(\alpha_l-\alpha_m)}{(\alpha_i-\alpha_m)(\alpha_l-\alpha_j)} \ - \ 1 
 \]
 for all $-1\leq i<j<l<m\leq k$.
\end{lemma}

\begin{proof}
 It is clear from the definition of $\cR_k$ that it is generated by the terms $\alpha_j-\alpha_i$ with $-1\leq i<j\leq k$ and $-1$ as a multiplicative monoid, where we can exclude $\alpha_0-\alpha_{-1}=1$ from the generating set. Since the multiplicative monoid of $\U_k$ agrees with that of $\cR_k$ by definition, this shows that $\U_k$ is generated by the terms $\alpha_j-\alpha_i$ with $-1\leq i<j\leq k$ and $j>0$ over $\Funpm$.
 
 The null set $N_{\U_k}$ of $\U_k$ is a union of $\U_k^\times$-orbits, and by \cite[Prop.\ 3.6]{Baker-Lorscheid21}, $N_{\U_k}/\U_k^\times$ is in bijection with the set of hexagons in $\U_k$. We show that the hexagons in $\U_k$ correspond to the terms in $S_k$.
 
 Semple determines in \cite[Thm.\ 3.2]{Semple97} the fundamental elements of $\cR_k$ as the elements
  \[
  \frac{\alpha_i-\alpha_l}{\alpha_j-\alpha_l}
 \]
 for pairwise distinct $i,j,l\in\{-1,\dotsc,k\}$ and 
 \[
  \frac{(\alpha_i-\alpha_l)(\alpha_j-\alpha_m)}{(\alpha_i-\alpha_m)(\alpha_j-\alpha_l)}
 \]
 for all pairwise distinct $i,j,l,m\in\{-1,\dotsc,k\}$. Note that we can permute $\{i,j,l,m\}$ with elements of the Klein-four group $V=\gen{(ij)(lm),(il)(jm)}$ in the second expression without changing the fundamental element. Thus we can assume that $i$ is the smallest index among $i,j,l,m$.
 
 These fundamental elements are grouped into the following hexagons. The first type of fundamental element appears in the hexagon
 \[
 \beginpgfgraphicnamed{tikz/fig12}
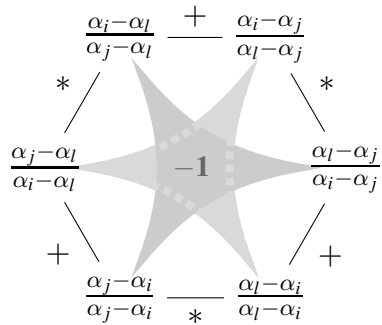
  \begin{tikzpicture}[baseline={([yshift=-.5ex]current bounding box.center)},x=1.0cm,y=1.0cm]
   \draw[line width=3pt,color=gray!30,fill=gray!30,bend angle=20] (60:1.4) to[bend left] (180:1.4) to[bend left] (300:1.4) to[bend left] cycle;
   \draw[line width=3pt,color=gray!40,fill=gray!40,bend angle=20] (0:1.4) to[bend left] (120:1.4) to[bend left] (240:1.4) to[bend left] cycle;
   \draw[line width=3pt,color=gray!30,dotted,bend angle=20] (60:1.4) to[bend left] (180:1.4) to[bend left] (300:1.4) to[bend left] cycle;
   \node (-1) at (0,0) {\small \textcolor{black!60}{{$\mathbf{-1\ }$}}};
   \node (a) at (120:2) {$\tfrac{\alpha_i-\alpha_l}{\alpha_j-\alpha_l}$};
   \node (b) at ( 60:2) {$\tfrac{\alpha_i-\alpha_j}{\alpha_l-\alpha_j}$};
   \node (c) at (180:2) {$\tfrac{\alpha_j-\alpha_l}{\alpha_i-\alpha_l}$};
   \node (d) at (  0:2) {$\tfrac{\alpha_l-\alpha_j}{\alpha_i-\alpha_j}$};
   \node (e) at (240:2) {$\tfrac{\alpha_j-\alpha_i}{\alpha_j-\alpha_i}$};
   \node (f) at (300:2) {$\tfrac{\alpha_l-\alpha_i}{\alpha_l-\alpha_i}$};
   \path (a) edge node[auto] {$+$} (b);
   \path (b) edge node[auto] {$*$} (d);
   \path (d) edge node[auto] {$+$} (f);
   \path (f) edge node[auto] {$*$} (e);
   \path (e) edge node[auto] {$+$} (c);
   \path (c) edge node[auto] {$*$} (a);
  \end{tikzpicture}
 \endpgfgraphicnamed
 \]
 and the second type of fundamental element appears in the hexagon
 \[
 \beginpgfgraphicnamed{tikz/fig13}
  \begin{tikzpicture}[baseline={([yshift=-.5ex]current bounding box.center)},x=1.8cm,y=1.0cm]
   \draw[line width=3pt,color=gray!30,fill=gray!30,bend angle=20] (60:1.4) to[bend left] (180:1.4) to[bend left] (300:1.4) to[bend left] cycle;
   \draw[line width=3pt,color=gray!40,fill=gray!40,bend angle=20] (0:1.4) to[bend left] (120:1.4) to[bend left] (240:1.4) to[bend left] cycle;
   \draw[line width=3pt,color=gray!30,dotted,bend angle=20] (60:1.4) to[bend left] (180:1.4) to[bend left] (300:1.4) to[bend left] cycle;
   \node (-1) at (0,0) {\small \textcolor{black!60}{{$\mathbf{-1\ }$}}};
   \node (a) at (120:2) {$\tfrac{(\alpha_i-\alpha_l)(\alpha_j-\alpha_m)}{(\alpha_i-\alpha_m)(\alpha_j-\alpha_l)}$};
   \node (b) at ( 60:2) {$\tfrac{(\alpha_i-\alpha_j)(\alpha_l-\alpha_m)}{(\alpha_i-\alpha_m)(\alpha_l-\alpha_j)}$};
   \node (c) at (180:2) {$\tfrac{(\alpha_i-\alpha_m)(\alpha_j-\alpha_l)}{(\alpha_i-\alpha_l)(\alpha_j-\alpha_m)}$};
   \node (d) at (  0:2) {$\tfrac{(\alpha_i-\alpha_m)(\alpha_l-\alpha_j)}{(\alpha_i-\alpha_j)(\alpha_l-\alpha_m)}$};
   \node (e) at (240:2) {$\tfrac{(\alpha_i-\alpha_j)(\alpha_m-\alpha_l)}{(\alpha_i-\alpha_l)(\alpha_m-\alpha_j)}$};
   \node (f) at (300:2) {$\tfrac{(\alpha_i-\alpha_l)(\alpha_m-\alpha_j)}{(\alpha_i-\alpha_j)(\alpha_m-\alpha_l)}$};
   \path (a) edge node[auto] {$+$} (b);
   \path (b) edge node[auto] {$*$} (d);
   \path (d) edge node[auto] {$+$} (f);
   \path (f) edge node[auto] {$*$} (e);
   \path (e) edge node[auto] {$+$} (c);
   \path (c) edge node[auto] {$*$} (a);
  \end{tikzpicture}
 \endpgfgraphicnamed
 \]
 The terms in $S_k$ correspond to the first rows in these hexagons, where we can assume that $i<j<l$ and $i<j<l<m$, respectively, after rotating or reflecting the hexagon. This shows that $S_k$ generates the null set $N_{\U_k}$, and concludes the proof.
\end{proof}

\begin{prop}\label{prop: U_k is the foundation of U2k+3}
 Let $k\geq0$. Then the foundation of $U^2_{k+3}$ is isomorphic to $\U_k$.
\end{prop}

\begin{proof}
 Van Zwam shows in \cite[Thm.\ 3.3.24]{vanZwam09} that the universal partial field $\P_M$ of $M=U^2_{k+3}$ is $\cR_k$ or, translated into the language of pastures, the $k$-regular partial field $\U_k$. By \autoref{prop: the universal partial field as quotient of the foundation}, we have 
 \[
  \P_M \ = \ \Pi(F_M) \ = \ \past{F_M}{\gen{a-b\mid a-b\in\gen{N_M}_\Z}},
 \]
 where $\gen{N_M}_\Z$ is the ideal of the group ring $\Z[F_M^\ast]$ generated by all terms $c+d+e$ in the null set $N_M$ of $F_M$. Moreover, \cite[Cor.\ 7.13]{Baker-Lorscheid21b} identifies the unit group $F_M^\times$ of the foundation with the inner Tutte group $\T^{(0)}_M$ of $M$.
 
 Dress and Wenzel show in \cite[Thm.\ 8.1]{Dress-Wenzel89} that 
 \[
  \T^{(0)}_M \ \simeq \ \{\pm1\} \times \Z^{\binom{k+3}{2}-(k+3)}
 \]
 for $M=U^2_{k+3}$. And by definition, we see that $\cR_k^\times$ is the product of $\{\pm 1\}$ with the free abelian group with basis $\alpha_i-\alpha_j$, where $(i,j)$ range over all pairs of integers with $-1\leq i< j\leq k$ and $j\neq0$. Thus the free rank of $\cR_k^\times$ is 
 \[
  \#\bigg\{(i,j)\in\Z^2 \, \bigg| \, -1\leq i<j\leq k,\ j\neq 0\bigg\} \ = \Bigg(\sum_{i=-1}^{k-1} (k-i)\Bigg)-1 \ = \ \binom{k+3}{2}-(k+3),
 \]
 which equals the free rank of $\T^{(0)}$. We conclude that the surjective group homomorphism
 \[
  \T^{(0)} \ \simeq \ F_M^\times \ \longrightarrow \ \U_k^\times \ \simeq \ \cR_k^\times
 \]
 is an isomorphism of groups, and thus the quotient map $F_M\to\U_k$ is a bijection.
 
 In order to show that the respective null sets agree, we count the number of fundamental elements in both pastures and show that $\U_k$ does not have more fundamental elements than $F_M$. The fundamental elements of $F_M$ are the cross ratios $\cross abcd{}$ of $M$. Since every $2$-subset of $U^2_n$ is independent, every $4$-tuple of pairwise different elements of $M$ yields a cross ratio. Modulo the relations \eqref{Rs}, which identifies the cross ratios for $4$ permutations of $a$, $b$, $c$ and $d$, we count $n(n-1)(n-2)(n-3)/4=6\cdot\binom n4$ cross ratios. 
 
 We have to exclude further identifications of cross ratios to ensure that this is the correct count. Relation \eqref{R-} does not appear since $U^2_n$ has no minors of types $F_7$ and $F_7^\ast$. Relation \eqref{R0} does not appear since all $2$-subsets are independent. Relation \eqref{R4} does not appear since $U^2_n$ has no minors of rank $3$ and relation \eqref{R5} does not appear since $U^2_n$ does not have minors with parallel elements. Relations \eqref{Rs}, \eqref{R1} and \eqref{R2} express that every $U^2_4$-minor leads to $6$ \textit{a priori} distinct cross ratios, which entered our count already, and \eqref{R3} are relations that are imposed by rank $2$ minors on $5$ elements, which are all isomorphic to $U^2_5$. By \autoref{prop: foundation of U25}, none of the cross ratios of distinct $U^2_4$-minors are identified in the foundation of $U^2_5$. By \cite[Thm.~4.23]{Baker-Lorscheid20} (also cf.\ \autoref{thm: fundamental presentation} for a more concise version), the foundation of $U^2_n$ is the colimit of the foundations of all embedded minors of types $U^2_4$ and $U^2_5$, which shows that there are no further identifications by minors on more than $5$ elements.
 
 The fundamental elements of $\U_k$ are described in \autoref{lemma: presentation of the k-regular partial field}: for each $3$-subset of $\{-1,0,1,\dotsc,k\}$, there are $6$ fundamental elements of the form
 \[
  \frac{\alpha_i-\alpha_l}{\alpha_j-\alpha_l},
 \]
 and for every $4$-subset of $\{-1,0,1,\dotsc,k\}$, there are $6$ fundamental elements of the form
 \[
  \frac{(\alpha_i-\alpha_l)(\alpha_j-\alpha_m)}{(\alpha_i-\alpha_m)(\alpha_j-\alpha_l)},
 \]
 Since $n=k+3$, the number of fundamental elements in $\U_k$ is at (most equal) to $6\cdot\Big(\,\binom{n-1}3+\binom{n-1}4\,\Big)=6\cdot\binom n4$, which is the number of fundamental elements in $F_M$. This shows that the bijection $F_M\to\U_k$ is an isomorphism of pastures, as desired.
\end{proof}


Semple characterizes all automorphisms of the $k$-regular partial field $\U_k$ in \cite[Thm.\ 4.2]{Semple97}. The group structure of $\Aut(\U_k)$ is, however, not visible from this description. Coupling Semple's result with our methods yields the following result:

\begin{prop}\label{prop: automorphism group of the k-regular partial field}
 The automorphism group of $\U$ is isomorphic to the symmetric group $S_3$, and for $k\geq 2$, the automorphism group of $\U_k$ is isomorphic to $S_{k+3}$.
\end{prop}

\begin{proof}
 The functoriality of the foundation with respect to minor embeddings, and in particular with respect to automorphisms, yields a group homomorphism 
 \[
  \varphi_k: \ S_{k+3} \ = \ \Aut(U^2_{k+3}) \ \longrightarrow \ \Aut(\U_{k}).
 \]
 In \cite[Prop.\ 5.6]{Baker-Lorscheid20}, we have shown that $\varphi_1:S_4=\Aut(U^2_4)\to \Aut(\U)$ is surjective with kernel $V=\gen{(12)(34),(14)(23)}$, i.e.\ $\Aut(\U)\simeq S_3$, which yields the first claim. 
 
 If $k\geq 2$, then every permutation $\sigma\in S_{k+3}$ that fixes $j$, $k$ and $l$ sends the cross ratio $\cross ijkl{}$ to $\cross {\sigma(i)}jkl{}$, which differs from $\cross ijkl{}$ if $\sigma(i)\neq i$. Since $k+3\geq 5$ for $k\geq 2$, we find such a cross ratio, which shows that $\varphi_k$ is injective.
 
 Semple's description of the automorphisms of $\cR_k$ in \cite[Thm.\ 4.2]{Semple97} determines all possible images of $\alpha_1,\dotsc,\alpha_k$. From the outset, it is not clear that the corresponding automorphisms are pairwise distinct---and they are not for $k=1$--- but this description gives an upper bound for $\#\Aut(\cR_k)$. There are $4$ cases in \cite[Thm.\ 4.2]{Semple97}: the first $3$ cases describe (up to) $(k+2)!$ automorphisms, the fourth case contains an additional element, which leads to (up to) $k\cdot(k+2)!$ automorphisms. We deduce the upper bound
 \[
  \#\Aut(\U_k) \ = \ \#\Aut(\cR_k) \ \leq \ 3\cdot(k+2)! + k\cdot (k+2)! \ = \ (k+3)!.
 \]
 Since $\varphi_k$ is injective for $k\geq2$ and $\# S_{k+3}=(k+3)!$, we conclude that $\varphi_k$ is surjective as well and thus an isomorphism.
\end{proof}

\begin{rem}\label{rem: comparison between V and U_2}
 Comparing \autoref{prop: U_k is the foundation of U2k+3} with \autoref{prop: foundation of U25} shows that $\U_2\simeq\V$, which yields an arguably more symmetric description of the $2$-regular partial field $\U_2$. An explicit isomorphism $\varphi:\V\to\U_2$ is given by sending the generators $x_1,\dotsc,x_5$ of $\V$ to
 \[
  \varphi(x_1) =\alpha, \quad \varphi(x_2) = 1-\beta, \quad \varphi(x_3) = \tfrac\beta\alpha, \quad \varphi(x_4) = \tfrac{\alpha-\beta}{\alpha(1-\beta)}, \quad \varphi(x_5) = \tfrac{1-\alpha}{1-\beta},
 \]
  where we write $\alpha=\alpha_1$ and $\beta=\alpha_2$ for better readability. It is also possible to deduce $\Aut(\V)\simeq S_5$ without too much effort directly from the definition of $\V$, without reference to Semple's and Pendavingh--van Zwam's results. 
  
  It seems desirable to develop more ``symmetric'' descriptions of the higher $k$-regular partial fields $\U_k$, in the vein of the isomorphisms $\U_1\simeq\U$ and $\U_2\simeq\V$.  One obstacle towards this goal is that van Zwam's result $\P_{U^2_{k+3}}\simeq\cR_k$ relies on partial field techniques for which it is not clear how to generalize them to pastures. 
\end{rem}

\subsection{Finite projective spaces}
\label{subsection: finite projective spaces}

Let $PG(d,q)$ be the $d$-dimensional projective space over the finite field $\F_q$ with $q$ elements, which is a rank $d+1$ matroid on $q^d+\dotsb+q+1$ elements. Results on the inner Tutte group of $GP(d,q)$ and the universal partial of a certain minor allow us to deduce that the foundation of $PG(d,q)$ is $\F_q$. By a similar type of reasoning, we determine the foundation of any non-Desarguesian planes as $\K$.

The \emph{extended rank $3$ Dowling geometry $Q_q^+$ of $\F_q^\times$} is the restriction of $PG(2,q)$ to $L_1\cup L_2\cup L_3\cup e$, where $L_1$, $L_2$ and $L_3$ are three lines in $PG(2,q)$ with empty intersection and $e$ is a point that lies on neither of these lines.

\begin{prop}\label{prop: foundation of projective spaces}
 The foundation of the $d$-dimensional projective space $PG(d,q)$ is $\F_q$ for all $d\geq2$ and all prime powers $q$.
\end{prop}

\begin{proof}
 By \cite[Thm.\ 3.6]{Dress-Wenzel90}, the inner Tutte group of $M=PG(d,q)$ is $\F_q^\times$, which equals the unit group of the foundation $F_M$ of $M$ by \cite[Cor.\ 7.13]{Baker-Lorscheid21b}. By \cite[Lemma 7.48]{Baker-Lorscheid21b}, the universal partial field $\P_M$ of $M$ is a quotient of $F_M$, and thus can have at most $q$ elements. By \cite[p.\ 660]{Oxley92}, $PG(2,q)$ is representable over $\F_q$, so there is a morphism $\P_M\to\F_q$.
 
 The extended rank $3$ Dowling geometry $N=Q_q^+$ of $\F_q^\times$ is a minor of $PG(2,q)$ whose universal partial field is $\P_N=\F_q$ by \cite[Thm.\ 3.3.25]{vanZwam09}. Since $PG(2,q)$ is a minor of $PG(d,q)$, we get morphisms $\F_q=\P_N\to\P_M\to\F_q$, which are isomorphisms since $\P_M$ has at most $q$ elements and since every homomorphism from a field into a partial field is injective. Thus $\P_M=\F_q$, which shows that $\P_M$ is the trivial quotient of $F_M$ and therefore $F_M=\P_M=\F_q$.
\end{proof}

The fact (\cite[Thm.\ 3.3.25]{vanZwam09}) that the universal partial field of extended Dowling geometries $Q_q^+$ is $\F_q$ suggests the following:

\begin{problem}
 Is $\F_q$ the foundation of the extended rank $3$ Dowling geometry $Q_q^+$ of $\F_q^\times$? More generally, what are the minimal matroids with foundation $\F_q$?
\end{problem}

The answer is yes for $\F_2$, $\F_3$ and $\F_4$, but the extended Dowling geometries $Q_3^+$ and $Q_4^+$ are not minimal for their respective foundations; cf.\ \autoref{Ffinitefields}.

\begin{prop}\label{prop: foundation of non-Desarguesian projective plane}
 The foundation of a non-Desarguesian projective plane is $\K$.
\end{prop}

\begin{proof}
 By \cite[Thm.\ 3.7]{Dress-Wenzel90}, the inner Tutte group of a non-Desarguesian projective plane is trivial. There are only two pastures $P$ with $P^\times=\{1\}$, which are $\F_2$ and $\K$. A non-Desarguesian plane cannot be represented over any field since Desargues's theorem is violated, so we conclude that its foundation is $\K$.
\end{proof}

\section{The foundation of a direct sum}
\label{subsection: foundation of direct sums}

Our goal in this section is to prove:

\begin{thm}\label{thm: foundations of direct sums}
 Let $M_1$ and $M_2$ be matroids. Then $F_{M_1\oplus M_2} \simeq F_{M_1}\otimes F_{M_2}$.
\end{thm}

Before giving the proof, we note that \autoref{thm: foundations of direct sums} is a special case of a more general result on the foundation of generalized parallel connections proved in \cite{Baker-Lorscheid-Walsh-Zhang}. However, the case considered here is substantially simpler, so it seems worthwhile to give a self-contained proof. A related result, also proved in \cite{Baker-Lorscheid-Walsh-Zhang}, is that the foundation of a $2$-sum of matroids is also the tensor product of the foundations. 

We begin with the following well-known facts about direct sums:

\begin{prop} \cite[4.2.15]{Oxley92} \label{prop:hyperplanes of the direct sum}
Let $M_1,M_2$ be matroids on $E_1$ and $E_2$, respectively. A subset $H$ of $E(M_1 \oplus M_2) = E_1 \sqcup E_2$ is a hyperplane of $M_1 \oplus M_2$ if and only if it satisfies one of the following:
\begin{enumerate}
\item $H\cap E_1$ is a hyperplane in $M_1$ and $H$ contains $E_2$.
\item $H\cap E_2$ is a hyperplane in $M_2$ and $H$ contains $E_1$.
\end{enumerate}
\end{prop}

\begin{prop} \cite[4.2.13]{Oxley92} \label{prop:rank function of the direct sum}
With notation as in the previous proposition, if $r,r_1,r_2$ are the rank functions of $M_1 \oplus M_2$, $M_1$, and $M_2$, respectively, then for any subset $X$ of $E(M_1 \oplus M_2)$ we have:
\[
r(X) = r_1(X\cap E_1) +r_2(X\cap E_2).
\]
\end{prop}

Let $F$ be a pasture and let $M$ be a matroid. 
Let $\cX^I_{M}(F)$ denote the set of (weak) $F$-representations of $M$ up to isomorphism, and let
$\cX^R_{M}(F)$ denote the set of rescaling equivalence classes of $F$-representations of $M$.
In order to prove \autoref{thm: foundations of direct sums}, it suffices to show that for every pasture $F$ there is a natural bijection from $\cX^I_{M_1 \oplus M_2}(F)$ to $\cX^I_{M_1}(F) \times \cX^I_{M_2}(F)$, and that these bijections are functorial in $F$.
Indeed, this immediately implies the same thing with $\cX^I$ replaced by $\cX^R$, and (in light of the universal property of tensor products) it also implies that both $F_{M_1 \oplus M_2}$ and $F_{M_1} \otimes F_{M_2}$ represent the functor $\cX^R_{M_1 \oplus M_2}(F)$. Hence $F_{M_1 \oplus M_2}$ and $F_{M_1} \otimes F_{M_2}$ are (canonically) isomorphic.\footnote{So, in fact, what we are really proving is that the \emph{universal pasture} (cf.~\cite[Definition 2.17 and Theorem 2.18]{Baker-Lorscheid20}) of the direct sum is the tensor product of the universal pastures.}

\begin{proof}[Proof of \autoref{thm: foundations of direct sums}]
As discussed above, it suffices to define a bijection which is functorial in $F$ from $\cX^I_{M}(F)$ to $\cX^I_{M_1}(F) \times \cX^I_{M_2}(F)$.
We use the description of $\cX^I_M(F)$ in terms of modular systems of $F$-hyperplane functions for $M$ given in \cite[Theorem 2.16]{Baker-Lorscheid20}. (This can be thought of as a cryptomorphic description of weak $F$-matroids in terms of hyperplanes.)

Given a modular system $\cH = \{ f_H \}_{H \in \cH(M)}$ of $F$-hyperplane functions for $M$, we can define corresponding sets of $F$-hyperplane functions $\cH_i$ for $i=1,2$ by the formulas $f_{H_1}(e) = f_{H_1 \oplus E_2}(e)$ and $f_{H_2}(e) = f_{E_1 \oplus H_2}(e)$.
Conversely, given a modular system $\cH_i = \{ f_H : E \to F \}_{H \in \cH(M_i)}$ of $F$-hyperplane functions for $M_i$ ($i=1,2$), we can define (using \autoref{prop:hyperplanes of the direct sum}) a corresponding set of $F$-hyperplane functions $\cH$ for $M := M_1 \oplus M_2$ by the formula $f_{H_1 \oplus E_2}(e) = 0$ if $e \in E_2$ and $f_{H_1 \oplus E_2}(e) = f_{H_1}(e)$ if $e \in E_1$, and similarly $f_{E_1 \oplus H_2}(e) = 0$ if $e \in E_1$ and $f_{E_1 \oplus H_2}(e) = f_{H_2}(e)$ if $e \in E_2$.

It is easy to see that these constructions are inverse to one another and functorial in $F$, so it suffices to prove that each construction yields a modular system.
We verify that the second map yields a modular system, and leave the similar verification for the first map to the reader.

Suppose $H_1,H_2,H_3$ are a modular triple of hyperplanes for $M$, with mutual pairwise intersection equal to the corank 2 flat $P$. Up to symmetry (swapping the roles of $E_1$ and $E_2$), there are just two cases to consider:

\begin{enumerate}
 \item Suppose $E_1$ is contained in $H_1$ and $E_2$ is contained in $H_2$. By symmetry, we can assume that $H_3$ contains $E_1$. Then, $H_3 \cap H_1 = P$ contains $E_1$, and hence $H_2$ contains $E_1$ as well, which contradicts the assumption that $E_2 \subseteq H_2$. So this case does not occur.
 \item Suppose $E_1$ is contained in both $H_1$ and $H_2$. Then $E_1$ is contained in $H_3$ as well, and so by \autoref{prop:rank function of the direct sum},
 the restrictions $H_1', H_2', H_3'$ of $H_1,H_2,H_3$ to $E_2$ form a modular triple of hyperplanes of $M_2$.
 Consequently, we know that the functions $f_{H_i'}$ are linearly dependent, which implies that the original functions $f_{H_i}$ are also linearly dependent.
\end{enumerate}
\end{proof}

\subsection{Indecomposable foundations}
\label{subsection: indecomposable foundations}

We conclude this section with a discussion about \emph{indecomposable foundations}, which are foundations $F_M$ that are not isomorphic to a non-trivial tensor product $P_1\otimes P_2$ of two pastures. In this context, {\em non-trivial} means that neither $P_1$ nor $P_2$ is a quotient of $\Funpm$, i.e.,\ neither is isomorphic to a pasture in $\{ \Funpm, \F_2, \F_3, \K \}$. The reason for excluding quotients $Q$ of $\Funpm$ is that $Q\otimes Q=Q$ for such a pasture.

\autoref{thm: foundations of direct sums} tells us that the direct sum decomposition $M=M_1\oplus M_2$ of a matroid induces a tensor decomposition $F_{M}\simeq F_{M_1}\otimes F_{M_2}$ of its foundation. In fact, the same is true for the decomposition into a $2$-sum $M=M_1\oplus_2 M_2$: its foundation is $F_{M}\simeq F_{M_1}\otimes F_{M_2}$; cf.\ \cite{Baker-Lorscheid-Walsh-Zhang}.

The foundation of many $3$-connected matroids is indecomposable. For example:

\begin{prop}\label{prop: indecomposable quarternary foundations}
 The foundation of a $3$-connected quarternary matroid is indecomposable.
\end{prop}

\begin{proof}
 Let $M$ be a quarternary matroid and let $F_M$ be its foundation. Suppose that $F_M=P_1\otimes P_2$ is a non-trivial decomposition into pastures $P_1$ and $P_2$, neither of which is a quotient of $\Funpm$. Since $F_M$ is generated by the universal cross ratios of $M$, which are fundamental elements, both $P_1$ and $P_2$ are generated by their respective fundamental elements. In particular, each factor contains a fundamental element.
 
 The fundamental elements of $\F_4=\{0,1,a,a+1\}$ are $a$ and $a+1$. Therefore the composition of any morphism $F_M\to\F_4$ with the canonical inclusion $P_i\to F_M$ (for $i=1,2$) is a surjection $P_i\twoheadrightarrow\F_4$. Composing this surjection with the nontrivial field automorphism $\F_4\to\F_4$ yields two additional morphisms $P_i\twoheadrightarrow\F_4$.
 
 The universal property of the tensor product can be expressed as a bijection 
 \[
 \Hom(F_M,\F_4)=\Hom(P_1,\F_4)\times\Hom(P_2,\F_4).
 \] 
 The previous paragraph thus guarantees that there are at least $4$ distinct morphisms $F_M\to\F_4$. By a well-known theorem of Kahn (\cite{Kahn88}), $M$ cannot be $3$-connected, which establishes our claim.
\end{proof}

\begin{problem}
 Instead of using Kahn's theorem to deduce Proposition~\ref{prop: indecomposable quarternary foundations}, can one give a direct proof of the proposition using the theory of foundations and deduce Kahn's theorem as a corollary?
\end{problem}
  
It is not true in general that the foundation of a $3$-connected matroid is indecomposable (this insight was shared with us by Nathan Bowler); cf.\ \autoref{Fdyadic} for a $3$-connected matroid with foundation $\D\otimes\D$. In so far, we wonder:

\begin{problem}
 What are necessary and sufficient (or even just sufficient) conditions on a matroid $M$ for its foundation to be indecomposable? For example, is the foundation of a $4$-connected matroid indecomposable?
\end{problem}

One can refine the notion of indecomposable foundations using the fundamental diagram of a matroid. Namely, given a diagram of pastures $\cF$ with connected components $\cF_1,\dotsc,\cF_r$, we have $\colim\cF=\otimes_{i=1}^r\colim\cF_i$. Embedded minors of type $F_7$ and $F_7^\ast$ appear as isolated points in the fundamental diagram $\cE_M$, and we call them the \emph{trivial connected components of $\cE_M$}. We call $\cE_M$ \emph{essentially connected} if it has at most one nontrivial connected component. In other words, $\cE_M$ is essentially connected if all vertices corresponding to embedded $U^2_4$-minors of $M$ lie in the same connected component of $\cE_M$. Since the foundation $\F_2$ of $F_7$ and $F_7^\ast$ is a trivial tensor factor, it follows that the foundation of a matroid $M$ with essentially connected fundamental diagram is indecomposable. We can therefore sharpen the previous problem as follows:

\begin{problem}
 What are necessary and sufficient (or even just sufficient) conditions on a matroid $M$ for its fundamental diagram to be essentially connected? 
\end{problem}

The fundamental diagram of Bowler's ternary spike mentioned above is not essentially connected, so the condition that $M$ is ternary is not sufficient.


\section{Fundamental presentations}
\label{section: presentations of the foundation}

In this section, we develop presentations of the foundation in terms of minors and in terms of sublattices of the lattice of flats of a matroid.


\subsection{The fundamental presentation by embedded minors}
\label{subsection: presentation of the foundation by embedded minors}

The relations between universal cross ratios from \autoref{thm: fundamental presentation of foundations in terms of bases} involve only few elements of the ground set, which means that these relations stem from minors of fairly small sizes. This has already been noted in \cite[Thm.\ 4.23]{Baker-Lorscheid20}, which describes the foundation of a matroid as the colimit of the foundations of its embedded minors of small size. We improve upon this result by narrowing down the list of embedded minors that we need to consider to a minimal possible set.

Let $M$ be a matroid with ground set $E$. An \emph{embedded minor of $M$} is a minor $M\minor JI$ of $M$ with a fixed choice of subsets $I$ and $J$ of $E$, where $I$ is coindependent and $J$ is independent. Note that every minor of $M$ can be expressed in this form; this is a consequence of the Scum Theorem, cf.\ \cite[section 1.3]{Baker-Lorscheid20}. 

A \emph{minor embedding} $N\hookrightarrow M$ is an isomorphism $N\simeq M\minor JI$ of $N$ with an embedded minor $M\minor JI$ of $M$. In particular, every embedded minor comes with a tautological minor embedding $M\minor JI\hookrightarrow M$.

Every minor embedding $N\simeq M\minor JI\hookrightarrow M$ induces a morphism of foundations:
\[
 \begin{array}{ccrcl}
  F_N & \stackrel\sim\longrightarrow & F_{M\minor JI} & \longrightarrow & F_M \\[5pt]
      &        & \cross abcd{J'}  & \longmapsto     & \cross abcd{J'\cup J}
 \end{array}
\]
where $(J';a,b,c,d)$ varies through all tuples in $\Omega_{M\minor JI}^\octa$; cf.\ \cite[Prop.\ 4.9]{Baker-Lorscheid20}. Note that this pasture morphism is in general not injective.

\begin{df} \label{df:fundamental diagram}
 Let $M$ be a matroid. The \emph{fundamental diagram of $M$} is the diagram $\cE_M$ of all embedded minors $N=M\minor JI$ of $M$ of types
 \[
  \begin{aligned}
   & U^2_4, && \text{($4$ elements)} \\
   &C_5,\quad C_5^\ast,\quad U^2_5,\quad U^3_5, &\hspace{2cm} & \text{($5$ elements)} \\
   & U^1_2\oplus U^2_4, &&\text{($6$ elements)} \\
   & F_7,\quad F_7^\ast && \text{($7$ elements),}
  \end{aligned}
 \]
 together with all minor embeddings. We denote by $F(\cE_M)$ the family of foundations of all embedded minors in $\cE_M$, together with the pasture morphisms that are induced by the minor embeddings.
 We write $\cS := \{ U^2_4, U^2_5, U^3_5, C_5, C_5^\ast, U^2_4\oplus U^1_2, F_7, F_7^\ast \}$, and refer to embedded minors of $M$ isomorphic to some matroid in $\cS$ as the \emph{special embedded minors} of $M$.
\end{df}

\begin{rem}
By the results in \autoref{subsection: foundations of the uniform matroid U24}, any minor embedding $U^2_4 \hookrightarrow M$ induces an isomorphism $\U\stackrel\sim\to F_M$ if $M$ is any of $C_5$, $C_5^\ast$, or $U^1_2\oplus U^2_4$. This follows for the series extension $C_5$ and the parallel extension $C_5^\ast$ at once from \cite[Prop.\ 4.9]{Baker-Lorscheid20}. For $M=U^1_2\oplus U^2_4$, this follows from \autoref{thm: foundations of direct sums}: since $U^1_2$ is regular, its foundation is $\Funpm$, and therefore the canonical inclusion $\U\to F_M=\Funpm\otimes\U=\U$ is an isomorphism.
\end{rem}

\begin{thm}\label{thm: fundamental presentation}
 Let $M$ be a matroid with foundation $F_M$ and fundamental diagram $\cE_M$. Then the canonical morphism $\colim F(\cE_M)\to F_M$ is an isomorphism.
\end{thm}

\begin{proof}
 In addition to the matroids listed in the definition of the fundamental diagram, the description $F_M=\colim F(\cE_M')$ in \cite[Thm.\ 4.23]{Baker-Lorscheid20} requires considering embedded minors of types $U^i_1\oplus U^2_4$ for $i=0,1$, as well as all embedded minors that are parallel extensions of rank $3$ matroids on $5$ elements with a $U^2_4$-minor.
 
 The claim of this theorem follows if we can show that the colimit $\colim F(\cE_M')$ does not change when we omit these additional embedded minors from $\cE_M'$. Our argument will make use of the following fact about colimits: if an object $F_{N}$ of the diagram $F(\cE_M')$ is the colimit $\colim F(\cE_N)$, where $\cE_N$ is the fundamental diagram of $N$, then we can omit $F_{N}$ from $\cE_M'$ without changing the value of $\colim F(\cE_M')$. We verify this case by case for all embedded minors $N=M\minor JI$ of the mentioned types.
 
 To start with, we consider $N=U^0_1\oplus U^2_4$. Let $1$ be the loop of the factor $U^0_1$. Then $\cE_N$ contains a unique embedded minor, which is $N\setminus1$, and thus $\colim F(\cE_N)=F_{N\setminus1}$. Since $1$ is a loop, the minor embedding $N\setminus1\hookrightarrow N$ induces an isomorphism $\colim F(\cE_N)=F_{N\setminus1}\to F_N$ by \cite[Prop.\ 4.9]{Baker-Lorscheid20}. As explained above, this shows that we can omit embedded minors of type $U^0_1\oplus U^2_4$ from $\cE_M'$ without changing its colimit.
 
 The case of embedded minors of type $N=U^1_1\oplus U^2_4$ is analogous. If $1$ is the coloop of the factor $U^1_1$, then $\cE_N$ consists of the unique embedded minor $N/1$ and thus $\colim F(\cE_N)=F_{N/1}$. Also in this case, \cite[Prop.\ 4.9]{Baker-Lorscheid20} applies and shows that we have an isomorphism $\colim F(\cE_N)=F_{N/1}\simeq F_N$, as desired.
 
 Next we turn to the parallel extensions of matroids of rank $3$ on $5$ elements. We begin with a classification of such parallel extensions. The isomorphism classes of rank $3$ matroids on $5$ elements with a $U^2_4$-minor are $U^1_1\oplus U^2_4$, $C_5$, and $U^3_5$. The matroid $U^1_1\oplus U^2_4$ has two parallel extensions: $U^1_2\oplus U^2_4$ and $U^1_1\oplus C_5^\ast$. The matroid $C_5$ has also two parallel extensions: one extends a series element, the other a non-series element. The uniform matroid $U^3_5$ has a unique parallel extension. 
 
 We continue with a case by case inspection of these parallel extensions, but for $U^1_2\oplus U^2_4$, which we cannot omit from the fundamental diagram for the reasons explained in \autoref{rem: fundamental diagram of U12+U24}. In each case, we identify the ground set with $E=\{1,\dotsc,6\}$ and assume that $1$ and $2$ are parallel. 
 
 As a first case, we consider $N=U^1_1\oplus C_5^\ast$. The two parallel elements $1$ and $2$ belong to $C_5^\ast$, and we denote by $6$ be the coloop of $U^1_1$. We illustrate in \autoref{fig: fundamental diagram of U11+C5*} the fundamental diagram $\cE_N$ of $N$, together with the isomorphism types of the embedded minors and the resulting diagram $F(\cE_N)$ of their foundations:

 \begin{figure}[htb]
 \[
  \beginpgfgraphicnamed{tikz/fig17}
  \begin{tikzpicture}[x=1.9cm,y=0.7cm]
   \draw[fill=blue!10!white,draw=blue!40!white,rounded corners=5pt]   (0.65,0.55) rectangle (1.35,2.45);
   \draw[fill=green!20!white,draw=green!80!black,rounded corners=5pt] (1.75,0.55) rectangle (2.25,2.45); 
   \draw[fill=blue!10!white,draw=blue!40!white,rounded corners=5pt]   (2.65,0.55) rectangle (3.35,2.45);
   \node[color=blue!80!white] at (1,2) {\footnotesize  $U^2_4$};
   \node[color=green!40!black] at (2,2) {\footnotesize  $C_5^\ast$};
   \node[color=blue!80!white] at (3,2) {\footnotesize  $U^2_4$};
   \node (N1/6) at (1,1) { $N\setminus1/6$};  
   \node (N/6)  at (2,1) { $N/6$};  
   \node (N2/6) at (3,1) { $N\setminus2/6$};  
   \draw[->] (N1/6) to (N/6);
   \draw[->] (N2/6) to (N/6);
   \node (F1/6) at (1,-0.5) { $\U$};
   \node (F/6)  at (2,-0.5) { $\U$};
   \node (F2/6) at (3,-0.5) { $\U$};
   \draw[->] (F1/6)  to node[sloped, above=-2pt] { $\sim$} (F/6);
   \draw[->] (F2/6)  to node[sloped, above=-2pt] { $\sim$} (F/6);
  \end{tikzpicture}
  \endpgfgraphicnamed
 \]
 \caption{The fundamental diagram of $U^1_1\oplus C_5^\ast$ and the associated diagram of foundations}
 \label{fig: fundamental diagram of U11+C5*}
 \end{figure}
 
 This shows that the colimit of $F(\cE_N)$ is $\U$, and identifies it canonically with the foundation of $N/6$. Since $6$ is a coloop, the foundation of $N$ is isomorphic to $F_{N/6}$ (cf.\ \cite[Prop.\ 4.9]{Baker-Lorscheid20}) and thus $\colim F(\cE_N)=F_{N/6}\to F_N$ is an isomorphism, as desired.
 
 Next we consider the parallel extension $N$ of $C_5$ by a series element. We denote the parallel elements by $1$ and $2$ and the complementary series element of $C_5$ by $3$. The fundamental diagram $\cE_N$ of $N$ and the resulting diagram $F(\cE_N)$ of their foundations are as illustrated in \autoref{fig: fundamental diagram of the parallel extension of C5 at a series element}.
 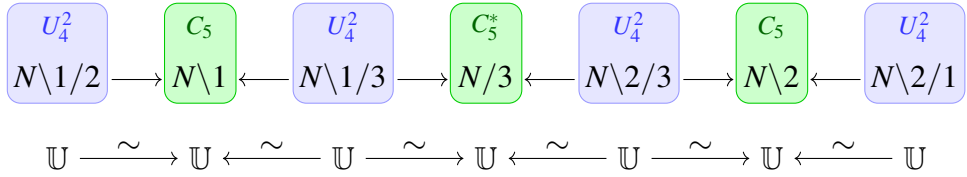
\begin{figure}[htb]
 \[
  \beginpgfgraphicnamed{tikz/fig18}
  \begin{tikzpicture}[x=1.9cm,y=0.7cm]
   \draw[fill=blue!10!white,draw=blue!40!white,rounded corners=5pt]   (0.65,0.55) rectangle (1.35,2.45);
   \draw[fill=green!20!white,draw=green!80!black,rounded corners=5pt] (1.75,0.55) rectangle (2.25,2.45); 
   \draw[fill=blue!10!white,draw=blue!40!white,rounded corners=5pt]   (2.65,0.55) rectangle (3.35,2.45);
   \draw[fill=green!20!white,draw=green!80!black,rounded corners=5pt] (3.75,0.55) rectangle (4.25,2.45); 
   \draw[fill=blue!10!white,draw=blue!40!white,rounded corners=5pt]   (4.65,0.55) rectangle (5.35,2.45);
   \draw[fill=green!20!white,draw=green!80!black,rounded corners=5pt] (5.75,0.55) rectangle (6.25,2.45); 
   \draw[fill=blue!10!white,draw=blue!40!white,rounded corners=5pt]   (6.65,0.55) rectangle (7.35,2.45);
   \node[color=blue!80!white]  at (1,2) {\footnotesize $U^2_4$};
   \node[color=green!40!black] at (2,2) {\footnotesize $C_5$};
   \node[color=blue!80!white]  at (3,2) {\footnotesize $U^2_4$};
   \node[color=green!40!black] at (4,2) {\footnotesize $C_5^\ast$};
   \node[color=blue!80!white]  at (5,2) {\footnotesize $U^2_4$};
   \node[color=green!40!black] at (6,2) {\footnotesize $C_5$};
   \node[color=blue!80!white]  at (7,2) {\footnotesize $U^2_4$};
   \node (N1/2) at (1,1) {$N\setminus1/2$};  
   \node (N1)   at (2,1) {$N\setminus1$};  
   \node (N1/3) at (3,1) {$N\setminus1/3$};  
   \node (N/3)  at (4,1) {$N/3$};  
   \node (N2/3) at (5,1) {$N\setminus2/3$};  
   \node (N2)   at (6,1) {$N\setminus2$};  
   \node (N2/1) at (7,1) {$N\setminus2/1$};  
   \draw[->] (N1/2) to (N1);
   \draw[->] (N1/3) to (N1);
   \draw[->] (N1/3) to (N/3);
   \draw[->] (N2/3) to (N/3);
   \draw[->] (N2/3) to (N2);
   \draw[->] (N2/1) to (N2);
   \node (F1/2) at (1,-0.5) {$\U$};
   \node (F1)   at (2,-0.5) {$\U$};
   \node (F1/3) at (3,-0.5) {$\U$};
   \node (F/3)  at (4,-0.5) {$\U$};
   \node (F2/3) at (5,-0.5) {$\U$};
   \node (F2)   at (6,-0.5) {$\U$};
   \node (F2/1) at (7,-0.5) {$\U$};
   \draw[->] (F1/2)  to node[sloped, above=-2pt] {$\sim$} (F1);
   \draw[->] (F1/3)  to node[sloped, above=-2pt] {$\sim$} (F1);
   \draw[->] (F1/3)  to node[sloped, above=-2pt] {$\sim$} (F/3);
   \draw[->] (F2/3)  to node[sloped, above=-2pt] {$\sim$} (F/3);
   \draw[->] (F2/3)  to node[sloped, above=-2pt] {$\sim$} (F2);
   \draw[->] (F2/1)  to node[sloped, above=-2pt] {$\sim$} (F2);
  \end{tikzpicture}
  \endpgfgraphicnamed
 \]
 \caption{The fundamental diagram of the parallel extension of $C_5$ by a series element and the associated diagram of foundations}
 \label{fig: fundamental diagram of the parallel extension of C5 at a series element}
 \end{figure}

 As in the previous case, the colimit of $F(\cE_N)$ is identified canonically with the foundation of $F_{N\setminus1}$, which is isomorphic to the foundation of the parallel extension $N$ of $N\setminus1$ by \cite[Prop.\ 4.9]{Baker-Lorscheid20}. Thus $\colim F(\cE_N)=F_{N\setminus1}\simeq F_N$, as desired.
 
 Next, we consider the parallel extension $N$ of $C_5$ by a non-series element. We denote the parallel elements by $1$ and $2$ and the series elements by $5$ and $6$. The fundamental diagram $\cE_N$ of $N$, and the resulting diagram $F(\cE_N)$ of their foundations, are as illustrated in \autoref{fig: fundamental diagram of the parallel extension of C5 at a non-series element}:

 \begin{figure}[htb]
 \[
  \beginpgfgraphicnamed{tikz/fig19}
  \begin{tikzpicture}[x=1.9cm,y=0.7cm]
   \draw[fill=green!20!white,draw=green!80!black,rounded corners=5pt] (1.75,1.55) rectangle (2.25,3.45); 
   \draw[fill=blue!10!white,draw=blue!40!white,rounded corners=5pt]   (2.65,0.55) rectangle (3.35,3.45);
   \draw[fill=green!20!white,draw=green!80!black,rounded corners=5pt] (3.75,0.55) rectangle (4.25,3.45); 
   \draw[fill=blue!10!white,draw=blue!40!white,rounded corners=5pt]   (4.65,0.55) rectangle (5.35,3.45);
   \draw[fill=green!20!white,draw=green!80!black,rounded corners=5pt] (5.75,1.55) rectangle (6.25,3.45); 
   \node[color=green!40!black] at (2,3) {\footnotesize $C_5$};
   \node[color=blue!80!white]  at (3,2) {\footnotesize $U^2_4$};
   \node[color=green!40!black] at (4,2) {\footnotesize $C_5^\ast$};
   \node[color=blue!80!white]  at (5,2) {\footnotesize $U^2_4$};
   \node[color=green!40!black] at (6,3) {\footnotesize $C_5$};
   \node (N1)   at (2,2) {$N\setminus1$};  
   \node (N1/6) at (3,1) {$N\setminus1/6$};  
   \node (N/6)  at (4,1) {$N/6$};  
   \node (N2/6) at (5,1) {$N\setminus2/6$};  
   \node (N1/5) at (3,3) {$N\setminus1/5$};  
   \node (N/5)  at (4,3) {$N/5$};  
   \node (N2/5) at (5,3) {$N\setminus2/5$};  
   \node (N2)   at (6,2) {$N\setminus2$};  
   \draw[->] (N1/5) to (N1);
   \draw[->] (N1/5) to (N/5);
   \draw[->] (N1/6) to (N1);
   \draw[->] (N1/6) to (N/6);
   \draw[->] (N2/5) to (N2);
   \draw[->] (N2/5) to (N/5);
   \draw[->] (N2/6) to (N2);
   \draw[->] (N2/6) to (N/6);
   \node (F1)   at (2,-1.5) {$\U$};  
   \node (F1/6) at (3,-2.5) {$\U$};  
   \node (F/6)  at (4,-2.5) {$\U$};  
   \node (F2/6) at (5,-2.5) {$\U$};  
   \node (F1/5) at (3,-0.5) {$\U$};  
   \node (F/5)  at (4,-0.5) {$\U$};  
   \node (F2/5) at (5,-0.5) {$\U$};  
   \node (F2)   at (6,-1.5) {$\U$};  
   \draw[->] (F1/5) to node[sloped, above=-2pt] {$\sim$} (F1);
   \draw[->] (F1/5) to node[sloped, above=-2pt] {$\sim$} (F/5);
   \draw[->] (F1/6) to node[sloped, above=-2pt] {$\sim$} (F1);
   \draw[->] (F1/6) to node[sloped, above=-2pt] {$\sim$} (F/6);
   \draw[->] (F2/5) to node[sloped, above=-2pt] {$\sim$} (F2);
   \draw[->] (F2/5) to node[sloped, above=-2pt] {$\sim$} (F/5);
   \draw[->] (F2/6) to node[sloped, above=-2pt] {$\sim$} (F2);
   \draw[->] (F2/6) to node[sloped, above=-2pt] {$\sim$} (F/6);
  \end{tikzpicture}
  \endpgfgraphicnamed
 \]
 \caption{The fundamental diagram of the parallel extension of $C_5$ at a non-series element and the associated diagram of foundations}
 \label{fig: fundamental diagram of the parallel extension of C5 at a non-series element}
 \end{figure}
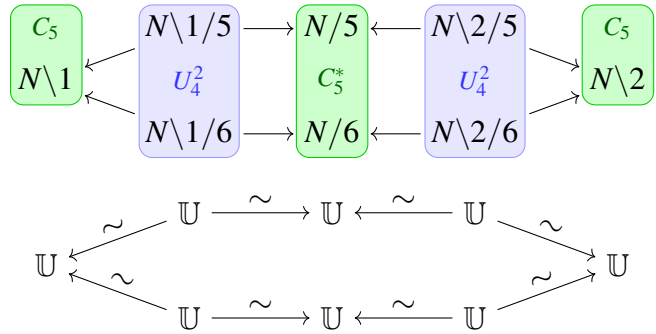

 In conclusion, the colimit of $F(\cE_N)$ is a quotient of $\U$ by an automorphism group that stems from the monodromy of the diagram. As a parallel extension of $N\setminus 1\simeq C_5$, the foundation of $N$ is isomorphic to $F_{N\setminus1}\simeq\U$. Since the diagram $F(\cE_N')$, which includes the foundation $F_N\simeq\U$ of $N$, is commutative, we conclude that the monodromy of $F(\cE_N)$ is trivial and thus $\colim F(\cE_N)\to F_N$ is an isomorphism, as claimed.

 The last case under investigation is the parallel extension $N$ of $U^3_5$. The fundamental diagram $\cE_N$ of $N$ and the resulting diagram $F(\cE_N)$ of their foundations are as in \autoref{fig: fundamental diagram of the parallel extension of U35}.
 \begin{figure}[htb]
  \centering
  \[
  \beginpgfgraphicnamed{tikz/fig20}
  \begin{tikzpicture}[x=1.9cm,y=0.7cm]
   \draw[fill=green!20!white,draw=green!80!black,rounded corners=5pt] (1.75,2.55) rectangle (2.25,6.45); 
   \draw[fill=blue!10!white,draw=blue!40!white,rounded corners=5pt]   (2.65,0.55) rectangle (3.35,6.45);
   \draw[fill=green!20!white,draw=green!80!black,rounded corners=5pt] (3.75,0.55) rectangle (4.25,6.45); 
   \draw[fill=blue!10!white,draw=blue!40!white,rounded corners=5pt]   (4.65,0.55) rectangle (5.35,6.45);
   \draw[fill=green!20!white,draw=green!80!black,rounded corners=5pt] (5.75,2.55) rectangle (6.25,6.45); 
   \node[color=green!40!black] at (2,6) {\footnotesize $U^3_5$};
   \node[color=blue!80!white]  at (3,6) {\footnotesize $U^2_4$};
   \node[color=green!40!black] at (4,6) {\footnotesize $C_5^\ast$};
   \node[color=blue!80!white]  at (5,6) {\footnotesize $U^2_4$};
   \node[color=green!40!black] at (6,6) {\footnotesize $U^3_5$};
   \node (N1)   at (2,3) {$N\setminus1$};  
   \node (N1/6) at (3,1) {$N\setminus1/6$};  
   \node (N/6)  at (4,1) {$N/6$};  
   \node (N2/6) at (5,1) {$N\setminus2/6$};  
   \node (N1/5) at (3,2) {$N\setminus1/5$};  
   \node (N/5)  at (4,2) {$N/5$};  
   \node (N2/5) at (5,2) {$N\setminus2/5$};  
   \node (N1/4) at (3,3) {$N\setminus1/4$};  
   \node (N/4)  at (4,3) {$N/4$};  
   \node (N2/4) at (5,3) {$N\setminus2/4$};  
   \node (N1/3) at (3,4) {$N\setminus1/3$};  
   \node (N/3)  at (4,4) {$N/3$};  
   \node (N2/3) at (5,4) {$N\setminus2/3$};  
   \node (N1/2) at (3,5) {$N\setminus1/2$};  
   \node (N2/1) at (5,5) {$N\setminus2/1$};  
   \node (N2)   at (6,3) {$N\setminus2$};  
   \draw[->] (N1/6) to (N1);
   \draw[->] (N1/6) to (N/6);
   \draw[->] (N1/5) to (N1);
   \draw[->] (N1/5) to (N/5);
   \draw[->] (N1/4) to (N1);
   \draw[->] (N1/4) to (N/4);
   \draw[->] (N1/3) to (N1);
   \draw[->] (N1/3) to (N/3);
   \draw[->] (N1/2) to (N1);
   \draw[->] (N2/6) to (N2);
   \draw[->] (N2/6) to (N/6);
   \draw[->] (N2/5) to (N2);
   \draw[->] (N2/5) to (N/5);
   \draw[->] (N2/4) to (N2);
   \draw[->] (N2/4) to (N/4);
   \draw[->] (N2/3) to (N2);
   \draw[->] (N2/3) to (N/3);
   \draw[->] (N2/1) to (N2);
   \node (F1)   at (2,-2.5) {$\V$};  
   \node (F1/6) at (3,-4.5) {$\U$};
   \node (F/6)  at (4,-4.5) {$\U$}; 
   \node (F2/6) at (5,-4.5) {$\U$};
   \node (F1/5) at (3,-3.5) {$\U$};
   \node (F/5)  at (4,-3.5) {$\U$};
   \node (F2/5) at (5,-3.5) {$\U$};
   \node (F1/4) at (3,-2.5) {$\U$};
   \node (F/4)  at (4,-2.5) {$\U$};
   \node (F2/4) at (5,-2.5) {$\U$};
   \node (F1/3) at (3,-1.5) {$\U$};
   \node (F/3)  at (4,-1.5) {$\U$};
   \node (F2/3) at (5,-1.5) {$\U$};
   \node (F1/2) at (3,-0.5) {$\U$};
   \node (F2/1) at (5,-0.5) {$\U$};
   \node (F2)   at (6,-2.5) {$\V$};  
   \draw[->] (F1/6) to (F1);
   \draw[->] (F1/6) to node[sloped, above=-2pt] {$\sim$} (F/6);
   \draw[->] (F1/5) to (F1);
   \draw[->] (F1/5) to node[sloped, above=-2pt] {$\sim$} (F/5);
   \draw[->] (F1/4) to (F1);
   \draw[->] (F1/4) to node[sloped, above=-2pt] {$\sim$} (F/4);
   \draw[->] (F1/3) to (F1);
   \draw[->] (F1/3) to node[sloped, above=-2pt] {$\sim$} (F/3);
   \draw[->] (F1/2) to (F1);
   \draw[->] (F2/6) to (F2);
   \draw[->] (F2/6) to node[sloped, above=-2pt] {$\sim$} (F/6);
   \draw[->] (F2/5) to (F2);
   \draw[->] (F2/5) to node[sloped, above=-2pt] {$\sim$} (F/5);
   \draw[->] (F2/4) to (F2);
   \draw[->] (F2/4) to node[sloped, above=-2pt] {$\sim$} (F/4);
   \draw[->] (F2/3) to (F2);
   \draw[->] (F2/3) to node[sloped, above=-2pt] {$\sim$} (F/3);
   \draw[->] (F2/1) to (F2);
  \end{tikzpicture}
  \endpgfgraphicnamed
  \]
  \caption{The fundamental diagram $\cE_N$ of the parallel extension of $U^3_5$ and the induced diagram of foundations}
  \label{fig: fundamental diagram of the parallel extension of U35}
 \end{figure}
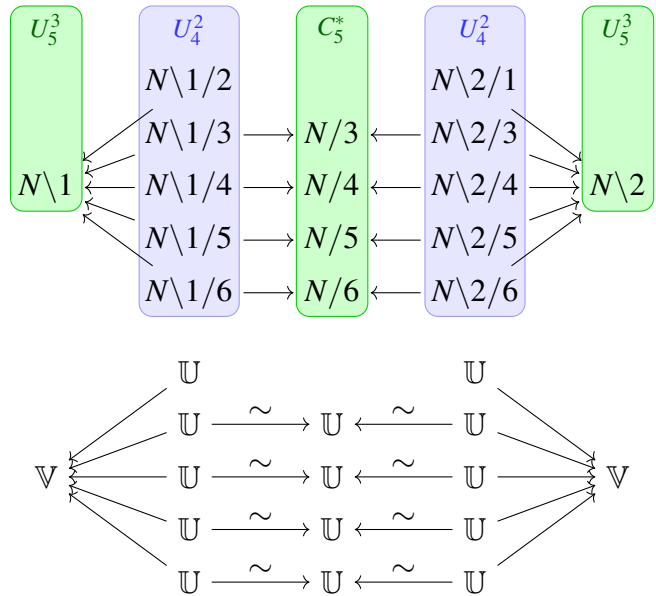
 By \cite[Prop.\ 4.9]{Baker-Lorscheid20}, the embedding of $N\setminus 1$ into its parallel extension $N$ induces an isomorphism $\V=F_{N\setminus1}\simeq F_N$. We aim to show that the canonical inclusion $F_{N\setminus1}\to\colim F(\cE_N)$ is an isomorphism. This follows if we can show that the morphisms
 \[
  \Psi_i: \ \U^{\otimes 4} \ \simeq \ \bigotimes_{j=3}^6 F_{N\minor ik} \ \longrightarrow \ F_{N\setminus i} \ \simeq \ \V
 \]
 are epimorphisms for $i=1,2$, since in this case, $F_{N\setminus1}$ and $F_{N\setminus2}$ will be identified in the colimit of $F(\cE_N)$. By symmetry, it suffices to show that $\Psi_1$ is an epimorphism. By \autoref{prop: foundation of U25}, we have
 \[
  F_{N\setminus 1} \ = \ \pastgenn{\Funpm}{x_2,\dotsc,x_6}{x_i+x_{i-1}x_{i+1}-1\mid i=2,\dotsc,6}
 \]
 for $x_i=\cross{i+1}{i+2}{i+3}{i+4}{}$ and $y_i=x_{i-1}x_{i+1}=\cross{i+1}{i+3}{i+2}{i+4}{}$, where we consider $i+k$ modulo $5$ as an element of $\{2,\dotsc,6\}$. For $i=3,\dotsc,6$, the cross ratio $x_i$ of $F_{N\setminus1}$ is the image of the corresponding cross ratio of $F_{N\minor ik}\simeq\U$. Since 
 \[
  x_2 \ = \ \cross3456{} \ = \ \cross3462{} \cdot \cross3425 \ = \ x_5^{-1}\cdot y_6,
 \]
 where $x_5^{-1}$ stems from $F_{N\minor15}$ and $y_6$ stems from $F_{N\minor16}$, $x_2$ also lies in the image of the map $\Psi_1$. This shows that $\Psi_1$ is surjective, and therefore an epimorphism of pastures. This concludes our argument that $\colim F(\cE_N)=F_{N\setminus1}\simeq F_N$, as desired.
\end{proof}

\begin{rem}\label{rem: fundamental diagram of U12+U24}
 To see that we cannot remove embedded minors of type $U^1_2\oplus U^2_4$ from the fundamental diagram, let us consider $M=U^1_2\oplus U^2_4$, whose fundamental diagram $\cE_M$, together with the induced diagram $F(\cE_M)$ of the associated foundations, is as illustrated in \autoref{fig: fundamental diagram of U12+U24}:

 \begin{figure}[htb]
 \[
  \beginpgfgraphicnamed{tikz/fig16}
  \begin{tikzpicture}[x=1.9cm,y=0.7cm]
   \draw[fill=blue!10!white,draw=blue!40!white,rounded corners=5pt] (0.65,0.55) rectangle (1.35,2.45);
   \draw[fill=green!20!white,draw=green!80!black,rounded corners=5pt] (1.65,0.55) rectangle (2.35,2.45); 
   \draw[fill=blue!10!white,draw=blue!40!white,rounded corners=5pt] (2.65,0.55) rectangle (3.35,2.45);
   \node[color=blue!80!white] at (1,2) {\footnotesize $U^2_4$};
   \node[color=green!40!black] at (2,2) {\footnotesize $U^1_2\oplus U^2_4$};
   \node[color=blue!80!white] at (3,2) {\footnotesize $U^2_4$};
   \node (M12)  at (1,1) {$M\minor12$};  
   \node (M)   at (2,1) {$M$};  
   \node (M21)  at (3,1) {$M\minor21$};  
   \draw[->] (M12) to (M);
   \draw[->] (M21) to (M);
   \node (F12) at (1,-0.5) {$\U$};
   \node (F)   at (2,-0.5) {$\U$};
   \node (F21) at (3,-0.5) {$\U$};
   \draw[->] (F12)  to node[sloped, above=-2pt] {$\sim$} (F);
   \draw[->] (F21)  to node[sloped, above=-2pt] {$\sim$} (F);
  \end{tikzpicture}
  \endpgfgraphicnamed
 \]
 \caption{The fundamental diagram of $U^1_2\oplus U^2_4$ and the associated diagram of foundations}
 \label{fig: fundamental diagram of U12+U24}
 \end{figure}
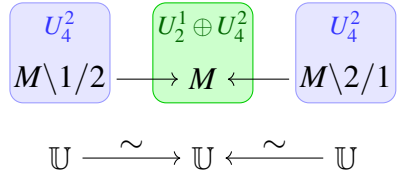
 Since $\U$ is a terminal object of the diagram $F(\cE_M)$, we conclude that $F_M\simeq\U$. If we were to omit the embedded minors of type $U^1_2\oplus U^2_4$ from $\cE_M$, then the resulting diagram of foundations would consist of two isolated copies of $\U$ whose colimit is $\U\otimes\U$. This shows that we cannot omit minors of this type in general.
 
 Also none of the other minors of types $U^2_4$, $U^2_5$, $U^3_5$, $C_5$, $C_5^\ast$, $F_7$ and $F_7^\ast$ can be omitted from the fundamental diagram without violating the result of \autoref{thm: fundamental presentation}. This can be seen as follows: all proper minors of $U^2_4$ are regular, so if we omit $M=U^2_4$ from the list, then $\colim\cE_M=\Funpm\neq\U=F_M$. Similarly, all proper minors of $M=F_7$ are regular, and omitting $F_7$ from the fundamental type yields $\colim \cE_M=\Funpm\neq\F_2=F_M$. Since the fundamental diagram and the fundamental presentation are both symmetric under duality, the same holds for $F_7^\ast$. The matroid $M=C_5$ has two non-regular proper minors, which are both of type $U^2_4$. Therefore the fundamental presentation is $\U\leftarrow\U\to\U$, equal to that of $\U^2_4\oplus \U^1_2$. Omitting $C_5$ from $\cE_M$ yields $\colim\cE_M=\U\otimes\U\neq\U=F_M$. By duality, the same holds for $C_5^\ast$. The matroid $U^2_5$ has $5$ non-regular proper minors, which are of type $U^2_4$. Thus omitting $M=U^2_5$ from the fundamental diagram yields $\colim\cE_M=\U^{\otimes5}\neq\V=F_M$. By duality, the same holds for $U^3_5$.
\end{rem}


\subsection{The fundamental lattice presentation}
\label{subsection: the fundamental lattice presentation}

In this section, we describe a presentation of the foundation in terms of the lattice of flats of a matroid. This approach has computational advantages, because several types of embedded minors correspond to the same sublattice, which leads to a more compact presentation of the foundation as a colimit.

By \cite[Prop.\ 4.9]{Baker-Lorscheid20}, the foundation $F_M$ of a matroid $M$ only depends on its simplification and therefore is determined by the lattice of flats $\Lambda$ of $M$, which justifies the notation $F_{\Lambda}=F_M$. As explained in \autoref{subsection: universal cross ratios}, the cross ratios $\cross abcdI=\cross{\gen{Ia}}{\gen{Ib}}{\gen{Ic}}{\gen{Id}}{}$ in $F_M=F_\Lambda$ only depend on the hyperplanes $\gen{Ia}$, $\gen{Ib}$, $\gen{Ic}$ and $\gen{Id}$, which are corank $1$ elements of $\Lambda$. This yields intrinsic generators of $F_\Lambda$ in terms of $\Lambda$.

\begin{df}
 An \emph{upper sublattice} is a matroid sublattice $\Lambda'$ (i.e.\ atomistic and semimodular) whose rank $r'$ equals the corank of its bottom element $F$ as an element in $\Lambda$ (or as a flat in $M$). 
\end{df}

\begin{df}
 If $\Lambda$ is a lattice and $S$ is a subset of atoms of $\Lambda$, we define the \emph{sublattice $\Lambda_S$ induced by $S$} to be the set of all $x \in \Lambda$ such that $x$ is a join of elements in $S$. If $\Lambda$ is a lattice and $x,y \in \Lambda$, we denote by $[x,y]$ the \emph{interval} $\{ z \in \Lambda \mid x \leq z \leq y \}$.
\end{df}

According to the ``Scum Theorem''\footnote{The name comes from the fact that, like scum which rises to the top of a pond, the lattice of flats of any minor of $M$ can be found ``hanging from the top'' of the lattice of flats of $M$.} of D.~A.~Higgs, $\Lambda_N$ is isomorphic to an upper sublattice of $\Lambda_M$ for every minor $N$ of $M$, and every upper sublattice of $\Lambda_M$ is isomorphic to $\Lambda_N$ for some minor $N$ of $M$. However, this is not a one-to-one correspondence. The following result, phrased in terms of embedded minors rather than minors, makes the correspondence more precise. 
To state the result, let $\Emb_M$ denote the set of embedded minors of a matroid $M$, and let $\USL_\Lambda$ denote the set of upper sublattices of a lattice $\Lambda$.
 
\begin{prop}\label{prop: embedded minor surject onto upper sublattices}
 Define $\Psi: \Emb_M \to \USL_{\Lambda_M}$ by sending $M\minor JI$ to $[\langle I \rangle, E]_{S}$ where $S=\{\gen{I,e}\mid e\notin I\cup J\}$ are the atoms of $[\langle I \rangle, E]$ that are generated by the elements of $M\minor JI$. Then $\Psi$ is surjective and $\Psi(N) \cong \Lambda_N$ for every embedded minor $N$ of $M$.
\end{prop}

\begin{proof}
 According to Proposition 3.3.7 in \cite{Oxley92}, the flats of $M\setminus J$ are the subsets of $E$ of the form $F-J$ such that $F$ is a flat of $M$, and the flats of $M/I$ are subsets $F'$ of $E$ such that $F' \cup I$ is a flat of $M$.  Thus a subset $F'$ of $E-(I\cup J)$ is a flat of $N = M\minor JI$ if and only if $F' \cup I = F -J$ for some flat $F$ of $M$. (The flat $F$ is determined by $F'$ since the closure of $F -J$ is equal to $F$ as $E-J$ is spanning.)

 It follows that for every embedded minor $N = M\minor JI$ and $S=\{\gen{I,e}\mid e\notin I\cup J\}$, the lattice $[\langle I \rangle, E]_S$ is an upper sublattice of $\Lambda_M$ that is isomorphic to $\Lambda_N$. The surjectivity of $\Psi$ follows since every upper sublattice of $\Lambda_M$ is of the form $[F,E]_S$ for some $F \in \Lambda$ and some subset $S$ of atoms of $[F,E]$, and since $[F,E]_S=\Psi(M\minor JI)$ for $I$ and $J$ such that $F=\gen I$ and $S=\{\gen{I,e}\mid e\notin I\cup J\}$.
\end{proof}
 
Given an upper sublattice $\Lambda'$ of $\Lambda_M$, we find an embedded minor $N=M\minor JI$ of $M$ with sublattice $\Lambda_N=\Lambda'$ by \autoref{prop: embedded minor surject onto upper sublattices}.
Consequently, the lattice inclusion $\Lambda' \hookrightarrow \Lambda$ induces a morphism of foundations $F_{\Lambda'}=F_N\to F_M=F_\Lambda$, which can be described intrinsically in terms of the tautological association 
\[
 \begin{array}{ccc}
  F_{\Lambda'}                 & \longrightarrow & F_\Lambda \\
  \cross{H_1}{H_2}{H_3}{H_4}{} & \longmapsto     & \cross{H_1}{H_2}{H_3}{H_4}{}
 \end{array}
\]
for modular quadruples $(H_1,H_2,H_3,H_4)$ of hyperplanes of $\Lambda'$. We say that an upper sublattice $\Lambda'$ of $\Lambda$ is an \emph{$N$-sublattice}, or \emph{of type $N$}, if $\Lambda'$ is isomorphic to the lattice of flats of the matroid $N$. 

\begin{df}
 Let $M$ be a matroid with lattice $\Lambda_M$. The \emph{fundamental lattice diagram $\cL_M$ of $M$} is the diagram of all upper sublattices $\Lambda'$ of $M$ of types $U^2_4$, $U^2_5$, $U^3_5$, $C_5$, $F_7$ and $F_7^*$ (as illustrated in \autoref{fig: relevant upper sublattices for the foundation}), together with all lattice inclusions $\Lambda'\hookrightarrow\Lambda''$. The $\cL_M$-presentation of $M$ is the induced diagram $F(\cL_M)$ of foundations of the lattices in $\cL_M$.
\end{df}

\begin{figure}[htb]
 \[
 \newsavebox\latticeUtwofour
 \sbox{\latticeUtwofour}{
  \beginpgfgraphicnamed{tikz/fig1}
  \begin{tikzpicture}[x=1.0cm,y=0.8cm]
   \node (0) at (1.5,0) {\footnotesize $\emptyset$};  
   \node (1) at (0,1) {\footnotesize $1$};  
   \node (2) at (1,1) {\footnotesize $2$};  
   \node (3) at (2,1) {\footnotesize $3$};  
   \node (4) at (3,1) {\footnotesize $4$};  
   \node (1234) at (1.5,2) {\footnotesize $1234$};  
   \draw (0) to (1);
   \draw (0) to (2);
   \draw (0) to (3);
   \draw (0) to (4);
   \draw (1) to (1234);
   \draw (2) to (1234);
   \draw (3) to (1234);
   \draw (4) to (1234);
   \node at (3.0,0.0) {\footnotesize ${U^2_4}$};  
  \end{tikzpicture}
  \endpgfgraphicnamed
 }
 \newsavebox\latticeUtwofive
 \sbox{\latticeUtwofive}{
  \beginpgfgraphicnamed{tikz/fig2}
  \begin{tikzpicture}[x=1.0cm,y=0.8cm]
   \node (0) at (2,0) {\footnotesize $\emptyset$};  
   \node (1) at (0,1) {\footnotesize $1$};  
   \node (2) at (1,1) {\footnotesize $2$};  
   \node (3) at (2,1) {\footnotesize $3$};  
   \node (4) at (3,1) {\footnotesize $4$};  
   \node (5) at (4,1) {\footnotesize $5$};  
   \node (12345) at (2,2) {\footnotesize $12345$};  
   \draw (0) to (1);
   \draw (0) to (2);
   \draw (0) to (3);
   \draw (0) to (4);
   \draw (0) to (5);
   \draw (1) to (12345);
   \draw (2) to (12345);
   \draw (3) to (12345);
   \draw (4) to (12345);
   \draw (5) to (12345);
   \node at (4.0,0.0) {\footnotesize ${U^2_5}$};  
  \end{tikzpicture}
  \endpgfgraphicnamed
 }
 \newsavebox\latticeUthreefive
 \sbox{\latticeUthreefive}{
  \beginpgfgraphicnamed{tikz/fig3}
  \begin{tikzpicture}[x=0.75cm,y=0.8cm]
   \node (0) at (4.5,0) {\footnotesize $\emptyset$};  
   \node (1) at (2.5,1) {\footnotesize $1$};  
   \node (2) at (3.5,1) {\footnotesize $2$};  
   \node (3) at (4.5,1) {\footnotesize $3$};  
   \node (4) at (5.5,1) {\footnotesize $4$};  
   \node (5) at (6.5,1) {\footnotesize $5$};  
   \node (12) at (0,2) {\footnotesize $12$};  
   \node (13) at (1,2) {\footnotesize $13$};  
   \node (14) at (2,2) {\footnotesize $14$};  
   \node (15) at (3,2) {\footnotesize $15$};  
   \node (23) at (4,2) {\footnotesize $23$};  
   \node (24) at (5,2) {\footnotesize $24$};  
   \node (25) at (6,2) {\footnotesize $25$};  
   \node (34) at (7,2) {\footnotesize $34$};  
   \node (35) at (8,2) {\footnotesize $35$};  
   \node (45) at (9,2) {\footnotesize $45$};  
   \node (12345) at (4.5,3) {\footnotesize $12345$};  
   \draw (0) to (1);
   \draw (0) to (2);
   \draw (0) to (3);
   \draw (0) to (4);
   \draw (0) to (5);
   \draw (2) to (12);
   \draw (3) to (13);
   \draw (4) to (14);
   \draw (5) to (15);
   \draw[white,line width=2pt] (3) to (23);
   \draw (3) to (23);
   \draw[white,line width=2pt] (4) to (24);
   \draw (4) to (24);
   \draw[white,line width=2pt] (5) to (25);
   \draw (5) to (25);
   \draw[white,line width=2pt] (4) to (34);
   \draw (4) to (34);
   \draw[white,line width=2pt] (5) to (35);
   \draw (5) to (35);
   \draw[white,line width=2pt] (1) to (12);
   \draw (1) to (12);
   \draw[white,line width=2pt] (1) to (13);
   \draw (1) to (13);
   \draw[white,line width=2pt] (1) to (14);
   \draw (1) to (14);
   \draw[white,line width=2pt] (1) to (15);
   \draw (1) to (15);
   \draw[white,line width=2pt] (2) to (23);
   \draw (2) to (23);
   \draw[white,line width=2pt] (2) to (24);
   \draw (2) to (24);
   \draw[white,line width=2pt] (2) to (25);
   \draw (2) to (25);
   \draw[white,line width=2pt] (3) to (34);
   \draw (3) to (34);
   \draw[white,line width=2pt] (3) to (35);
   \draw (3) to (35);
   \draw[white,line width=2pt] (4) to (45);
   \draw (4) to (45);
   \draw[white,line width=2pt] (5) to (45);
   \draw (5) to (45);
   \draw (12) to (12345);
   \draw (13) to (12345);
   \draw (14) to (12345);
   \draw (15) to (12345);
   \draw (23) to (12345);
   \draw (24) to (12345);
   \draw (25) to (12345);
   \draw (34) to (12345);
   \draw (35) to (12345);
   \draw (45) to (12345);
   \node at (9,0.0) {\footnotesize ${U^3_5}$};  
  \end{tikzpicture}
  \endpgfgraphicnamed
 }
 \newsavebox\latticeCfive
 \sbox{\latticeCfive}{
  \beginpgfgraphicnamed{tikz/fig4}
  \begin{tikzpicture}[x=0.75cm,y=0.8cm]
   \node (0) at (3.5,0) {\footnotesize $\emptyset$};  
   \node (1) at (1.5,1) {\footnotesize $1$};  
   \node (2) at (2.5,1) {\footnotesize $2$};  
   \node (3) at (3.5,1) {\footnotesize $3$};  
   \node (4) at (4.5,1) {\footnotesize $4$};  
   \node (5) at (5.5,1) {\footnotesize $5$};  
   \node (123) at (0,2) {\footnotesize $123$};  
   \node (14) at (1,2) {\footnotesize $14$};  
   \node (15) at (2,2) {\footnotesize $15$};  
   \node (24) at (3,2) {\footnotesize $24$};  
   \node (25) at (4,2) {\footnotesize $25$};  
   \node (34) at (5,2) {\footnotesize $34$};  
   \node (35) at (6,2) {\footnotesize $35$};  
   \node (45) at (7,2) {\footnotesize $45$};  
   \node (12345) at (3.5,3) {\footnotesize $12345$};  
   \draw (0) to (1);
   \draw (0) to (2);
   \draw (0) to (3);
   \draw (0) to (4);
   \draw (0) to (5);
   \draw (2) to (123);
   \draw (3) to (123);
   \draw (4) to (14);
   \draw (5) to (15);
   \draw[white,line width=2pt] (4) to (24);
   \draw (4) to (24);
   \draw[white,line width=2pt] (5) to (25);
   \draw (5) to (25);
   \draw[white,line width=2pt] (4) to (34);
   \draw (4) to (34);
   \draw[white,line width=2pt] (5) to (35);
   \draw (5) to (35);
   \draw (1) to (123);
   \draw[white,line width=2pt] (1) to (14);
   \draw (1) to (14);
   \draw[white,line width=2pt] (1) to (15);
   \draw (1) to (15);
   \draw[white,line width=2pt] (2) to (24);
   \draw (2) to (24);
   \draw[white,line width=2pt] (2) to (25);
   \draw (2) to (25);
   \draw[white,line width=2pt] (3) to (34);
   \draw (3) to (34);
   \draw[white,line width=2pt] (3) to (35);
   \draw (3) to (35);
   \draw[white,line width=2pt] (4) to (45);
   \draw (4) to (45);
   \draw[white,line width=2pt] (5) to (45);
   \draw (5) to (45);
   \draw (123) to (12345);
   \draw (14) to (12345);
   \draw (15) to (12345);
   \draw (24) to (12345);
   \draw (25) to (12345);
   \draw (34) to (12345);
   \draw (35) to (12345);
   \draw (45) to (12345);
   \node at (7,0.0) {\footnotesize ${C_5}$};  
  \end{tikzpicture}
  \endpgfgraphicnamed
 }
 \newsavebox\latticeFseven
 \sbox{\latticeFseven}{
  \beginpgfgraphicnamed{tikz/fig5}
  \begin{tikzpicture}[x=0.9cm,y=0.8cm]
   \node (0) at (3,0) {\footnotesize $\emptyset$};  
   \node (1) at (0,1) {\footnotesize $1$};  
   \node (2) at (1,1) {\footnotesize $2$};  
   \node (3) at (2,1) {\footnotesize $3$};  
   \node (4) at (3,1) {\footnotesize $4$};  
   \node (5) at (4,1) {\footnotesize $5$};  
   \node (6) at (5,1) {\footnotesize $6$};  
   \node (7) at (6,1) {\footnotesize $7$};  
   \node (124) at (0,2) {\footnotesize $124$};  
   \node (235) at (1,2) {\footnotesize $235$};  
   \node (346) at (2,2) {\footnotesize $346$};  
   \node (457) at (3,2) {\footnotesize $457$};  
   \node (561) at (4,2) {\footnotesize $561$};  
   \node (672) at (5,2) {\footnotesize $672$};  
   \node (713) at (6,2) {\footnotesize $713$};  
   \node (1234567) at (3,3) {\footnotesize $1234567$};  
   \draw (0) to (1);
   \draw (0) to (2);
   \draw (0) to (3);
   \draw (0) to (4);
   \draw (0) to (5);
   \draw (0) to (6);
   \draw (0) to (7);
   \draw (7) to (457);
   \draw (6) to (346);
   \draw (5) to (235);
   \draw (4) to (124);
   \draw (1) to (124);
   \draw (3)[white,line width=2pt] to (713);
   \draw (3) to (713);
   \draw (2)[white,line width=2pt] to (672);
   \draw (2) to (672);
   \draw (1)[white,line width=2pt] to (561);
   \draw (1) to (561);
   \draw (1)[white,line width=2pt] to (713);
   \draw (1) to (713);   
   \node[fill=white,rounded corners,inner sep=1pt] at (5,2) {\footnotesize $672$};  
   \node[fill=white,rounded corners,inner sep=1pt] at (1,1) {\footnotesize $2$};  
   \draw (7)[white,line width=2pt] to (672);
   \draw (7) to (672);
   \draw (7)[white,line width=2pt] to (713);
   \draw (7) to (713);
   \draw (6)[white,line width=2pt] to (561);
   \draw (6) to (561);
   \draw (6)[white,line width=2pt] to (672);
   \draw (6) to (672);
   \draw (5)[white,line width=2pt] to (457);
   \draw (5) to (457);
   \draw (5)[white,line width=2pt] to (561);
   \draw (5) to (561);
   \draw (4)[white,line width=2pt] to (346);
   \draw (4) to (346);
   \draw (4)[white,line width=2pt] to (457);
   \draw (4) to (457);
   \draw (3)[white,line width=2pt] to (235);
   \draw (3) to (235);
   \draw (3)[white,line width=2pt] to (346);
   \draw (3) to (346);
   \draw (2)[white,line width=2pt] to (124);
   \draw (2) to (124);
   \draw (2)[white,line width=2pt] to (235);
   \draw (2) to (235);
   \draw (124) to (1234567);
   \draw (235) to (1234567);
   \draw (346) to (1234567);
   \draw (457) to (1234567);
   \draw (561) to (1234567);
   \draw (672) to (1234567);
   \draw (713) to (1234567);
   \node at (6,0.0) {\footnotesize ${F_7}$};  
  \end{tikzpicture}
  \endpgfgraphicnamed
 }
 \newsavebox\latticeFsevendual
 \sbox{\latticeFsevendual}{
  \beginpgfgraphicnamed{tikz/fig6}
  \begin{tikzpicture}[x=0.9cm,y=0.8cm]
   \node (0) at (3,0) {\footnotesize $\emptyset$};  
   \node (1) at (0,1) {\footnotesize $1$};  
   \node (2) at (1,1) {\footnotesize $2$};  
   \node (3) at (2,1) {\footnotesize $3$};  
   \node (4) at (3,1) {\footnotesize $4$};  
   \node (5) at (4,1) {\footnotesize $5$};  
   \node (6) at (5,1) {\footnotesize $6$};  
   \node (7) at (6,1) {\footnotesize $7$};  
   \node (12) at (0,2) {\footnotesize $12$};  
   \node (dots1) at (1,2) {\footnotesize $\dotsb$};  
   \node (rk2) at (3,2) {\footnotesize $\text{(all $2$-subsets)}$};  
   \node (dots2) at (5,2) {\footnotesize $\dotsb$};  
   \node (67) at (6,2) {\footnotesize $67$};  
   \node (124) at (0,3) {\footnotesize $124$};  
   \node (dots3) at (0.6,3) {\footnotesize $\dotsb$};  
   \node (713) at (2.1,3) {\footnotesize $\text{(cohyperplanes)}$};  
   \node (2456) at (4.4,3) {\footnotesize $\text{(circuits)}$};  
   \node (dots4) at (5.35,3) {\footnotesize $\dotsb$};  
   \node (3567) at (6,3) {\footnotesize $3567$};  
   \node (1234567) at (3,4) {\footnotesize $1234567$};  
   \draw (0) to (1);
   \draw (0) to (2);
   \draw (0) to (3);
   \draw (0) to (4);
   \draw (0) to (5);
   \draw (0) to (6);
   \draw (0) to (7);
   \draw (1) to (12);
   \draw (2) to (12);
   \node at (1.5,1.5) {\footnotesize $\dotsb$};  
   \draw[dashed] (3) to (rk2);
   \draw[dashed] (4) to (rk2);
   \draw[dashed] (5) to (rk2);
   \node at (4.5,1.5) {\footnotesize $\dotsb$};  
   \draw (6) to (67);
   \draw (7) to (67);
   \draw (12) to (124);
   \node at (1,2.5) {\footnotesize $\dotsb$};  
   \draw[dashed] (rk2) to (713);
   \draw[dashed] (rk2) to (2456);
   \node at (5,2.5) {\footnotesize $\dotsb$};  
   \draw (67) to (3567);
   \draw (124) to (1234567);
   \node at (2.15,3.5) {\footnotesize $\dotsb$};  
   \draw[dashed] (713) to (1234567);
   \draw[dashed] (2456) to (1234567);
   \node at (3.95,3.5) {\footnotesize $\dotsb$};  
   \draw (3567) to (1234567);
   \node at (6,0.0) {\footnotesize ${F_7^\ast}$};  
  \end{tikzpicture}
  \endpgfgraphicnamed
 }
 \begin{tikzpicture}
  \node at (0,6.5) {\usebox{\latticeUtwofour}};
  \node at (7,6.5) {\usebox{\latticeUtwofive}};
  \node at (0,3.5) {\usebox{\latticeUthreefive}};
  \node at (7,3.5) {\usebox{\latticeCfive}};
  \node at (0,0) {\usebox{\latticeFseven}};
  \node at (7,0) {\usebox{\latticeFsevendual}};
 \end{tikzpicture}
 \]
 \caption{Lattices of types $U^2_4$, $U^2_5$, $U^3_5$, $C_5$, $F_7$ and $F_7^*$}
 \label{fig: relevant upper sublattices for the foundation}
\end{figure}
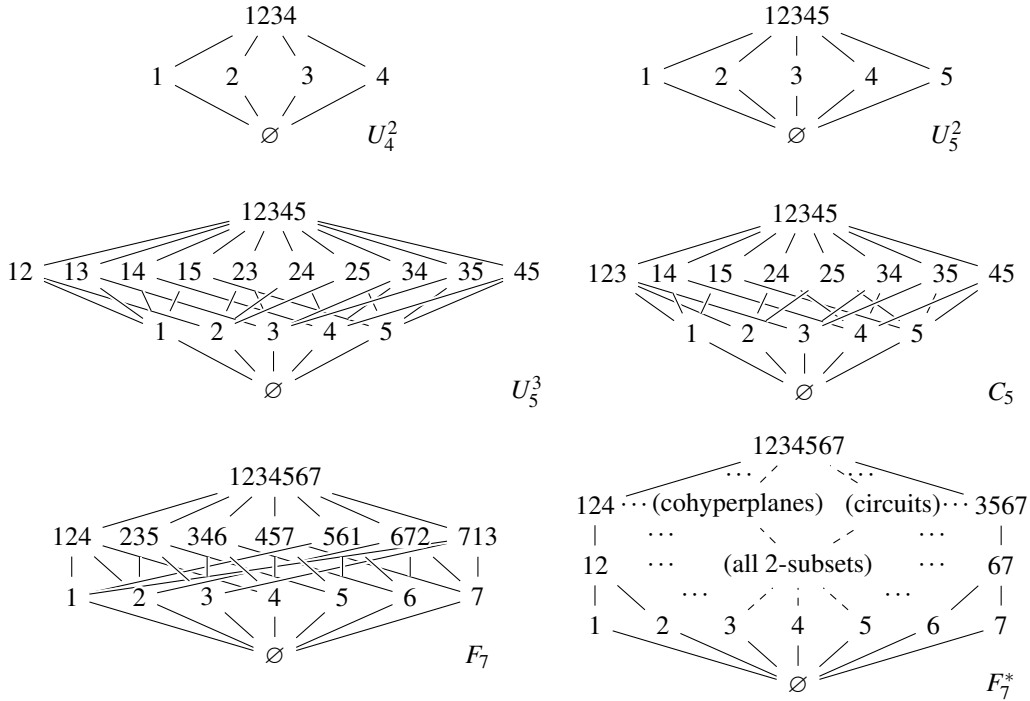

\begin{thm}\label{thm: fundamental lattice presentation}
 Let $M$ be a matroid with fundamental lattice diagram $\cL_M$. Then $F_M\simeq\colim F(\cL_M)$.
\end{thm}

\begin{proof}
 By \autoref{thm: fundamental presentation}, the foundation $F_M$ is the colimit of $F(\cE_M)$, where $\cE_M$ is the fundamental diagram of $M$. By our previous considerations, the association $\cE_M\mapsto F(\cE_M)$ factors through the diagram $\cL_M'$ of upper sublattices $\Lambda_N$ of $M$ that correspond to the embedded minors $N=M\minor JI$ in $\cE_M$. Thus $\Lambda_M'$ consists of all upper sublattices of $M$ of types $U^2_4$, $U^2_5$, $U^3_5$, $C_5$, $C_5^\ast$, $U^1_2\oplus U^2_4$, $F_7$ and $F_7^*$ and $F_M=\colim F(\cE_M)=\colim F(\cL_M')$. Since the lattice of $N=C_5^\ast$ is equal to the lattice of type $U^2_4$, this case is subsumed by type $U^2_4$.
 
 For upper sublattices of type $N=U^1_2\oplus U^2_4$, we proceed as in the proof of \autoref{thm: fundamental presentation}: we show that the foundation $F_{\Lambda_N}$ is the colimit of $F(\cL_N)$, where $\cL_N$ is the fundamental lattice diagram of $N$. Let $E=\{1,\dotsc,6\}$ be the ground set of $N$, where we assume that $1$ and $2$ are the parallel elements corresponding to the direct summand $U^1_2$. The lattice of $N$ is as illustrated in \autoref{fig: lattice of U12+U24}:
 \begin{figure}[htb]
  \[
  \beginpgfgraphicnamed{tikz/fig21}
  \begin{tikzpicture}[x=1.2cm,y=0.8cm]
   \node (0) at (3,0) {\footnotesize $\emptyset$};  
   \node (12) at (1,1) {\footnotesize $12$};  
   \node (3) at (2,1) {\footnotesize $3$};  
   \node (4) at (3,1) {\footnotesize $4$};  
   \node (5) at (4,1) {\footnotesize $5$};  
   \node (6) at (5,1) {\footnotesize $6$};  
   \node (123)  at (1,2) {\footnotesize $123$};  
   \node (124)  at (2,2) {\footnotesize $124$};  
   \node (125)  at (3,2) {\footnotesize $125$};  
   \node (126)  at (4,2) {\footnotesize $126$};  
   \node (3456) at (5,2) {\footnotesize $3456$};  
   \node (123456) at (3,3) {\footnotesize $123456$};  
   \draw (0) to (12);
   \draw (0) to (3);
   \draw (0) to (4);
   \draw (0) to (5);
   \draw (0) to (6);
   \draw (3) to (3456);
   \draw (4) to (3456);
   \draw (5) to (3456);
   \draw (6) to (3456);
   \draw[white,line width=2pt] (3) to (123);
   \draw (3) to (123);
   \draw[white,line width=2pt] (4) to (124);
   \draw (4) to (124);
   \draw[white,line width=2pt] (5) to (125);
   \draw (5) to (125);
   \draw[white,line width=2pt] (6) to (126);
   \draw (6) to (126);
   \draw[white,line width=2pt] (12) to (123);
   \draw (12) to (123);
   \draw[white,line width=2pt] (12) to (124);
   \draw (12) to (124);
   \draw[white,line width=2pt] (12) to (125);
   \draw (12) to (125);
   \draw[white,line width=2pt] (12) to (126);
   \draw (12) to (126);
   \draw (123) to (123456);
   \draw (124) to (123456);
   \draw (125) to (123456);
   \draw (126) to (123456);
   \draw (3456) to (123456);
  \end{tikzpicture}
  \endpgfgraphicnamed
  \]
  \caption{The lattice of $U^1_2\oplus U^2_4$}
  \label{fig: lattice of U12+U24}
 \end{figure}
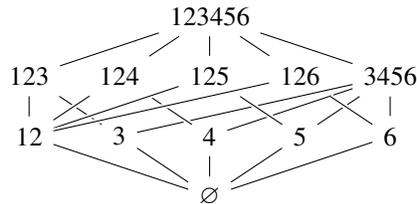
 
 The fundamental lattice diagram $\cL_N$ of $N$ consists of a unique upper sublattice $\Lambda'$ of type $U^2_4$, which is characterized by its hyperplanes $123$, $124$, $125$, and $126$. Since the foundation $F_N\simeq\U$ of $N$ is isomorphic to the foundation of $U^2_4$, we obtain the desired isomorphism $\colim F(\cL_N)=\U\simeq F_N$.
\end{proof}

The nice thing about the lattice presentation $F(\cL_M)$ of the foundation $F_M$, in contrast to the embedded minor presentation $F(\cE_M)$, is that it is more economical and therefore better for explicit computations; cf.\ \autoref{subsection: foundation of AG23-e}.

\begin{rem}\label{rem: relations of cross ratios from lattices}
 \textit{A posteriori,} we can associate with the different types of sublattices in $\cL_M$ relations between non-degenerate cross ratios $\cross{H_1}{H_2}{H_3}{H_4}{}$, as follows. (We denote by $H_F$ the hyperplane that appear as `$F$' in the illustration in \autoref{fig: relevant upper sublattices for the foundation}.)\footnote{To be precise, we will reorder the elements of a flat $F$ when this yields a more systematic description of a relation. The relations have to read in the sense that all indices $1$, $2$, $3$, $4$ (and $5$ in \eqref{H3} and \eqref{H4}) can be permuted. In \eqref{H4'}, the permutation of elements must preserve the circuit $123$, so we cannot exchange $4$ and $5$ by any of $1$, $2$ and $3$ in this relation.}

 \begin{itemize}
  \item Type $U^2_4$ corresponds to the relations
 \[\tag{H+}\label{H+}
  \cross {H_1}{H_2}{H_3}{H_4}{} + \cross {H_1}{H_3}{H_2}{H_4}{}  \ = \ 1;
 \]
 \[\tag{H$\sigma$}\label{Hs}
  \cross{H_1}{H_2}{H_3}{H_4}{} \ = \ \cross{H_2}{H_1}{H_4}{H_3}{} \ = \ \cross{H_3}{H_4}{H_1}{H_2}{} \ = \ \cross{H_4}{H_3}{H_2}{H_1}{};
 \]
 \[\tag{H1}\label{H1}
  \cross {H_1}{H_2}{H_4}{H_3}{} \ = \ \crossinv {H_1}{H_2}{H_3}{H_4}{};
 \]
 \[\tag{H2}\label{H2}
  \cross {H_1}{H_2}{H_3}{H_4}{} \cdot \cross {H_1}{H_3}{H_4}{H_2}{}  \cdot \cross {H_1}{H_4}{H_2}{H_3}{} \ = \ -1.
 \]
 \item Type $U^2_5$ corresponds to the relation
 \[\tag{H3}\label{H3}
  \cross {H_1}{H_2}{H_3}{H_4}{} \cdot \cross {H_1}{H_2}{H_4}{H_5}{} \cdot \cross {H_1}{H_2}{H_5}{H_3}{} \ = \ 1.
 \]
 \item Type $U^3_5$ corresponds to the relation
 \[\tag{H4}\label{H4}
  \cross {H_{15}}{H_{25}}{H_{35}}{H_{45}}{} \cdot \cross {H_{13}}{H_{23}}{H_{43}}{H_{53}}{} \cdot \cross {H_{14}}{H_{24}}{H_{54}}{H_{34}}{} \ = \ 1.
 \]
 \item Type $C_5$ (with circuit $123$) corresponds to the relation 
 \[\tag{H4'}\label{H4'}
  \cross {H_{14}}{H_{24}}{H_{34}}{H_{54}}{} \ = \ \cross {H_{15}}{H_{25}}{H_{35}}{H_{45}}{} .
 \]
 \item Types $F_7$ and $F_7^\ast$ correspond to the relation
 \[\tag{H--}\label{H-}
   -1 \ = \ 1.
 \]
 \end{itemize} 
\end{rem}

As a consequence of \autoref{thm: fundamental lattice presentation}, we deduce the following presentation for the foundation, which is helpful for inductive computations of foundations of matroids of large rank in terms of matroids of smaller rank (but with possibly more elements). For a concrete application, see \autoref{prop: foundation of T8}.

Let $M$ be a matroid with lattice $\Lambda_M$. A \emph{full upper sublattice of $\Lambda_M$} is an upper sublattice $\Lambda'$ of $\Lambda_M$ for which an inclusion $F'\subset F$ of flats $F'\in\Lambda'$ and $F\in \Lambda_M$ implies that $F\in\Lambda'$. In other words, if $F_0$ is the bottom of $\Lambda'$, then $\Lambda'$ consists of precisely all flats $F\in\Lambda_M$ that contain $F_0$. In particular, the full upper sublattices of $\Lambda_M$ are in bijection with the elements of $\Lambda_M$.

We write $\cL_M^{\leq r}$ for the diagram of all full upper sublattices of $\Lambda_M$ of rank less or equal to $r$ that contain a sublattice of type $U^2_4$, $F_7$ or $F_7^\ast$, together with all lattice inclusions. Note that every upper sublattice in $\cL_M$ contains an upper sublattice of type $U^2_4$, $F_7$ or $F_7^\ast$.

\begin{thm}\label{thm: fundamental lattice presentation by upper sublattices of small rank}
 Let $M$ be a matroid with lattice $\Lambda_M$. Then the canonical morphism $\colim F(\cL_M^{\leq4})\to F_M$ is an isomorphism. If $M$ either contains a minor of type $F_7$ or \emph{does not} contain a minor of type $F_7^\ast$, then the canonical morphism $\colim F(\cL_M^{\leq3})\to F_M$ is an isomorphism.
\end{thm}

\begin{proof}
 These claims follow from general considerations about colimits. First of all, we observe that every every upper sublattice $\Lambda'$ of $\cL_M$ embeds into a (unique) full upper sublattice $\overline{\Lambda'}$ with the same bottom element, which is in $\Lambda_M^{\leq4}$. Moreover, every lattice inclusion $\Lambda'\to\Lambda''$ in $\cL_M$ extends to a lattice inclusion $\overline{\Lambda'}\to\overline{\Lambda''}$ in $\Lambda_M^{\leq4}$. 
 
Consequently, these lattice inclusions induce a morphism $F_M\simeq\colim F(\cL_M)\to\colim F(\cL_M^{\leq4})$ that is a section to the canonical morphism $\colim F(\cL_M^{\leq4})\to F_M$. For a given upper sublattice $\Lambda'$, we denote by $\cL_{\Lambda'}^{\rk}$ the family of all sublattices $\Lambda''$ of $\Lambda'$ of the same rank that are in $\cL_M$. The inclusion $\Lambda''\to\Lambda'$ induces a morphism $F_{\Lambda''}\to F_{\Lambda'}$ and thus a morphism
 \[
  \bigotimes_{\Lambda''\in\cL_{\Lambda'}^{\rk}} F_{\Lambda''} \ \longrightarrow \ F_{\Lambda'}.
 \]
 This morphism is surjective, and thus an epimorphism, since $F_{\Lambda'}=\colim F(\cL_{M(\Lambda')})$ (by \autoref{thm: fundamental lattice presentation}), and since $\cL_{\Lambda'}^{\rk}$ contains all maximal elements of $\cL_{M(\Lambda')}$, where $M(\Lambda')$ is the simple matroid with lattice $\Lambda'$. We conclude that the section $F_M\to \colim F(\cL_M^{\leq4})$ is an epimorphism, and thus an isomorphism that is inverse to $\colim F(\cL_M^{\leq4})\to F_M$. The first claim follows.
 
 If $M$ has no minor of type $F_7^\ast$, then every sublattice in $\cL_M$ has rank at most $3$ and the above arguments hold with $\cL_M^{\leq4}$ replaced by $\cL_M^{\leq3}$. If $M$ has a minor of type $F_7$, then the corresponding upper sublattice is of rank $3$ and induces the relation $1=-1$ on $F_M$. Therefore we can remove all upper sublattices of type $F_7^\ast$ from $\cL_M$ without changing the colimit and, once again, the above arguments yield the desired result with $\cL_M^{\leq4}$ replaced by $\cL_M^{\leq3}$. The second claim follows.
\end{proof}

Similar arguments lead to the following variant of \autoref{thm: fundamental lattice presentation by upper sublattices of small rank} (we omit a formal proof):

\begin{thm}\label{thm: variant of foundation as colimit of all upper sublattices of small rank}
 Let $M$ be a matroid of rank $\geq 3$ and $\Lambda_M$ its lattice of flats. Let $\cL_M^{3+}$ be the diagram of all upper sublattices of $\Lambda_M$ of the types $U^2_4$ and $F_7^\ast$ and all full upper sublattices of rank $3$, together with all lattice inclusions. Then $F_M\simeq\colim F(\cL_M^{3+})$.
\end{thm}

Finally we mention the following result for future reference.

\begin{lemma}\label{lemma: going up}
 Let $M$ be a matroid with lattice $\Lambda_M$, let $N=M\minor JI\hookrightarrow M$ be an embedded minor with lattice $\Lambda_N$, and let $\Lambda_N\to\Lambda'$ be an inclusion of upper sublattices of $\Lambda_M$. If $\Lambda'$ has the same rank as $\Lambda_N$, then there exists an embedded minor $N'=M\minor{J'}{I}$ with $N\hookrightarrow N'$, i.e.\ $J'\subset J$ and $\Lambda'=\Lambda_{N'}$ as upper sublattices of $\Lambda_M$.
\end{lemma}

\begin{proof}
 Let $G=\gen{I}$ be the bottom element of $\Lambda_N$, which is the smallest flat of $M$ that is contained in $\Lambda_N$. Since $\Lambda'$ is an upper sublattice of the same rank as $\Lambda_N$ and $\Lambda_N\subset\Lambda'$, the flat $G$ is also the bottom element of $\Lambda'$. Consider an atom $F$ of $\Lambda'$, which is a flat of $M$ that covers $G$ and is contained in $\Lambda'$. If $F$ is not in $\Lambda_N$, then $F-G\subset J$. We choose, for every atom $F$ of $\Lambda'-\Lambda_N$, an element $e_F$ in $F-G$ and define 
 \[
  J' \ = \ J \ - \ \{ \, e_F\mid \text{atoms }F\text{ of }\Lambda'-\Lambda_N \, \}
 \]
 and $N'=M\minor{J'}I$. Then $F=\gen{Ie_F}$ is an atom of $\Lambda_{N'}$ and all other atoms of $\Lambda_{N'}$ are atoms of $\Lambda_N$. This shows that the atoms of $\Lambda_{N'}$ and $\Lambda'$ agree. Since upper sublattices are atomistic, this shows that $\Lambda'=\Lambda_{N'}$. 
\end{proof}

\begin{rem}
 The condition that $\Lambda'$ and $\Lambda_N$ have the same rank cannot be dropped in \autoref{lemma: going up} in general. For instance, for $M=U^2_3$ with ground set $E=\{1,2,3\}$, the upper sublattice $\Lambda_N$ of the minor $N=M/12$ consists of the single element $123$ and thus is a sublattice of $\Lambda'=\Lambda_{M/3}$. But the lattice inclusion $\Lambda_N\to\Lambda'$ is not induced by a minor embedding $N\hookrightarrow N'$ for any embedded minor $N'$ of $M$.
\end{rem}

We conclude with the observation that equalities between universal cross ratios are invariant under automorphisms of the matroid and simultaneous permutation of the $4$ entries in a cross ratio.

\begin{prop}\label{prop: invariance of cross ratio equalities under permutations}
 Let $M$ be a matroid with ground set $E$ and foundation $F_M$. Let $\varphi:E\to E$ be an automorphism of $M$ and $\sigma\in S_4$ a permutation. Let $H_1,\dotsc, H_4$ and $H'_1,\dotsc, H'_4$ be two modular quadruples of hyperplanes such that $\cross{H_1}{H_2}{H_3}{H_4}{}=\cross{H'_1}{H'_2}{H'_3}{H'_4}{}$ as elements of $F_M$. Then 
 \[
  \Bigg[\begin{matrix} \ \varphi(H_{\sigma(1)}) & \varphi(H_{\sigma(2)}) \ \\[5pt] \varphi(H_{\sigma(3)}) & \varphi(H_{\sigma(4)}) \end{matrix} \Bigg] \quad = \quad \Bigg[\begin{matrix} \ \varphi(H'_{\sigma(1)}) & \varphi(H'_{\sigma(2)}) \ \\[5pt] \varphi(H'_{\sigma(3)}) & \varphi(H'_{\sigma(4)}) \end{matrix} \Bigg]. 
 \]
\end{prop}

\begin{proof}
 By \autoref{thm: fundamental lattice presentation} and \autoref{rem: relations of cross ratios from lattices}, all relations between cross ratios are a concatenation of the relations of types \eqref{H+}--\eqref{H-}. Since an automorphism of $M$ induces an automorphism of the lattice $\Lambda_M$ of flats of $M$, which permutes the upper sublattices of $M$ in an inclusion and type preserving way, it also preserves equalities between cross ratios.
 
 To establish our assertion for a permutation $\sigma$ of the $4$ entries of cross ratios, we note that any equality between cross ratios stems from a chain of relations of types \eqref{Hs} and \eqref{H4'}. Type \eqref{Hs} is evidently invariant under $\sigma$. To deduce the invariance for \eqref{H4'}, we note that the inclusions $\Lambda'/i\to\Lambda'\leftarrow\Lambda'/j$ of the two upper sublattices $\Lambda'/i$ and $\Lambda'/j$ of types $U^2_4$ into an upper sublattice $\Lambda'$ of type $C_5$ yields isomorphisms $\U\to\U\leftarrow\U$. We conclude that all cross ratios of $\Lambda'/i$ are identified with those of $\Lambda/j$, in a way that is stable under permutations of the four entries of cross ratios.
\end{proof}

\subsection{Foundations of matroids without large uniform minors}
\label{subsection: foundations of matroids without large uniform minors}

A matroid is \emph{without large uniform minors} if it has no minors of type $U^2_5$ or $U^3_5$. These are the only types of minors in the fundamental presentation $\cE_M$ of $M$ whose foundation is $\V$. 

Thus, for a matroid $M$ without large uniform minors, the fundamental diagram $F(\cE_M)$ consists of copies of $\F_2$ and $\U$. More precisely, copies of $\F_2$ appear as isolated vertices in $F(\cE_M)$, and any morphism between two copies of $\U$ is an isomorphism. Thus, the colimit of a connected component $\cC$ of $F(\cE_M)$ is either $\F_2$ or a \emph{symmetry quotient of $\U$}, by which we mean a quotient of $\U$ by a group of automorphisms.

The automorphism group of $\U$ is $\Aut(\U)=S_3$ (cf.\ \cite[Prop.\ 5.6]{Baker-Lorscheid20}), and the non-trivial symmetry quotients of $\U$ are 
\begin{align*}
 \D \ &= \ \pastgenn{\Funpm}{x}{x-1-1},\\
 \H \ &= \ \pastgenn{\Funpm}{\zeta_6}{\zeta_6^3+1,\ \zeta_6+\zeta_6^{-1}-1}, \\
 \F_3 \ &= \ \past{\Funpm}{}{\genn{1+1+1}}
\end{align*}
(cf.\ \cite[Prop.\ 5.8]{Baker-Lorscheid20}). Since the colimit of a diagram is the coproduct of the colimits of its connected components, this recovers the \emph{structure theorem for foundations of matroids without large uniform minors} (cf.\ \cite[Thm.\ 5.9]{Baker-Lorscheid20}):

\begin{thm}\label{thm: structure theorem for foundations of matroids without large uniform minors}
 Let $M$ be a matroid without large uniform minors, $F_M$ its foundation, and $r$ the number of connected components of the fundamental diagram of $M$. Then
 \begin{equation} \label{eq:WLUM}
  F_M \ \simeq \ F_1\otimes\dotsb\otimes F_r
 \end{equation}
 with $F_1,\dotsc,F_r\in\{\U,\ \D,\ \H,\ \F_3,\ \F_2\}$. 
 \qed
\end{thm}


Note that (the underlying graph of) the fundamental lattice diagram $\cL_M$ is a contraction of (the underlying graph of) the fundamental diagram $\cE_M$ and has, in particular, the same number of connected components. A result that we apply in the computations of several examples in \autoref{section: more examples of foundations} and \autoref{section: final examples} is the following.


\begin{cor}\label{cor: foundation of a matroid wlum with one component}
 Let $M$ be a matroid without minors of types $U^2_5$, \ $U^3_5$,\ $F_7$ or $F_7^\ast$, and assume that its fundamental (lattice) diagram is connected and non-empty. Then the foundation of $M$ is a symmetry quotient of $\U$. \qed
\end{cor}

\section{More examples}
\label{section: more examples of foundations}

We apply the methods from the previous sections to compute the foundation of $Q_6$ (following Oxley's notation in \cite[p.\ 641]{Oxley92}) and the one element restriction $AG(2,3)\setminus e$ of the ternary affine plane (cf.\ \cite[p.\ 653]{Oxley92}).


\subsection{The foundation of \texorpdfstring{$Q_6$}{Q6}}
\label{subsection: foundation of Q6}

The matroid $Q_6$ is the rank $3$ matroid on $6$ elements whose $3$-circuits are illustrated in \autoref{fig: circuits of Q6}.

\begin{figure}[htb]
 \[
 \beginpgfgraphicnamed{tikz/fig11}
 \begin{tikzpicture}[x=0.5cm,y=0.5cm]
  \filldraw ( 30:1) circle (2pt);  
  \node at ( 30:1.6) {\footnotesize $6$};  
  \filldraw (150:1) circle (2pt);  
  \node at (150:1.6) {\footnotesize $4$};  
  \filldraw (270:1) circle (2pt);  
  \node at (270:1.6) {\footnotesize $2$};  
  \filldraw ( 90:2) circle (2pt);  
  \node at ( 90:2.6) {\footnotesize $5$};  
  \filldraw (210:2) circle (2pt);  
  \node at (210:2.6) {\footnotesize $1$};  
  \filldraw (330:2) circle (2pt);  
  \node at (330:2.6) {\footnotesize $3$};  
  \draw [thick] ( 90:2) -- (210:2);
  \draw [thick] (210:2) -- (330:2);
 \end{tikzpicture}
 \endpgfgraphicnamed
 \]
 \caption{The circuits of $Q_6$}
 \label{fig: circuits of Q6} 
\end{figure}
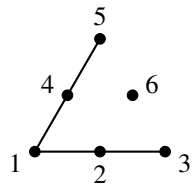

\begin{prop}\label{prop: foundation of Q6}
The foundation of $Q_6$ is $\V$.
\end{prop}

\begin{proof}
 The matroid $M=Q_6$ has one embedded minor of type $U^2_5$, which is $M/6$, and one embedded minor of type $U^3_5$, which is $M\setminus1$. It has several embedded minors of types $U^2_4$, $C_5$, and $C_5^\ast$, and none of types $F_7$ and $F_7^\ast$. The fundamental diagram is illustrated in \autoref{fig: fundamental diagram of Q6}:
 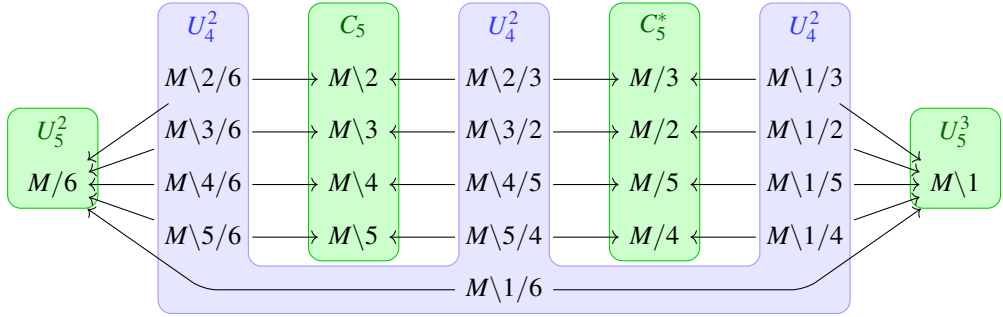
\begin{figure}[htb]
 \[
  \beginpgfgraphicnamed{tikz/fig15}
  \begin{tikzpicture}[x=2.0cm,y=0.7cm]
   \draw[fill=green!20!white,draw=green!80!black,rounded corners=5pt] (-0.3,1.55) rectangle (0.3,3.45); 
   \draw[fill=blue!10!white,draw=blue!40!white,rounded corners=5pt] (0.7,-0.45) -- (0.7,5.45) -- (1.3,5.45) -- (1.3,0.45) -- (2.7,0.45) -- (2.7,5.45) -- (3.3,5.45) -- (3.3,0.45) -- (4.7,0.45) -- (4.7,5.45) -- (5.3,5.45) -- (5.3,-0.45) -- cycle; 
   \draw[fill=green!20!white,draw=green!80!black,rounded corners=5pt] (1.7,0.55) rectangle (2.3,5.45); 
   \draw[fill=green!20!white,draw=green!80!black,rounded corners=5pt] (3.7,0.55) rectangle (4.3,5.45); 
   \draw[fill=green!20!white,draw=green!80!black,rounded corners=5pt] (5.7,1.55) rectangle (6.3,3.45); 
   \node (M6) at (0,2) {\footnotesize $M/6$};
   \node (M26) at (1,4) {\footnotesize $M\minor26$};
   \node (M36) at (1,3) {\footnotesize $M\minor36$};
   \node (M46) at (1,2) {\footnotesize $M\minor46$};
   \node (M56) at (1,1) {\footnotesize $M\minor56$};
   \node (M-2) at (2,4) {\footnotesize $M\setminus2$};
   \node (M-3) at (2,3) {\footnotesize $M\setminus3$};
   \node (M-4) at (2,2) {\footnotesize $M\setminus4$};
   \node (M-5) at (2,1) {\footnotesize $M\setminus5$};
   \node (M23) at (3,4) {\footnotesize $M\minor23$};
   \node (M32) at (3,3) {\footnotesize $M\minor32$};
   \node (M45) at (3,2) {\footnotesize $M\minor45$};
   \node (M54) at (3,1) {\footnotesize $M\minor54$};
   \node (M16) at (3,0) {\footnotesize $M\minor16$};
   \node (M3) at (4,4) {\footnotesize $M/3$};
   \node (M2) at (4,3) {\footnotesize $M/2$};
   \node (M5) at (4,2) {\footnotesize $M/5$};
   \node (M4) at (4,1) {\footnotesize $M/4$};
   \node (M13) at (5,4) {\footnotesize $M\minor13$};
   \node (M12) at (5,3) {\footnotesize $M\minor12$};
   \node (M15) at (5,2) {\footnotesize $M\minor15$};
   \node (M14) at (5,1) {\footnotesize $M\minor14$};
   \node (M1) at (6,2) {\footnotesize $M\setminus1$};
   \node[color=green!40!black] at (0,3) {\footnotesize $U^2_5$};
   \node[color=blue!80!white] at (1,5) {\footnotesize $U^2_4$};
   \node[color=green!40!black] at (2,5) {\footnotesize $C_5$};
   \node[color=blue!80!white] at (3,5) {\footnotesize $U^2_4$};
   \node[color=green!40!black] at (4,5) {\footnotesize $C_5^\ast$};
   \node[color=blue!80!white] at (5,5) {\footnotesize $U^2_4$};
   \node[color=green!40!black] at (6,3) {\footnotesize $U^3_5$};
   \draw[->] (M26) to (M6);
   \draw[->] (M36) to (M6);
   \draw[->] (M46) to (M6);
   \draw[->] (M56) to (M6);
   \draw[->] (M26) to (M-2);
   \draw[->] (M36) to (M-3);
   \draw[->] (M46) to (M-4);
   \draw[->] (M56) to (M-5);
   \draw[->] (M23) to (M-2);
   \draw[->] (M32) to (M-3);
   \draw[->] (M45) to (M-4);
   \draw[->] (M54) to (M-5);
   \draw[->] (M23) to (M3);
   \draw[->] (M32) to (M2);
   \draw[->] (M45) to (M5);
   \draw[->] (M54) to (M4);
   \draw[->] (M13) to (M3);
   \draw[->] (M12) to (M2);
   \draw[->] (M15) to (M5);
   \draw[->] (M14) to (M4);
   \draw[->] (M13) to (M1);
   \draw[->] (M12) to (M1);
   \draw[->] (M15) to (M1);
   \draw[->] (M14) to (M1);
   \draw[->,rounded corners=5pt] (M16) -- (1,0) -- (M6);
   \draw[->,rounded corners=5pt] (M16) -- (5,0) -- (M1);
  \end{tikzpicture}
  \endpgfgraphicnamed
  \] 
  \caption{The fundamental diagram $\cE_{Q_6}^\min$ of embedded minors of $Q_6$}
  \label{fig: fundamental diagram of Q6} 
 \end{figure}

 In order to compute the colimit of $F(\cE_M)$, we investigate how the foundations of the embedded minors become identified. Recall from \autoref{prop: foundation of U25} that the foundation of $M/6\simeq U^2_5$ is isomorphic to 
 \[
  \V \ = \ \pastgenn{\Funpm}{x_i\mid i=1,\dotsc,5}{x_i+x_{i-1}x_{i+1}-1\mid i=1,\dotsc,5}
 \]
 via the isomorphism $\alpha:\V\to F_{M/6}$ with 
 \[
  x_1 \mapsto \cross{2}{3}{5}{4}{6}, \quad \ \ 
  x_2 \mapsto \cross{3}{4}{1}{5}{6}, \quad \ \ 
  x_3 \mapsto \cross{4}{5}{2}{1}{6}, \quad \ \ 
  x_4 \mapsto \cross{5}{1}{3}{2}{6}, \quad \ \ 
  x_5 \mapsto \cross{1}{2}{4}{3}{6}.
 \]
 
 Similarly, there is an isomorphism $\beta:\V\to F_{M\setminus1}$ with
 \[
  x_1 \mapsto \cross{2}{3}{5}{4}{6}, \quad \ \ 
  x_2 \mapsto \cross{2}{5}{6}{4}{3}, \quad \ \ 
  x_3 \mapsto \cross{5}{4}{3}{6}{2}, \quad \ \ 
  x_4 \mapsto \cross{4}{6}{2}{3}{5}, \quad \ \ 
  x_5 \mapsto \cross{6}{3}{5}{2}{4}.
 \]
 
 The isomorphisms $\alpha$ and $\beta$ are compatible with the minor embeddings $M/6\hookrightarrow M$ and $M\setminus 1\hookrightarrow M$ in the sense that the respective images of the $x_i$ become identified in $F_M$, as follows from the computation below. Recall from \cite[Prop.\ 5.3]{Baker-Lorscheid20} that the relations \eqref{R0}, \eqref{R1}, and \eqref{R4} imply that the cross ratios of the two embedded $U^2_4$-minors of a $C_5$-minor are pairwise identified, e.g.,\
 \[
  \cross 3ijk6 \ = \ \cross 6ijk3
 \]
 in $M\setminus 2$, which also applies if the four entries of the cross ratios are simultaneously permuted. Similarly, we find identifications
 \[
  \cross ijk12 \ = \ \cross ijk32
 \]
 in $M/2$; cf.\ also \cite[Prop.\ 5.3]{Baker-Lorscheid20}. We label these equalities by the respective embedded minors that induce them in the following equations.
 
 Note that every cross ratio stems from a (unique) embedded minor of type $U^2_4$. We verify that
 \begin{align*}
  \alpha(x_1) \ &= \ \cross{2}{3}{5}{4}{6} \ = \ \beta(x_1), \\
  \alpha(x_2) \ &= \ \cross{3}{4}{1}{5}{6} \ \underset{M\setminus2}= \ \cross{6}{4}{1}{5}{3} \ \underset{M/3}= \ \cross{6}{4}{2}{5}{3} \ \underset{\eqref{Hs}}= \ \cross{2}{5}{6}{4}{3} \ = \ \beta(x_2), \\
  \alpha(x_3) \ &= \ \cross{4}{5}{2}{1}{6} \ \underset{M\setminus3}= \ \cross{4}{5}{6}{1}{2} \ \underset{M/2}= \ \cross{4}{5}{6}{3}{2} \ \underset{\eqref{Hs}}= \ \cross{5}{4}{3}{6}{2} \ = \ \beta(x_3), \\
  \alpha(x_4) \ &= \ \cross{5}{1}{3}{2}{6} \ \underset{M\setminus4}= \ \cross{6}{1}{3}{2}{5} \ \underset{M/5}= \ \cross{6}{4}{3}{2}{5} \ \underset{\eqref{Hs}}= \ \cross{4}{6}{2}{3}{5} \ = \ \beta(x_4), \textrm{\; and} \\
  \alpha(x_5) \ &= \ \cross{1}{2}{4}{3}{6} \ \underset{M\setminus5}= \ \cross{1}{2}{6}{3}{4} \ \underset{M/4}= \ \cross{5}{2}{6}{3}{4} \ \underset{\eqref{Hs}}= \ \cross{6}{3}{5}{2}{4} \ = \ \beta(x_5)
 \end{align*}
 as elements of $F_M$. 
 
 Since every upper sublattice of type $U^2_4$ is contained in either $\Lambda/6$ or $\Lambda\setminus1$, the foundation $F_M$ is generated by the images of $F_{\Lambda/6}$ and $F_{\Lambda\setminus1}$. Since we have taken all relations into consideration (note that the $y_i\in\V$ are the unique elements with $x_i+y_i=1$, so $\alpha(y_i)$ equals $\beta(y_i)$ in $F_M$ as well), it follows that the canonical morphisms $F_{\Lambda/6}\to F_M$ and $F_{\Lambda\setminus1}\to F_M$ are isomorphisms, and thus $F_M\simeq\V$.
\end{proof}


\subsection{The foundation of the M\"obius-Kantor configuration \texorpdfstring{$AG(2,3)\setminus e$}{AG(2,3)-e}}
\label{subsection: foundation of AG23-e}

The M\"obius-Kantor configuration is the single-element deletion of the ternary affine plane $AG(2,3)$. Since the automorphism group of $AG(2,3)$ acts transitively on its points, all of the single-element deletions of $AG(2,3)$ are isomorphic to each other. We write $AG(2,3)\setminus e$ for one of these minors. \autoref{fig: circuit diagram of AG23-e} contains a depiction of its $3$-circuits, and \autoref{fig: the lattice of flats of AG23-e} illustrates its lattice of flats.

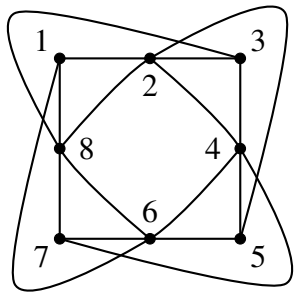
\begin{figure}[htb] 
\[
 \beginpgfgraphicnamed{tikz/fig8}
 \begin{tikzpicture}[x=1.2cm,y=1.2cm]
  \filldraw ( 1, 1) circle (2pt);
  \filldraw ( 0, 1) circle (2pt);
  \filldraw (-1, 1) circle (2pt);
  \filldraw ( 1,-1) circle (2pt);
  \filldraw ( 0,-1) circle (2pt);
  \filldraw (-1,-1) circle (2pt);
  \filldraw ( 1, 0) circle (2pt);
  \filldraw (-1, 0) circle (2pt);
  \draw[thick] (-1,-1) -- ( 1,-1);
  \draw[thick] ( 1,-1) -- ( 1, 1);
  \draw[thick] ( 1, 1) -- (-1, 1);
  \draw[thick] (-1, 1) -- (-1,-1);
  \draw[thick] plot [smooth] coordinates { (-1, 0) ( 0, 1) ( 1.5, 1.5) ( 1,-1) };
  \draw[thick] plot [smooth] coordinates { ( 0, 1) ( 1, 0) ( 1.5,-1.5) (-1,-1) };
  \draw[thick] plot [smooth] coordinates { ( 1, 0) ( 0,-1) (-1.5,-1.5) (-1, 1) };
  \draw[thick] plot [smooth] coordinates { ( 0,-1) (-1, 0) (-1.5, 1.5) ( 1, 1) };
  \node at (-1.2, 1.2) {$1$};
  \node at ( 1.2, 1.2) {$3$};
  \node at ( 1.2,-1.2) {$5$};
  \node at (-1.2,-1.2) {$7$};
  \node at (   0, 0.7) {$2$};
  \node at ( 0.7, 0  ) {$4$};
  \node at (   0,-0.7) {$6$};
  \node at (-0.7, 0  ) {$8$};
 \end{tikzpicture}
 \endpgfgraphicnamed
 \]
 \caption{The $3$-circuits of $AG(2,3)\setminus e$} 
 \label{fig: circuit diagram of AG23-e}
\end{figure}

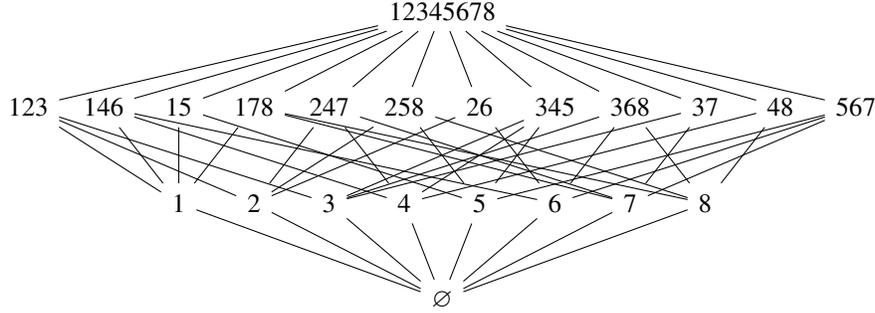
\begin{figure}[htb] 
\[
 \beginpgfgraphicnamed{tikz/fig35}
\begin{tikzpicture}[scale=0.8, vertices/.style={draw, fill=black, circle, inner sep=0pt},font=\footnotesize]
              \node  (0) at (-0+0,0) {$\emptyset$};
              \node  (1) at (-35/8+0,8/5){1};
              \node  (2) at (-35/8+5/4,8/5){2};
              \node  (3) at (-35/8+5/2,8/5){3};
              \node  (4) at (-35/8+15/4,8/5){4};
              \node  (5) at (-35/8+5,8/5){5};
              \node  (6) at (-35/8+25/4,8/5){6};
              \node  (7) at (-35/8+15/2,8/5){7};
              \node  (8) at (-35/8+35/4,8/5){8};
              \node  (9) at (-55/8+0,16/5){123};
              \node  (10) at (-55/8+5/4,16/5){146};
              \node  (11) at (-55/8+5/2,16/5){15};
              \node  (12) at (-55/8+15/4,16/5){178};
              \node  (13) at (-55/8+5,16/5){247};
              \node  (14) at (-55/8+25/4,16/5){258};
              \node  (15) at (-55/8+15/2,16/5){26};
              \node  (16) at (-55/8+35/4,16/5){345};
              \node  (17) at (-55/8+10,16/5){368};
              \node  (18) at (-55/8+45/4,16/5){37};
              \node  (19) at (-55/8+25/2,16/5){48};
              \node  (20) at (-55/8+55/4,16/5){567};
              \node  (21) at (-0+0,24/5){12345678};
      \foreach \to/\from in {0/1, 0/2, 0/3, 0/4, 0/5, 0/6, 0/7, 0/8, 1/9, 1/10, 1/11, 1/12, 2/9, 2/13, 2/14, 2/15, 3/16, 3/17, 3/9, 3/18, 4/10, 4/13, 4/16, 4/19, 5/11, 5/14, 5/20, 5/16, 6/10, 6/15, 6/17, 6/20, 7/12, 7/20, 7/13, 7/18, 8/17, 8/19, 8/12, 8/14, 9/21, 10/21, 11/21, 12/21, 13/21, 14/21, 15/21, 16/21, 17/21, 18/21, 19/21, 20/21}
      \draw [-] (\to)--(\from);
      \end{tikzpicture}
 \endpgfgraphicnamed
 \] 
  \caption{The lattice of flats of $AG(2,3)\setminus e$} 
 \label{fig: the lattice of flats of AG23-e}     
\end{figure}

\begin{prop}\label{prop: foundation of AG23-e}
 The foundation of $M = AG(2,3)\setminus e$ is isomorphic to $\H$. 
\end{prop}

\begin{proof}
Let $\Lambda=\Lambda_M$ be the lattice of flats of $M=AG(2,3)\setminus e$ as illustrated in \autoref{fig: the lattice of flats of AG23-e}. It contains 8 upper sublattices of type $U^2_4$, whose respective bottom elements are precisely the $8$ atoms of $\Lambda$. We denote the $U^2_4$-sublattice with bottom element $i$ by $\Lambda/i$ for $i = 1,\dotsc,8$. The lattice $\Lambda$ contains $32$ upper sublattices of type $C_5$, which can be obtained from $AG(2,3)\setminus e$ by deleting $3$ elements $i$, $j$ and $k$ that either form a $3$-circuit (e.g.\ $\{i,j,k\}=\{1,2,3\}$) or such that $|i-j|=4$ (e.g.\ $\{i,j,k\}=\{1,5,6\}$). We denote these sublattices by $\Lambda\setminus{jkl}$. These are all upper sublattices that appear in the fundamental lattice diagram $\cL_M$ of $M$, which is illustrated in \autoref{fig: fundamental lattice subdiagram of AG23-e}, where we label the sublattices $\Lambda\setminus ijk$ of type $C_5$ by ``$ijk$'' due to space limitations.

Since the fundamental lattice diagram $\cL_M$ is connected and does not contain upper sublattices of types $U^2_5$,\ $U^3_5$,\ $F_7$ or $F_7^\ast$, \autoref{cor: foundation of a matroid wlum with one component} implies that the foundation $F_M$ of $M$ is a symmetry quotient of $\U$.

\newcounter{tikz-counter}

\begin{figure}[tb] 
 \[
  \beginpgfgraphicnamed{tikz/fig37}
   \begin{tikzpicture}[scale=1, font=\footnotesize, vertices/.style={draw, fill=black, circle, inner sep=0pt},decoration={markings,mark=at position 0.33 with {\arrow{>}}}]
    \draw[fill=blue!10!white,draw=blue!40!white,rounded corners=5pt] (22.5:5.3cm) -- (67.5:5.3cm) -- (112.5:5.3cm) -- (157.5:5.3cm) -- (202.5:5.3cm) -- (247.5:5.3cm) -- (292.5:5.3cm) -- (337.5:5.3cm) -- cycle;
    \draw[fill=white,draw=blue!40!white,rounded corners=5pt] (22.5:4.5cm) -- (45:3.7cm) -- (67.5:4.5cm) -- (90:3.7cm) -- (112.5:4.5cm) -- (135:3.7cm) -- (157.5:4.5cm) -- (180:3.7cm) -- (202.5:4.5cm) -- (225:3.7cm) -- (247.5:4.5cm) -- (270:3.7cm) -- (292.5:4.5cm) -- (315:3.7cm) -- (337.5:4.5cm) -- (0:3.7cm) -- cycle;
    \draw[fill=green!20!white,draw=green!80!black,rounded corners=5pt] (22.5:4.4cm) -- (45:3.6cm) -- (67.5:4.4cm) -- (90:3.6cm) -- (112.5:4.4cm) -- (135:3.6cm) -- (157.5:4.4cm) -- (180:3.6cm) -- (202.5:4.4cm) -- (225:3.6cm) -- (247.5:4.4cm) -- (270:3.6cm) -- (292.5:4.4cm) -- (315:3.6cm) -- (337.5:4.4cm) -- (0:3.6cm) -- cycle;
   \node[color=blue!80!white] at (22.5:4.8) {\footnotesize  $U^2_4$};
   \node[color=blue!80!white] at (112.5:4.8) {\footnotesize  $U^2_4$};
   \node[color=blue!80!white] at (202.5:4.8) {\footnotesize  $U^2_4$};
   \node[color=blue!80!white] at (292.5:4.8) {\footnotesize  $U^2_4$};
   \node[color=green!40!black] at (0,0) {\normalsize $C_5$};
    \foreach \a in {1,...,16}{\draw (135-\a*360/8: 4cm) node [draw,circle,inner sep=2pt,fill=black] (L\a) {};}
    \foreach \a in {1,...,8} {\draw ((135-\a*360/8: 4.5cm) node {$\Lambda/\a$};}
    \foreach \a in {1,...,8} {\setcounter{tikz-counter}{\a};
                              \addtocounter{tikz-counter}{1};
                              \draw (L\a) edge [-,postaction={decorate},bend left=0] node[midway,draw,circle,inner sep=1.5pt,fill=black]{} (L\arabic{tikz-counter});
                              \addtocounter{tikz-counter}{1};
                              \draw (L\a) edge [-,postaction={decorate},bend left=15] node[midway,draw,circle,inner sep=1.5pt,fill=black]{} (L\arabic{tikz-counter});
                              \addtocounter{tikz-counter}{1};
                              \draw (L\a) edge [-,postaction={decorate},bend left=22] node[midway,draw,circle,inner sep=1.5pt,fill=black]{} (L\arabic{tikz-counter});
                              \addtocounter{tikz-counter}{1};
                              \draw (L\a) edge [-,postaction={decorate},bend left=25] node[midway,draw,circle,inner sep=1.5pt,fill=black]{} (L\arabic{tikz-counter});
                              \draw (L\a) edge [-,postaction={decorate},bend right=25] node[midway,draw,circle,inner sep=1.5pt,fill=black]{} (L\arabic{tikz-counter});
                              \addtocounter{tikz-counter}{1};
                              \draw (L\a) edge [-,postaction={decorate},bend right=22] node[midway,draw,circle,inner sep=1.5pt,fill=black]{} (L\arabic{tikz-counter});
                              \addtocounter{tikz-counter}{1};
                              \draw (L\a) edge [-,postaction={decorate},bend right=15] node[midway,draw,circle,inner sep=1.5pt,fill=black]{} (L\arabic{tikz-counter});
                              \addtocounter{tikz-counter}{1};
                              \draw (L\a) edge [-,postaction={decorate},bend right=0] node[midway,draw,circle,inner sep=1.5pt,fill=black]{} (L\arabic{tikz-counter});
                              }
   \foreach \angle/\label in {1/483,2/481,3/265,4/263,5/487,6/485,7/261,8/267}
    {\draw ((112.5-\angle*360/8:4.05cm) node {\tiny$\label$};}
   \foreach \angle/\label in {1/375,2/482,3/157,4/264,5/371,6/486,7/153,8/268}
    {\draw ((135-\angle*360/8:2.9cm) node {\tiny$\label$};}
   \foreach \angle/\label in {1/156,2/376,3/378,4/158,5/152,6/372,7/374,8/154}
    {\draw ((112.5-\angle*360/8:2.67cm) node {\tiny$\label$};}
   \foreach \angle/\label in {1/146,2/123,3/247,4/178,5/258,6/567,7/368,8/345}
    {\draw ((135-\angle*360/8:0.7cm) node {\tiny$\label$};}
   \end{tikzpicture}
  \endpgfgraphicnamed
 \] 
  \caption{The fundamental lattice diagram of $AG(2,3)\setminus e$} 
 \label{fig: fundamental lattice subdiagram of AG23-e}     
\end{figure}


By \cite[p.\ 653]{Oxley92}, $AG(2,3)$ is quarternary, and so is $AG(2,3)\setminus e$. Therefore $F_M$ cannot be $\D$ nor $\F_3$, which leaves only the possibilities $\U$ and $\H$. We exhibit the defining relation $x\sim y^{-1}$ of $\H=\past{\U}{\{x-y^{-1}\}}$ in the following.

Consider the map $\alpha:\U\simeq F_{\Lambda/1}\to F_M$ with $\alpha(x)=\cross{123}{146}{15}{178}{}$, and thus $\alpha(y^{-1})=\cross{123}{15}{178}{146}{}$. We use the following part of $\cL_M$ to identify $\alpha(x)$ with $\alpha(y^{-1})$ in $F_M$:
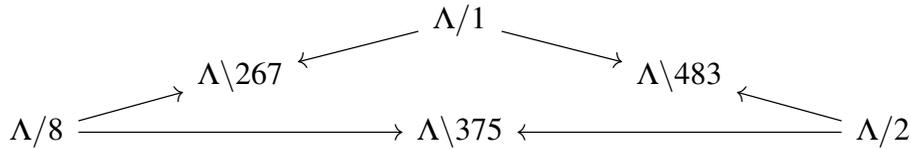
\[
 \begin{tikzcd}[row sep=0,column sep=40]
  & & \Lambda/1 \ar[dl] \ar[dr] \\
  & \Lambda\setminus 267 && \Lambda\setminus 483 \\
  \Lambda / 8 \ar[rr] \ar[ur] && \Lambda\setminus 375 && \Lambda/2 \ar[ll] \ar[ul] 
 \end{tikzcd}
\]
In the following computation, we label an equality of cross ratios in $F_M$ that is induced by an upper sublattice of type $C_5$ by the corresponding sublattice. We have
\begin{multline*}
 \alpha(x) \ = \ \cross{123}{146}{15}{178}{} \ \underset{\Lambda\setminus 267}= \ \cross{368}{48}{258}{178}{} \ \underset{\Lambda\setminus 375}= \ \cross{26}{247}{258}{123}{} \ \underset{\Lambda\setminus 483}= \ \cross{146}{178}{15}{123}{} \\ 
 \underset{\eqref{Hs}}= \ \cross{123}{15}{178}{146}{} \ = \ \alpha(y^{-1}),
\end{multline*}
which proves our assertion that $F_M\simeq \past{\U}{\{x-y^{-1}\}}=\H$.
\end{proof}

\section{Fundamental types}
\label{subsection: fundamental types}

\subsection{The fundamental type of a class of matroids}
\label{subsection: fundamental presentation for subclasses}

In \autoref{subsection: presentation of the foundation by embedded minors}, we have seen that the foundation of a matroid $M$ can be presented as the colimit of the foundations of all special embedded minors of $M$, which consists of the isomorphism types of $U^2_4$, $U^2_5$, $U^3_5$, $C_5$, $C_5^\ast$, $U^2_4\oplus U^1_2$, $F_7$ and $F_7^\ast$. In fact, this presentation is minimal: the foundation of none of the matroids in this list can be written as the colimit of the foundations of its proper minors; cf.\ \autoref{rem: fundamental diagram of U12+U24}.

If we are given a class $\cC$ of matroids, it is natural to wonder if there is a minimal set $\cC_0$ of isomorphism classes of matroids in $\cC$ such that the foundation of every matroid $M$ in $\cC$ is the colimit of the foundations of all embedded minors of $M$ whose isomorphism type is in $\cC_0$. 

Given a matroid $M$ and a family $\cC_0$ of isomorphism types of matroids, we denote by $\cE_{M,\cC_0}$ the diagram of all embedded minors $N$ of $M$ whose isomorphism class belongs to $\cC_0$, together with all minor embeddings. We denote by $\Omega_{M,\cC_0}:\colim F(\cE_{M,\cC_0})\to F_M$ the canonical morphism from the colimit of the foundations of embedded minors of $M$ with isomorphism type $\cC_0$ into the foundation of $M$.

\begin{df}
 Let $\cC$ be a class of matroids. A \emph{fundamental type of $\cC$} is a minimal family $\cC_0$ of isomorphism classes of matroids such that $\Omega_{M,\cC_0}:\colim F(\cE_{M,\cC_0})\to F_M$ is an isomorphism for all $M\in\cC$.
\end{df}

It turns out that every class of matroids possesses a unique fundamental type. Before we prove this, let us consider some examples:
\begin{itemize}
 \item The fundamental type of all matroids consists of $U^2_4$, $U^2_5$, $U^3_5$, $C_5$, $C_5^\ast$, $U^2_4\oplus U^1_2$, $F_7$ and $F_7^\ast$ (by \autoref{thm: fundamental presentation}).
 \item The fundamental type of all regular matroids is empty since the foundation of a regular matroid is $\Funpm$ (by \cite[Thm.\ 7.35]{Baker-Lorscheid21b}).
 \item The fundamental type of all binary matroids consists of $F_7$ and $F_7^\ast$, since the foundation of a binary matroid is $\Funpm$ or $\F_2$ (by \cite[Thm.\ 7.32]{Baker-Lorscheid21b}).
 \item The fundamental type of matroids without $U^2_5$ and $U^3_5$-minors consists of $U^2_4$, $C_5$, $C_5^\ast$, $U^2_4\oplus U^1_2$, $F_7$ and $F_7^\ast$ (by \autoref{thm: fundamental presentation}; cf.\ also \cite[Thm.\ 5.9]{Baker-Lorscheid20}).
\end{itemize}

In \autoref{subsection: presentation of the foundation by embedded minors for 2-connected matroids} (resp. \autoref{subsection: presentation of the foundation by embedded minors for 3-connected matroids}), we describe the fundamental types of the classes of all $2$-connected (resp. $3$-connected) matroids.

The next result shows that a fundamental type always exists, and is uniquely determined through an explicit description. For $M$ in $\cC$, we denote by $\cE^{<}_{M,\cC}$ the diagram of all \emph{proper} embedded minors $N=M\minor JI$ (i.e.\ $N\neq M$) that are in $\cC$, together with all minor embeddings. We denote by $\Omega^<_{M,\cC}:\colim F(\cE^{<}_{M,\cC})\to F_M$ the canonical morphism.

\begin{thm}\label{thm: fundamental type}
 Every class $\cC$ of matroids has a unique fundamental type $\cC_\fundtype$, which consists of all isomorphism classes of matroids $M$ in $\cC$ for which $\Omega^<_{M,\cC}:\colim F(\cE^{<}_{M,\cC})\to F_M$ is \emph{not} an isomorphism.
\end{thm}

\begin{proof}
 Let $\cC_\fundtype$ be the collection of isomorphism classes of matroids $M$ in $\cC$ for which the canonical map $\Omega^<_{M,\cC}:\colim F(\cE^{<}_{M,\cC})\to F_M$ is not an isomorphism. As a first step we show that $\cC_\fundtype$ is contained in every fundamental type $\cC_0$ of $\cC$.
 
 Let $M$ be a matroid that does not belong to $\cC_0$. Then the diagram $\cE_{M,\cC_0}$ of embedded minors of $M$ with isomorphism class in $\cC_0$ is a subdiagram of $\cE^<_{M,\cC}$, which induces a morphism $\Phi_{M}:\colim F(\cE_{M,\cC_0})\to\colim F(\cE^{<}_{M,\cC})$ that commutes with the canonical morphisms to $F_M$:
 \[
  \begin{tikzcd}[row sep=40pt, column sep=60pt]
   \colim F(\cE_{M,\cC_0}) \ar[rr,"\Phi_{M}"] \ar[dr,"\Omega_{M,\cC_0}"',"\sim"] &     & \colim F(\cE^{<}_{M,\cC}) \ar[dl,"\Omega^<_{M,\cC}"] \\
                                                                         & F_M
  \end{tikzcd}
 \]
 If we define $\widetilde\Phi_M=\Phi_M\circ(\Omega_{M,\cC_0})^{-1}$, then this means that $\widetilde\Phi_M\circ\Omega^<_{M,\cC}=\id_{F_M}$.
 
 In order to show that $\Omega^<_{M,\cC}\circ\widetilde\Phi_M$ is the identity on $\colim F(\cE^{<}_{M,\cC})$, we note that the foundation $F_N$ of every embedded minor $N$ of $M$ is generated by the universal cross ratios of $N$, and that every universal cross ratio stems from a $U^2_4$-minor of $N$. Since $\Omega_{N,\cC_0}:\colim F(\cE_{N,\cC_0})\to F_N$ is an isomorphism, every embedded $U^2_4$-minor of $N$ is contained in a (minimal) embedded minor $N'$ in $\cE_{N,\cC_0}$. Thus the canonical morphism
 \[
  \pi: \bigotimes_{N\in\cE^\min_{M,\cC_0}} F_N \ \longrightarrow \ \colim F(\cE^{<}_{M,\cC})
 \]
 is an epimorphism, where $\cE^\min_{M,\cC_0}$ is the collection of minimal embedded minors of $\cE_{M,\cC_0}$ that contain a $U^2_4$-minor. Since $\cE^\min_{M,\cC_0}$ is contained in $\cE_{M,\cC_0}$, we have naturally that
 \[
  \Omega^<_{M,\cC}\circ\widetilde\Phi_M\circ\pi \ = \ \pi \ = \ \id_{\colim F(\cE^{<}_{M,\cC})}\circ\pi.
 \]
 Since $\pi$ is an epimorphism, this implies that $\Omega^<_{M,\cC}\circ\widetilde\Phi_M$ is the identity on $\colim F(\cE^{<}_{M,\cC})$.
 
 We conclude that $\Omega^<_{M,\cC}:\colim F(\cE^{<}_{M,\cC})\to F_M$ is an isomorphism and that the isomorphism class of $M$ does not belong to $\cC_\fundtype$. This shows that $\cC_\fundtype$ is contained in $\cC_0$ as claimed. Once we have shown that $\cC_\fundtype$ is a fundamental type, the uniqueness of $\cC_\fundtype$ follows by the minimality of fundamental types.
 
 Given a matroid $M$, we write $\cE^\fundtype_{M,\cC}=\cE_{M,\cC_\fundtype}$ and $\Omega^\fundtype_{M,\cC}=\Omega_{M,\cC_\fundtype}$. We aim to show that $\Omega^\fundtype_{M,\cC}:\colim F(\cE^\fundtype_{M,\cC})\to F_M$ is an isomorphism for all $M\in\cC$. If $\Omega^{<}_{M,\cC}$ is not an isomorphism, then $M$ belongs to $\cE^\fundtype_{M,\cC}$ and forms its terminal object. Therefore $\Omega^{\fundtype}_{M,\cC}$ is tautologically an isomorphism.
 
 Assume that $\Omega^{<}_{M,\cC}$ is an isomorphism. Let $\cE$ be a full subdiagram of $\cE^<_{M,\cC}$, i.e.,\ whenever it contains two embedded minors $N$ and $N'$ of $M$ such that $N$ is an embedded minor of $N'$, then the minor embedding $N\hookrightarrow N'$ is in $\cE$. The inclusion of diagrams $\cE\hookrightarrow\cE^<_{M,\cC}$ induces a pasture morphism
 \[
  \Phi_\cE: \ \colim F(\cE) \ \longrightarrow \ \colim F(\cE^<_{M,\cC}) \ \simeq \ F_M.
 \]
 Since the isomorphism type of $M$ does not belong to $\cC_\fundtype$, the diagram $\cE^\fundtype_{M,\cC}$ is a subdiagram of $\cE^<_{M,\cC}$. We show by induction on the number $k$ of embedded minors in $\cE^<_{M,\cC}$ that are not in $\cE$ that $\Phi_\cE$ is an isomorphism provided that $\cE$ contains $\cE^\fundtype_{M,\cC}$. If $k=0$, then $\cE=\cE^<_{M,\cC}$, so $\Phi_\cE$ is the identity, which establishes the base case of the induction.
 
 Let $k>0$. Let $N\in \cE^<_{M,\cC}$ be an embedded minor of minimal size in the complement of $\cE$. Then all proper embedded minors of $N$ are in $\cE$, i.e.,\ $\cE^<_{N,\cC}$ is a subdiagram of $\cE$. Since $N$ is not in $\cE^\fundtype_{M,\cC}\subset\cE$, the canonical morphism $\colim F(\cE^<_{N,\cC})\to F_N$ is an isomorphism. Let $\cE'$ be the full subdiagram of $\cE^<_{M,\cC}$ that contains $\cE$ and $N$. Thus $\Phi_\cE$ factors into
 \[
  \Phi_\cE: \ \colim F(\cE) \ \stackrel\sim\longrightarrow \ \colim F(\cE') \ \stackrel{\Phi_{\cE'}}\longrightarrow \ F_M.
 \]
 By the inductive hypothesis, $\Phi_{\cE'}$ is an isomorphism, and so is $\Phi_\cE$, which completes the inductive step.
 
 Since the number of embedded minors of $M$ is finite, this shows that the morphism $\Phi_{\cE^\fundtype_{M,\cC}}:\colim F(\cE^{\fundtype}_{M,\cC})\to F_M$ is an isomorphism, which concludes the proof.
\end{proof}

Given a matroid $N$ and a class of matroids $\cC$, we call a matroid $M$ in $\cC$ together with a minor embedding $N\simeq M\minor JI\hookrightarrow M$ a \emph{minimal $\cC$-extension of $N$} if all intermediate embedded minors $N\hookrightarrow N'\hookrightarrow M$ with $N'\neq M$ are not in $\cC$. The following result allows us to determine new fundamental types from known ones in terms of their minimal extensions.

\begin{prop}\label{prop: fundamental type of subclasses with U24}
 Let $\cD\subset\cC$ be classes of matroids with respective fundamental types $\cD_\fundtype$ and $\cC_\fundtype$. Then $\cD_\fundtype$ consists of the isomorphism types of all minimal $\cD$-extensions $M$ of matroids $N$ with isomorphism class in $\cC_\fundtype$ for which $\colim F(\cE^<_{M,\cD})\to F_M$ is not an isomorphism.
\end{prop}

\begin{proof}
 By \autoref{thm: fundamental type}, the isomorphism class of $M$ belongs to the fundamental type of $\cD$ if and only if the canonical morphism $\Phi_M: \colim F(\cE^<_{M,\cD})\to F_M$ is not an isomorphism. Thus the claim follows at once for the minimal $\cD$-extensions $M$ of matroids with isomorphism type in $\cC_\fundtype$. 
 
 Due to the characterization of $\cD_\fundtype$ in \autoref{thm: fundamental type}, we are left with verifying that $\Omega^<_{M,\cD}: \colim F(\cE^<_{M,\cD})\to F_M$ is an isomorphism if $M$ is in $\cD$ and not a minimal $\cD$-extension of any of its embedded minors $N$ with isomorphism class in $\cC_\fundtype$.  Note that this means in particular that the isomorphism class of $M$ is not in $\cC_\fundtype$, and thus $\Omega^<_{M,\cC}: \colim F(\cE^<_{M,\cC})\to F_M$ is an isomorphism. 
 
 Let $\cE^\fundtype_{N,\cC}=\cE_{N,\cC_\fundtype}$ be the diagram of all embedded minors of a matroid $N$ with isomorphism class in $\cC_\fundtype$, and let $F^\fundtype_{N,\cC}=\colim F(\cE^\fundtype_{N,\cC})$. Then by the defining property of the fundamental type, the canonical morphism $\Omega^\fundtype_{N,\cC}: F^\fundtype_{N,\cC} \to F_N$ is an isomorphism. 
 
 A minor embedding $N\hookrightarrow N'$ induces an inclusion of diagrams $\cE^\fundtype_{N,\cC}\hookrightarrow \cE^\fundtype_{N',\cC}$ and a morphism $F^\fundtype_{N,\cC} \to F^\fundtype_{N',\cC}$ between the colimits of the respective diagrams of foundations, which corresponds to the canonical morphism $F_N\to F_{N'}$ under the isomorphisms $\Omega^\fundtype_{N,\cC}$ and $\Omega^\fundtype_{N',\cC}$. If $\cF^<_{M,\cD}$ is the diagram of all colimits $F^\fundtype_{N,\cC}$, together with the morphisms $F^\fundtype_{N,\cC} \to F^\fundtype_{N',\cC}$ induced by the inclusions $N\hookrightarrow N'$ of embedded minors $N$ and $N'$ of $M$ that are in $\cD$, then the previous observations imply that the colimit over the isomorphisms $\Omega^\fundtype_{N,\cC}$ yields an isomorphism $\Theta_{M,\cD}:\colim \cF^<_{M,\cD} \to \colim F(\cE^<_{M,\cD})$.
 
 Since $M$ is not a minimal $\cD$-extension of any matroid $N$ for which $[N]\in\cC_\fundtype$, we find for every embedded minor $N$ of $M$ with $[N]\in\cC_\fundtype$ a proper embedded minor $N'\in\cD$ of $M$ that contains $N$. Thus the union of the diagrams $\cE^\fundtype_{N,\cC}$, where $N$ ranges over all proper embedded minors of $M$ that are in $\cD$, is all of $\cE^\fundtype_{M,\cC}$. Therefore we have a canonical identification $\colim F(\cE^\fundtype_{M,\cC}) \cong \colim  \cF^<_{M,\cD}$. This yields a commutative diagram
 \[
  \begin{tikzcd}[row sep=25pt, column sep=60pt]
   \colim F(\cE^\fundtype_{M,\cC}) \ar[r,"\sim"] \ar[d,"\Omega^\fundtype_{M,\cC}"',"\sim"] & \colim  \cF^<_{M,\cD} \ar[r,"\sim","\Theta_{M,\cD}"'] & \colim F(\cE^<_{M,\cD}) \ar[d,"\Omega^<_{M,\cD}"] \\
   F_M \ar[rr,"\id"]                             &                              & F_M
  \end{tikzcd}
 \]
 which establishes our claim that $\Omega^<_{M,\cD}$ is an isomorphism and completes the proof.
\end{proof}

If $\cC$ is the class of all matroids, \autoref{prop: fundamental type of subclasses with U24} reads as follows:

\begin{cor}\label{cor: fundamental type with U24}
 Let $\cD$ be a class of matroids. Then its fundamental type consists of the isomorphism types of all minimal $\cD$-extensions $M$ of 
 \[
  U^2_4, \qquad U^2_5, \qquad U^3_5, \qquad C_5, \qquad C_5^\ast, \qquad U^2_4\oplus U^1_2, \qquad F_7 \qquad \text{and} \qquad F_7^\ast
 \]
 for which $\colim F(\cE^<_{M,\cD})\to F_M$ is not an isomorphism. 
\end{cor}

\begin{proof}
 This follows at once from \autoref{prop: fundamental type of subclasses with U24} applied to the class $\cC$ of all matroids and \autoref{thm: fundamental presentation}, which characterizes the fundamental type in this case as the isomorphism classes of $U^2_4$, \ $U^2_5$, \ $U^3_5$, \ $C_5$, \ $C_5^\ast$, \ $U^2_4\oplus U^1_2$, \ $F_7$ and $F_7^\ast$ and satisfies the hypotheses of \autoref{prop: fundamental type of subclasses with U24}.
\end{proof}

\begin{rem}
 In general, two different classes $\cC$ and $\cC'$ can have the same fundamental type. Since their union has the same type, there is a maximal class for each given fundamental type. Examples of such maximal classes are regular matroids, binary matroids, and matroids without $U^2_5$ or $U^3_5$-minors. The classes $\cC^{(k)}$ of $k$-connected matroids (for $k\geq 2$) are not maximal, since we can add a loop to a matroid in $\cC^{(k)}$ without changing the foundation. This leads us to ask:
  
 \begin{problem}
  Can we characterize which families of matroids are the maximal classes for some fundamental type?
 \end{problem}

 Not all fundamental types are finite, as there are infinite families of matroids that are not minors of each other, such as the class of excluded minors for orientable matroids. This suggests: 
 
 \begin{problem}
  Is there a simple set of sufficient conditions, or, more ambitiously, both necessary and sufficient conditions, for a class of matroids to have a finite fundamental type? 
 \end{problem}
\end{rem}

\begin{rem}
 Fundamental types are in some sense complementary to excluded minors when applied to the class $\cC$ of matroids representable over some pasture $F$: while excluded minors are isomorphism types of minimal matroids that do not appear in $\cC$, the fundamental type of $\cC$ consists of isomorphism types of certain matroids that \emph{do} appear in $\cC$. 
 
 The difficulty of these complementary approaches is, however, of a very different nature. While it is in general hard to find a complete list of excluded minors, the fundamental type of $\cC$ is easily determined---the difficulty lies rather in the structure of the fundamental diagram. 
 
 More explicitly, the fundamental type of $\cC$ consists of the isomorphism types of those matroids among $U^2_4$, $U^2_5$, $U^3_5$, $C_5$, $C_5^\ast$, $U^2_4\oplus U^1_2$, $F_7$ and $F_7^\ast$ that are representable over $F$. This follows from \autoref{cor: fundamental type with U24} and the fact that if a matroid $N$ from this list is \emph{not} representable over $F$, then no extension of $N$ is $F$-representable. In particular, this means that the fundamental type of the class of $F$-representable matroids is finite for every pasture $F$.
\end{rem}

For connectivity considerations, however, the question of determining fundamental types is more interesting:

\begin{problem}
 Can one explicitly determine the fundamental type for the class of $4$-connected, or vertically $4$-connected, matroids?
\end{problem}


\subsection{The 2-connected fundamental presentation}
\label{subsection: presentation of the foundation by embedded minors for 2-connected matroids}

If $M$ is $2$-connected, then it turns out that we can omit the embedded minors of type $U^2_4\oplus U^1_2$ from the fundamental diagram without changing the fundamental presentation $\colim F(\cE_M)$ of the foundation $F_M$ of $M$. The key idea for the proof of this claim was communicated to us by Nathan Bowler.

\begin{df}
 The \emph{$2$-connected fundamental diagram of $M$} is the diagram $\cE_M^{(2)}$ of embedded minors of isomorphism types
 \[
  U^2_4, \qquad U^2_5, \qquad U^3_5, \qquad C_5,\qquad C_5^\ast, \qquad F_7, \qquad F_7^\ast,
 \]
 together with all minor embeddings. 
\end{df}

Note that $\cE_M^{(2)}$ is a subdiagram of the fundamental diagram $\cE_M$. More precisely, it consists of all embedded minors of $\cE_M$ that are $2$-connected. If $M$ is $2$-connected, then we call the associated diagram $F(\cE_M^{(2)})$ of foundations the \emph{$2$-connected fundamental presentation of $M$}, which is motivated by the following result. Let $\cC^{(2)}$ be the class of all $2$-connected matroids.

\begin{thm}\label{thm: fundamental presentation for 2-connected matroids}
 The fundamental type of $\cC^{(2)}$ consists of the isomorphism types of the matroids $U^2_4$, $U^2_5$, $U^3_5$, $C_5$, $C_5^\ast$, $F_7$ and $F_7^\ast$. This is, if $M$ is a $2$-connected matroid with foundation $F_M$ and $2$-connected fundamental diagram $\cE_M^{(2)}$, then $F_M\simeq \colim F(\cE_M^{(2)})$.
\end{thm}

The rest of this section is dedicated to the proof of this theorem. By \autoref{cor: fundamental type with U24}, we only need to determine which of the minimal $2$-connected extensions of $U^2_4$, $U^2_5$, $U^3_5$, $C_5$, $C_5^\ast$, $U^2_4\oplus U^1_2$, $F_7$ and $F_7^\ast$ belong to the fundamental type $\cC^{(2)}_\fundtype$ of $\cC^{(2)}$. Since each of $U^2_4$, $U^2_5$, $U^3_5$, $C_5$, $C_5^\ast$, $F_7$ and $F_7^\ast$ is $2$-connected, they are each their own unique minimal $2$-connected extension. And since the foundation of none of these matroids is a colimit of foundations of proper embedded minors, each belongs to $\cC^{(2)}_\fundtype$.

We are left with an investigation of the minimal $2$-connected extensions of $U^2_4\oplus U^1_2$. We assume that the reader is familiar with Cunningham and Edmond's canonical tree decomposition of $2$-connected matroids; for details see \cite{Cunningham73} and \cite[Thm.\ 8.3.10]{Oxley92}.

\begin{lemma}\label{lemma: minimal 2-connected extensions of U24+U12}
 Let $M$ be a minimal $2$-connected extension of $N=U^2_4\oplus U^1_2$ with $N=M\minor{J}{I}$. Then $n=\# E_M\geq7$, and up to duality, the tree decomposition of $M$ is 
 \[
  \begin{tikzcd}
   M_1 \ar[-,r,"e_1"] & M_2 \ar[-,r,"e_2"] & \dotsb \ar[-,r,"e_{n-5}"] & M_{n-4},
  \end{tikzcd}
 \]
 where $M_i\simeq U^2_3$ for $i\geq 2$ even, $M_i\simeq U^1_3$ for $i\geq2$ odd, and $M_1\simeq U^2_5$. Moreover,
 \[
  J \ = \ E_M \ \cap \ \bigcup_{2\leq i\leq n-5\text{ odd}} E_{M_i} \qquad \text{and} \qquad I \ = \ E_M \ \cap \ \bigcup_{2\leq i\leq n-5\text{ even}} E_{M_i}.
 \]
\end{lemma}

\begin{figure}[htb]
 \[
  \beginpgfgraphicnamed{tikz/fig22}
  \begin{tikzpicture}[x=1cm,y=1cm]
   \draw[thick] (0,1) arc ( 90:270:0.5);
   \node[yshift=0.5cm] at (112.5:0.8) {\footnotesize $1$};
   \node[yshift=0.5cm] at (157.5:0.8) {\footnotesize $2$};
   \node[yshift=0.5cm] at (202.5:0.8) {\footnotesize $3$};
   \node[yshift=0.5cm] at (247.5:0.8) {\footnotesize $4$};
   \draw[line width=1.5pt,color=green] (0,0) to node[below=0pt,color=black] {\footnotesize $5$} (1,0);
   \draw[line width=1.5pt,color=green] (0,1) to node[above=0pt,color=black] {\footnotesize $7$} (2,1);
   \draw[line width=1.5pt,color=green] (1,0) to node[below=0pt,color=black] {\footnotesize $9$} (3,0);
   \draw[line width=1.0pt,color=red] (0,1) to node[left=2pt,color=black] {\footnotesize $6$} (1,0);
   \draw[line width=1.0pt,color=red] (1,0) to node[left=2pt,color=black] {\footnotesize $8$} (2,1);
   \draw[line width=1.0pt,color=red] (2,1) to node[left=2pt,color=black] {\footnotesize $10$} (3,0);
   \draw[thick] (2,1) to node[above=0pt,color=black] {\footnotesize $11$} (4,1);
   \draw[thick] (3,0) to node[left=2pt,color=black] {\footnotesize $12$} (4,1);
   \fill (-0.5,0.5) circle (3pt);
   \fill (-0.36,0.86) circle (3pt);
   \fill (-0.36,0.14) circle (3pt);
   \fill (0,0) circle (3pt);
   \fill (0,1) circle (3pt);
   \fill (1,0) circle (3pt);
   \fill (2,1) circle (3pt);
   \fill (3,0) circle (3pt);
   \fill (4,1) circle (3pt);
  \end{tikzpicture}
  \endpgfgraphicnamed
  \hspace{2cm}
  \beginpgfgraphicnamed{tikz/fig23}
  \begin{tikzpicture}[x=1cm,y=1cm]
   \draw[thick] (0,1) arc ( 90:270:0.5);
   \node[yshift=0.5cm] at (112.5:0.8) {\footnotesize $1$};
   \node[yshift=0.5cm] at (157.5:0.8) {\footnotesize $2$};
   \node[yshift=0.5cm] at (202.5:0.8) {\footnotesize $3$};
   \node[yshift=0.5cm] at (247.5:0.8) {\footnotesize $4$};
   \draw[line width=1.5pt,color=green] (0,0) to node[below=0pt,color=black] {\footnotesize $5$} (1,0);
   \draw[line width=1.5pt,color=green] (0,1) to node[above=0pt,color=black] {\footnotesize $7$} (2,1);
   \draw[line width=1.5pt,color=green] (1,0) to node[below=0pt,color=black] {\footnotesize $9$} (3,0);
   \draw[line width=1.5pt,color=green] (2,1) to node[above=0pt,color=black] {\footnotesize $11$} (4,1);
   \draw[line width=1.0pt,color=red] (0,1) to node[left=2pt,color=black] {\footnotesize $6$} (1,0);
   \draw[line width=1.0pt,color=red] (1,0) to node[left=2pt,color=black] {\footnotesize $8$} (2,1);
   \draw[line width=1.0pt,color=red] (2,1) to node[left=2pt,color=black] {\footnotesize $10$} (3,0);
   \draw[thick,bend left=15] (3,0) to node[left=0pt,color=black] {\footnotesize $12$} (4,1);
   \draw[thick,bend right=15] (3,0) to node[right=0pt,color=black] {\footnotesize $13$} (4,1);
   \fill (-0.5,0.5) circle (3pt);
   \fill (-0.36,0.86) circle (3pt);
   \fill (-0.36,0.14) circle (3pt);
   \fill (0,0) circle (3pt);
   \fill (0,1) circle (3pt);
   \fill (1,0) circle (3pt);
   \fill (2,1) circle (3pt);
   \fill (3,0) circle (3pt);
   \fill (4,1) circle (3pt);
  \end{tikzpicture}
  \endpgfgraphicnamed
 \]
 \caption{Minimal $2$-connected extensions of $U^2_4\oplus U^1_2$}
 \label{fig: minimal extension of U24+U12}
\end{figure}
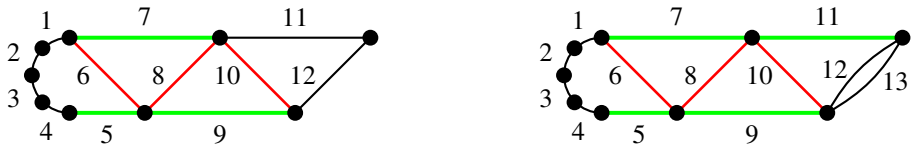

\begin{rem}\label{rem: minimal 2-connected extensions of U24+U12}
 The minimal $2$-connected extensions of $N=U^2_4\oplus U^1_2$ with $M_1=U^2_5$ are illustrated in \autoref{fig: minimal extension of U24+U12} for $n=\# E_M$ equal to $12$ and to $13$, which exemplify the cases $M_{n-4}=U^2_3$ and $M_{n-4}=U^1_3$, respectively.
 
 The black half-circle with edges $1$, $2$, $3$ and $4$ stands symbolically for the $2$-summand $M_1=U^2_5$, and might be seen as a stylized version of the common illustration of the point line configuration of $U^2_5$. We use the convention $E_M\cap E_{M_1}=\{1,\dotsc,4\}$, as well as $E_M\cap E_{M_i}=\{i+3\}$ for $i=2,\dotsc,n-5$ and $E_M\cap E_{M_{n-4}}=\{n-1,n\}$. The elements of $J$ are illustrated as green edges and the elements of $I$ are illustrated as red edges.
 
 Since $M$ is a parallel-series extension of $U^2_5$, we can illustrate the deletion and contraction of elements $i\geq 5$ of $M$ in terms of deleting or contracting the corresponding edge. In particular, a subset $J$ of $E_M-\{1,\dotsc,4\}$ is coindependent if it is not an edge cut and a subset $I$ of $E_M-\{1,\dotsc,4\}$ is independent if $I\cup\{1,\dotsc,4\}$ is a forest.
 
 Excluding the green edges and contracting the red edges results in both cases of \autoref{fig: minimal extension of U24+U12} in
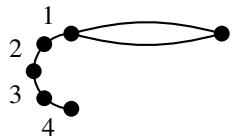
 \[
  \beginpgfgraphicnamed{tikz/fig24}
  \begin{tikzpicture}[x=1cm,y=1cm]
   \draw[thick] (0,1) arc ( 90:270:0.5);
   \node[yshift=0.5cm] at (112.5:0.8) {\footnotesize $1$};
   \node[yshift=0.5cm] at (157.5:0.8) {\footnotesize $2$};
   \node[yshift=0.5cm] at (202.5:0.8) {\footnotesize $3$};
   \node[yshift=0.5cm] at (247.5:0.8) {\footnotesize $4$};
   \draw[thick,bend left=15] (0,1) to node[above=0pt,color=black] {\footnotesize } (2,1);
   \draw[thick,bend right=15] (0,1) to node[below=0pt,color=black] {\footnotesize } (2,1);
   \fill (-0.5,0.5) circle (3pt);
   \fill (-0.36,0.86) circle (3pt);
   \fill (-0.36,0.14) circle (3pt);
   \fill (0,0) circle (3pt);
   \fill (0,1) circle (3pt);
   \fill (2,1) circle (3pt);
  \end{tikzpicture}
  \endpgfgraphicnamed
 \]
 which illustrates the matroid $N_0=U^2_4\oplus U^1_2$.
\end{rem}

\begin{proof}[Proof of {\autoref{lemma: minimal 2-connected extensions of U24+U12}}]
 Let us identify $E_M$ with $\{1,\dotsc,n\}$ such that $E_{N_0}=\{1,\dotsc,4,a,b\}$ with $1,\dotsc,4$ being the elements of the summand $U^2_4$ of $N_0=U^2_4\oplus U^1_2$, and $a=n-1$ and $b=n$ being elements of $U^1_2$. Since $N_0$ is not $2$-connected and has $6$ elements, it is clear that $n\geq7$.

 Let $M_1,\dotsc,M_r$ be the $2$-summands of the canonical tree decomposition of $M$. This means that each $M_i$ is either $3$-connected, a circuit (i.e.\ isomorphic to $U^1_k$) or a cocircuit (i.e.\ isomorphic to $U^{k-1}_k$). Since $U^2_4$ is $3$-connected and appears as a minor of $N_0$, and thus of $M$, we conclude that $1,\dotsc,4$ belong to the same $2$-summand of $M$, say $M_1$. Since $U^2_4$ is not contained in a circuit or cocircuit, we conclude that $M_1$ is $3$-connected. 

 The assumption that $M$ is a minimal $2$-connected extension of $N_0$ has a series of implications. Since the removal of a single element of a $3$-connected matroid results in a $2$-connected minor, none of $M_2,\dotsc,M_r$ are $3$-connected, and thus each element $e\in E_M-E_{N_0}$ is contained in a circuit $M_i\simeq U^1_{n_i}$ or in a cocircuit $M_i\simeq U^{n_i-1}_{n_i}$ . Note that deleting $e$ from a circuit $M_i\simeq U^1_{n_i}$, as well as contracting $e$ from a cocircuit $M_i\simeq U^{n_i-1}_{n_i}$, leaves $M$ $2$-connected. By the minimality of $M$, we conclude that if $e\in M_i\simeq U^1_{n_i}$, then $e\in I$, and if $e\in M_i\simeq U^{n_i-1}_{n_i}$, then $e\in J$. Moreover, since $J$ is coindependent and $I$ is independent by the definition of $N_0$ as an embedded minor, we conclude that $M_i$ contains at most one element that is also in $M$.
 
 Since $U^2_4\oplus U^1_2$ is not $2$-connected, $a$ and $b$ are not in $M_1$. Since $N_0\minor{12}{34}\simeq U^1_2$ is $2$-connected, but $M\setminus e$ for $e\in J$ and $M/e$ for $e\in I$ are not, we conclude that $a$ and $b$ are contained in the same component $M_i$, say $M_r$, and that $M_r$ does not contain any other element of $M$. Once again, by the minimality of $M$, we conclude that all other components $M_2,\dotsc,M_{r-1}$ must lie on a path between $M_1$ and $M_r$, i.e.,\ that the tree decomposition of $M$ is of the form
 \[
  \begin{tikzcd}
   M_1 \ar[-,r,"e_1"] & M_2 \ar[-,r,"e_2"] & \dotsb \ar[-,r,"e_{r-1}"] & M_{r}
  \end{tikzcd}
 \]
 after reordering the indices appropriately. Thus $M_1$ has elements $1,2,3,4,e_1$ and is isomorphic to $U^2_5$ or $U^3_5$ as a $3$-connected matroid on $5$ elements. After taking duals (as allowed in the claim of the lemma), we can assume that $M_1\simeq U^2_5$. Since $U^1_2$ does not appear in a canonical tree decomposition, we conclude that that every component $M_2,\dotsc,M_{r-1}$ contains precisely $1$ element of $M$ and $M_r$ has elements $a,b,e_{r-1}$. Thus $M_i$ is isomorphic to $U^1_3$ or to $U^2_3$ for every $i=2,\dotsc,r$. Counting elements, we conclude that $r=n-4$.
 
 Let $5$ be the unique element in $E_{M_2}\cap E_M$. Since $M\minor{(J-5)}{(I-5)}$ has a $U^2_5$-minor by \cite[Prop.\ 7.1.21]{Oxley92}, but $M\minor JI$ is not $2$-connected, we conclude that $5\in I$. In conclusion, $M_2\simeq U^2_3$. Since adjacent $2$-summands cannot both be circuits, nor can they both be cocircuits, in the canonical tree decomposition of a matroid, this determines the isomorphism types of $M_3,\dotsc,M_{n-4}$, which proves all claims of the lemma.
\end{proof}

In the following arguments, we use a chain of series extensions of $U^2_5$ that appear as embedded minors of $M$. We denote the series extension of $U^2_5$ by $D_6$. We fix the ground set $E_{D_6}=\{1,\dotsc,6\}$, where $5$ and $6$ are the series elements. 

\begin{lemma}\label{lemma: cross ratios in D6}
 For every permutation $\sigma\in S_4$, we have
 \[
  \cross{\sigma(1)}{\sigma(2)}{\sigma(3)}{\sigma(4)}{5} \ = \ \cross{\sigma(1)}{\sigma(2)}{\sigma(3)}{\sigma(4)}{6}
 \]
 in $D_6$. 
\end{lemma}

\begin{proof}
 The dual $D_6^\ast$ of $D_6$ is a parallel extension of $U^3_5$ with parallel elements $5$ and $6$. Thus $\cross{\sigma(1)}{\sigma(2)}{\sigma(3)}{\sigma(4)}{5} \ = \ \cross{\sigma(1)}{\sigma(2)}{\sigma(3)}{\sigma(4)}{6}$ in $D_6^\ast$ by relation \eqref{R5} of \autoref{thm: fundamental presentation of foundations in terms of bases}. Applying the canonical isomorphism $\varphi:F_{D_6}\to F_{D_6^\ast}$ to the cross ratios in question yields 
 \[
  \varphi\bigg(\cross{\sigma(1)}{\sigma(2)}{\sigma(3)}{\sigma(4)}{5}\bigg) \ = \ \cross{\sigma(1)}{\sigma(2)}{\sigma(3)}{\sigma(4)}{6} \ = \ \cross{\sigma(1)}{\sigma(2)}{\sigma(3)}{\sigma(4)}{5} \ = \ \varphi\bigg(\cross{\sigma(1)}{\sigma(2)}{\sigma(3)}{\sigma(4)}{6}\bigg),
 \]
 and thus the relation claimed in the lemma.
\end{proof}

We use the previous insights to study the foundation of a fixed minimal $2$-connected extension $M$ of $N_0$. Since foundations are insensitive to dualization, we can assume that $M$ is a parallel-series extension of $M_1=U^2_5$. We label the elements of $M$ according to the conventions of \autoref{fig: minimal extension of U24+U12}: if 
 \[
  \begin{tikzcd}
   M_1 \ar[-,r,"e_1"] & M_2 \ar[-,r,"e_2"] & \dotsb \ar[-,r,"e_{n-5}"] & M_{n-4}
  \end{tikzcd}
 \]
 is the tree decomposition of $M$, then 
 \[
  E_{M_1}\cap M \ = \ \{1,2,3,4\}, \qquad E_{M_i}\cap E_M \ = \ \{i+3\}, \qquad E_{M_{n-4}}\cap E_M =\{a,b\}
 \]
 for $i=2,\dotsc,n-5$, and where $a=n-1$ and $b=n$.
 
 Since $D_6\simeq U^2_5\oplus_2 U^2_3$, we can illustrate $D_6$ as 
 \[
  \beginpgfgraphicnamed{tikz/fig25}
  \begin{tikzpicture}[x=1cm,y=1cm]
   \draw[thick] (0,1) arc ( 90:270:0.5);
   \node[yshift=0.5cm] at (112.5:0.8) {\footnotesize $1$};
   \node[yshift=0.5cm] at (157.5:0.8) {\footnotesize $2$};
   \node[yshift=0.5cm] at (202.5:0.8) {\footnotesize $3$};
   \node[yshift=0.5cm] at (247.5:0.8) {\footnotesize $4$};
   \fill (-0.5,0.5) circle (3pt);
   \fill (-0.36,0.86) circle (3pt);
   \fill (-0.36,0.14) circle (3pt);
   \fill (0,0) circle (3pt);
   \fill (0,1) circle (3pt);
   \draw[line width=0.8pt,color=black] (0,0) to node[below=0pt,color=black] {\footnotesize $5$} (1,0.5);
   \draw[line width=0.8pt,color=black] (0,1) to node[above=0pt,color=black] {\footnotesize $6$} (1,0.5);
   \fill (1,0.5) circle (3pt);
  \end{tikzpicture}
  \endpgfgraphicnamed
 \]
 following the conventions of \autoref{rem: minimal 2-connected extensions of U24+U12} and \autoref{fig: minimal extension of U24+U12}. Thus an embedded minor $N=M\minor JI$ that contains $\{1,\dotsc,4\}$ is of type $D_6$ if excluding the edges in $J$ and contracting the edges labeled by $I$ results in the above picture. 
 
 For odd $k$ between $5$ and $n-2$, we define
 \begin{align*}
  J_k \ &= \ \big\{ i\in \{5,\dotsc,n-2\} \, \big| \, \text{$i$ odd if $i<k$ and $i$ even if $i\geq k$} \big\}, \\
  I_k \ &= \ \big\{ i\in \{5,\dotsc,n-2\} \, \big| \, \text{$i$ even if $i<k$ and $i$ odd if $i\geq k$} \big\}, 
 \end{align*}
 as well as $J_k'=J_k-k$ and $I_k'=I_k-(k+1)$. Then all of the following embedded minors are of type $D_6$, as visible from the corresponding illustrations where the green edges are deleted and the red edges are contracted.
 \begin{enumerate}
  \item $M\minor{J_k'a}{I_k'b}$ and $M\minor{J_k'b}{I_k'a}$ for $k=5,7,\dotsc,n-3$ (note that $a$ and $b$ are symmetric, so we only illustrate the first case for $k=9$, $n=16$ and $k=9$, $n=17$):
 \[
  \beginpgfgraphicnamed{tikz/fig26}
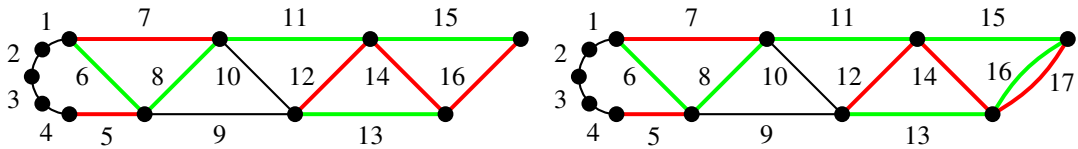
  \begin{tikzpicture}[x=1cm,y=1cm]
   \draw[thick] (0,1) arc ( 90:270:0.5);
   \node[yshift=0.5cm] at (112.5:0.8) {\footnotesize $1$};
   \node[yshift=0.5cm] at (157.5:0.8) {\footnotesize $2$};
   \node[yshift=0.5cm] at (202.5:0.8) {\footnotesize $3$};
   \node[yshift=0.5cm] at (247.5:0.8) {\footnotesize $4$};
   \draw[line width=1.5pt,color=red] (0,0) to node[below=0pt,color=black] {\footnotesize $5$} (1,0);
   \draw[line width=1.5pt,color=red] (0,1) to node[above=0pt,color=black] {\footnotesize $7$} (2,1);
   \draw[line width=0.8pt,color=black] (1,0) to node[below=0pt,color=black] {\footnotesize $9$} (3,0);
   \draw[line width=1.5pt,color=green] (2,1) to node[above=0pt,color=black] {\footnotesize $11$} (4,1);
   \draw[line width=1.5pt,color=green] (3,0) to node[below=0pt,color=black] {\footnotesize $13$} (5,0);
   \draw[line width=1.5pt,color=green] (0,1) to node[left=2pt,color=black] {\footnotesize $6$} (1,0);
   \draw[line width=1.5pt,color=green] (1,0) to node[left=2pt,color=black] {\footnotesize $8$} (2,1);
   \draw[line width=0.8pt,color=black] (2,1) to node[left=2pt,color=black] {\footnotesize $10$} (3,0);
   \draw[line width=1.5pt,color=red] (3,0) to node[left=2pt,color=black] {\footnotesize $12$} (4,1);
   \draw[line width=1.5pt,color=red] (4,1) to node[left=2pt,color=black] {\footnotesize $14$} (5,0);
   \draw[line width=1.5pt,color=green] (4,1) to node[above=0pt,color=black] {\footnotesize $15$} (6,1);
   \draw[line width=1.5pt,color=red] (5,0) to node[left=2pt,color=black] {\footnotesize $16$} (6,1);
   \fill (-0.5,0.5) circle (3pt);
   \fill (-0.36,0.86) circle (3pt);
   \fill (-0.36,0.14) circle (3pt);
   \fill (0,0) circle (3pt);
   \fill (0,1) circle (3pt);
   \fill (1,0) circle (3pt);
   \fill (2,1) circle (3pt);
   \fill (3,0) circle (3pt);
   \fill (4,1) circle (3pt);
   \fill (5,0) circle (3pt);
   \fill (6,1) circle (3pt);
  \end{tikzpicture}
  \endpgfgraphicnamed
  \hspace{0.2cm}
  \beginpgfgraphicnamed{tikz/fig27}
  \begin{tikzpicture}[x=1cm,y=1cm]
   \draw[line width=0.8pt] (0,1) arc ( 90:270:0.5);
   \node[yshift=0.5cm] at (112.5:0.8) {\footnotesize $1$};
   \node[yshift=0.5cm] at (157.5:0.8) {\footnotesize $2$};
   \node[yshift=0.5cm] at (202.5:0.8) {\footnotesize $3$};
   \node[yshift=0.5cm] at (247.5:0.8) {\footnotesize $4$};
   \draw[line width=1.5pt,color=red] (0,0) to node[below=0pt,color=black] {\footnotesize $5$} (1,0);
   \draw[line width=1.5pt,color=red] (0,1) to node[above=0pt,color=black] {\footnotesize $7$} (2,1);
   \draw[line width=0.8pt,color=black] (1,0) to node[below=0pt,color=black] {\footnotesize $9$} (3,0);
   \draw[line width=1.5pt,color=green] (2,1) to node[above=0pt,color=black] {\footnotesize $11$} (4,1);
   \draw[line width=1.5pt,color=green] (3,0) to node[below=0pt,color=black] {\footnotesize $13$} (5,0);
   \draw[line width=1.5pt,color=green] (4,1) to node[above=0pt,color=black] {\footnotesize $15$} (6,1);
   \draw[line width=1.5pt,color=green] (0,1) to node[left=2pt,color=black] {\footnotesize $6$} (1,0);
   \draw[line width=1.5pt,color=green] (1,0) to node[left=2pt,color=black] {\footnotesize $8$} (2,1);
   \draw[line width=0.8pt,color=black] (2,1) to node[left=2pt,color=black] {\footnotesize $10$} (3,0);
   \draw[line width=1.5pt,color=red] (3,0) to node[left=2pt,color=black] {\footnotesize $12$} (4,1);
   \draw[line width=1.5pt,color=red] (4,1) to node[left=2pt,color=black] {\footnotesize $14$} (5,0);
   \draw[bend left=15,line width=1.5pt,color=green] (5,0) to node[left=0pt,color=black] {\footnotesize $16$} (6,1);
   \draw[bend right=15,line width=1.5pt,color=red] (5,0) to node[right=0pt,color=black] {\footnotesize $17$} (6,1);
   \fill (-0.5,0.5) circle (3pt);
   \fill (-0.36,0.86) circle (3pt);
   \fill (-0.36,0.14) circle (3pt);
   \fill (0,0) circle (3pt);
   \fill (0,1) circle (3pt);
   \fill (1,0) circle (3pt);
   \fill (2,1) circle (3pt);
   \fill (3,0) circle (3pt);
   \fill (4,1) circle (3pt);
   \fill (5,0) circle (3pt);
   \fill (6,1) circle (3pt);
  \end{tikzpicture}
  \endpgfgraphicnamed
 \]
 \item $M\minor{J'_{n-1}}{I'_{n-1}}$ for even $n$:
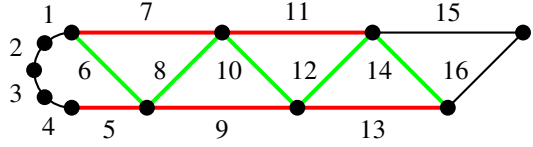
 \[
  \beginpgfgraphicnamed{tikz/fig28}
  \begin{tikzpicture}[x=1cm,y=1cm]
   \draw[thick] (0,1) arc ( 90:270:0.5);
   \node[yshift=0.5cm] at (112.5:0.8) {\footnotesize $1$};
   \node[yshift=0.5cm] at (157.5:0.8) {\footnotesize $2$};
   \node[yshift=0.5cm] at (202.5:0.8) {\footnotesize $3$};
   \node[yshift=0.5cm] at (247.5:0.8) {\footnotesize $4$};
   \draw[line width=1.5pt,color=red] (0,0) to node[below=0pt,color=black] {\footnotesize $5$} (1,0);
   \draw[line width=1.5pt,color=red] (0,1) to node[above=0pt,color=black] {\footnotesize $7$} (2,1);
   \draw[line width=1.5pt,color=red] (1,0) to node[below=0pt,color=black] {\footnotesize $9$} (3,0);
   \draw[line width=1.5pt,color=red] (2,1) to node[above=0pt,color=black] {\footnotesize $11$} (4,1);
   \draw[line width=1.5pt,color=red] (3,0) to node[below=0pt,color=black] {\footnotesize $13$} (5,0);
   \draw[line width=1.5pt,color=green] (0,1) to node[left=2pt,color=black] {\footnotesize $6$} (1,0);
   \draw[line width=1.5pt,color=green] (1,0) to node[left=2pt,color=black] {\footnotesize $8$} (2,1);
   \draw[line width=1.5pt,color=green] (2,1) to node[left=2pt,color=black] {\footnotesize $10$} (3,0);
   \draw[line width=1.5pt,color=green] (3,0) to node[left=2pt,color=black] {\footnotesize $12$} (4,1);
   \draw[line width=1.5pt,color=green] (4,1) to node[left=2pt,color=black] {\footnotesize $14$} (5,0);
   \draw[line width=0.8pt,color=black] (4,1) to node[above=0pt,color=black] {\footnotesize $15$} (6,1);
   \draw[line width=0.8pt,color=black] (5,0) to node[left=2pt,color=black] {\footnotesize $16$} (6,1);
   \fill (-0.5,0.5) circle (3pt);
   \fill (-0.36,0.86) circle (3pt);
   \fill (-0.36,0.14) circle (3pt);
   \fill (0,0) circle (3pt);
   \fill (0,1) circle (3pt);
   \fill (1,0) circle (3pt);
   \fill (2,1) circle (3pt);
   \fill (3,0) circle (3pt);
   \fill (4,1) circle (3pt);
   \fill (5,0) circle (3pt);
   \fill (6,1) circle (3pt);
  \end{tikzpicture}
  \endpgfgraphicnamed
 \]
 \item $M\minor{J'_{n-2}a}{I'_{n-2}}$ and $M\minor{J'_{n-2}b}{I'_{n-2}}$ for odd $n$:
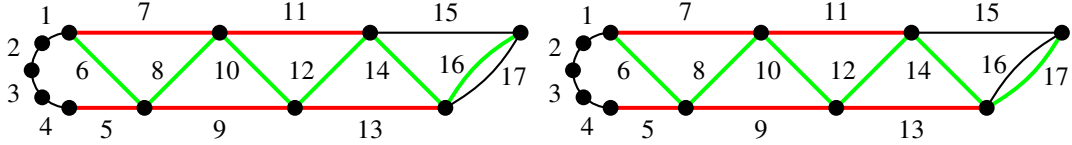
 \[
  \beginpgfgraphicnamed{tikz/fig30}
  \begin{tikzpicture}[x=1cm,y=1cm]
   \draw[line width=0.8pt] (0,1) arc ( 90:270:0.5);
   \node[yshift=0.5cm] at (112.5:0.8) {\footnotesize $1$};
   \node[yshift=0.5cm] at (157.5:0.8) {\footnotesize $2$};
   \node[yshift=0.5cm] at (202.5:0.8) {\footnotesize $3$};
   \node[yshift=0.5cm] at (247.5:0.8) {\footnotesize $4$};
   \draw[line width=1.5pt,color=red] (0,0) to node[below=0pt,color=black] {\footnotesize $5$} (1,0);
   \draw[line width=1.5pt,color=red] (0,1) to node[above=0pt,color=black] {\footnotesize $7$} (2,1);
   \draw[line width=1.5pt,color=red] (1,0) to node[below=0pt,color=black] {\footnotesize $9$} (3,0);
   \draw[line width=1.5pt,color=red] (2,1) to node[above=0pt,color=black] {\footnotesize $11$} (4,1);
   \draw[line width=1.5pt,color=red] (3,0) to node[below=0pt,color=black] {\footnotesize $13$} (5,0);
   \draw[line width=0.8pt,color=black] (4,1) to node[above=0pt,color=black] {\footnotesize $15$} (6,1);
   \draw[line width=1.5pt,color=green] (0,1) to node[left=2pt,color=black] {\footnotesize $6$} (1,0);
   \draw[line width=1.5pt,color=green] (1,0) to node[left=2pt,color=black] {\footnotesize $8$} (2,1);
   \draw[line width=1.5pt,color=green] (2,1) to node[left=2pt,color=black] {\footnotesize $10$} (3,0);
   \draw[line width=1.5pt,color=green] (3,0) to node[left=2pt,color=black] {\footnotesize $12$} (4,1);
   \draw[line width=1.5pt,color=green] (4,1) to node[left=2pt,color=black] {\footnotesize $14$} (5,0);
   \draw[bend left=15,line width=1.5pt,color=green] (5,0) to node[left=0pt,color=black] {\footnotesize $16$} (6,1);
   \draw[bend right=15,line width=0.8pt,color=black] (5,0) to node[right=0pt,color=black] {\footnotesize $17$} (6,1);
   \fill (-0.5,0.5) circle (3pt);
   \fill (-0.36,0.86) circle (3pt);
   \fill (-0.36,0.14) circle (3pt);
   \fill (0,0) circle (3pt);
   \fill (0,1) circle (3pt);
   \fill (1,0) circle (3pt);
   \fill (2,1) circle (3pt);
   \fill (3,0) circle (3pt);
   \fill (4,1) circle (3pt);
   \fill (5,0) circle (3pt);
   \fill (6,1) circle (3pt);
  \end{tikzpicture}
  \endpgfgraphicnamed
  \hspace{0.2pt}
  \beginpgfgraphicnamed{tikz/fig29}
  \begin{tikzpicture}[x=1cm,y=1cm]
   \draw[line width=0.8pt] (0,1) arc ( 90:270:0.5);
   \node[yshift=0.5cm] at (112.5:0.8) {\footnotesize $1$};
   \node[yshift=0.5cm] at (157.5:0.8) {\footnotesize $2$};
   \node[yshift=0.5cm] at (202.5:0.8) {\footnotesize $3$};
   \node[yshift=0.5cm] at (247.5:0.8) {\footnotesize $4$};
   \draw[line width=1.5pt,color=red] (0,0) to node[below=0pt,color=black] {\footnotesize $5$} (1,0);
   \draw[line width=1.5pt,color=red] (0,1) to node[above=0pt,color=black] {\footnotesize $7$} (2,1);
   \draw[line width=1.5pt,color=red] (1,0) to node[below=0pt,color=black] {\footnotesize $9$} (3,0);
   \draw[line width=1.5pt,color=red] (2,1) to node[above=0pt,color=black] {\footnotesize $11$} (4,1);
   \draw[line width=1.5pt,color=red] (3,0) to node[below=0pt,color=black] {\footnotesize $13$} (5,0);
   \draw[line width=0.8pt,color=black] (4,1) to node[above=0pt,color=black] {\footnotesize $15$} (6,1);
   \draw[line width=1.5pt,color=green] (0,1) to node[left=2pt,color=black] {\footnotesize $6$} (1,0);
   \draw[line width=1.5pt,color=green] (1,0) to node[left=2pt,color=black] {\footnotesize $8$} (2,1);
   \draw[line width=1.5pt,color=green] (2,1) to node[left=2pt,color=black] {\footnotesize $10$} (3,0);
   \draw[line width=1.5pt,color=green] (3,0) to node[left=2pt,color=black] {\footnotesize $12$} (4,1);
   \draw[line width=1.5pt,color=green] (4,1) to node[left=2pt,color=black] {\footnotesize $14$} (5,0);
   \draw[bend left=15,line width=0.8pt,color=black] (5,0) to node[left=0pt,color=black] {\footnotesize $16$} (6,1);
   \draw[bend right=15,line width=1.5pt,color=green] (5,0) to node[right=0pt,color=black] {\footnotesize $17$} (6,1);
   \fill (-0.5,0.5) circle (3pt);
   \fill (-0.36,0.86) circle (3pt);
   \fill (-0.36,0.14) circle (3pt);
   \fill (0,0) circle (3pt);
   \fill (0,1) circle (3pt);
   \fill (1,0) circle (3pt);
   \fill (2,1) circle (3pt);
   \fill (3,0) circle (3pt);
   \fill (4,1) circle (3pt);
   \fill (5,0) circle (3pt);
   \fill (6,1) circle (3pt);
  \end{tikzpicture}
  \endpgfgraphicnamed  
 \]
\end{enumerate}

With this, we are prepared to prove \autoref{thm: fundamental presentation for 2-connected matroids}. We are left with showing that for all minimal $2$-connected extensions $M$ of $N_0=U^2_4\oplus U^1_2$, the natural map $\colim F(\cE^{<}_{M,\cC^{(2)}})\to F_M$ is an isomorphism. We show this by induction on $n=\# E_M$. Since there is no $2$-connected extension of $N_0=U^2_4\oplus U^1_2$ with $6$ elements, the base case for $n=6$ is trivially true.

Let $n\geq 7$. By \autoref{thm: fundamental presentation}, the natural map $\colim F(\cE^{<,+}_{M,\cC^{(2)}})\to F_M$ is an isomorphism, where $\cE^{<,+}_{M,\cC^{(2)}}$ is the diagram $\cE^{<}_{M,\cC^{(2)}}$ enriched with all embedded minors of type $U^2_4\oplus U^1_2$.

By \autoref{lemma: minimal 2-connected extensions of U24+U12}, $N_0=M\minor{J_0}{I_0}$ is the unique embedded minor of $M$ of type $U^2_4\oplus U^1_2$ such that $M$ is a minimal $2$-connected extension. Thus every other embedded $U^2_4\oplus U^1_2$-minor $N\hookrightarrow M$ is contained in a proper $2$-connected embedded minor $M'$ of $M$. By our inductive hypothesis, the natural map $\colim F(\cE^<_{M',\cC^{(2)}})\to F_{M'}$ is an isomorphism. Since the map $F_{N}\to F_M$ factors through $F_{M'}$, we conclude that we can exclude $N$ from $\cE^{<,+}_{M,\cC^{(2)}}$ without changing the colimit of $F(\cE^{<,+}_{M',\cC^{(2)}})$.

We are left with $N_0$, whose effect on the cross ratios of $M$ consists of the relations $\cross{\sigma(1)}{\sigma(2)}{\sigma(3)}{\sigma(4)}{I_0a}=\cross{\sigma(1)}{\sigma(2)}{\sigma(3)}{\sigma(4)}{I_0b}$ for all $\sigma\in S_4$. These relations are also implied by a sequence of embedded $D_6$-minors, as considered above in (1)--(3), which identify cross ratios according to \autoref{lemma: cross ratios in D6}. Namely, if $n$ is even, then
\begin{multline*}
 \cross{\sigma(1)}{\sigma(2)}{\sigma(3)}{\sigma(4)}{I_0a} \ = \ \cross{\sigma(1)}{\sigma(2)}{\sigma(3)}{\sigma(4)}{I_5a} \ = \ \dotsb \ = \ \cross{\sigma(1)}{\sigma(2)}{\sigma(3)}{\sigma(4)}{I_{n-4}a} \ = \ \cross{\sigma(1)}{\sigma(2)}{\sigma(3)}{\sigma(4)}{I_{n-2}}   \\
 \ = \ \cross{\sigma(1)}{\sigma(2)}{\sigma(3)}{\sigma(4)}{I_{n-4}b} \ = \ \dotsb \ = \ \cross{\sigma(1)}{\sigma(2)}{\sigma(3)}{\sigma(4)}{I_5b} \ = \ \cross{\sigma(1)}{\sigma(2)}{\sigma(3)}{\sigma(4)}{I_0b}.
\end{multline*}
If $n$ is odd, then
\begin{multline*}
 \cross{\sigma(1)}{\sigma(2)}{\sigma(3)}{\sigma(4)}{I_0a} \ = \ \cross{\sigma(1)}{\sigma(2)}{\sigma(3)}{\sigma(4)}{I_5a} \ = \ \dotsb \ = \ \cross{\sigma(1)}{\sigma(2)}{\sigma(3)}{\sigma(4)}{I_{n-1}a} \\
 \ = \ \cross{\sigma(1)}{\sigma(2)}{\sigma(3)}{\sigma(4)}{I_{n-1}b} \ = \ \dotsb \ = \ \cross{\sigma(1)}{\sigma(2)}{\sigma(3)}{\sigma(4)}{I_5b} \ = \ \cross{\sigma(1)}{\sigma(2)}{\sigma(3)}{\sigma(4)}{I_0b}.
\end{multline*}
We see that in either case, the relations between the cross ratios of $N_0$ are already induced by other embedded minors of $F_M$ of type $D_6$. This allows us to exclude $N_0$ from $\cE^{<,+}_{M,\cC^{(2)}}$ without changing the colimit of $F(\cE^{<,+}_{M',\cC^{(2)}})$. We deduce that the natural map $\colim F(\cE^{<}_{M,\cC^{(2)}})\to F_M$ is an isomorphism, which concludes the proof of \autoref{thm: fundamental presentation for 2-connected matroids}. \hfill\qed


\subsection{The 3-connected fundamental presentation}
\label{subsection: presentation of the foundation by embedded minors for 3-connected matroids}

If $M$ is $3$-connected, then we can present its foundation $F_M$ as the colimit of $3$-connected embedded minors of certain isomorphism types. In comparison with the $2$-connected fundamental diagram $\cE_M^{(2)}$, we replace the embedded minors of types $C_5$ and $C_5^\ast$, which are not $3$-connected, by their minimal $3$-connected extensions, which are the rank $3$ whirl $W^3$ and the matroids $Q_6$ and $P_6$ (following Oxley's notation in \cite[p.\ 641]{Oxley92}). This presentation of $F_M$ results from a version of the splitter theorem ``for a marked element'' by Bixby and Coullard (\cite{Bixby-Coullard87}), together with an exhaustive computer search.

\begin{figure}[htb]
 \[
 \beginpgfgraphicnamed{tikz/fig7}
 \begin{tikzpicture}[x=0.65cm,y=0.65cm]
  \filldraw[xshift=-4.5cm] ( 30:1) circle (2pt);  
  \filldraw[xshift=-4.5cm] (150:1) circle (2pt);  
  \filldraw[xshift=-4.5cm] (270:1) circle (2pt);  
  \filldraw[xshift=-4.5cm] ( 90:2) circle (2pt);  
  \filldraw[xshift=-4.5cm] (210:2) circle (2pt);  
  \filldraw[xshift=-4.5cm] (330:2) circle (2pt);  
  \draw [xshift=-4.5cm,thick] ( 90:2) -- (210:2);
  \draw [xshift=-4.5cm,thick] (210:2) -- (330:2);
  \draw [xshift=-4.5cm,thick] (330:2) -- ( 90:2);
  \draw[xshift=-4.5cm] (0,-2) node {$W^3$};
  \filldraw[xshift=0cm] ( 30:1) circle (2pt);  
  \filldraw[xshift=0cm] (150:1) circle (2pt);  
  \filldraw[xshift=0cm] (270:1) circle (2pt);  
  \filldraw[xshift=0cm] ( 90:2) circle (2pt);  
  \filldraw[xshift=0cm] (210:2) circle (2pt);  
  \filldraw[xshift=0cm] (330:2) circle (2pt);  
  \draw [xshift=0cm,thick] ( 90:2) -- (210:2);
  \draw [xshift=0cm,thick] (210:2) -- (330:2);
  \draw[xshift=0cm] (0,-2) node {$Q_6$};
  \filldraw[xshift=4.5cm] ( 30:1) circle (2pt);  
  \filldraw[xshift=4.5cm] (150:1) circle (2pt);  
  \filldraw[xshift=4.5cm] (270:1) circle (2pt);  
  \filldraw[xshift=4.5cm] ( 90:2) circle (2pt);  
  \filldraw[xshift=4.5cm] (210:2) circle (2pt);  
  \filldraw[xshift=4.5cm] (330:2) circle (2pt);  
  \draw [xshift=4.5cm,thick] (210:2) -- (330:2);
  \draw[xshift=4.5cm] (0,-2) node {$P_6$};
 \end{tikzpicture}
 \endpgfgraphicnamed
 \]
 \caption{The $3$-circuits of $W^3$, $Q_6$ and $P_6$}
 \label{fig: 3-circuits of W3, Q6 and P6}
\end{figure}
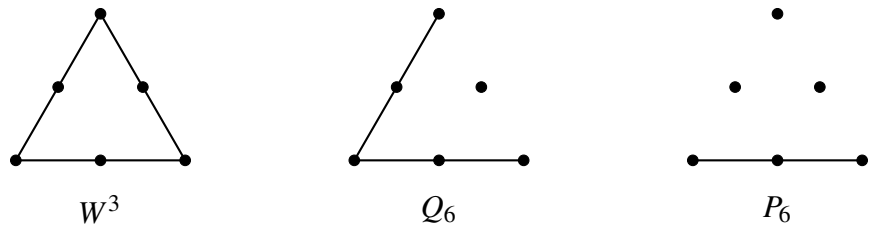
The matroids $W^3$, $Q_6$ and $P_6$ are all rank $3$ matroids on $6$ elements. Their respective $3$-circuits are illustrated in \autoref{fig: 3-circuits of W3, Q6 and P6}. 
The foundations of these three matroids are isomorphic to the foundations of the first rank $2$ uniform matroids:
\[
 F_{W^3} \ \simeq \ F_{U^2_4} \ = \ \U, \qquad\qquad  F_{Q_6} \ \simeq \ F_{U^2_5} \ = \ \V, \qquad\qquad  F_{P_6} \ \simeq \ F_{U^2_6}.
\]

See \ \autoref{subsection: foundation of Q6} for the computation of the foundation of $Q_6$ and \autoref{subsection: foundation of whirls} for $W^3$.
We have verified that the foundations of $P_6$ and $U^2_6$ are isomorphic with the assistance of the Macaulay2 package \textsc{Pastures}; however, it also follows from the fact (suggested to us by Nathan Bowler) that
the foundation of a matroid $M$ is equal to the foundation of any Delta-Wye exchange $M'$ of $M$; this is proved in \cite{Baker-Lorscheid-Walsh-Zhang}.

\begin{df}
 The \emph{$3$-connected fundamental diagram of $M$} is the diagram $\cE_M^{(3)}$ of embedded minors of isomorphism types
 \[
   U^2_4, \qquad U^2_5, \qquad U^3_5, \qquad W^3, \qquad Q_6, \qquad P_6, \qquad F_7, \textrm{ \; and \; } F_7^\ast,
 \]
together with all minor embeddings. 
\end{df}

If $M$ is $3$-connected, we call the associated diagram $F(\cE_M^{(3)})$ of foundations the \emph{$3$-connected fundamental presentation of $M$}, which is motivated by the following result. Let $\cC^{(3)}$ be the class of all $3$-connected matroids.

\begin{thm}\label{thm: fundamental presentation for 3-connected matroids}
  The fundamental type of $\cC^{(3)}$ consists of the isomorphism types of the matroids $U^2_4$, $U^2_5$, $U^3_5$, $W^3$, $Q_6$, $P_6$, $F_7$ and $F_7^\ast$. In other words, if $M$ is a $3$-connected matroid with foundation $F_M$ and $3$-connected fundamental diagram $\cE_M^{(3)}$, then $F_M\simeq \colim F(\cE_M^{(3)})$.
\end{thm}

\begin{proof}
 By \autoref{prop: fundamental type of subclasses with U24} applied to $\cC=\cC^{(2)}$ and \autoref{thm: fundamental presentation for 2-connected matroids}, it suffices to show that the minimal $3$-connected extensions of $C_5$ and $C_5^*$ are $W^3$, $P_6$, and $Q_6$. Since each of the latter three matroids is self-dual, it suffices to establish the result for $C_5$.

 Let $e$ be a series element of $C_5$, so that $C_5/e=U^2_4$. Let $M$ be a minimal $3$-connected extension of $C_5$, which is the same as a minimal $3$-connected extension of $U^2_4=C_5/e$ that uses $e$. By \cite[Theorem 12.3.6]{Oxley-Whittle98}, $M$ has at most $4$ elements that are not in $U^2_4$. An exhaustive search among all matroids with up to $8$ elements using the Macaulay2 package \textsc{Pastures} then shows that $W^3$, $Q_6$, and $P_6$ are the only minimal $3$-connected extensions of $C_5$.
\end{proof}

\begin{df}
 Let $M$ be a matroid with lattice of flats $\Lambda$. The \emph{$3$-connected fundamental lattice diagram} is the diagram $\cL_M^{(3)}$ of all upper sublattices of $\Lambda$ of types $U^2_4$, $U^2_5$, $U^3_5$, $W^3$, $Q_6$, $P_6$, $F_7$, and $F_7^\ast$, together with all inclusions as sublattices.
\end{df}

\begin{thm}\label{thm: fundamental lattice presentation for 3-connected matroids}
 Let $M$ be a matroid with foundation $F_M$ and $3$-connected fundamental lattice diagram $\cL_M^{(3)}$. Then $F_M\simeq\colim F(\cL_M^{(3)})$.
\end{thm}

\begin{proof}
 This follows at once from \autoref{thm: fundamental presentation for 3-connected matroids}, using the observation that the foundations of two embedded minors with the same upper sublattice are identified in $\colim F(\cE^{(3)})\simeq F_M$, and thus this isomorphism factors through $\colim F(\cL_M^{(3)})$.
\end{proof}

If we exclude certain minors from the fundamental type of $\cC^{(3)}$, then these minors get replaced by their minimal proper $3$-connected extensions. In this way, we can derive new results for subclasses of $\cC^{(3)}$. For example, the minimal proper $3$-connected extensions of $U^2_5$ are $U^2_6$, $U^3_6$, $W^3$, $Q_6$, and $P_6$. The same holds for $U^3_5$ if we replace $U^2_6$ by $U^4_6$ in this list. By excluding some of these matroids, we find fundamental types for various subclasses of $\cC^{(3)}$. One such result that will prove useful in a follow-up paper is the following.

\begin{cor}\label{cor: fundamental type of U24 and 3-connected on at least 6 elements without some minors}
 Let $\cC$ be the class of all $3$-connected matroids of type $U^2_4$, or 3-connected matroids with at least $6$ elements that are without minors of type $U^2_6$, $U^3_6$, $U^4_6$, or $P_6$. Then the fundamental class of $\cC$ consists of the isomorphism types of $U^2_4$, $W^3$, $Q_6$, $F_7$, and $F_7^\ast$.
\end{cor}

\begin{proof}
 Let $M$ be in $\cC$ and let $N=M\minor JI$ an embedded minor of type $U^2_5$ or $U^3_5$. By the Splitter Theorem (\cite[Thm.\ 12.1.2]{Oxley92}), $N$ is embedded in a $3$-connected $N'$-minor of $M$ on $6$ elements, which must be of type $Q_6$ since we excluded all other possible types in the hypothesis of the corollary. By \autoref{prop: foundation of Q6}, the induced morphism $F_N\to F_{N'}$ is an isomorphism, and the map $F_N\to F_M$ factors through $F_N\to F_{N'}$. Since $P_6$ is excluded, we deduce from \autoref{thm: fundamental presentation for 3-connected matroids} that the fundamental type of $\cC$ consists of $U^2_4$, $W^3$, $Q_6$, $F_7$, and $F_7^\ast$, as claimed.
\end{proof}

\section{Final examples}
\label{section: final examples}

In this section, we compute the foundation of whirls (\cite[p.\ 659]{Oxley92}), the non-Fano matroid $F_7^-$ (\cite[p.\ 643]{Oxley92}), and the matroids $P_7$ and $T_8$ (\cite[p.\ 644, 649]{Oxley92}).


\subsection{The foundation of whirls}
\label{subsection: foundation of whirls}

Let us recall the definition of whirls. Let $W_r$ be the $r$-spoked wheel, considered as a matroid. Let $a_1,\dotsc,a_r$ correspond to the rim edges and $b_1,\dotsc,b_r$ to the spoke edges, as illustrated in \autoref{fig: the r-spoked wheel}. The rim $\{a_1,\ldots,a_r\}$ is the unique circuit-hyperplane of $W_r$, and the rank $r$ whirl $W^r$ is the corresponding relaxation of $W_r$. Note that wheels are regular, so the foundation of $W_r$ is $\Funpm$, cf.\ \autoref{subsection: foundations of binary and regular matroids}.

 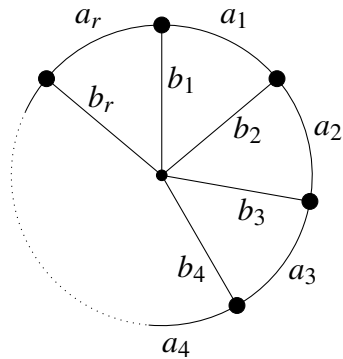
\begin{figure}[htb]
 \[
 \beginpgfgraphicnamed{tikz/fig9}
  \begin{tikzpicture}
   \filldraw (0:0) circle (2pt);  
   \filldraw (140:2) circle (3pt);  
   \draw (0:0) -- (140:2);
   \node at (128:1.3) {$b_{r}$};
   \node at (115:2.3) {$a_{r}$};
   \foreach \a in {1,2,3,4} {
                           \filldraw (140-50*\a:2) circle (3pt);  
                           \draw (0:0) -- (140-50*\a:2);
                           \node at (128-50*\a:1.3) {$b_{\a}$};
                           \node at (115-50*\a:2.3) {$a_{\a}$};
                           } 
   \centerarc[](0,0)(155:-95:2);
   \centerarc[dotted](0,0)(155:265:2);
  \end{tikzpicture}
 \endpgfgraphicnamed
 \]
 \caption{The graphic representation of the $r$-spoked wheel}
 \label{fig: the r-spoked wheel} 
\end{figure}

\begin{prop}\label{prop: foundations of whirls}
 The foundation of the whirl $W^r$ is isomorphic to $\U$ for all $r\geq 2$. 
\end{prop}

\begin{proof}
 By \cite[p.\ 660]{Oxley92}, whirls are $2$-connected and near-regular. In particular, they do not contain any minors of types $U^2_5$ and $U^3_5$. Thus by the structure theorem for foundations of matroids without large uniform minors (cf.\ \autoref{thm: structure theorem for foundations of matroids without large uniform minors}), the foundation of $W^r$ is a tensor product of copies of $\U$, $\D$, $\H$, $\F_3$, and $\F_2$. Since $W^r$ is near-regular, but none of $\D$, $\H$, $\F_3$, and $\F_2$ allow for a morphism to $\U$, we conclude that the foundation of $W^r$ is a tensor power of $\U$. Moreover, if $\cE^{(2)}(W^r)$ is connected then there is only one tensor factor. 
 
 We prove that $\cE^{(2)}(W^r)$ is connected by induction on $r$. The $2$-whirl $W^2$ is isomorphic to $U^2_4$ and thus $\cE^{(2)}(W^2)$ consists of a single vertex (of type $U^2_4$), which establishes the base case.
 
 Assume that $r>2$. We need to show that every pair of embedded $U^2_4$-minors of $W^r$ lie in the same connected component of $\cE^{(2)}(W^r)$. Let $E=\{a_1,\dotsc,a_r,b_1,\dotsc,b_r\}$ be the common ground set of $W_r$ and $W^r$ and $i\in\{1,\dotsc,r\}$. As a first step, we observe that the bases $W^r/b_i$ are those subsets $B$ of $E-b_i$ for which $B\cup \{b_i\}$ is a basis of $W^r$. These subsets agree with the bases of the regular matroid $W_r/b_i$, which shows that $W^r/b_i$ is regular and does not contain any $U^2_4$-minor.
 
 Since the rank of $W^r$ is $r>2$, we conclude that every embedded $U^2_4$-minor of $W^r$ is contained in $W^r/a_i$ for some $i$. Since $b_i$ and $b_{i+1}$ are parallel in $W^r/a_i$ (where we use $b_{r+1}=b_1$), every $U^2_4$-minor of $W^r/a_i$ is contained in either $W^r\minor{b_i}{a_i}$ or $W^r\minor{b_{i+1}}{a_{i}}$, which are both isomorphic to $W^{r-1}$. By the inductive hypothesis, the $2$-connected fundamental diagram of $W^{r-1}$ is connected, so we are left with showing that the $2$-connected fundamental diagrams of $W^r\minor{b_i}{a_i}$ or $W^r\minor{b_{i+1}}{a_{i}}$ (for varying $i$) are connected as subdiagrams of $\cE^{(2)}(W^r)$.
 
 The subdiagrams $W^r\minor{b_i}{a_i}$ and $W^r\minor{b_{i+1}}{a_{i}}$ are connected by the embedded $C_5^\ast$-minor
 \[
  W^r\;\minor{\;b_1\dotsc b_{i-2}\;b_{i+2}\dotsc b_r\;}{\;a_1\dotsc a_{i-2}\;a_i\;a_{i+2}\dotsc a_r}
 \]
 (where we read indices modulo $r$, as appropriate) and the subdiagrams $W^r\minor{b_i}{a_i}$ and $W^r\minor{b_i}{a_{i-1}}$ are connected by the embedded $C_5$-minor 
 \[
  W^r\;\minor{\;b_1\dotsc b_{i-2}\;b_i\;b_{i+2}\dotsc b_r\;}{\;a_1\dotsc a_{i-2}\;a_{i+2}\dotsc a_r},
 \]
 which shows that $\cE^{(2)}(W^r)$ is connected as claimed.
\end{proof}

\begin{ex}\label{ex: 3-connected fundamental diagrams of W3}
 We determine $\cE_M^{(3)}$ and $\cL_M^{(3)}$ for the rank $3$ whirl $M=W^3$ as a first example. Its $3$-circuits and its lattice of flats $\Lambda$ are illustrated in \autoref{fig: circuits and lattice of flats of W3}.

 \begin{figure}[htb]
 \[
 \beginpgfgraphicnamed{tikz/fig36}
  \begin{tikzpicture}[cross line/.style={normal line,preaction={draw=white, -, line width=2pt}}]
   \node (C) at (-2,1) {};
   \foreach \a in {2,4,6} {
                           \filldraw (C)++(150+\a*60:0.75) coordinate (P\a) circle (2pt);  
                           \draw[color=white] (C)++(150+\a*60:1.1) coordinate (L\a) circle (2pt);  
                           \node at (L\a) {\a};
                           } 
   \foreach \a in {1,3,5} {
                           \filldraw (C)++(150+\a*60:1.5) coordinate (P\a) circle (2pt);  
                           \draw[color=white] (C)++(150+\a*60:1.9) coordinate (L\a) circle (2pt);  
                           \node at (L\a) {\a};
                           } 
   \draw (P1) -- (P3) -- (P5) -- (P1);                          
   \node (0) at (5,0) {$\emptyset$};
   \node (1) at (2.5,1) {$1$};
   \node (2) at (3.5,1) {$2$};
   \node (3) at (4.5,1) {$3$};
   \node (4) at (5.5,1) {$4$};
   \node (5) at (6.5,1) {$5$};
   \node (6) at (7.5,1) {$6$};
   \node (14) at (1,2){$14$};
   \node (123) at (2,2){$123$};
   \node (24) at (3,2){$24$};
   \node (36) at (4,2){$36$};
   \node (345) at (5,2){$345$};
   \node (46) at (6,2){$46$};
   \node (52) at (7,2){$52$};
   \node (561) at (8,2){$561$};
   \node (62) at (9,2){$62$};
   \node (E) at (5,3){$123456$};
   \foreach \from/\to in {0/1, 0/2, 0/3, 0/4, 0/5, 0/6, 1/14, 4/14, 1/123, 2/123, 3/123, 2/24, 4/24, 3/36, 6/36, 3/345, 4/345, 5/345, 4/46, 6/46, 5/52, 2/52, 5/561, 6/561, 1/561, 6/62, 2/62, 14/E, 123/E, 24/E, 36/E, 345/E, 46/E, 52/E, 561/E, 62/E} \draw [-] (\from)--(\to);
  \end{tikzpicture}
 \endpgfgraphicnamed
 \]
 \caption{The $3$-circuits and the lattice of flats of the $3$-whirl $W^3$}
 \label{fig: circuits and lattice of flats of W3} 
\end{figure}
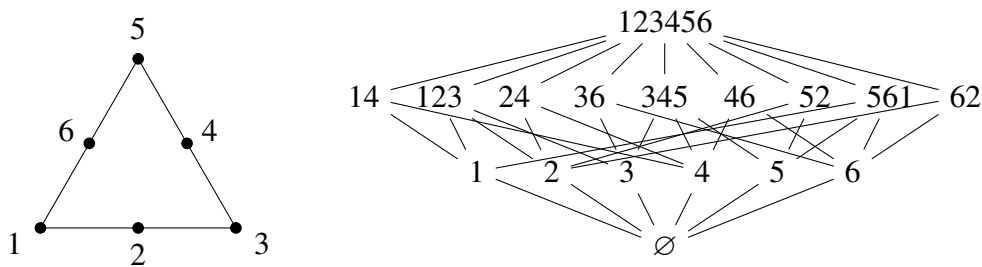

The embedded minors that appear in $\cE^{(3)}_M$ are $M$ itself and the $U^2_4$-minors $M\minor ji$ with $j$ even, $i$ odd, both contained in a common $3$-circuit. The upper sublattice $\Lambda/i$ defined by $M\minor ji$ is insensitive to $j$, and therefore $\Lambda/i$ corresponds to two distinct embedded $U^2_4$-minors. The diagrams $\cE^{(3)}_M$ and $\cL_M^{(3)}$ are illustrated in \autoref{fig: 3-connected fundamental diagram of W3}.

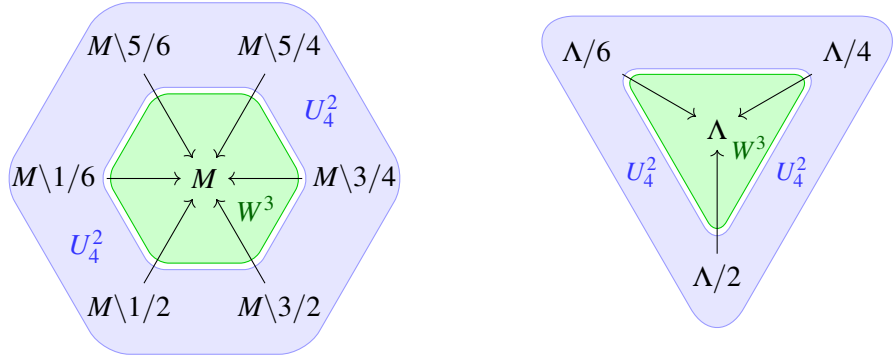
\begin{figure}[htb]
 \[
  \beginpgfgraphicnamed{tikz/fig38}
   \begin{tikzpicture}[font=\small]
    \draw[fill=blue!10!white,draw=blue!40!white,rounded corners=10pt] (0:2.7) -- (60:2.7) -- (120:2.7) -- (180:2.7) -- (240:2.7) -- (300:2.7) -- cycle;
    \draw[fill=white,draw=blue!40!white,rounded corners=5pt] (0:1.4) -- (60:1.4) -- (120:1.4) -- (180:1.4) -- (240:1.4) -- (300:1.4) -- cycle;
    \draw[fill=green!20!white,draw=green!80!black,rounded corners=5pt] (0:1.3) -- (60:1.3) -- (120:1.3) -- (180:1.3) -- (240:1.3) -- (300:1.3) -- cycle;
    \node[color=blue!80!white] at (30:1.8) {$U^2_4$};
    \node[color=blue!80!white] at (210:1.8) {$U^2_4$};
    \node[color=green!40!black] at (330:0.8) {$W^3$};
    \node (M) at (0:0) {$M$};
    \node (M34) at (  0:2) {$M\minor34$};
    \node (M54) at ( 60:2) {$M\minor54$};
    \node (M56) at (120:2) {$M\minor56$};
    \node (M16) at (180:2) {$M\minor16$};
    \node (M12) at (240:2) {$M\minor12$};
    \node (M32) at (300:2) {$M\minor32$};
    \foreach \from in {34, 54, 56, 16, 12, 32} \draw[->] (M\from) -- (M);
   \end{tikzpicture}
   \hspace{1.5cm}
   \begin{tikzpicture}[font=\small]
    \draw[fill=blue!10!white,draw=blue!40!white,rounded corners=20pt] (30:3.0) -- (150:3.0) -- (270:3.0) -- cycle;
    \draw[fill=white,draw=blue!40!white,rounded corners=10pt] (30:1.6) -- (150:1.6) -- (270:1.6) -- cycle;
    \draw[fill=green!20!white,draw=green!80!black,rounded corners=7pt] (30:1.45) -- (150:1.45) -- (270:1.45) -- cycle;
    \node[color=blue!80!white] at (330:1.15) {\footnotesize $U^2_4$};
    \node[color=blue!80!white] at (210:1.15) {\footnotesize $U^2_4$};
    \node[color=green!40!black] at (330:0.5) {\footnotesize $W^3$};
    \node (L) at (0:0) {$\Lambda$};
    \node (L4) at ( 30:2) {$\Lambda/4$};
    \node (L6) at (150:2) {$\Lambda/6$};
    \node (L2) at (270:2) {$\Lambda/2$};
    \foreach \from in {2,4,6} \draw[->] (L\from) -- (L);
   \end{tikzpicture}
  \endpgfgraphicnamed
 \]
 \caption{The $3$-connected fundamental diagram and lattice diagram of $W^3$}
 \label{fig: 3-connected fundamental diagram of W3} 
\end{figure}
\end{ex}

Since the foundation of the $3$-whirl $W^3$ is $\U$ by \autoref{prop: foundations of whirls}, the embedding of each minor $W^3\minor ji$ of type $U^2_4$ into $W^3$ induces an isomorphism $F_{W^3\minor ji}\to F_{W^3}$. This means that the cross ratios stemming from the six $U^2_4$-minors of $W^3$ are identified in $F_{W^3}$. Explicitly, these identifications are as follows.

\begin{lemma}\label{lemma: cross ratios in W3}
 The cross ratios in the foundation of $W^3$ satisfy the following relations (assuming the enumeration of elements as in \autoref{fig: circuits and lattice of flats of W3}):
 \[
  \cross{123}{24}{25}{26}{} \ = \ \cross{345}{46}{14}{24}{}  \ = \ \cross{156}{26}{36}{46}{}.
 \]
\end{lemma}

\begin{proof}
 For each pair of cross ratios in this equation, all of their hyperplanes are contained in an upper sublattice of $W^3$ of type $C_5$, thus the claim follows directly from \eqref{H4'}.
\end{proof}


\subsection{The foundation of \texorpdfstring{$F_7^-$}{F7-}}
\label{subsection: foundation of F7-}

In this section, we compute the foundation of the non-Fano matroid $F_7^-$, whose $3$-circuits are as illustrated in \autoref{fig: circuits of nonfano}. 

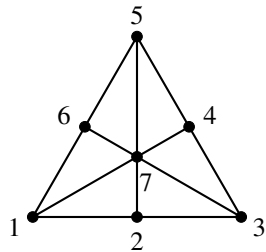
\begin{figure}[htb]
 \[
 \beginpgfgraphicnamed{tikz/fig31}
 \begin{tikzpicture}[x=0.4cm,y=0.4cm]
  \filldraw ( 30:2) circle (2pt);  
  \node at ( 30:2.8) {\footnotesize $4$};  
  \filldraw (150:2) circle (2pt);  
  \node at (150:2.8) {\footnotesize $6$};  
  \filldraw (270:2) circle (2pt);  
  \node at (270:2.8) {\footnotesize $2$};  
  \filldraw ( 90:4) circle (2pt);  
  \node at ( 90:4.7) {\footnotesize $5$};  
  \filldraw (210:4) circle (2pt);  
  \node at (210:4.7) {\footnotesize $1$};  
  \filldraw (330:4) circle (2pt);  
  \node at (330:4.7) {\footnotesize $3$};
  \node at (290:0.9) {\footnotesize $7$};
  \filldraw (330:0) circle (2pt);  
  \draw [thick] ( 90:4) -- (210:4);
  \draw [thick] (210:4) -- (330:4);
  \draw [thick] (90:4) -- (330:4);
  \draw [thick] (150:2) -- (330:4);
  \draw [thick] (270:2) -- (90:4);
  \draw [thick] (210:4) -- (30:2);
 \end{tikzpicture}
 \endpgfgraphicnamed
 \]
 \caption{The $3$-circuits of the non-Fano matroid $F_7^-$}
 \label{fig: circuits of nonfano} 
\end{figure}

\begin{prop}\label{prop: foundation of the non-Fano matroid}
 The foundation of $F_7^-$ is isomorphic to $\D$.
\end{prop}

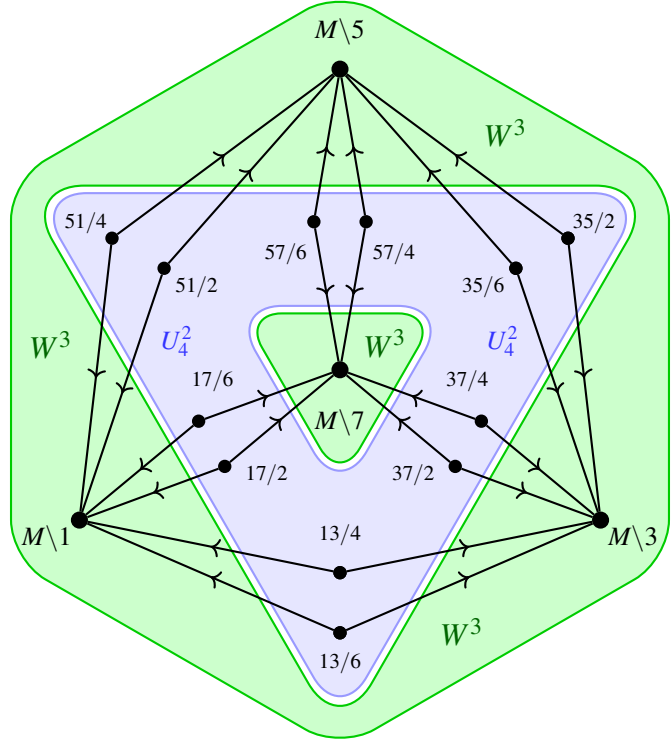
\begin{figure}[htb]
 \[
  \beginpgfgraphicnamed{tikz/fig32}
   \begin{tikzpicture}[x=1cm,y=1cm,scale=1, font=\footnotesize, vertices/.style={draw, fill=black, circle, inner sep=0pt},thick,decoration={markings,mark=at position 0.5 with {\arrow{>}}}]
    \draw[fill=green!20!white,draw=green!80!black,rounded corners=15pt] (30:5.05) -- (90:5.05) -- (150:5.05) -- (210:5.05) -- (270:5.05) -- (330:5.05) -- cycle;
    \draw[fill=white,draw=green!80!black,rounded corners=24pt] (270:4.9) -- (30:4.9) -- (150:4.9) -- cycle;
    \draw[fill=blue!10!white,draw=blue!40!white,rounded corners=20pt] (270:4.7) -- (30:4.7) -- (150:4.7) -- cycle;
    \draw[fill=white,draw=blue!40!white,rounded corners=20pt] (270:1.7) -- (30:1.7) -- (150:1.7) -- cycle;
    \draw[fill=green!20!white,draw=green!80!black,rounded corners=15pt] (270:1.5) -- (30:1.5) -- (150:1.5) -- cycle;
    \node[color=blue!80!white] at (170:2.2) {\footnotesize  $U^2_4$};
    \node[color=blue!80!white] at (10:2.2) {\footnotesize  $U^2_4$};
    \node[color=green!40!black] at (30:0.7) {\normalsize $W^3$};
    \node[color=green!40!black] at (55:3.85) {\normalsize $W^3$};
    \node[color=green!40!black] at (175:3.85) {\normalsize $W^3$};
    \node[color=green!40!black] at (295:3.85) {\normalsize $W^3$};
    \foreach \a in {2,4,...,10}{
                           \draw ((150+\a*360/6: 3.5) node [draw,circle,inner sep=1.5pt,fill=black] (U\a) {};
                           \draw ((150+\a*360/6: 2.7) node [draw,circle,inner sep=1.5pt,fill=black] (V\a) {};
                           }
    \draw (0:0) node [draw,circle,inner sep=2pt,fill=black] (W7) {};
    \foreach \a in {1,3,5}{
                           \draw (150+\a*360/6: 4) node [draw,circle,inner sep=2pt,fill=black] (W\a) {};
                           \draw (140+\a*360/6: 2) node [draw,circle,inner sep=1.5pt,fill=black] (X\a) {};
                           \draw (160+\a*360/6: 2) node [draw,circle,inner sep=1.5pt,fill=black] (Y\a) {};
                           }
    \foreach \a in {1,3,5} {\setcounter{tikz-counter}{\a};
                              \draw [-,postaction={decorate}] (X\arabic{tikz-counter}) -- (W7);
                              \draw [-,postaction={decorate}] (Y\arabic{tikz-counter}) -- (W7);
                              \draw [-,postaction={decorate}] (X\arabic{tikz-counter}) -- (W\a);
                              \draw [-,postaction={decorate}] (Y\arabic{tikz-counter}) -- (W\a);
                              \addtocounter{tikz-counter}{1};
                              \draw [-,postaction={decorate}] (U\arabic{tikz-counter}) -- (W\a);
                              \draw [-,postaction={decorate}] (V\arabic{tikz-counter}) -- (W\a);
                              \addtocounter{tikz-counter}{4};
                              \draw [-,postaction={decorate}] (U\arabic{tikz-counter}) -- (W\a);
                              \draw [-,postaction={decorate}] (V\arabic{tikz-counter}) -- (W\a);
                              }
    \draw (270:0.7) node {$M\setminus7$};                              
    \foreach \a in {1,3,5}{\draw (150+\a*360/6: 4.5) node {$M\setminus\a$};} 
    \foreach \angle/\exclude/\contract in  {1/17/6,3/37/2,5/57/4} {\draw ((125+\angle*360/6:1.7) node {\tiny$\exclude/\contract$};}
    \foreach \angle/\exclude/\contract in {1/17/2,3/37/4,5/57/6} {\draw ((175+\angle*360/6:1.7) node {\tiny$\exclude/\contract$};}
    \foreach \angle/\exclude/\contract in {2/13/4,4/35/6,6/51/2} {\draw ((150+\angle*360/6:2.2) node {\tiny$\exclude/\contract$};}
    \foreach \angle/\exclude/\contract in {2/13/6,4/35/2,6/51/4} {\draw ((150+\angle*360/6:3.9) node {\tiny$\exclude/\contract$};}
   \end{tikzpicture}
  \endpgfgraphicnamed
 \]
 \caption{The 3-connected fundamental diagram of $F_7^-$}
 \label{fig: 3-conn fund diagram of the nonfano matroid} 
\end{figure}

\begin{proof}
 Since $F_7^-$ is $3$-connected, we can use the $3$-connected fundamental diagram $\cE^{(3)}(F_7^-)$ of $F_7^-$, as illustrated in \autoref{fig: 3-conn fund diagram of the nonfano matroid}, to compute $F_M$ by \autoref{thm: fundamental presentation for 3-connected matroids}
 
 Since $\cE^{(3)}(F_7^-)$ is connected and $F_7^-$ does not have any minors of types $U^2_5$,\ $U^3_5$,\ $F_7$ and $F_7^\ast$, \autoref{cor: foundation of a matroid wlum with one component} implies that the foundation $F_{F^-_7}$ of $F_7^-$ is a symmetry quotient of $\U$.
 
 By \cite[p.\ 644]{Oxley92}, the non-Fano matroid $F_7^-$ is dyadic, i.e.,\ $\D$-representable. Since neither $\H$ nor $\F_3$ map to $\D$, the foundation $F_M$ is either $\U$ or $\D$. In order to exhibit the defining relation of $\D=\past\U{\{x-y\}}$, we consider the subdiagram 
 
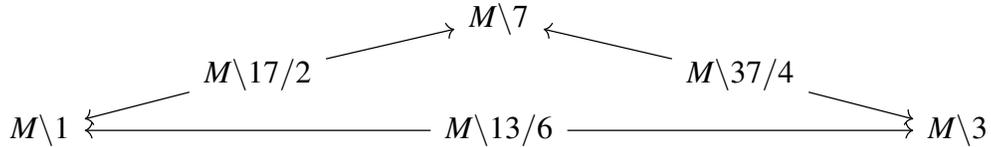
\[
  \begin{tikzcd}[row sep=0,column sep=40]
   & & M\setminus 7  \\
   & M\minor{17}2 \ar[dl] \ar[ur] && M\minor{37}4 \ar[dr] \ar[ul] \\
   M\setminus 1 && M\minor{13}6 \ar[rr] \ar[ll] && M\setminus 3 
  \end{tikzcd}
 \]
 of $\cE^{(3)}(F_7^-)$ and the morphism $\alpha:\U\simeq F_{W^3\minor{13}6}\to F_{W^3}$ with $\alpha(x)=\cross{367}{46}{26}{56}{}$ and $\alpha(y)=\cross{367}{26}{46}{56}{}$. By \autoref{lemma: cross ratios in W3}, we find the following chain of equalities (where we label the equalities by the corresponding $W^3$-minor to which we apply the lemma):
 \[
  \alpha(x) \ \ = \ \ \cross{367}{46}{26}{56}{} \ \ \underset{M\setminus1}= \ \ \cross{123}{24}{26}{25}{} \ \ \underset{M\setminus7}= \ \ \cross{147}{24}{46}{45}{} \ \ \underset{M\setminus3}= \ \ \cross{367}{26}{46}{56}{} \ \ = \ \ \alpha(y),
 \]
 which shows that $\alpha(x)=\alpha(y)$ in $F_{F^-_7}$ and thus $F_{F^-_7}\simeq\past\U{\{x-y\}}=\D$, as claimed. 
\end{proof}

\begin{rem}
 Note that it is known that $F_7^-$ is not near-regular: by \cite[p.\ 644]{Oxley92}), $F_7^-$ is not representable over any field of characteristic $2$. This could have been used as a shortcut in last step of the proof of \autoref{prop: foundation of the non-Fano matroid}.
\end{rem}

\begin{lemma}\label{lemma: cross ratios in F7-}
 There is a unique isomorphism $\alpha:\D=\pastgenn\Funpm{x}{x+x-1}\to F_{F_7^-}$, and the following equalities between cross ratios hold in $F_7^-$ (assuming the enumeration of elements as in \autoref{fig: circuits of nonfano}):
 \begin{align*}
  \alpha(x) \ &= \ \cross{367}{46}{26}{156}{} \ = \ \cross{367}{26}{46}{156}{}; \\
  \alpha(x^{-1}) \ &= \ \ \cross{367}{26}{156}{46}{} \ \; = \ \ \cross{367}{46}{156}{26}{}; \\
  \alpha(-1) \ &= \ \cross{367}{156}{46}{26}{} \ = \ \cross{367}{156}{26}{46}{}.
 \end{align*}
\end{lemma}

\begin{proof}
 The existence of $\alpha:\D\to F_{F_7^-}$ is the content of \autoref{prop: foundation of the non-Fano matroid}, and its uniqueness follows from the fact that $\D$ does not have any nontrivial automorphisms, since $x$ is the unique element of $\D$ that satisfies $x+x-1\in N_\D$. The first row of equalities of cross ratios reflects the computation in the proof of \autoref{prop: foundation of the non-Fano matroid}, and the second and third row are deduced from this by \autoref{prop: invariance of cross ratio equalities under permutations}.
\end{proof}


\subsection{The foundation of \texorpdfstring{$P_7$}{P7}} 
\label{subsection: foundation of P7}

The matroid $P_7$, following Oxley's notation in \cite[p.\ 644]{Oxley92}, is the rank $3$ matroid on $7$ elements whose $3$-circuits and lattice of flats are as illustrated in \autoref{fig: circuits and lattice of flats of P7}.

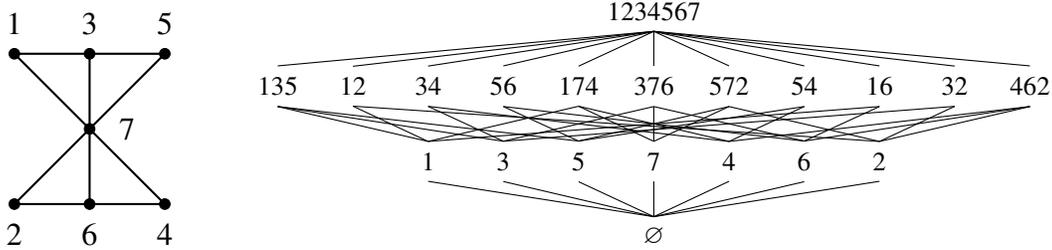
\begin{figure}[tb]
\[
  \beginpgfgraphicnamed{tikz/fig33}
   \begin{tikzpicture}[x=1cm, y=1cm]
    \filldraw (0,0) circle (2pt);
    \filldraw (1,0) circle (2pt);
    \filldraw (2,0) circle (2pt);
    \filldraw (1,1) circle (2pt);
    \filldraw (0,2) circle (2pt);
    \filldraw (1,2) circle (2pt);
    \filldraw (2,2) circle (2pt);
    \draw[thick] (0,0) -- (2,0) -- (0,2) -- (2,2) -- (0,0);
    \draw[thick] (1,0) -- (1,2);
    \node at (0,2.4) {$1$};
    \node at (1,2.4) {$3$};
    \node at (2,2.4) {$5$};
    \node at (1.5,1) {$7$};
    \node at (0,-0.4) {$2$};
    \node at (1,-0.4) {$6$};
    \node at (2,-0.4) {$4$};
   \end{tikzpicture}
  \endpgfgraphicnamed
  \hspace{0.85cm}
  \beginpgfgraphicnamed{tikz/fig39}
   \begin{tikzpicture}[x=1cm, y=1cm, font=\footnotesize]
              \node  (0) at (5,0){$\emptyset$};
              \node  (1) at (2,1) {$1$};
              \node  (3) at (3,1) {$3$};
              \node  (5) at (4,1) {$5$};
              \node  (7) at (5,1) {$7$};
              \node  (4) at (6,1) {$4$};
              \node  (6) at (7,1) {$6$};
              \node  (2) at (8,1) {$2$};
              \node  (135) at (0,2) {$135$};
              \node  (174) at (4,2) {$174$};
              \node  (376) at (5,2) {$376$};
              \node  (572) at (6,2) {$572$};
              \node  (462) at (10,2) {$462$};
              \node  (12) at (1,2) {$12$};
              \node  (34) at (2,2) {$34$};
              \node  (56) at (3,2) {$56$};
              \node  (54) at (7,2) {$54$};
              \node  (16) at (8,2) {$16$};
              \node  (32) at (9,2) {$32$};
              \node  (E) at (5,3) {$1234567$};
      \foreach \to/\from in {0/1, 0/2, 0/3, 0/4, 0/5, 0/6, 0/7, 1/135, 3/135, 5/135, 1/174, 7/174, 4/174, 3/376, 7/376, 6/376, 5/572, 7/572, 2/572, 2/462, 6/462, 4/462, 1/12, 2/12, 3/34, 4/34, 5/56, 6/56, 5/54, 4/54, 1/16, 6/16, 3/32, 2/32, 135/E, 174/E, 376/E, 572/E, 462/E, 12/E, 34/E, 56/E, 54/E, 16/E, 32/E} \draw [-] (\to.north)--(\from.south);
  \end{tikzpicture}
 \endpgfgraphicnamed
\]
 \caption{The $3$-circuits and the lattice of flats of $P_7$}
 \label{fig: circuits and lattice of flats of P7} 
\end{figure}

\begin{prop}\label{prop: foundation of P7}
 The foundation of $P_7$ is isomorphic to $\U$. 
\end{prop}

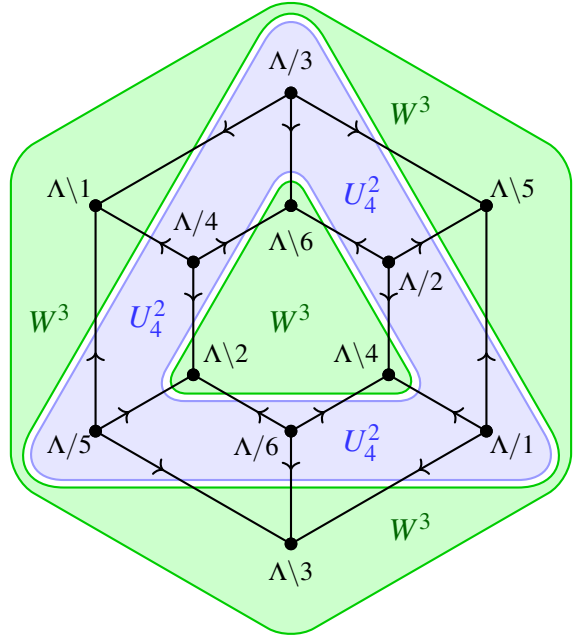
\begin{figure}[tb] 
 \[
  \beginpgfgraphicnamed{tikz/fig34}
   \begin{tikzpicture}[x=1cm,y=1cm, font=\footnotesize, vertices/.style={draw, fill=black, circle, inner sep=0pt},thick,decoration={markings,mark=at position 0.35 with {\arrow{>}}}]
    \draw[fill=green!20!white,draw=green!80!black,rounded corners=10pt] (270:4.3) -- (330:4.3) -- ( 30:4.3) -- ( 90:4.3) -- (150:4.3) -- (210:4.3) -- cycle;
    \draw[fill=white,draw=green!80!black,rounded corners=25pt] (330:4.5) -- ( 90:4.5) -- (210:4.5) -- cycle; 
    \draw[fill=blue!10!white,draw=blue!40!white,rounded corners=20pt] (330:4.3) -- ( 90:4.3) -- (210:4.3) -- cycle; 
    \draw[fill=white,draw=blue!40!white,rounded corners=14pt] (330:2.2) -- ( 90:2.2) -- (210:2.2) -- cycle;
    \draw[fill=green!20!white,draw=green!80!black,rounded corners=10pt] (330:2.0) -- ( 90:2.0) -- (210:2.0) -- cycle;
   \node[color=blue!80!white] at ( 60:1.9) {\normalsize $U^2_4$};
   \node[color=blue!80!white] at (180:1.9) {\normalsize $U^2_4$};
   \node[color=blue!80!white] at (300:1.9) {\normalsize $U^2_4$};
   \node[color=green!40!black] at (0:0) {\normalsize $W^3$};
   \node[color=green!40!black] at ( 60:3.2) {\normalsize $W^3$};
   \node[color=green!40!black] at (180:3.2) {\normalsize $W^3$};
   \node[color=green!40!black] at (300:3.2) {\normalsize $W^3$};
    \foreach \a in {1,3,...,11} {\filldraw ( 90+\a*60:3.0) circle (2pt) coordinate (L-\a); 
                                 \filldraw (270+\a*60:3.0) circle (2pt) coordinate (L/\a);}
    \foreach \a in {2,4,...,12} {\filldraw ( 90+\a*60:1.5) circle (2pt) coordinate (L-\a); 
                                 \filldraw (270+\a*60:1.5) circle (2pt) coordinate (L/\a);}
    \foreach \a in {1,3,5} {\node at ( 90+\a*60:3.4) {$\Lambda\setminus\a$}; 
                            \node at (270+\a*60:3.4) {$\Lambda/\a$};}
    \foreach \a in {2,4,6} {\node at ( 90+\a*60:1.0) {$\Lambda\setminus\a$}; 
                            \node at (255+\a*60:1.79) {$\Lambda/\a$};}
    \foreach \a in {1,...,6} {\setcounter{tikz-counter}{\a};
                              \addtocounter{tikz-counter}{2};
                              \draw [-,postaction={decorate}] (L/\a) -- (L-\arabic{tikz-counter});
                              \addtocounter{tikz-counter}{1};
                              \draw [-,postaction={decorate}] (L/\a) -- (L-\arabic{tikz-counter});
                              \addtocounter{tikz-counter}{1};
                              \draw [-,postaction={decorate}] (L/\a) -- (L-\arabic{tikz-counter});
                              }
   \end{tikzpicture}
  \endpgfgraphicnamed
 \] 
  \caption{The $3$-connected fundamental lattice diagram $\cL^{(3)}_{P_7}$ of $P_7$} 
 \label{fig: 3-connected fundamental lattice diagram of P7} 
\end{figure}

\begin{proof}
 Since $M=P_7$ is near-regular (cf.\ \cite[p.\ 644]{Oxley92}), its foundation $F_M$ is isomorphic to $\U\otimes\dotsc\otimes\U$ by \cite[Thm.\ 5.9]{Baker-Lorscheid20}. Our result $F_M\simeq\U$ follows from the fact that $\cL^{(3)}_M$ is connected, as visible in \autoref{fig: 3-connected fundamental lattice diagram of P7}.
\end{proof}

\begin{lemma}\label{lemma: cross ratios in P7}
 The following equalities between cross ratios hold in the foundation of $P_7$ (assuming the enumeration of elements as in \autoref{fig: circuits and lattice of flats of P7}):
 \[
  \cross{135}{147}{12}{16}{} \ = \ \cross{135}{45}{257}{56}{} \ = \ \cross{135}{24}{23}{367}{} \ = \ \cross{246}{147}{45}{34}{} \ = \ \cross{246}{12}{257}{23}{} \ = \ \cross{246}{16}{56}{367}{}.
 \]
\end{lemma}

\begin{proof}
 As visible in \autoref{fig: 3-connected fundamental lattice diagram of P7}, every $U^2_4$-minor of $P_7$ is contained in a $W^3$-minor. This allows us to deduct the lemma by a repeated application of \autoref{lemma: cross ratios in W3}.
\end{proof}

\subsection{The foundation of \texorpdfstring{$T_8$}{T8}} 
\label{subsection: foundation of T8}

In this section, we determine the foundation of the ternary spike $T_8$ (using Oxley's notation in \cite[p.\ 649]{Oxley92}) as $\F_3$ using the fundamental presentation by upper sublattices of rank $\leq3$ as in \autoref{thm: fundamental lattice presentation by upper sublattices of small rank}.

We realize $T_8$ as the matroid on $E=\{1,\dotsc,8\}$ whose $4$-circuits are
\[
 1238,\ 1247,\ 1346,\ 2345,\ 1256,\ 1357,\ 1458,\ 2367,\ 2468,\ 3478,\ 5678
\]
and whose other circuits have all $5$ elements.

\begin{prop}\label{prop: foundation of T8}
 The foundation of $T_8$ is isomorphic to $\F_3$.
\end{prop}

 
 

\begin{proof}
 As a ternary matroid, $M=T_8$ does not have minors of type $F_7^\ast$, and thus \autoref{thm: fundamental lattice presentation by upper sublattices of small rank} implies that $F_M=\colim F(\cL^{\leq3}_M)$ where $\cL_M^{\leq 3}$ consists of all full upper sublattices of $\Lambda=\Lambda_M$ of rank less or equal to $3$ that contain an upper sublattice of type $U^2_4$.
 
 The lattice $\Lambda$ of flats of $T_8$ contains the following elements: its atoms are $1,\dotsc,8$, its $2$-flats are all $2$-subsets, and its hyperplanes are all $4$-circuits and all $3$-subsets $ijk$ with $1\leq i\leq 4<j<k\leq 8$ with $j-i\neq4\neq k-i$. The full upper sublattice $\Lambda/ij$ is of type $U^2_4$ if $j\geq5$ and $j-i\neq 4$; otherwise it is regular. For $i\leq 4$, the lattice $\Lambda/i$ is of type $F_7^-$, while for $i\geq5$, it is of type $P_7$. The fundamental diagram $\cL^{(\leq3)}$ is illustrated in \autoref{fig: fundamental diagram of T8}. Because of space limitations, we label the upper sublattices $\Lambda/i$ and $\Lambda/ij$ by their respective bottom elements $i$ and $ij$.

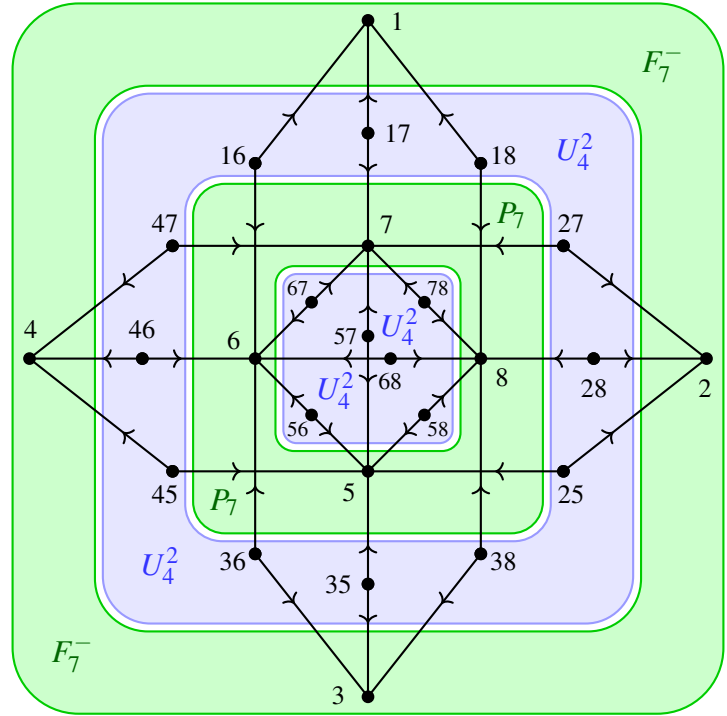
\begin{figure}[tb]
 \[
  \beginpgfgraphicnamed{tikz/fig42}
   \begin{tikzpicture}[x=1.5cm,y=1.5cm, font=\footnotesize, vertices/.style={draw, fill=black, circle, inner sep=0pt},thick,decoration={markings,mark=at position 0.35 with {\arrow{>}}}]
    \draw[fill=green!20!white,draw=green!80!black,rounded corners=25pt] (-3.15,-3.15) rectangle (3.15,3.15);
    \draw[fill=white,draw=green!80!black,rounded corners=20pt] (-2.42,-2.42) rectangle (2.42,2.42);
    \draw[fill=blue!10!white,draw=blue!40!white,rounded corners=18pt] (-2.35,-2.35) rectangle (2.35,2.35);
    \draw[fill=white,draw=blue!40!white,rounded corners=14pt] (-1.62,-1.62) rectangle (1.62,1.62);
    \draw[fill=green!20!white,draw=green!80!black,rounded corners=12pt] (-1.55,-1.55) rectangle (1.55,1.55);
    \draw[fill=white,draw=green!80!black,rounded corners=7pt] (-0.82,-0.82) rectangle (0.82,0.82);
    \draw[fill=blue!10!white,draw=blue!40!white,rounded corners=5pt] (-0.75,-0.75) rectangle (0.75,0.75);
    \node[color=green!40!black] at (45:3.7) {\normalsize $F_7^-$};
    \node[color=blue!80!white] at (45:2.6) {\normalsize $U^2_4$};
    \node[color=green!40!black] at (45:1.8) {\normalsize $P_7$};
    \node[color=blue!80!white] at (45:0.4) {\normalsize $U^2_4$};
    \node[color=green!40!black] at (225:3.7) {\normalsize $F_7^-$};
    \node[color=blue!80!white] at (225:2.6) {\normalsize $U^2_4$};
    \node[color=green!40!black] at (225:1.8) {\normalsize $P_7$};
    \node[color=blue!80!white] at (225:0.4) {\normalsize $U^2_4$};
    \foreach \a in {1,...,4} {
                              \filldraw (180-\a*90:3.0) circle (2pt) coordinate (\a);
                              \node at (175-\a*90:3.0) {$\a$};
                             } 
    \foreach \a in {5,...,8} {
                              \filldraw (-\a*90:1.0) circle (2pt) coordinate (\a);
                              \node at (-8-\a*90:1.2) {$\a$};
                             } 
    \foreach \a in {0,...,23} {\filldraw (120-\a*30:2.0) circle (2pt) coordinate (U\a);} 
    \filldraw (-0.5,-0.5) circle (2pt) coordinate (U56);
    \filldraw ( 0.5,-0.5) circle (2pt) coordinate (U58);
    \filldraw (-0.5, 0.5) circle (2pt) coordinate (U67);
    \filldraw ( 0.5, 0.5) circle (2pt) coordinate (U78);
    \filldraw ( 0  , 0.2) circle (2pt) coordinate (U57);
    \filldraw ( 0.2, 0  ) circle (2pt) coordinate (U68);
    \foreach \a in {1,...,4} {\setcounter{tikz-counter}{3*\a-3};
                              \draw [-,postaction={decorate}] (U\arabic{tikz-counter}) -- (\a);
                              \addtocounter{tikz-counter}{1};
                              \draw [-,postaction={decorate}] (U\arabic{tikz-counter}) -- (\a);
                              \addtocounter{tikz-counter}{1};
                              \draw [-,postaction={decorate}] (U\arabic{tikz-counter}) -- (\a);
                              }
    \foreach \a in {5,...,8} {\setcounter{tikz-counter}{3*\a-10};
                              \draw [-,postaction={decorate}] (U\arabic{tikz-counter}) -- (\a);
                              \addtocounter{tikz-counter}{2};
                              \draw [-,postaction={decorate}] (U\arabic{tikz-counter}) -- (\a);
                              \addtocounter{tikz-counter}{2};
                              \draw [-,postaction={decorate}] (U\arabic{tikz-counter}) -- (\a);
                              }
    \draw [-,postaction={decorate}] (U56) -- (5);
    \draw [-,postaction={decorate}] (U56) -- (6);
    \draw [-,postaction={decorate}] (U57) -- (5);
    \draw [-,postaction={decorate}] (U57) -- (7);
    \draw [-,postaction={decorate}] (U58) -- (5);
    \draw [-,postaction={decorate}] (U58) -- (8);
    \draw [-,postaction={decorate}] (U67) -- (6);
    \draw [-,postaction={decorate}] (U67) -- (7);
    \draw [-,postaction={decorate}] (U68) -- (6);
    \draw [-,postaction={decorate}] (U68) -- (8);
    \draw [-,postaction={decorate}] (U78) -- (7);
    \draw [-,postaction={decorate}] (U78) -- (8);
    \node at ( 45:0.88) {\tiny $78$};
    \node at (135:0.88) {\tiny $67$};
    \node at (225:0.88) {\tiny $56$};
    \node at (315:0.88) {\tiny $58$};
    \node at (135:0.28) {\scriptsize $57$};
    \node at (315:0.28) {\scriptsize $68$};
    \node at ( 1.8, 1.2) {$27$};
    \node at ( 1.8,-1.2) {$25$};
    \node at (-1.8, 1.2) {$47$};
    \node at (-1.8,-1.2) {$45$};
    \node at ( 1.2, 1.8) {$18$};
    \node at ( 1.2,-1.8) {$38$};
    \node at (-1.2, 1.8) {$16$};
    \node at (-1.2,-1.8) {$36$};
    \node at ( 82.5:2.02) {$17$};
    \node at (172.5:2.02) {$46$};
    \node at (262.5:2.02) {$35$};
    \node at (352.5:2.02) {$28$};
   \end{tikzpicture}
  \endpgfgraphicnamed
 \] 
 \caption{The fundamental lattice diagram $\cL^{(\leq3)}$ of $T_8$}
 \label{fig: fundamental diagram of T8} 
\end{figure}

 Since $\cL^{\leq3}_M$ is connected and since, as a ternary matroid, $M$ is without minors of types $U^2_5$,\ $U^3_5$,\ $F_7$ and $F_7^\ast$, \autoref{cor: foundation of a matroid wlum with one component} implies that $F_{M}$ is a symmetry quotient of $\U$.
 
 By \autoref{prop: foundation of the non-Fano matroid}, the foundation of the upper sublattices $\Lambda/i$ of type $F_7^-$ (for $i=1,\dotsc,4$) have foundation $\D$. In conclusion, $F_M$ is isomorphic to a symmetry quotient of $\D$, which is either $\D$ or $\F_3$.
 
 In order to exhibit the defining relation of $\F_3\simeq\past\D{\{x+1\}}$, we consider the chain of lattice inclusions
 \[
  \begin{tikzcd}[column sep=40, row sep=0]
   \Lambda/1 & \Lambda/18 \ar[r] \ar[l] & \Lambda/8 & \Lambda/28 \ar[r] \ar[l] & \Lambda/2,
  \end{tikzcd}
 \]
 which induces a chain of morphisms
 \[
  \begin{tikzcd}[column sep=40, row sep=0]
   \D & \U \ar[r,"\sim"] \ar[l] & \U & \U \ar[l,"\sim"'] \ar[r] & \D
  \end{tikzcd}
 \]
 between the respective foundations. The $3$-circuits of $\Lambda/1$, $\Lambda/8$ and $\Lambda/2$ are depicted in \autoref{fig: circuits of 3 minors of T8}.
 
 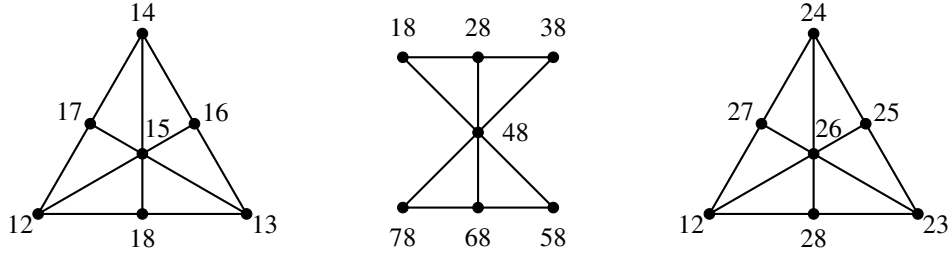
\begin{figure}[htb]
 \[
  \beginpgfgraphicnamed{tikz/fig10}
  \begin{tikzpicture}[x=0.4cm,y=0.4cm]
   \filldraw ( 30:2) circle (2pt);  
   \node at ( 30:2.8) {\footnotesize $16$};  
   \filldraw (150:2) circle (2pt);  
   \node at (150:2.8) {\footnotesize $17$};  
   \filldraw (270:2) circle (2pt);  
   \node at (270:2.8) {\footnotesize $18$};  
   \filldraw ( 90:4) circle (2pt);  
   \node at ( 90:4.7) {\footnotesize $14$};  
   \filldraw (210:4) circle (2pt);  
   \node at (210:4.7) {\footnotesize $12$};  
   \filldraw (330:4) circle (2pt);  
   \node at (330:4.7) {\footnotesize $13$};
   \node at (60:1.0) {\footnotesize $15$};
   \filldraw (330:0) circle (2pt);  
   \draw [thick] ( 90:4) -- (210:4);
   \draw [thick] (210:4) -- (330:4);
   \draw [thick] (90:4) -- (330:4);
   \draw [thick] (150:2) -- (330:4);
   \draw [thick] (270:2) -- (90:4);
   \draw [thick] (210:4) -- (30:2);
  \end{tikzpicture}
  \endpgfgraphicnamed
  \hspace{1.2cm}
  \beginpgfgraphicnamed{tikz/fig14}
   \begin{tikzpicture}[x=1cm, y=1cm]
    \filldraw (0,0) circle (2pt);
    \filldraw (1,0) circle (2pt);
    \filldraw (2,0) circle (2pt);
    \filldraw (1,1) circle (2pt);
    \filldraw (0,2) circle (2pt);
    \filldraw (1,2) circle (2pt);
    \filldraw (2,2) circle (2pt);
    \draw[thick] (0,0) -- (2,0) -- (0,2) -- (2,2) -- (0,0);
    \draw[thick] (1,0) -- (1,2);
    \node at (0,2.4) {\footnotesize $18$};
    \node at (1,2.4) {\footnotesize $28$};
    \node at (2,2.4) {\footnotesize $38$};
    \node at (1.5,1) {\footnotesize $48$};
    \node at (0,-0.4) {\footnotesize $78$};
    \node at (1,-0.4) {\footnotesize $68$};
    \node at (2,-0.4) {\footnotesize $58$};
   \end{tikzpicture}
  \endpgfgraphicnamed
  \hspace{1.2cm}
  \beginpgfgraphicnamed{tikz/fig41}
  \begin{tikzpicture}[x=0.4cm,y=0.4cm]
   \filldraw ( 30:2) circle (2pt);  
   \node at ( 30:2.8) {\footnotesize $25$};  
   \filldraw (150:2) circle (2pt);  
   \node at (150:2.8) {\footnotesize $27$};  
   \filldraw (270:2) circle (2pt);  
   \node at (270:2.8) {\footnotesize $28$};  
   \filldraw ( 90:4) circle (2pt);  
   \node at ( 90:4.7) {\footnotesize $24$};  
   \filldraw (210:4) circle (2pt);  
   \node at (210:4.7) {\footnotesize $12$};  
   \filldraw (330:4) circle (2pt);  
   \node at (330:4.7) {\footnotesize $23$};
   \node at (60:1.0) {\footnotesize $26$};
   \filldraw (330:0) circle (2pt);  
   \draw [thick] ( 90:4) -- (210:4);
   \draw [thick] (210:4) -- (330:4);
   \draw [thick] (90:4) -- (330:4);
   \draw [thick] (150:2) -- (330:4);
   \draw [thick] (270:2) -- (90:4);
   \draw [thick] (210:4) -- (30:2);
  \end{tikzpicture}
  \endpgfgraphicnamed
 \]
 \caption{The $3$-circuits of $T_8/1$, \ \ $T_8/8$ and $T_8/2$}
 \label{fig: circuits of 3 minors of T8} 
 \end{figure}

 Let $\alpha_i:\D\to F_{\Lambda/i}\to F_M$ (for $i=1,2$) be the composition of the unique isomorphisms $\D\to F_{\Lambda/i}$ (cf.\ \autoref{lemma: cross ratios in F7-}) with the morphism induced by the lattice inclusion $\Lambda/i\hookrightarrow\Lambda$. Let $x\in\D$ be the unique element with $x+x-1\in N_\D$. Then 
 \[
  \alpha_1(x) \ \ \underset{\Lambda/1}= \ \ \cross{1458}{178}{168}{1238}{} \ \ \underset{\Lambda/8}= \ \ \cross{258}{278}{2468}{1238}{} \ \ \underset{\Lambda/2}= \ \ \alpha_2(-1),
 \]
 where we label the equalities with the upper sublattice $\Lambda/i$ to which we apply one of \autoref{lemma: cross ratios in F7-} and \autoref{lemma: cross ratios in P7}, depending on the type of $\Lambda/i$. This equation shows that $x$ and $-1$ are identified in $F_M$, and thus $F_M\simeq\past\D{\genn{x+1}}\simeq\F_3$, as claimed.
\end{proof}
 
\begin{rem}
 Note that we can replace the last step in the proof of \autoref{prop: foundation of T8}, which exhibits the relation $\alpha_1(x)=\alpha_2(-1)$, by the fact that $T_8$ is not representable over any field of characteristic different from $3$ by \cite[p.\ 649]{Oxley92}, which rules out the possibility that its foundation is $\D$.
\end{rem}


\section{The structure theorem for matroids without large uniform minors revisited}
\label{subsection: the structure theorem for matroids without large uniform minors revisited}

A matroid is \emph{without large uniform minors} if it does not have any minors of type $U^2_5$ or $U^3_5$. A central result of the first two author's paper \cite{Baker-Lorscheid20} is that the foundation of a matroid without large uniform minors decomposes into a tensor product of pastures that are isomorphic to $\F_2$, $\F_3$, $\U$, $\D$ and $\H$. The previous computations show that every such tensor product occurs as the foundation of a matroid without large uniform minors.


\begin{thm}\label{thm: structure theorem for matroids without large uniform minors - revisited}
 The isomorphism classes of the foundations of matroids without large uniform minors are represented by all pastures of the form $F_1\otimes\dotsc\otimes F_r$ for some $r\geq 0$ and $F_i\in\{\F_2,\ \F_3,\ \U,\ \D,\ \H\}$.
\end{thm}

\begin{proof}
 By \cite[Thm.\ 5.9]{Baker-Lorscheid20}, every foundation of a matroid without large uniform minor is of the described form. Thus we are left with showing that every pasture of the form $F_1\otimes\dotsc\otimes F_r$ with $F_i\in\{\F_2,\ \F_3,\ \U,\ \D,\ \H\}$ is isomorphic to the foundation of a matroid without large uniform minors. By \autoref{thm: foundations of direct sums}, it suffices to show that each of $F_i\in\{\F_2,\ \F_3,\ \U,\ \D,\ \H\}$ appears as such a foundation. This is indeed the case: the foundation of $F_7$ is $\F_2$ (\autoref{subsection: foundations of binary and regular matroids}), the foundation of $T_8$ is $\F_3$ (\autoref{prop: foundation of T8}), the foundation of $U^2_4$ is $\U$ (\autoref{subsection: foundations of the uniform matroid U24}), the foundation of $F_7^-$ is $\D$ (\autoref{prop: foundation of the non-Fano matroid}), and the foundation of $AG(2,3)\setminus e$ is $\H$ (\autoref{prop: foundation of AG23-e}).
\end{proof}

\appendix

\section{Some interesting foundations}
\label{appendix: some interesting foundations}

Not every pasture appears as the foundation of a matroid. A concrete example is the pasture $P=\pastgenn{\Funpm}{x}{x^4+x-1}$. On the one hand, there is no morphism from $\V$ to $P$, and hence a matroid with a large uniform minor (i.e., a minor isomorphic to $U^2_5$ or $U^3_5$) cannot have foundation $P$. But, on the other hand, it is easy to check that $P$ is not a tensor product of copies of $\U$, $\D$, $\H$, $\F_3$, and $\F_2$, so by \autoref{thm: structure theorem for foundations of matroids without large uniform minors} $P$ cannot be the foundation of a matroid without large uniform minors. In general, it is a wide open problem to characterize which pastures are foundations. 

In this appendix, we list some foundations, many of which we found with the help of the Macaulay2 package \textsc{Pastures} developed by Chen and the third author; cf.\ \cite{Chen-Zhang}. Since the foundation of a direct sum and of a $2$-sum of two matroids is the tensor product of the foundations of the summands (cf.\ \autoref{thm: foundations of direct sums} and \cite{Baker-Lorscheid-Walsh-Zhang}), we concentrate on the description of foundations of $3$-connected matroids, which can be thought of as the building blocks for all foundations.

The following list contains descriptions of all small foundations of $3$-connected matroids on up to $8$ elements, where ``small'' means that the foundation has at most $7$ hexagons, along with various other examples of interest. We discuss notable properties of these foundations, in particular whether they admit morphisms into a field, into the sign hyperfield $\S$, or into the tropical hyperfield $\T$. We call a pasture \emph{rigid} if every morphism to $\T$ factors through $\K$, i.e.\ its image is contained in $\{0,1\}$. By \cite[Prop.~B.1]{Baker-Lorscheid23}, a matroid is rigid if and only if its foundation is rigid. We gather the information about representability in \autoref{table: morphisms of foundations into other pastures}. 

We freely use Oxley's notation from \cite{Oxley92} throughout. Where we lack a better description of a matroid $M$, we list its \emph{short circuits}, which are all circuits of size less than or equal to the rank $r$ of $M$. Note that the whole circuit set $\cC_M$ of $M$ can be recovered from the subset $\cC_M^{\rank}$ of short circuits by adding all $(r+1)$-subsets that do not contain a short circuit. 


Some of the examples of foundations that we mention below are too large for a complete description to be meaningful. In these cases, we restrict ourselves to a \emph{numerical description} of the foundation $F_M$ which mentions: 
\begin{itemize}
 \item the \emph{rank of $F_M$}, which is the free rank of the unit group $F_M^\times$;
 \item the \emph{torsion of $F_M$}, which is the torsion subgroup of $F_M^\times$;
 \item whether $-1=1$ or not; note that in many cases the torsion of $F_M$ is generated by $-1$, in which case it equals $\{1\}$ or $\{1,\ -1\}$;
 \item the number of \emph{hexagons} of $F_M$.
\end{itemize} 

\subsection*{Some open problems}

Before discussing the list of foundations, we would like to mention a few interesting problems:

\begin{problem}
 For a pasture $F$, let $\cC_F^\min$ be the class of minimal isomorphism types of matroids with foundation $F$. For which $F$ is $\cC_F^\min$ finite?
\end{problem}

Note that this problem is intimately related to Rota's conjecture, which can be separated into two parts: 
\begin{enumerate}
 \item Is there a finite list $\cL(q)$ of \emph{excluded foundations} for every finite field $\F_q$? (More precisely, we're asking here for a list with the following property: if $M$ is a matroid with foundation $F_M$ such that no foundation $F\in\cL(q)$ maps to $F_M$, then there is a morphism $F_M\to\F_q$.)
 \item Is $\cC_F^\min$ finite for every finite field $\F_q$ and $F\in\cL(q)$?
\end{enumerate}
We do not know the answer to either part.

Another highly interesting task is to develop a better understanding of which pastures appear as the foundation of a matroid. Results in this direction can be used as tools to study matroid representations. Thus we formulate the task:

\begin{problem}
 Find (easily verifiable) criteria that imply that a given pasture is a foundation or that it is \emph{not} a foundation.
\end{problem}

\subsection{Uniform foundations}
We call the foundation of a uniform matroid $U^r_n$ a \emph{uniform foundation} and denote it by $\U^r_n$. 

The first example of a uniform foundation is the regular partial field $\Funpm$, which is the foundation of every uniform matroid $U^r_n$ with $r\in\{0,1,n-1,n\}$. The unique minimal matroid with foundation $\Funpm$ is the trivial (empty) matroid.

The second (class of) examples of uniform matroids are the $k$-regular partial fields $\U_k=\U^2_{k+3}=\U^{k+1}_{k+3}$ for $k\geq1$, which appear as the foundation of $U^2_{k+3}$ and $U^{k+1}_{k+3}$, and of no other uniform matroid. These two matroids are minimal matroids with this foundation. In particular, we have $\U_1=\U$ and $\U_2=\V$. The unique minimal matroid with foundation $\U$ is $U^2_4$, and the unique minimal matroids with foundation $\V$ are $U^2_5$ and $U^3_5$. For $k\geq 3$, the regular partial field $\U_k$ is the foundation of minimal matroids that are not uniform; 
e.g.\ $P_6$ is a minimal matroid with foundation $\U_3$.

The smallest example of a uniform foundation that is not a $k$-regular partial field is $\U^3_6$, which has rank $14$, whose torsion group is generated by $-1\neq1$, and which has $30$ hexagons.

In general, the unit group of the uniform matroid $\U^r_n$ is isomorphic to $\Z/2\Z\times\Z^{\binom rn-n}$ by \cite[Thm.~8.1]{Dress-Wenzel89} and \cite[Cor.~7.13]{Baker-Lorscheid21b}. This shows that the rank of uniform foundations grows very quickly with the size of the ground set. Note that the rank of every non-uniform foundation of a rank $r$ matroid on $n$ elements is strictly smaller than $\binom rn-n$.

The foundations of all $3$-connected matroids on up to $6$ elements are uniform. The first examples of non-uniform foundations appear for $3$-connected matroids on $7$ elements, cf.\ \autoref{subsection: foundations of matroids on 7 elements}.

\subsection{Finite fields}\label{Ffinitefields}
By \autoref{prop: foundation of projective spaces}, every finite field $\F_q$ is a foundation, namely of the projective geometry $PG(d,q)$ of any dimension $d\geq2$.

In the case of $\F_2$, the Fano plane $F_7=PG(2,2)$ and its dual $F_7^\ast$ are the unique minimal matroids with foundation $\F_2$. This follows from the fact that a matroid is binary if and only if it does not contain a $U^2_4$-minor, which is equivalent with its foundation being equal to $\Funpm$ or $\F_2$; coupling this with \autoref{thm: fundamental presentation} yields our claim.

Since $PG(2,q)$ is a proper minor of $PG(d,q)$ for $d\geq3$, none of the higher dimensional projective spaces is minimal with foundation $\F_q$. Since $\F_q$ is the universal partial field of the extended Dowling geometry for $\F_q^\times$, which is a proper minor of $PG(2,q)$ for $q>2$, we conjecture that $F_7=PG(2,2)$ is the only (Desarguesian) projective plane that is minimal for its foundation.

Minimal matroids with foundation $\F_3$ are $T_8$ and $R_9$. Minimal size matroids for the foundations $\F_4$, $\F_5$, $\F_7$ and $\F_8$ all have $9$ elements. Two minimal matroids with foundation $\F_4$ are those represented by the matrices
{\footnotesize
\[
 \begin{pmatrix}
  1 & 0 & 0 & 1 & 1 & 1 & 0 & 0 & 1 & b \\
  0 & 1 & 0 & 1 & 1 & a & 1 & 1 & 0 & 0 \\
  0 & 0 & 1 & 1 & 0 & 0 & 1 & a & 1 & 1 
 \end{pmatrix}
 \qquad \text{\normalsize and} \qquad
 \begin{pmatrix}
  1 & 0 & 0 & 0 & a & b & 1 & 0 & 0 \\
  0 & 1 & 0 & 0 & 1 & a & 0 & a & 1 \\
  0 & 0 & 1 & 0 & 1 & 0 & b & 1 & 0 \\
  0 & 0 & 0 & 1 & 0 & 1 & 1 & 1 & 1 
 \end{pmatrix}
\]  
}%
over $\F_4=\{0,1,a,b\}$. Two minimal matroids with foundation $\F_5$ are those represented by the matrices
{\footnotesize
\[
 \left(
 \begin{array}{cccccccccccccccc}
  1 & 0 & 0 & 1 & 1 & 1 & 1 & 0 & 0 & 1 & 2 \\
  0 & 1 & 0 & 1 & 1 & 2 & 3 & 1 & 1 & 0 & 0 \\
  0 & 0 & 1 & 1 & 0 & 0 & 0 & 1 & 2 & 1 & 1 
 \end{array}
 \right)
 \quad \text{\normalsize and} \quad
 \begin{pmatrix}
  1 & 0 & 0 & 0 & 4 & 3 & 1 & 0 & 0 \\
  0 & 1 & 0 & 0 & 3 & 1 & 0 & 1 & 1 \\
  0 & 0 & 1 & 0 & 1 & 0 & 4 & 1 & 1 \\
  0 & 0 & 0 & 1 & 0 & 1 & 1 & 1 & 0
 \end{pmatrix}
\]
}%
over $\F_5$. A minimal size matroid with foundation $\F_7$ is represented by the matrix
{\footnotesize
\[
 \begin{pmatrix}
  1 & 0 & 0 & 0 & 0 & 0 & 1 & 1 & 1 \\
  0 & 1 & 0 & 0 & 1 & 1 & 0 & 1 & 1 \\
  0 & 0 & 1 & 0 & 1 & 2 & 2 & 4 & 2 \\
  0 & 0 & 0 & 1 & 0 & 1 & 2 & 2 & 3
 \end{pmatrix}
\]
}%
over $\F_7$. A minimal size matroid with foundation $\F_8$ is represented by the matrix
{\footnotesize
\[
 \begin{pmatrix}
  1 & 0 & 0 & 0 & 0 & 0 & 1       & 1     & 1       \\
  0 & 1 & 0 & 0 & 1 & 1 & 1       & a+1   & 1       \\
  0 & 0 & 1 & 0 & 1 & a & a^2+1   & a^2+1 & a^2+a+1 \\
  0 & 0 & 0 & 1 & 0 & 1 & a^2+a+1 & a     & 1
 \end{pmatrix}
\]
}%
over $\F_8=\F_2(a)$ where $a\in\F_8$ is a primitive element over $\F_2$ with minimal polynomial $T^3+T+1$.

Let $p$ be a prime number. Then \cite[Thm.\ 4.1.1]{Silins24} determines a rank $3$ matroid whose universal partial field is $\F_p$ in terms of an explicit matrix representation. Namely, let $l=\lfloor{\log_2(p+1)}\rfloor$ and $b_i=\lfloor(p+1)/2^{l-i+1}\rfloor$ for $i=1,\dotsc,l$. Then the matroid represented by the matrix
{\footnotesize
\[
 \begin{pmatrix}
  1 & 0 & 0 & 1 & 1 & 1 & 0   & 1 \ \quad \dotsc \quad \ 0 & 1   \\
  0 & 1 & 0 & 1 & 1 & 0 & 1   & 2 \ \quad \dotsc \quad \ 1 & 2   \\
  0 & 0 & 1 & 1 & 0 & 1 & b_1 & b_1 \quad \dotsc \quad b_l & b_l 
 \end{pmatrix}
\]
}%
over $\F_p$ has universal partial field $\F_p$. We do not know at the time of writing if the foundation of this matroid is equal to $\F_p$ for all primes $p$.

\subsection{Some small foundations}
\label{subsection: some small foundations}

In the following, we describe a series of ``small'' foundations of $3$-connected matroids. We searched exhaustively among all matroids on up to $8$ elements and list all foundations with up to $7$ hexagons (for $3$-connected matroids on up to $8$ elements). There are several examples of foundations of larger matroids. We order these examples by $(h,r)$ (as indicated in the section headers) where $h$ is the number of hexagons and $r$ is the rank of the foundation.

\subsubsection{(1,0)} \label{FKrasner}
The \emph{Krasner hyperfield} $\K=\F_2\otimes\F_3$ appears as the foundation of several matroids on $9$ elements (we found at least $44$ such matroids). There are no matroids with foundation $\K$ that have less than $9$ elements. A concrete example of a matroid on $10$ elements with foundation $\K$ is the modular sum of $F_7$ with $R_9$ along a common $3$-circuit as indicated in \autoref{fig: modular sum of F7 and R9} (where the circled vertices of the generalized parallel connection are deleted). 

\begin{figure}[htb] 
\[
 \beginpgfgraphicnamed{tikz/fig49}
 \begin{tikzpicture}[x=0.08cm,y=0.07cm]
  \draw ( 0,-6) -- ( 0,54) -- (54,41.5) -- (54,-18.5) --cycle;
  \draw ( 0,-6) -- ( 0,54) -- (-54,41.5) -- (-54,-18.5) --cycle;
  \draw[thick] ( 0, 0) -- ( 0,48);
  \draw[thick] ( 0, 0) -- (16,36);
  \draw[thick] ( 0, 0) -- (24,30);
  \draw[thick] ( 0, 0) -- (48,12);
  \draw[thick] ( 0,24) -- (24,30);
  \draw[thick] ( 0,24) -- (48,12);
  \draw[thick] ( 0,48) -- (24, 6);
  \draw[thick] ( 0,48) -- (48,12);
  \draw[thick] ( 0, 0) -- (-24,30);
  \draw[thick] ( 0, 0) -- (-48,12);
  \draw[thick] ( 0,24) -- (-48,12);
  \draw[thick] ( 0,48) -- (-24, 6);
  \draw[thick] ( 0,48) -- (-48,12);
  \draw[thick] plot [smooth] coordinates { (16,36) ( 0,24) (24, 6) };
  \draw[thick] plot [smooth] coordinates { (-24,30) ( 0,24) (-24, 6) };
  \draw[fill=white] ( 0, 0) circle (4pt);
  \draw[fill=white] ( 0,24) circle (4pt);
  \draw[fill=white] ( 0,48) circle (4pt);
  \filldraw (12,27) circle (3pt);
  \filldraw (16,20) circle (3pt);
  \filldraw (16,36) circle (3pt);
  \filldraw (24, 6) circle (3pt);
  \filldraw (24,30) circle (3pt);
  \filldraw (48,12) circle (3pt);
  \filldraw (-16,20) circle (3pt);
  \filldraw (-24, 6) circle (3pt);
  \filldraw (-24,30) circle (3pt);
  \filldraw (-48,12) circle (3pt);
 \end{tikzpicture}
 \endpgfgraphicnamed
 \]
 \caption{A modular sum of $F_7$ with $R_9$} 
 \label{fig: modular sum of F7 and R9}
\end{figure}
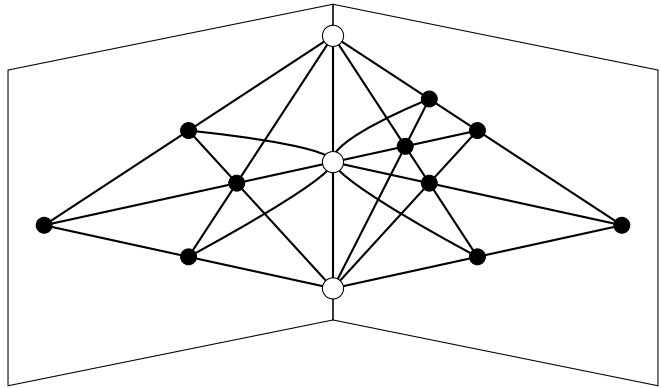

\subsubsection{(1,0)} \label{Fsign}
The sign hyperfield $\S=\past\Funpm{\genn{1+1-1}}$ appears as the foundation of the rank $4$ matroid on $9$ elements with short circuits
\begin{align*}
 & 
1234, \qquad 
1256, \qquad 
1357, \qquad 
1369, \qquad 
1468, \qquad 
1589, \qquad 
2358, \qquad 
 \\ & 
2379, \qquad 
2457, \qquad 
2469, \qquad 
2678, \qquad 
3456, \qquad 
4789, \qquad 
5679.
\end{align*}
This matroid is $3$-connected and of minimal size for foundation $\S$. The sign hyperfield is orientable and rigid, but not representable over any field.

Note that this implies that also the weak sign hyperfield 
\[
 \W \ = \ \past{\Funpm}{\genn{1+1+1,\ 1+1-1}} \ = \ \F_3\otimes\S
\]
(cf.~\cite{Baker-Bowler19,Bland-Jensen87,Wagowski89}) is a foundation, e.g.\ for the direct sum of the above matroid with $PG(2,3)$. We do not know if the weak sign hyperfield appears as the foundation of a $3$-connected matroid.

\subsubsection{(1,0)} \label{Fhexagonal}
The \emph{hexagonal partial field} 
\[
 \H \ = \ \pastgenn{\Funpm}{\zeta_6}{\zeta_6^3+1,\ \zeta_6+\zeta_6^{-1}-1}
\]
is of importance for its role in the structure theorem for foundations of matroids without large uniform minors (\autoref{thm: structure theorem for matroids without large uniform minors - revisited}). It also appears under the name of the \emph{sixth-root-of-unity partial field} (or \emph{$\sqrt[6]{1}$ partial field}), see, e.g., \cite{Pendavingh-vanZwam10a}. The smallest matroids with foundation $\H$ have $8$ elements, and there are $3$ of them: $AG(2,3)\setminus e$, its dual, and the self-dual rank $4$ matroid with short circuits
\begin{align*}
 & 
 1234, \qquad 
 1235, \qquad 
 1236, \qquad 
 1237, \qquad 
 1238, \qquad 
 1245, \qquad 
 1345,
 \\ & 
 1467, \qquad 
 1578, \qquad 
 2345, \qquad 
 2468, \qquad 
 2567, \qquad 
 3478, \qquad 
 3568.
\end{align*}
The hexagonal partial field is representable (in all fields with a solution to $x^2-x+1=0$) and rigid, but not orientable.



\subsubsection{(1,1)} \label{Fdyadic}
The \emph{dyadic partial field} 
\[
 \D \ = \ \pastgenn\Funpm{x}{x-1-1}
\]
is another building block for foundations of matroids without large uniform minors (\autoref{thm: structure theorem for matroids without large uniform minors - revisited}). The smallest matroids with foundation $\D$ are the non-Fano matroid $F_7^-$ and its dual. There is one minimal matroid for $\D$ with $8$ elements, which is the self-dual matroid $P_8$. All other minimal matroids with foundation $\D$ have more than $8$ elements.

Interestingly, certain tensor products of $\D$ with other pastures appear as the foundation of $3$-connected matroids. For instance, $\D\otimes\F_2=\pastgenn{\F_2}{x}{x+x+1}$ is the foundation of $AG(3,2)'$, and $\D\otimes\D$ is the foundation the rank $4$ spike on $8$ elements with short circuits
\begin{align*}
 &
 1256, \qquad 
 1278, \qquad 
 1357, \qquad 
 1368, \qquad 
 1458, \qquad 
 1467, 
 \\ & 
 2367, \qquad 
 2468, \qquad 
 2358, \qquad 
 2457, \qquad 
 3456, \qquad 
 3478. 
\end{align*}

The dyadic partial field is the only pasture that we know which is not a quotient of $\Funpm$ but appears as a non-trivial tensor factor of the foundation of a $3$-connected matroid. 

In so far, we wonder (cf.\ the discussion about indecomposable foundations in \autoref{subsection: indecomposable foundations}):

\begin{problem}
 Let $M$ be a $3$-connected matroid whose foundation $F_M\simeq P_1\otimes P_2$ is the tensor product of two pastures $P_1$ and $P_2$. Is then one of $P_1$ and $P_2$ necessarily a tensor product of copies of $\F_2$, $\F_3$, $\S$, and/or $\D$? 
\end{problem}

\subsubsection{(1,1)} \label{Fgoldenratio}
The \emph{golden ratio partial field} 
\[
 \G \ = \ \pastgenn{\Funpm}{x}{x^2+x-1}
\] 
plays a prominent role in the representation theory of matroids; cf.\ \cite{Pendavingh-vanZwam10a} and \cite{Pendavingh-vanZwam10b}. It is the quotient of $\V=\pastgenn{\Funpm}{x_1,\dotsc,x_5}{x_i+x_{i-1}x_{i+1}-1}$ obtained by identifying all generators $x_1\sim x_2\sim x_3\sim x_4\sim x_5$. The smallest matroid with foundation $\G$ is the self-dual rank $4$ matroid on $8$ elements with short circuits
\[
 1234, \quad
 1256, \quad
 1357, \quad
 1368, \quad
 1478, \quad
 2358, \quad
 2457, \quad
 2678, \quad 
 3456,
\]
and this is the only matroid with foundation $\G$ that has less than $9$ elements. A minimal matroid for foundation $\G$ with $9$ elements is the minor $B_{11}\setminus23$ of the Betsy Ross matroid $B_{11}$ (cf.~\cite[Figure 3.3]{vanZwam09}), where $123$ is a $3$-circuit and $1$ is the ``center'' of $M$. Also the Betsy Ross matroid itself has foundation $\G$.

The golden ratio partial field is representable (over all fields with a solution to $x^2+x=1$), orientable, and rigid.

\subsubsection{(2,0)} \label{F20a}

The pasture
\[
 \FF{F20a} \ = \ \pastgenn\S{x}{x^2+1,\ x+1-1}
\]
appears as the foundation of the $3$-connected rank $4$ matroid on $9$ elements with short circuits
\begin{align*}
 & 
 1234, \qquad 
 1256, \qquad 
 1289, \qquad 
 1357, \qquad 
 1368, \qquad 
 1469, \qquad 
 1478, \qquad 
 \\ & 
 2358, \qquad 
 2379, \qquad 
 2457, \qquad 
 2678, \qquad 
 3456, \qquad 
 3489, \qquad 
 5689,
\end{align*}
which is a matroid of minimal size for $\FF{F20a}$. The foundation $\FF{F20a}$ is rigid, but not orientable nor representable (over any field).

\subsubsection{(2,1)} \label{F21a}
The pasture
\[
 \FF{F21a} \ = \ \pastgenn\Funpm{x}{x^3-x-1,\ x^5-x^4-1}
\]
appears as the foundation of the $3$-connected rank $4$ matroid on $9$ elements with short circuits
\begin{align*}
 & 
 1234, \qquad 
 1235, \qquad 
 1245, \qquad 
 1267, \qquad 
 1345, \qquad 
 1368, \qquad 
 1478, \qquad 
 1579,
 \\ & 
 2345, \qquad 
 2369, \qquad 
 2489, \qquad 
 2568, \qquad 
 3467, \qquad 
 3789, \qquad 
 4569.
\end{align*}
which is a matroid of minimal size for $\FF{F21a}$. Since 
\[
 (x^2-x+1)\cdot(x^3-x-1) \ = \ x^5-x^4-1, 
\]
the foundation $\FF{F21a}$ injects into its the universal ring $R=\Z[x]/\gen{x^3-x-1}$ (note that $x\cdot(x^2-1)=1$, so $x$ is a unit of $R$). Solving the $S$-unit equation $a+b=1$ in $R$ (with SageMath) exhibits $(x^3,x)$ and $(x^5,-x^4)$ as its only solutions, up to $S_3$-conjugates. Therefore $\FF{F21a}$ is a partial field, and is thus equal to its universal partial field. The foundation $\FF{F21a}$ is representable (over all fields with a root of $T^3-T-1$), orientable, and rigid. 


The pastures \ref{F21a} and $\G$ are the only two examples we currently know of infinite rigid pastures which occur as foundations of $3$-connected matroids.

\begin{problem}
Are there infinitely many different infinite rigid pastures?
\end{problem}


\subsubsection{(2,1)} \label{FGaussian} 
The \emph{Gaussian partial field} 
\[
 \H_2 \ = \ \pastgenn{\Funpm}{i,x}{i^2+1, \quad x-i-1, \quad x^2-i-i}
\]
(cf.\ \cite[p.~543]{Pendavingh-vanZwam10a}) is the foundation of the $3$-connected rank $4$ matroid on $9$ elements with short circuits
\begin{align*}
 &
 1234, \quad 
 1267, \quad 
 1235, \quad 
 1245, \quad 
 1345, \quad 
 1368, \quad 
 1379, \quad 
 1489, \quad 
 \\ &
 2345, \quad 
 2469, \quad 
 2568, \quad 
 2789, \quad 
 3467, \quad 
 3589, \quad 
 4578, \quad 
 5679, \quad 
\end{align*}
which is minimal for this foundation. It embeds into its universal ring $\Z[i,1/2]$ via $i\mapsto i$ and $x\mapsto i+1$. The partial field $\H_2$ is representable (over all fields of characteristic $\neq2$ with a square root of $-1$), but it is neither orientable nor rigid.


\subsubsection{(2,1)} \label{FDowling}
The Dowling lift of $\F_4$
\[
 \FF{FDowling} \ = \ \pastgenn\H{x}{x-\zeta_6-1} 
\]
(cf.\ \cite[p.~543]{Pendavingh-vanZwam10a}) appears as the foundation of the $3$-connected rank $4$ matroid on $9$ elements with short circuits
\begin{align*}
 & 
 1234, \quad 1235, \quad 1245, \quad 1267, \quad 1268, \quad 1278, \quad 1345, \quad 1389, \quad 1469, \quad 1579, 
 \\ & 
 1678, \quad 2345, \quad 2379, \quad 2489, \quad 2569, \quad 2678, \quad 3467, \quad 3568, \quad 4578,
\end{align*}
which is a matroid of minimal size for $\FF{FDowling}$. It embeds into its universal ring $\Z[\zeta_6,\ (1+\zeta_6)^{-1}]$ via $\zeta_6\mapsto\zeta_6$ and $x\mapsto1+\zeta_6$ and is equal to its universal partial field. The Dowling lift of $\F_4$ is representable (over all fields with a primitive 3rd root of unity). It is neither orientable nor rigid.


\subsubsection{(2,1)} \label{F21d}
The pasture
\[
 \FF{F21d} \ = \ \pastgenn{\F_4}{x}{x+a+1}, 
\]
where $\F_4=\{0,1,a,b\}$, appears as the foundation of the $3$-connected rank $4$ matroid on $9$ elements with short circuits
\begin{align*}
 & 
 \ \ 123, \quad \ \ 
 1245, \quad \ \ 
 1267, \quad \ \ 
 1345, \quad \ \ 
 1367, \quad \ \ 
 1468, \quad \ \ 
 1789, \quad \ \ 
 2345,
 \\ & 
 2367, \quad \ \ 
 2469, \quad \ \ 
 2568, \quad \ \ 
 2579, \quad \ \ 
 3478, \quad \ \ 
 3689, \quad \ \ 
 4567, \quad \ \ 
 4589,
\end{align*}
which is a matroid of minimal size for $\FF{F21d}$. It surjects onto its universal ring $\F_4=\F_4[x^{\pm 1}]/\gen{x+a+1}$ (with $x\mapsto a+1$), which is equal to the universal partial field of $\FF{F21d}$. The foundation $\FF{F21d}$ is representable (over field extensions of $\F_4$), but it is neither orientable nor rigid.


\subsubsection{(2,1)} \label{F21e} 
The pasture 
\[
 \FF{F21e} \ = \ \pastgenn{\F_3}{x}{x+1-1} \ = \ \F_3\otimes\Big(\pastgenn\Funpm{x}{x+1-1}\Big)
\]
is the foundation of the $3$-connected rank $4$ matroid on $8$ elements with short circuits
\begin{align*}
 &
 1234, \quad 
 1256, \quad 
 1357, \quad 
 1368, \quad 
 1458, \quad 
 2358, \quad 
 2367, \quad 
 2457, \quad 
 3456, \quad 
 4678,
 \end{align*}
which is minimal for this foundation. It is neither representable, nor orientable, nor rigid.

\subsubsection{(2,1)} \label{F21f}
The pasture
\[
 \FF{F21f} \ = \ \pastgenn\S{x}{x+1-1} \ = \ \S\otimes\Big(\pastgenn\Funpm{x}{x+1-1}\Big)
\]
appears as the foundation of the $3$-connected rank $4$ matroid on $9$ elements with short circuits
\begin{align*}
 & 
 1234, \qquad 
 1256, \qquad 
 1279, \qquad 
 1357, \qquad 
 1389, \qquad 
 1459, \qquad 
 1468, \qquad 
 \\ & 
 2358, \qquad 
 2367, \qquad 
 2457, \qquad 
 2489, \qquad 
 3456, \qquad 
 3479, \qquad 
 6789.
\end{align*}
which is a matroid of minimal size for $\FF{F21f}$. The foundation $\FF{F21f}$ is orientable, but is neither representable nor rigid.

\subsubsection{(2,1)} \label{F21g}
The pasture
\[
 \FF{F21g} \ = \ \pastgenn\K{x}{x+1+1} \ = \ \K\otimes\Big(\pastgenn\Funpm{x}{x+1-1}\Big)
\]
appears as the foundation of the $3$-connected rank $4$ matroid on $9$ elements with short circuits
\begin{align*}
 & 
 1234, \quad \ \ 
 1256, \quad \ \ 
 1289, \quad \ \ 
 1357, \quad \ \ 
 1468, \quad \ \ 
 1479, \quad \ \ 
 2358, \quad \ \ 
 2379,
 \\ & 
 2457, \quad \ \ 
 2469, \quad \ \ 
 2678, \quad \ \ 
 3456, \quad \ \ 
 3689, \quad \ \ 
 4589, \quad \ \ 
 5679.
\end{align*}
which is a matroid of minimal size for $\FF{F21g}$. The foundation $\FF{F21g}$ is neither representable, nor orientable, nor rigid.

\subsubsection{(2,1)} \label{F21h}
The pasture
\[
 \FF{F21h} \ = \ \pastgenn{\F_2}{x}{x+x+1,\ x^2+x^2+1}
\]
appears as the foundation of the $3$-connected rank $4$ matroid on $9$ elements with short circuits
\begin{align*}
 & 
 1234,\ \quad 
 1235,\ \quad 
 1245,\ \quad 
 1267,\ \quad 
 1345,\ \quad 
 1368,\ \quad 
 1479,\ \quad 
 1589,\ \quad 
 2345,
 \\ & 
 2378,\ \quad 
 2469,\ \quad 
 2568,\ \quad 
 2579,\ \quad 
 3467,\ \quad 
 3489,\ \quad 
 3569,\ \quad 
 4578,\ \quad 
 6789.
\end{align*}
which is a matroid of minimal size for $\FF{F21h}$. The foundation $\FF{F21h}$ is neither representable, nor orientable, nor rigid.

\subsubsection{(2,2)} \label{F22a} 
The pasture 
\[
 \FF{F22a} \ = \ \pastgenn{\F_2}{x,y}{x+y+1, \quad x^2+y^2+1}
\]
is the foundation of the matroid $F_7^+$ (cf.\ \cite[p.~91]{vanZwam09}), which can be depicted as
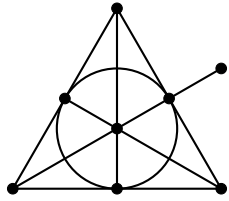
\[
 \beginpgfgraphicnamed{tikz/fig47}
 \begin{tikzpicture}[x=0.4cm,y=0.4cm]
  \filldraw (  0:0) circle (2pt);  
  \filldraw ( 30:2) circle (2pt);  
  \filldraw ( 30:4) circle (2pt);  
  \filldraw (150:2) circle (2pt);  
  \filldraw (270:2) circle (2pt);  
  \filldraw ( 90:4) circle (2pt);  
  \filldraw (210:4) circle (2pt);  
  \filldraw (330:4) circle (2pt);  
  \draw [thick] (0:0) circle (2);  
  \draw [thick] ( 90:4) -- (210:4);
  \draw [thick] (210:4) -- (330:4);
  \draw [thick] (90:4) -- (330:4);
  \draw [thick] (90:4) -- (270:2);
  \draw [thick] (210:4) -- (30:4);
  \draw [thick] (330:4) -- (150:2);
 \end{tikzpicture}
 \endpgfgraphicnamed
\]
Another minimal $3$-connected matroid for this foundation is represented over $\F_4=\{0,1,a,b\}$ by the matrix
{\footnotesize
\[
 \begin{pmatrix}
  1 & 0 & 0 & 0 & b & a & 1 & 1 \\
  0 & 1 & 0 & 0 & a & 1 & 0 & 1 \\
  0 & 0 & 1 & 0 & 1 & 0 & a & 1 \\
  0 & 0 & 0 & 1 & 0 & 1 & 1 & 1 \\
 \end{pmatrix}
\]
}
Note that $-1=1$ in $\FF{F22a}$ even though this matroid has no minors of type $F_7$ and $F_7^\ast$. 

The foundation $\FF{F22a}$ embeds into its universal ring $\F_2[x^{\pm1},y^{\pm1}]/(x+y-1)$ via the tautological map $x\mapsto x$ and $y\mapsto y$ (note that $x^2+y^2-1\in(x+y-1)$), and its universal partial field is $\pastgenn{\F_2}{x,y}{x^{2^k}+y^{2^k}-1\mid k\geq0}$ (which recovers \cite[Thm.~3.3.27.$($iv$)$]{vanZwam09}). The foundation $\FF{F22a}$ is representable (in all fields of characteristic $2$ with at least $4$ elements), but is neither orientable nor rigid.

\subsubsection{(3,2)} \label{F32a}  
The pasture 
\[
 \FF{F32a} \ = \ \pastgenn{\F_2}{x,y}{x+1+1, \quad y+1+1, \quad xy+1+1}
\]
is the foundation of (the $3$-connected rank $4$ matroid) $F_8$, which is of minimal size for this foundation. It is neither representable, nor orientable, nor rigid.

\subsubsection{(3,2)} \label{F32b} 
Let $\zeta_6$ be a fundamental element of $\H$ (cf.\ \autoref{Fhexagonal}). The pasture 
\[
 \FF{F32b} \ = \ \pastgenn{\H}{x,y}{x-\zeta_6-1, \quad y-\zeta_6^{-1}-1}
\]
is the foundation of the $3$-connected rank $4$ matroid on $8$ elements with short circuits
\begin{align*}
 &
 1234, \quad 
 1256, \quad 
 1278, \quad 
 1357, \quad 
 1368, \quad 
 2358, \quad 
 2457, \quad 
 3456, \quad 
 4678,
\end{align*}
which is minimal for this foundation. It maps to its universal ring $\Z[\zeta_6,1/3]$ via $x\mapsto\zeta_6+1$ and $y\mapsto\zeta_6^{-1}+1$ (since $(\zeta_6+1)(\zeta_6^{-1}+1)=3$, the integer $3$ is invertible). Note that $\zeta_6\cdot(\zeta_6^{-1}+1)=\zeta_6+1$, which shows that this map is not injective and that its universal partial field is $\FF{FDowling}=\pastgenn{\H}{x}{x-\zeta_6-1}$, the \emph{Dowling lift of $\F_4$} (cf.\ \autoref{FDowling}). The foundation $\FF{F32b}$ is representable (over all fields that contain a primitive third root of unity), but is neither orientable nor rigid.

\subsubsection{(3,2)} \label{F32c} \label{FnonGersonides}
The \emph{non-Gersonides foundation}\footnote{The Gersonides partial field $\G\E$ is named after Gersonides, who determined all solutions to $x+y=1$ in powers of $-1$, $2$ and $3$ as $2-1=1$,\ \ $3-2=1$,\ \ $4-3=1$ and $9-8=1$. Since the last relation is missing in the pasture $\FF{F32c}$, we call it the non-Gersonides foundation.} 
\[
 \FF{F32c} \ = \ \pastgenn{\Funpm}{x,y}{x-1-1,  \quad y-x-1,  \quad x^2-y-1}
\]
is the foundation of the $3$-connected rank $4$ matroid on $8$ elements with short circuits
\begin{align*}
 &
 1234, \quad 
 1256, \quad 
 1278, \quad 
 1357, \quad 
 1368, \quad 
 1458, \quad 
 2367, \quad 
 2457, \quad 
 3456, \quad 
 4678,
\end{align*}
which is minimal for this foundation. It maps injectively into its universal ring $\Z[1/6]$ via $x\mapsto2$ and $y\mapsto3$, and its universal partial field is the \emph{Gersonides partial field} $\G\E=\pastgenn{\Funpm}{x,y}{x-1-1,\ y-x-1,\ x^2-y-1,\ y^2-x^3-1}$ (cf.\ \cite[p.~543]{Pendavingh-vanZwam10a}), which acquires the additional relation $y^2-x^3-1=0$. This means that the quotient map $\FF{F32c}\to\G\E$ is a bijection, but not an isomorphism. The non-Gersonides foundation is representable (over all fields of characteristic $\neq2,3$) and orientable, but not rigid.

\subsubsection{(3,3)} \label{F2cyclotomic}
The \emph{$2$-cyclotomic partial field} 
\[
 \K_2 \ = \ \pastgenn{\Funpm}{x,y,z}{y-x-1, \quad z-y-1, \quad y^2+xz-1}
\]
(cf.\ \cite[p.~542]{Pendavingh-vanZwam10a}), which embeds into $\Z[x^{\pm1},\frac1{x+1},\frac1{x+2}]$ via $x\mapsto x$, \ \ $y\mapsto x+1$ \ \ and \ \ $z\mapsto x+2$, appears as the foundation the rank $3$ matroid on $7$ elements that is depicted as
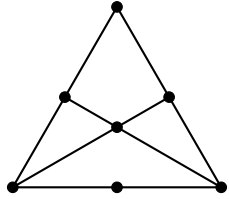
\[
 \beginpgfgraphicnamed{tikz/fig48}
 \begin{tikzpicture}[x=0.4cm,y=0.4cm]
  \filldraw (  0:0) circle (2pt);  
  \filldraw ( 30:2) circle (2pt);  
  \filldraw (150:2) circle (2pt);  
  \filldraw (270:2) circle (2pt);  
  \filldraw ( 90:4) circle (2pt);  
  \filldraw (210:4) circle (2pt);  
  \filldraw (330:4) circle (2pt);  
  \draw [thick] ( 90:4) -- (210:4);
  \draw [thick] (210:4) -- (330:4);
  \draw [thick] (90:4) -- (330:4);
  \draw [thick] (210:4) -- (30:2);
  \draw [thick] (330:4) -- (150:2);
 \end{tikzpicture}
 \endpgfgraphicnamed
\]
This matroid is of minimal size for $\K_2$ and $3$-connected, and its dual is the only other matroid on $7$ elements with foundation $\K_2$. The $2$-cyclotomic partial field is representable (in all fields with at least $4$ elements) and orientable, but not rigid. It plays a role in \cite[Thm.~4.17]{Pendavingh-vanZwam10a}: a matroid is quarternary and representable over the Gaussian partial field $\H_2$ (cf.\ \autoref{FGaussian}) if and only if it is $\K_2$-representable.




\subsubsection{(4,2)} \label{F42a} 
The pasture 
\[
 \FF{F42a} \ = \ \pastgenn{\Funpm}{x,y}{x+1+1, \quad y+1-1, \quad xy+1-1, \quad x^2y+1-1}
\]
is the foundation of the $3$-connected rank $4$ matroid on $8$ elements with short circuits
\begin{align*}
 &
 1234, \quad 
 1256, \quad 
 1278, \quad 
 1357, \quad 
 1468, \quad 
 2358, \quad 
 2467, \quad 
 3456, \quad 
 4578,  
\end{align*}
which is minimal for this foundation. This foundation is neither representable nor rigid, but is orientable.


\subsubsection{(4,3)} \label{FHydra3} 
The \emph{Hydra $3$ partial field}
\[
 \H_3 \ = \ \pastgenn{\Funpm}{x,y,z}{x+y-1, \quad xy+z-1, \quad x+y^2-z, \quad x^2+y-z}
\]
(cf.~\cite[p.~543]{Pendavingh-vanZwam10a}) is the foundation of the $3$-connected rank $3$ matroid on $8$ elements, depicted as 
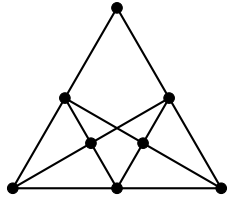
\[
 \beginpgfgraphicnamed{tikz/fig45}
 \begin{tikzpicture}[x=0.4cm,y=0.4cm]
  \filldraw (210:1) circle (2pt);  
  \filldraw (330:1) circle (2pt);  
  \filldraw ( 30:2) circle (2pt);  
  \filldraw (150:2) circle (2pt);  
  \filldraw (270:2) circle (2pt);  
  \filldraw ( 90:4) circle (2pt);  
  \filldraw (210:4) circle (2pt);  
  \filldraw (330:4) circle (2pt);  
  \draw [thick] ( 90:4) -- (210:4);
  \draw [thick] (210:4) -- (330:4);
  \draw [thick] ( 90:4) -- (330:4);
  \draw [thick] (210:4) -- ( 30:2);
  \draw [thick] (150:2) -- (330:4);
  \draw [thick] (270:2) -- ( 30:2);
  \draw [thick] (150:2) -- (270:2);
 \end{tikzpicture}
 \endpgfgraphicnamed
\]
which is minimal for this foundation. The foundation $\H_3$ is representable (in fields with at least $5$ elements) and orientable, but not rigid.

\subsubsection{(5,3)} \label{F53a} 
Let $\zeta_6$ be a fundamental element of $\H$ (cf.~\autoref{Fhexagonal}). The pasture
\[\textstyle
 \FF{F53a} \ = \ \pastgenn{\H}{x,y,z}{z-x+1, \quad x-y+\zeta_6^2, \quad y-z+\zeta_6^4, \quad x+\zeta_6^2y+\zeta_6^4z}
\]
is the foundation of the $3$-connected rank $4$ matroid on $8$ elements with short circuits
\begin{align*}
 &
 1234, \quad 
 1256, \quad 
 1357, \quad 
 1678, \quad 
 2378, \quad 
 2467, \quad 
 3456, \quad 
 4578,
\end{align*}
which is minimal for this foundation. The foundation $\FF{F53a}$ embeds into its universal ring $R_M=\Z[\zeta_6,\ x^{\pm1},\ (x-1)^{-1},\ (x+\zeta_6^2)^{-1}]$ via \ \ $\zeta_6\mapsto\zeta_6$, \ \ $x\mapsto x$, \ \ $y\mapsto x+\zeta_6^2$ \ \ and \ \ $z\mapsto x-1$. \ \ Since all solutions to the $S$-unit equation $a+b-1=0$ for $R_M$ and $S=\{x,\ x-1,\ x+\zeta_6^2\}$ are in the null set of \FF{F53a}, the foundation $\FF{F53a}$ is a partial field and is equal to the universal partial field of $M$. The foundation $\FF{F53a}$ is representable (in every field with a root of $T^2-T+1$ and at least $4$ elements), but is neither orientable nor rigid.



    
\subsubsection{(5,3)} \label{F53b} 
The pasture 
\[
 \FF{F53b} \ = \ \pastgenn{\Funpm}{x,y,z}{x-1-1, \ \ y-1-1, \ \ z+1-1, \ \ xz+1-1, \ \ yz+1-1}
\]
is the foundation of the $3$-connected rank $4$ matroid on $8$ elements with short circuits
\begin{align*}
 &
 1234, \qquad 
 1256, \qquad 
 1278, \qquad 
 1357, \qquad 
 1368, \qquad 
 1458, 
 \\ &
 2367, \qquad 
 2457, \qquad 
 2468, \qquad 
 3456, \qquad 
 3478,  
\end{align*}
which is minimal for this foundation. This foundation is neither representable nor rigid, but is orientable.


\subsubsection{(5,3)} \label{F53c} 
The pasture 
\[
 \FF{F53c} \ = \ \pastgenn{\F_2}{x,y,z}{x+y+1,\ \ x^2+y^2+1,\ \ z+1+1,\ \ xz+1+1,\ \ yz+1+1}
\]
is the foundation of the $3$-connected rank $4$ matroid on $8$ elements with short circuits
\begin{align*}
 &
 1234, \quad 
 1256, \quad 
 1357, \quad 
 1468, \quad 
 2358, \quad 
 2467, \quad 
 3456, \quad 
 3478, \quad 
 5678, \quad 
\end{align*}
which is minimal for this foundation. Another minimal ($3$-connected) matroid for this foundation is the rank $4$ matroid on $8$ elements with short circuits
\[
  1234, \quad 
  1256, \quad 
  1357, \quad 
  1468, \quad 
  2368, \quad 
  2457, \quad 
  3456, \quad 
  5678. 
\]
Interestingly, this matroid has no minors of type $F_7$ and $F_7^\ast$, even though $-1=1$ holds in $\FF{F53c}$. This foundation is neither representable, nor orientable, nor rigid.


\subsubsection{(5,4)} \label{F54a} 

The partial field 
\[
 \P_4 \ = \ \pastgenn{\Funpm}{x,y,z,w}{x+y-1, \ \ z-x-1, \ \ w-y-1, \ \ x^2+yz-1, \ \ y^2+xw-1}
\]
(cf.\ \cite[p.~543]{Pendavingh-vanZwam10a}) is the foundation of the $3$-connected matroid $M_{8591}^{Y\Delta}$ (cf.\ \cite[p.\ 91]{vanZwam09}). It is representable (over every field with at least $4$ elements) and orientable, but not rigid.



\subsubsection{(6,3)} \label{F63a} 
The pasture $\FF{F63a}=\pastgenn{\Funpm}{x,y,z}{S}$, where $S$ consists of
\[
 x-1-1,  \quad y+1-1,  \quad xy+1-1,  \quad yz+1-1,  \quad z-x-1,  \quad x^2-z-1, 
\]
is the foundation of the $3$-connected rank $4$ matroid on $8$ elements with short circuits
\begin{align*}
 &
 1234, \quad 
 1256, \quad 
 1357, \quad 
 1368, \quad 
 1478, \quad 
 2358, \quad 
 2457, \quad 
 3456, \quad 
 5678, \quad 
\end{align*}
which is minimal for this foundation. It is neither representable, nor rigid, nor orientable.

   
\subsubsection{(6,3)} \label{F63b} 
The pasture $\FF{F63b}=\pastgenn{\Funpm}{x,y,z}{S}$, where $S$ consists of
\[
 x-1-1,  \quad y-x-1,  \quad z-x-1,  \quad x^2-y-1,  \quad x^2-z-1,  \quad yz-x^3-1,
\]
is the foundation of the $3$-connected rank $4$ matroid on $8$ elements with short circuits
\begin{align*}
 &
 1234, \quad 
 1256, \quad 
 1357, \quad 
 1458, \quad 
 1678, \quad 
 2358, \quad 
 2367, \quad 
 2457, \quad 
 3456, \quad 
\end{align*}
which is minimal for this foundation. The foundation $\FF{F63b}$ maps (non-injectively) to its universal ring $\Z[\frac16]$ via \ \ $x\mapsto2$, \ \ $y\mapsto3$ \ \ and \ \ $z\mapsto3$. \ \ Its universal partial field is the Gersonides partial field $\G\E=\pastgenn{\Funpm}{x,y}{x-1-1,\ y-x-1,\ x^2-y-1,\ y^2-x^3-1}$ (cf.\ \autoref{FnonGersonides}). The foundation $\FF{F63b}$ is representable (in characteristic $\neq2,3$) and orientable, but not rigid.

    
\subsubsection{(6,3)} \label{F63c} 
The pasture $\FF{F63c}=\pastgenn{\F_2}{x,y,z}{S}$, where $S$ consists of 
\[
 x+1+1,  \quad y+1+1,  \quad z+1+1,  \quad xy+1+1,  \quad xz+1+1,  \quad yz+1+1,
\]
is the foundation of the $3$-connected rank $4$ matroid on $8$ elements with short circuits
\begin{align*}
 &
 1234, \qquad 
 1256, \qquad 
 1368, \qquad 
 1357, \qquad 
 1458, \qquad 
 2358, \qquad 
 \\ &
 2367, \qquad 
 2457, \qquad 
 3456, \qquad 
 3478, \qquad 
 5678, \qquad 
\end{align*}
which is minimal for this foundation. It is neither representable, nor orientable, nor rigid.


\subsubsection{(6,4)} \label{F64a} 
The pasture $\FF{F64a}=\pastgenn{\Funpm}{x,y,z,w}{S}$, where $S$ consists of
\[
 w-1-1, \quad \   
 y-x-1, \quad \   
 z-y-1, \quad \ 
 y^2-xz-1, \quad \ 
 \tfrac{wy}x-\tfrac zx-1, \quad \ 
 \tfrac xz+\tfrac wz-1, 
\]
 is the foundation of the $3$-connected rank $3$ matroid on $8$ elements depicted as
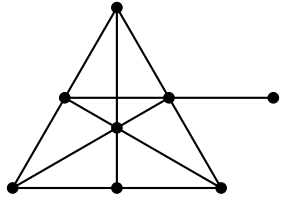
\[
 \beginpgfgraphicnamed{tikz/fig46}
 \begin{tikzpicture}[x=0.4cm,y=0.4cm]
  \filldraw (  0:0) circle (2pt);  
  \filldraw (5.2,1) circle (2pt);  
  \filldraw ( 30:2) circle (2pt);  
  \filldraw (150:2) circle (2pt);  
  \filldraw (270:2) circle (2pt);  
  \filldraw ( 90:4) circle (2pt);  
  \filldraw (210:4) circle (2pt);  
  \filldraw (330:4) circle (2pt);  
  \draw [thick] ( 90:4) -- (210:4);
  \draw [thick] (210:4) -- (330:4);
  \draw [thick] ( 90:4) -- (330:4);
  \draw [thick] (210:4) -- ( 30:2);
  \draw [thick] (150:2) -- (330:4);
  \draw [thick] (150:2) -- (5.2,1);
  \draw [thick] (270:2) -- ( 90:4);
 \end{tikzpicture}
 \endpgfgraphicnamed
\]
which is minimal for this foundation. It embeds into its universal ring $\Z[\frac12,\ x^{\pm1},\ \frac1{x+1},\ \frac1{x+2}]$ via \ \ $w\mapsto2$, \ \ $x\mapsto x$, \ \  $y\mapsto x+1$, \ \ and \ \ $z\mapsto x+2$, \ \ but is not equal to its universal partial field, which is isomorphic to the universal partial field of \FF{F74a} (cf.\ \autoref{F74a}). The foundation $\FF{F64a}$ is representable (in all fields of characteristic $\neq2$ with at least $5$ elements) and orientable, but not rigid.

\subsubsection{(7,4)} \label{F74a} 
The pasture $\FF{F74a}=\pastgenn{\Funpm}{x,y,z,w}{S}$, where $S$ consists of
\[
 \begin{array}{lllll}
  & w-1-1, \quad \ \ 
  & z-y-1, \quad \ \ 
  & \frac{wy}x-\frac zx-1, \quad \ \ 
  & \frac{x^2}{z^2}+\frac{w^2y}{z^2}-1, \quad \ \ 
  \\ 
  & y-x-1, \quad \ \ 
  & y^2-xz-1, \quad \ \ 
  & \frac xz+\frac wz-1, \quad \ \ 
 \end{array}
\]
is the foundation of the $3$-connected rank $4$ matroid on $8$ elements with short circuits
\begin{align*}
 &
 1234, \quad \ \ 
 1256, \quad \ \ 
 1357, \quad \ \ 
 1478, \quad \ \ 
 2368, \quad \ \ 
 2457, \quad \ \ 
 3456, \quad \ \ 
 5678, 
\end{align*}
which is minimal for this foundation. It embeds into its universal ring $\Z[\frac12,\ x^{\pm1},\ \frac1{x+1},\ \frac1{x+2}]$ via \ \ $w\mapsto2$, \ \ $x\mapsto x$, \ \  $y\mapsto x+1$, \ \ and \ \ $z\mapsto x+2$, \ \ and is equal to its universal partial field (verified with the method of \cite{Verschoore24}). The foundation $\FF{F74a}$ is representable (over every field with at least $5$ elements) and orientable, but not rigid.

\subsubsection{(7,4)} \label{F74b} 
The pasture $\FF{F74b}=\pastgenn{\Funpm}{x,y,z,w}{S}$, where $S$ consists of
\[
 \begin{array}{lllll}
  & x-y-1, \quad \ \ 
  & \frac{x^2}w+\frac yw-1, \quad \ \ 
  & \frac zw+\frac{xy}w-1, \quad \ \ 
  & \frac{x^3}{yw}-\frac z{yw}-1,
  \\[5pt] 
  & \frac zx-\frac yx-1, \quad \ \ 
  & \frac z{x^2}+\frac{y^2}{x^2}-1, \quad \ \ 
  & \frac{xz}w-\frac{y^2}w-1, 
 \end{array}
\]
is the foundation of the $3$-connected rank $4$ matroid on $8$ elements with short circuits
\begin{align*}
 &
 1234, \quad \ \ 
 1256, \quad \ \ 
 1357, \quad \ \ 
 1468, \quad \ \ 
 2358, \quad \ \ 
 2457, \quad \ \ 
 2678, \quad \ \ 
 3456, 
\end{align*}
which is minimal for this foundation. The foundation $\FF{F74b}$ embeds into its universal ring $\Z[x^{\pm1},\ (x-1)^{-1},\ (2x-1)^{-1},\ (x^2+x-1)^{-1}]$ via \ \ $x\mapsto x$, \ \ $y\mapsto x-1$, \ \ $z\mapsto 2x-1$ \ \ and \ \ $w\mapsto x^2+x-1$. \ \  It is not equal to its universal partial field
\[
 \Pi(\FF{F74b}) \ = \ \past{\FF{F74b}}{\genn{\tfrac{xw}{y^3}-\tfrac{z^2}{y^3}-1}},
\]
whose null set features an additional relation (verified with the method from \cite{Verschoore24}). The foundation $\FF{F74b}$ is representable (over every field with at least $5$ elements) and orientable, but not rigid.

%

\subsubsection{(9,7)} \label{FHydra4} 
The \emph{Hydra $4$ partial field} $\H_4=\pastgenn{\Funpm}{x,y,z,s,t,w}{S}$, where $S$ consists of
\[
 \begin{array}{lllllll}
  & x+s-1,
  & \frac sz+\frac{xt}z-1,
  & \frac{xt}w+\frac{ys}w-1,
  \\[4pt]
  & y+t-1,
  & \frac tz+\frac{ys}z-1,
  & \frac w{yz}-\frac {xt^2}{yz}-1,
  \\[4pt]
  & xy+z-1,
  \qquad
  & \frac wz+\frac{st}z-1,
  \qquad
  & \frac w{xz}-\frac{ys^2}{xz}-1,
 \end{array}
\]
(cf.\ \cite[p.~543]{Pendavingh-vanZwam10a}) is the foundation of the $3$-connected rank $3$ matroid on $7$ elements depicted as
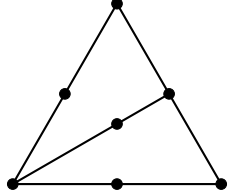
\[
 \beginpgfgraphicnamed{tikz/fig43}
 \begin{tikzpicture}[x=0.4cm,y=0.4cm]
  \filldraw (  0:0) circle (2pt);  
  \filldraw ( 30:2) circle (2pt);  
  \filldraw (150:2) circle (2pt);  
  \filldraw (270:2) circle (2pt);  
  \filldraw ( 90:4) circle (2pt);  
  \filldraw (210:4) circle (2pt);  
  \filldraw (330:4) circle (2pt);  
  \draw [thick] ( 90:4) -- (210:4);
  \draw [thick] (210:4) -- (330:4);
  \draw [thick] (90:4) -- (330:4);
  \draw [thick] (210:4) -- (30:2);
 \end{tikzpicture}
 \endpgfgraphicnamed
\]
which is minimal for this foundation. The Hydra 4 partial field is representable (over every field with at least $5$ elements) and orientable, but not rigid.

\subsubsection{(11,7)}\label{FPappus}
The \emph{Pappus foundation} $\FF{FPappus}=\pastgenn{\Funpm}{x,\ y,\ z,\ s,\ t,\ v,\ w}{S}$, where $S$ consists of 
\begin{align*}
    & s+x-1,
   && \tfrac wx-\tfrac{st}x-1,
   && \tfrac xz-\tfrac yz-1,
 \\[2pt]
    & t+y-1, 
   && \tfrac vy-\tfrac{st}y-1,
   && \tfrac tz-\tfrac sz-1,
 \\[2pt] 
    & v+xt-1,
   && \tfrac v{xy}-\tfrac s{xy}-1,
   && \tfrac{xt}z-\tfrac{ys}z-1,
 \\[2pt]
    & w+ys-1,
   && \tfrac w{xy}-\tfrac t{xy}-1,
\end{align*}
is the foundation of the Pappus matroid, depicted as
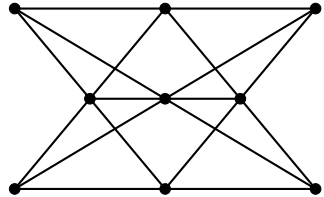
\[
 \beginpgfgraphicnamed{tikz/fig54}
   \begin{tikzpicture}[x=1cm, y=1.2cm]
    \filldraw (0,0) circle (2pt);
    \filldraw (2,0) circle (2pt);
    \filldraw (4,0) circle (2pt);
    \filldraw (1,1) circle (2pt);
    \filldraw (2,1) circle (2pt);
    \filldraw (3,1) circle (2pt);
    \filldraw (0,2) circle (2pt);
    \filldraw (2,2) circle (2pt);
    \filldraw (4,2) circle (2pt);
    \draw[thick] (0,0) -- (4,0) -- (0,2) -- (4,2) -- (0,0);
    \draw[thick] (0,0) -- (2,2) -- (4,0);
    \draw[thick] (0,2) -- (2,0) -- (4,2);
    \draw[thick] (1,1) -- (3,1);
   \end{tikzpicture}
 \endpgfgraphicnamed
\]
which is a minimal matroid for this foundation. The Pappus foundation embeds into its universal ring 
\[
 \Z\big[x^{\pm1},\ y^{\pm1},\ (1-x)^{-1},\ (1-y)^{-1},\ (x-y)^{-1},\ (1-x+xy)^{-1},\ (1-y+xy)^{-1}\big]
\]
via 
\begin{align*}
   x &\mapsto x, \qquad 
 & s &\mapsto 1-x, \qquad 
 & v &\mapsto 1-x+xy, \qquad 
 & z &\mapsto x-y,
 \\
   y &\mapsto y, \qquad 
 & t &\mapsto 1-y, \qquad 
 & w &\mapsto 1-y+xy.
\end{align*}
The Pappus foundation differs from its universal partial field 
\[
 \Pi(\FF{FPappus}) \ = \ \past{\FF{FPappus}}{\genn{\tfrac wz-\tfrac vz-1, \quad \tfrac{xtw}z-\tfrac{ysv}z-1, \quad \tfrac{tv}{xyz}-\tfrac{sw}{xyz}-1, \quad \tfrac{xv}{stz}-\tfrac{yw}{stz}-1}}
\]
(verified with the method from \cite{Verschoore24}). It is representable (over $\F_4$ and all fields with at least $7$ elements) and orientable, but not rigid.

\subsection{Morphisms of foundations into other pastures}

A matroid $M$ is representable over a pasture $P$ if and only if there is a morphism from the foundation $F_M$ of $M$ into $P$. This allows us to study matroid representations in terms of algebraic properties, a technique that has been applied successfully in the previous papers \cite{Baker-Lorscheid20} and \cite{Baker-Lorscheid21}. 

In \autoref{table: morphisms of foundations into other pastures}, we display for a range of foundations $F$ (in the leftmost column) and pastures $P$ (in the top row) whether there exists a morphism from $F$ to $P$. The second column (``repr.'') indicates if the foundation is representable over any field, the third column (``PF'') indicates if the foundation is a partial field, and the last column (``rigid'') indicates if the foundation is rigid.

\begin{table}[ptb]
 \caption{Morphisms from foundations into other pastures}
 \label{table: morphisms of foundations into other pastures}
 \begin{tabular}{|c|c|c|c|c|c|c|c|c|c|c|c|c|c|c|c|}
  \hline
  found.     & repr.  & PF     & $\F_2$ & $\F_3$ & $\F_4$ & $\F_5$ & $\F_7$ & $\F_8$ & $\F_9$ & $\G$   & $\R$   & $\C$   & $\S$   & $\P$   & rigid  
  \\ \hline \hline
  $\Funpm$   & \check & \check & \check & \check & \check & \check & \check & \check & \check & \check & \check & \check & \check & \check & \check 
  \\ \hline
  $\U$       & \check & \check & $-$    & \check & \check & \check & \check & \check & \check & \check & \check & \check & \check & \check & $-$    
  \\ \hline
  $\V$       & \check & \check & $-$    & $-$    & \check & \check & \check & \check & \check & \check & \check & \check & \check & \check & $-$    
  \\ \hline
  $\F_2$     & \check & \check & \check & $-$    & \check & $-$    & $-$    & \check & $-$    & $-$    & $-$    & $-$    & $-$    & $-$    & \check 
  \\ \hline
  $\F_3$     & \check & \check & $-$    & \check & $-$    & $-$    & $-$    & $-$    & \check & $-$    & $-$    & $-$    & $-$    & $-$    & \check 
  \\ \hline
  $\K$       & \norep & $-$    & $-$    & $-$    & $-$    & $-$    & $-$    & $-$    & $-$    & $-$    & $-$    & $-$    & $-$    & $-$    & \check 
  \\ \hline
  $\S$       & \norep & $-$    & $-$    & $-$    & $-$    & $-$    & $-$    & $-$    & $-$    & $-$    & $-$    & $-$    & \check & \check & \check 
  \\ \hline
  $\H$       & \check & \check & $-$    & \check & \check & $-$    & \check & $-$    & \check & $-$    & $-$    & \check & $-$    & \check & \check 
  \\ \hline
  $\D$       & \check & \check & $-$    & \check & $-$    & \check & \check & $-$    & \check & $-$    & \check & \check & \check & \check & $-$    
  \\ \hline
  $\G$       & \check & \check & $-$    & $-$    & \check & \check & $-$    & $-$    & \check & \check & \check & \check & \check & \check & \check 
  \\ \hline
  \FF{F20a}  & \norep & $-$    & $-$    & $-$    & $-$    & $-$    & $-$    & $-$    & $-$    & $-$    & $-$    & $-$    & $-$    & $-$    & \check 
  \\ \hline
  \FF{F21a}  & \check & \check & $-$    & $-$    & $-$    & \check & \check & \check & $-$    & $-$    & \check & \check & \check & \check & \check 
  \\ \hline
  $\H_2$     & \check & \check & $-$    & $-$    & $-$    & \check & $-$    & $-$    & \check & $-$    & $-$    & \check & $-$    & \check & $-$
  \\ \hline
\FF{FDowling}& \check & \check & $-$    & $-$    & \check & $-$    & \check & $-$    & $-$    & $-$    & $-$    & \check & $-$    & \check & $-$    
  \\ \hline
  \FF{F21d}  & \check & $-$    & $-$    & $-$    & \check & $-$    & $-$    & $-$    & $-$    & $-$    & $-$    & $-$    & $-$    & $-$    & $-$    
  \\ \hline
  \FF{F21e}  & \norep & $-$    & $-$    & $-$    & $-$    & $-$    & $-$    & $-$    & $-$    & $-$    & $-$    & $-$    & $-$    & $-$    & $-$    
  \\ \hline
  \FF{F21f}  & \norep & $-$    & $-$    & $-$    & $-$    & $-$    & $-$    & $-$    & $-$    & $-$    & $-$    & $-$    & \check & $-$    & $-$    
  \\ \hline
  \FF{F21g}  & \norep & $-$    & $-$    & $-$    & $-$    & $-$    & $-$    & $-$    & $-$    & $-$    & $-$    & $-$    & $-$    & $-$    & $-$    
  \\ \hline
  \FF{F21h}  & \norep & $-$    & $-$    & $-$    & $-$    & $-$    & $-$    & $-$    & $-$    & $-$    & $-$    & $-$    & $-$    & $-$    & $-$    
  \\ \hline
  \FF{F22a}  & \check & $-$    & $-$    & \check & $-$    & $-$    & $-$    & $-$    & $-$    & $-$    & $-$    & $-$    & \check & \check & $-$    
  \\ \hline
  \FF{F32a}  & \norep & $-$    & $-$    & $-$    & $-$    & $-$    & $-$    & $-$    & $-$    & $-$    & $-$    & $-$    & $-$    & $-$    & $-$    
  \\ \hline
  \FF{F32b}  & \check & $-$    & $-$    & $-$    & $-$    & \check & $-$    & $-$    & $-$    & $-$    & $-$    & \check & $-$    & \check & $-$    
  \\ \hline
  \FF{F32c}  & \check & $-$    & $-$    & $-$    & \check & \check & $-$    & $-$    & $-$    & $-$    & \check & \check & \check & \check & $-$    
  \\ \hline
  $\K_2$     & \check & \check & $-$    & \check & \check & \check & \check & \check & $-$    & $-$    & \check & \check & \check & \check & $-$    
  \\ \hline
  \FF{F42a}  & \norep & $-$    & $-$    & $-$    & $-$    & $-$    & $-$    & $-$    & $-$    & $-$    & $-$    & $-$    & \check & \check & $-$    
  \\ \hline
  $\H_3$     & \check & \check & $-$    & $-$    & $-$    & \check & \check & \check & \check & \check & \check & \check & \check & \check & $-$    
  \\ \hline
  \FF{F53a}  & \check & \check & $-$    & $-$    & \check & \check & \check & \check & \check & \check & \check & \check & $-$    & \check & $-$    
  \\ \hline
  \FF{F53b}  & \norep & $-$    & $-$    & $-$    & $-$    & $-$    & $-$    & $-$    & $-$    & $-$    & $-$    & $-$    & \check & \check & $-$    
  \\ \hline
  \FF{F53c}  & \norep & $-$    & $-$    & $-$    & $-$    & $-$    & $-$    & $-$    & $-$    & $-$    & $-$    & $-$    & $-$    & $-$    & $-$    
  \\ \hline
  $\P_4$     & \check & \check & $-$    & $-$    & \check & \check & \check & \check & \check & \check & \check & \check & \check & \check & $-$    
  \\ \hline
  \FF{F63a}  & \norep & $-$    & $-$    & $-$    & $-$    & $-$    & $-$    & $-$    & $-$    & $-$    & $-$    & $-$    & \check & \check & $-$    
  \\ \hline
  \FF{F63b}  & \check & $-$    & $-$    & $-$    & $-$    & \check & \check & $-$    & $-$    & \check & \check & \check & \check & \check & $-$    
  \\ \hline
  \FF{F63c}  & \norep & $-$    & $-$    & $-$    & $-$    & $-$    & $-$    & $-$    & $-$    & $-$    & $-$    & $-$    & $-$    & $-$    & $-$    
  \\ \hline
  \FF{F64a}  & \check & $-$    & $-$    & $-$    & $-$    & \check & \check & $-$    & \check & \check & \check & \check & \check & \check & $-$    
  \\ \hline
  \FF{F74a}  & \check & \check & $-$    & $-$    & $-$    & \check & \check & $-$    & \check & \check & \check & \check & \check & \check & $-$    
  \\ \hline
  \FF{F74b}  & \check & $-$    & $-$    & $-$    & $-$    & \check & \check & \check & \check & \check & \check & \check & \check & \check & $-$    
  \\ \hline
  $\H_4$     & \check & \check & $-$    & $-$    & $-$    & \check & \check & \check & \check & \check & \check & \check & \check & \check & $-$    
  \\ \hline
 \FF{FPappus}& \check & $-$    & $-$    & $-$    & \check & $-$    & \check & \check & \check & \check & \check & \check & \check & \check & $-$    
  \\ \hline
 \end{tabular}
\end{table}

\subsection{Numerical description of foundations of \texorpdfstring{$3$}{3}-connected matroids on \texorpdfstring{$7$}{7} and \texorpdfstring{$8$}{8} elements}
\label{subsection: foundations of matroids on 7 elements}

The pastures we've encountered already which appear as foundations of 3-connected matroids on 7 elements are $\F_2$, \ $\D$, \ $\K_2$, \ $\H_4$, \ $\U^2_7$, \ and \ $\U^3_7$. There are $4$ additional pastures that appear as foundation of $3$-connected matroids on $7$ elements; they occur for the following rank $3$ matroids:
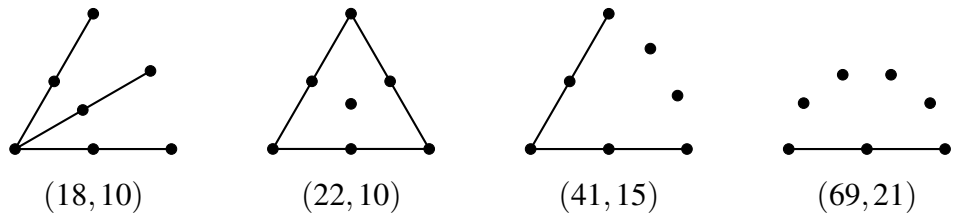
\[
  \begin{array}{cccc}
   \quad 
   \beginpgfgraphicnamed{tikz/fig50}
   \begin{tikzpicture}[x=0.3cm,y=0.3cm]
    \filldraw (  0:0) circle (2pt);  
    \filldraw (  0:3.464) circle (2pt);  
    \filldraw (  0:6.928) circle (2pt);  
    \filldraw ( 60:3.464) circle (2pt);  
    \filldraw ( 60:6.928) circle (2pt);  
    \filldraw ( 30:3.464) circle (2pt);  
    \filldraw ( 30:6.928) circle (2pt);  
    \draw [thick] (  0:0) -- (  0:6.928);
    \draw [thick] (  0:0) -- ( 30:6.928);
    \draw [thick] (  0:0) -- ( 60:6.928);
   \end{tikzpicture}
   \endpgfgraphicnamed
   \quad & \quad
   \beginpgfgraphicnamed{tikz/fig51}
   \begin{tikzpicture}[x=0.3cm,y=0.3cm]
    \filldraw (  0:0) circle (2pt);  
    \filldraw ( 30:2) circle (2pt);  
    \filldraw (150:2) circle (2pt);  
    \filldraw (270:2) circle (2pt);  
    \filldraw ( 90:4) circle (2pt);  
    \filldraw (210:4) circle (2pt);  
    \filldraw (330:4) circle (2pt);  
    \draw [thick] ( 90:4) -- (210:4);
    \draw [thick] (210:4) -- (330:4);
    \draw [thick] (90:4) -- (330:4);
   \end{tikzpicture}
   \endpgfgraphicnamed
   \quad & \quad
   \beginpgfgraphicnamed{tikz/fig52}
   \begin{tikzpicture}[x=0.3cm,y=0.3cm]
    \filldraw (  0:0) circle (2pt);  
    \filldraw (  0:3.464) circle (2pt);  
    \filldraw (  0:6.928) circle (2pt);  
    \filldraw ( 60:3.464) circle (2pt);  
    \filldraw ( 60:6.928) circle (2pt);  
    \filldraw ( 20:6.928) circle (2pt);  
    \filldraw ( 40:6.928) circle (2pt);  
    \draw [thick] (  0:0) -- (  0:6.928);
    \draw [thick] (  0:0) -- ( 60:6.928);
   \end{tikzpicture}
   \endpgfgraphicnamed
   \quad & \quad
   \beginpgfgraphicnamed{tikz/fig53}
   \begin{tikzpicture}[x=0.3cm,y=0.3cm]
    \filldraw (  0:0) circle (2pt);  
    \filldraw (  0:3.464) circle (2pt);  
    \filldraw ( 36:3.464) circle (2pt);  
    \filldraw ( 72:3.464) circle (2pt);  
    \filldraw (108:3.464) circle (2pt);  
    \filldraw (144:3.464) circle (2pt);  
    \filldraw (180:3.464) circle (2pt);  
    \draw [thick] (  0:3.464) -- (180:3.464);
   \end{tikzpicture}
   \endpgfgraphicnamed
   \quad \\[5pt]
   (18,10) & (22,10)& (41,15) & (69,21)
  \end{array}
\]
These matroids are labeled by the numerical type $(h,r)$ of their foundations (where $h$ is the number of hexagons and $r$ is the rank), which are too large for a meaningful explicit description. In each case, the torsion is generated by $-1$ and $-1\neq 1$. 

An exhaustive search on $3$-connected matroids on $8$ elements reveals (at least) $196$ different isomorphism classes of foundations. For up to $20$ hexagons, we verified that there are precisely $66$ isomorphism types, but beyond $20$ hexagons, the available computational power limited us to a search for numerical types: there are $130$ different ones, and these might very well split into different isomorphism classes. This means that there might be more than $196$ isomorphism classes.

\subsection{Large non-representable foundations}

According to Nelson's theorem (\cite{Nelson18}), most matroids are not representable. Accordingly, we expect that most foundations are not representable. We have seen already a significant number of small non-representable foundations in \autoref{subsection: some small foundations}: $\K$, $\S$, \FF{F20a}, \FF{F21e}, \FF{F21f}, \FF{F21g}, \FF{F21h}, \FF{F32a}, \FF{F42a}, \FF{F53b}, \FF{F53c}, \FF{F63a}, and \FF{F63c}. 

The foundations of some prominent non-representable matroids are much larger than the examples in this appendix: the non-Pappus matroid has numerical type $(29,9)$; the V{\'a}mos matroid has numerical type $(76,21)$; the Desarguesian matroid has numerical type $(95,17)$.

In so far, we wonder if the analog of Nelson's theorem for foundations holds:

\begin{problem}
 Are almost all foundations non-representable? This is, let $\cF(n)$ be set of isomorphism classes of foundations of matroids on $n$ elements and $\cN(n)\subset\cF(n)$ the subset of non-representable classes. Is 
 \[
  \lim\limits_{n\to\infty} \ \ \frac{\#\cN(n)}{\#\cF(n)} \ \ = \ \  1 \ ?
 \]
\end{problem}

Many examples of non-representable foundations contain $1$ as a fundamental element, which is an obstruction for representability. A minimal size non-representable matroid for which $1$ is not a fundamental element is the matroid $R_9^A$ (see \cite{Sage} for the notation), as observed in \cite[Section 5]{Chen-Zhang}: it has foundation is $\pastgenn{\H}{\epsilon,x}{\epsilon^2-1,\ S}$, where $S$ consists of 
\[
\begin{array}{llllll}
   \epsilon + \zeta_6\epsilon x - 1, \
 & \zeta_6^2-1-1, \
 & \zeta_6^2\epsilon x-1-1, \
 & \epsilon/x+1/x-1,  \
 & -\epsilon -\epsilon x-1. \
 \\
   \epsilon - \zeta_6\epsilon x - 1, \
 & -\zeta_6^2 \epsilon -\zeta_6-1, \
 & -\zeta_6\epsilon x-1-1, \
 & \zeta_6/x+\zeta_6/x-1, \
\end{array} 
\]
and where $\zeta_6$ is a fundamental element of $\H$ (cf.\ \autoref{Fhexagonal}). The element $1$ does not appear as a fundamental element; the reason for non-representability are the relations
\[
 \epsilon + \zeta_6\epsilon x - 1, \qquad \epsilon - \zeta_6\epsilon x - 1, \qquad \zeta_6^2-1-1.
\]
Indeed, suppose we had a homomorphism $\varphi : F_{R_9^A} \to K$ from the foundation of $R_9^A$ to a field $K$. The first two displayed relations imply that $2\varphi(\zeta_6\epsilon x)=0$, and thus $2=0$ in $K$, while the third relation implies that $2=\varphi(\zeta_6^2)\neq0$ in $K$, a contradiction.

Chen and the third author of this paper find in \cite{Chen-Zhang} a criterion for non-representability that holds in all known cases. Namely, consider 
\[
 P \ = \ \pastgenn{\Funpm}{x,y,z,w}{x+y-1, \ \ x+z-1, \ \ w+y/z-1},
\]
which is a pasture that does not map into any field, but does map into every non-representable foundation known to us. The first part of the following problem is stated in \cite[Section 5]{Chen-Zhang}.

\begin{problem}
 Is there a morphism $P\to F$ for every non-representable foundation $F$? If not, can we describe an explicit (and possibly even finite?) list of pastures $P_1,P_2,\dotsc$ such that a foundation $F$ is non-representable if and only if there is a morphism $P_i\to F$ for some $i$?
\end{problem}

\begin{small}
 \bibliographystyle{plain}
 \bibliography{matroid}

\begin{thebibliography}{10}

\bibitem{Baker-Bowler19}
Matthew Baker and Nathan Bowler.
\newblock Matroids over partial hyperstructures.
\newblock {\em Adv. Math.}, 343:821--863, 2019.

\bibitem{Baker-Lorscheid20}
Matthew Baker and Oliver Lorscheid.
\newblock Foundations of matroids. {P}art 1: Matroids without large uniform
  minors.
\newblock Preprint, \arxiv{2008.00014}. To appear in Memoirs of the AMS.

\bibitem{Baker-Lorscheid23}
Matthew Baker and Oliver Lorscheid.
\newblock On a theorem of {L}afforgue.
\newblock Preprint, \arxiv{2309.01746}. To appear in IMRN.

\bibitem{Baker-Lorscheid21}
Matthew Baker and Oliver Lorscheid.
\newblock Lift theorems for representations of matroids over pastures.
\newblock Preprint, \arxiv{2107.00981}, 2021.

\bibitem{Baker-Lorscheid21b}
Matthew Baker and Oliver Lorscheid.
\newblock The moduli space of matroids.
\newblock {\em Adv. Math.}, 390:Paper No. 107883, 118, 2021.

\bibitem{Baker-Lorscheid-Walsh-Zhang}
Matthew Baker, Oliver Lorscheid, Zach Walsh, and Tianyi Zhang.
\newblock The foundation of generalized parallel connections, 2-sums, and
  generalized delta-wye exhanges of matroids.
\newblock Preprint, \arxiv{2404.10656}, 2024.

\bibitem{Bixby-Coullard87}
Robert~E. Bixby and C.~R. Coullard.
\newblock Finding a small {$3$}-connected minor maintaining a fixed minor and a
  fixed element.
\newblock {\em Combinatorica}, 7(3):231--242, 1987.

\bibitem{Bland-Jensen87}
Robert~G. Bland and David~L. Jensen.
\newblock Weakly oriented matroids.
\newblock Cornell University School of OR/IE Technical Report No. 732, 1987.

\bibitem{Chen-Zhang}
Justin Chen and Tianyi Zhang.
\newblock Representing matroids via pasture morphisms.
\newblock Preprint, \arxiv{2307.14275}, 2023.

\bibitem{Creech21}
Steven Creech.
\newblock Limits and colimits in the category of pastures.
\newblock Preprint, \arxiv{2103.08655}, 2021.

\bibitem{Cunningham73}
W.H. Cunningham.
\newblock A combinatorial decomposition theory.
\newblock PhD thesis, University of Waterloo, 1973.

\bibitem{Dress-Wenzel89}
Andreas W.~M. Dress and Walter Wenzel.
\newblock Geometric algebra for combinatorial geometries.
\newblock {\em Adv. Math.}, 77(1):1--36, 1989.

\bibitem{Dress-Wenzel90}
Andreas W.~M. Dress and Walter Wenzel.
\newblock On combinatorial and projective geometry.
\newblock {\em Geom. Dedicata}, 34(2):161--197, 1990.

\bibitem{Gelfand-Rybnikov-Stone95}
Israel~M. Gelfand, Grigori~L. Rybnikov, and David~A. Stone.
\newblock Projective orientations of matroids.
\newblock {\em Adv. Math.}, 113(1):118--150, 1995.

\bibitem{Kahn88}
Jeff Kahn.
\newblock On the uniqueness of matroid representations over {${\rm GF}(4)$}.
\newblock {\em Bull. London Math. Soc.}, 20(1):5--10, 1988.

\bibitem{Nelson18}
Peter Nelson.
\newblock Almost all matroids are nonrepresentable.
\newblock {\em Bull. Lond. Math. Soc.}, 50(2):245--248, 2018.

\bibitem{Oxley-Whittle98}
James Oxley and Geoff Whittle.
\newblock On weak maps of ternary matroids.
\newblock {\em European J. Combin.}, 19(3):377--389, 1998.

\bibitem{Oxley92}
James~G. Oxley.
\newblock {\em Matroid theory}.
\newblock Oxford Science Publications. The Clarendon Press, Oxford University
  Press, New York, 1992.

\bibitem{Pendavingh-vanZwam10a}
Rudi~A. Pendavingh and Stefan H.~M. van Zwam.
\newblock Confinement of matroid representations to subsets of partial fields.
\newblock {\em J. Combin. Theory Ser. B}, 100(6):510--545, 2010.

\bibitem{Pendavingh-vanZwam10b}
Rudi~A. Pendavingh and Stefan H.~M. van Zwam.
\newblock Lifts of matroid representations over partial fields.
\newblock {\em J. Combin. Theory Ser. B}, 100(1):36--67, 2010.

\bibitem{Sage}
Sage.
\newblock Documentation for the matroids in the catalog.
\newblock Available at
  \url{https://doc.sagemath.org/html/en/reference/matroids/sage/matroids/catalog.html}.

\bibitem{Semple97}
Charles Semple.
\newblock {$k$}-regular matroids.
\newblock In {\em Combinatorics, complexity, \& logic ({A}uckland, 1996)},
  Springer Ser. Discrete Math. Theor. Comput. Sci., pages 376--386. Springer,
  Singapore, 1997.

\bibitem{Semple98}
Charles Semple.
\newblock {$k$}-{R}egular matroids.
\newblock Thesis, Victory University of Wellington, 1998.

\bibitem{Silins24}
Peteris Silins.
\newblock Cross ratios for finite field geometries.
\newblock Bachelor thesis, Groningen, 2024.

\bibitem{Tutte58a}
William~T. Tutte.
\newblock A homotopy theorem for matroids, {I}.
\newblock {\em Trans. Amer. Math. Soc.}, 88:144--160, 1958.

\bibitem{vanZwam09}
Stefan H.~M. van Zwam.
\newblock Partial fields in matroid theory.
\newblock PhD thesis, Eindhoven, 2009. Online available at
  \url{http://www.matroidunion.org/stefan/pdf/thesis-online.pdf}.

\bibitem{Verschoore24}
Willard~A. Verschoore de~la Houssaije.
\newblock Solving the {$S$}-unit equation in function fields.
\newblock Bachelor thesis, Groningen, 2024.

\bibitem{Wagowski89}
Marc Wagowski.
\newblock Matroid signatures coordinatizable over a semiring.
\newblock {\em European J. Combin.}, 10(4):393--398, 1989.

\bibitem{Wenzel91}
Walter Wenzel.
\newblock Projective equivalence of matroids with coefficients.
\newblock {\em J. Combin. Theory Ser. A}, 57(1):15--45, 1991.

\end{thebibliography}
\end{small}

\end{document}